\documentclass[english, 10pt]{amsart}
\usepackage{euscript,graphicx,epstopdf,amscd,amsgen,amsfonts,amssymb,latexsym,amsmath,amsthm,graphicx,mathrsfs,times,color}

\usepackage{yfonts}

\usepackage[all]{xy}
\usepackage{enumitem}%,mathabx}
\usepackage{overpic, fourier}
 \usepackage[colorlinks=true,backref=page]{hyperref}
\usepackage{xcolor}
\definecolor{forestgreen}{rgb}{0.0, 0.27, 0.13}

\usepackage[T1]{fontenc}
\renewcommand{\thefigure}{%
  \ifnum\value{subsection}>0
    \thesubsection.%
  \else
    \thesection.%
  \fi
  \arabic{figure}%
}

\usepackage{graphicx}
\renewcommand\thefigure{\arabic{figure}}

\usepackage{lineno}
\usepackage[normalem]{ulem}

\usepackage{chngcntr}

\def\blR{\mathcal{R}}
\def\torT{\textgoth{T}}

\def\scR{\mathscr{R}}

\def\yj{{\mathtt{j}}}
\def\yi{{\mathtt{i}}}

\def\ya{{\mathtt{a}}}

\def\yb{{\mathtt{b}}}

\def\yc{{\mathtt{c}}}

\def\yx{{\mathtt{x}}}
\def\yy{{\mathtt{y}}}

\def\yr{{\mathtt{r}}}
\def\yd{{\mathtt{d}}}
\def\ye{{\mathtt{e}}}

\def\ve{{\mathbf{e}}}
\def\vv{{\mathbf{v}}}

\def\qA{{\mathtt{A}}}
\def\qB{{\mathtt{B}}}
\def\qC{{\mathtt{C}}}
\def\qD{{\mathtt{D}}}

\def\qH{{\mathtt{H}}}

\def\qO{{\mathtt{O}}}
\def\qP{{\mathtt{P}}}
\def\qR{{\mathtt{R}}}
\def\qQ{{\mathtt{Q}}}
\def\qU{{\mathtt{U}}}
\def\qV{{\mathtt{V}}}
\def\qX{{\mathtt{X}}}

\newcommand{\eqdef}{\stackrel{\scriptscriptstyle\rm def}{=}}

\theoremstyle{plain}
\newtheorem{mainthm}{Theorem} 
\newtheorem{thm}{Theorem}[section]

\newtheorem{lem}[thm]{Lemma}
\newtheorem{prop}[thm]{Proposition}
\newtheorem{claim}[thm]{Claim}

\theoremstyle{definition}
\newtheorem{defi}[thm]{Definition}

\newtheorem{Question}{Question}

\newtheorem{caveat}[thm]{Caveat}
\newtheorem{remark}[thm]{Remark}
\newtheorem{notation}[thm]{Notation}
\newtheorem{rem}[thm]{Remark}

\numberwithin{equation}{section}

\newcommand{\diam}{\mathrm{diam}\,}

\newcommand{\fC}{\mathfrak{C}}
\newcommand{\tfS}{\mathfrak{X}}

\newcommand{\tfB}{\textfrak{D}}

\newcommand{\tum}{\mathtt{1}}

\newcommand{\tr}{\mathtt{r}}

\newcommand{\ta}{\mathtt{a}}
\newcommand{\tb}{\mathtt{b}}
\newcommand{\tc}{\mathtt{c}}
\newcommand{\td}{\mathtt{d}}
\newcommand{\te}{\mathtt{e}}

\newcommand{\ti}{\mathtt{i}}
\newcommand{\tj}{\mathtt{j}}

\newcommand{\tk}{\mathtt{k}}

\newcommand{\tp}{\mathtt{p}}
\newcommand{\tq}{\mathtt{q}}

\newcommand{\tx}{\mathtt{x}}
\newcommand{\ty}{\mathtt{y}}

\def\cB{\textswab{B}}
\def\fB{\textswab{B}}

\def\tfB{\textfrak{D}}

\def\norte{{Q}}
\def\sul{{P}}

\newcommand{\fS}{\mathfrak{S}}

\def\bfR{\mathbf{R}}

\def\fL{\mathfrak{L}}

\renewcommand{\leq}{\leqslant}
\renewcommand{\geq}{\geqslant}

\def\loc{\mathrm{loc}}
\def\sss{\mathrm{sss}}
\def\ss{\mathrm{ss}}
\def\uuu{\mathrm{uuu}}
\def\uu{\mathrm{uu}}
\def\st{\mathrm{s}}
\def\ut{\mathrm{u}}
\def\ct{\mathrm{c}}
\def\cs{\mathrm{cs}}
\def\cu{\mathrm{cu}}

\small\normalsize
\let\oldmarginpar\marginpar
\renewcommand\marginpar[1]{\-\oldmarginpar[\raggedleft\tiny #1]%
{\raggedright\tiny #1}}

\newcommand{\relationarrow}[4]{\raise1ex\hbox{$\overset{\text{#1}}{\longrightarrow}$}\hspace{#3}\lower1ex\hbox{$\underset{\text{#2}}{\longleftarrow}$}\hspace{#4}\lower1ex\hbox{$/$}}

\title[Robust heterodimensional cycles of  co-index two via split blending machines]{Robust heterodimensional cycles of  co-index two via split blending machines}

\date{}

\author[Pablo G. Barrientos]{Pablo G.~Barrientos}
\address[P. G.~Barrientos]{Instituto de Matem\'atica e Estat\'{\i}stica UFF, Rua Marcos Waldemar de Freitas Reis, S/N - Campus do Gragoat\'{a}, Niter\'{o}i, Rio de Janeiro, 
Brazil}
\email{pgbarrientos@id.uff.br}

\author[Lorenzo J.D\'\i az]{Lorenzo J. D\'\i az}
\address[L. J. D\'\i az]{Departamento de Matem\'atica PUC-Rio, Marqu\^es de S\~ao Vicente 225, G\'avea, Rio de Janeiro 225453-900, Brazil}
\email{lodiaz@puc-rio.br}

\author[Yuri Ki]{Yuri Ki}
\address[Y. Ki]{Instituto de Matem\'atica e Estat\'{\i}stica UFF, Rua Marcos Waldemar de Freitas Reis, S/N - Campus do Gragoat\'{a}, Niter\'{o}i, Rio de Janeiro, 
Brazil}
\email{yuriki@id.uff.br}

\author[Cristina Lizana]{Cristina Lizana}
\address[C. Lizana]{Departamento de Matem\'atica. Instituto de Matem\'atica e Estat\'\i stica. Universidade Federal da Bahia. Av. Milton Santos s/n, 40170-110. Salvador, Bahia, Brazil}
\email{clizana@ufba.br}

\author[Sebastián A. P\'erez]{Sebastián A. P\'erez}
\address[S. A. P\'erez]{Instituto de Matem\'aticas, Pontificia Universidad Cat\'olica de Valpara\'\i so, Blanco Viel 596, Cerro Bar\'on, Valpara\'iso, Chile}
\email{sebastian.perez.o@pucv.cl}

\begin{document}
\date{\today}

\begin{abstract}
We consider diffeomorphisms $f$ with heterodimensional cycles of co-index two, associated with saddles $P$ and $Q$ having unstable indices 
$\ell$ and $\ell+2$, respectively. In a partially hyperbolic setting, where a two-dimensional center direction and strong invariant manifolds are defined, we introduce 
the class of  \emph{non-escaping cycles}, where the strong stable manifold of $P$ and the strong unstable manifold of $Q$ are
involved in the cycle. This configuration guarantees the existence of orbits that remain in a neighbourhood of the cycle.

We show that such diffeomorphisms $f$ can be $C^1$ approximated by diffeomorphisms exhibiting simultaneously $C^1$ robust heterodimensional cycles of co-indices one and two, encompassing all possible combinations among hyperbolic sets of unstable indices $\ell$, $\ell+1$, and $\ell+2$.

The proof relies on the construction of \emph{split blending machines}. This tool extends Asaoka's blending machines to a partially hyperbolic setting, providing a mechanisms to generate and control robust intersections within a two-dimensional central bundle. 

We also present simple dynamical settings where such cycles occur, namely skew product dynamics with surface  fiber maps.
Non-escaping cycles also appear  in contexts such as Derived from Anosov diffeomorphisms and matrix cocycles on $\mathrm{GL}(3,\mathbb{R})$.
\end{abstract}

\thanks{
%P. Barrientos acknowledges support from grant PID2020-113052GB-I00 funded by MCIN, PQ 305352/2020-2 (CNPq), and JCNE E-26/201.305/2022 (FAPERJ). 
%L.J.D\'iaz was partially supported by CAPES-Finance Code 001, by CNPq-grants 310069/2020-3 and 305327/2022-4, CNPq Projeto Universal 430154/2018-6 and 404943/2023-3. This study was funded by FAPERJ-Carlos Chagas Filho Foundation for Research Support of the State of Rio de Janeiro, Processes SEI E-26/200.371/2023 and E-16/2014 INCT/FAPERJ. C.Lizana  were partially supported by CNPq projects 406750/2021-1 and 404943/2023-3, Brazil. S. Pérez was partially supported FONDECYT Regular 1250697.
%The authors thank the hospitality of UFBA (Brazil), PUC-Rio (Brazil), PUCV (Chile), CMUP (Portugal)
P. Barrientos was supported by PID2020-113052GB-I00 funded by MCIN, PQ 305352/2020-2 (CNPq), and JCNE E-26/201.305/2022 (FAPERJ), Brazil.
L. J. D\'iaz is partially supported by CAPES-Finance Code 001, CNPq 310069/2020-3, CNPq Projeto Universal 404943/2023-3, INCT-FAPERJ E-26/200866/2018, and CNE-FAPERJ E26/204046/2024, Brazil.
C. Lizana was partially supported by CNPq projects 406750/2021-1 and 404943/2023-3, Brazil.
S. A. Pérez was partially supported by FONDECYT Regular 1250697, Chile.
The authors thank the hospitality of PUC-Rio (Brazil), UFBA (Brazil), PUCV (Chile), and CMUP (Portugal).
}


\maketitle

\begin{flushright}
To Jacob Palis, in memoriam
\end{flushright}
\section{introduction}
\label{s.introduction}
%\vmargem{reducir e encontrar um lugar o eliminar este paragrafo (LS a favor de eliminar): 
%This kind of dynamics with blender-like properties appears in several contexts, including the \emph{standard affine blender} \cite{AviCroWil:21}, the \emph{prototypical blender} \cite{BonDia:12}, the \emph{symbolic blender} ~\cite{BarKiRai:14},  the \emph{Grassmannian blender} ~\cite{BarRai:17},  and the \emph{jet blender} and \emph{jet-Grassmannian blender} ~\cite{BarRai:21}.}

\subsection{Framework}
\label{ss.frame}
Homoclinic tangencies and heterodimensional cycles are conjectured to be the two primary sources of nonhyperbolic dynamics, see \cite{Pal:00}.  These bifurcations give rise to cascades of new bifurcations and, in many cases, to robustly non-hyperbolic dynamics.
Here we  focus on  \emph{heterodimensional cycles}, a configuration in which the invariant manifolds of two 
hyperbolic basic sets with different \emph{$\ut$-indices} (dimension of the unstable bundle)
 intersect cyclically. Due to a deficiency in dimension, one of these intersections  cannot be transverse. 
A threshold for the ``lack of transversality" is given by the \emph{co-index} of the cycle, defined as the absolute value of the difference between the u-indices of the sets.
Heterodimensional cycles can only occur in dimensions three or higher. In dimension three, heterodimensional cycles always have a co-index  one.

Recall that a basic set $\Lambda$ of diffeomorphism $f$ has a unique and well-defined continuation $\Lambda_g$ for every diffeomorphism $g$ sufficiently close to $f$. A heterodimensional cycle between two basic sets $\Lambda$ and $\Gamma$ of
a diffeomorphism $f$ is said to be 
{\em{$C^r$ robust}} if there exists a $C^r$ neighbourhood $\mathcal{U}$ of $f$ such that, for every $g \in \mathcal{U}$, the continuations $\Lambda_g$ and $\Gamma_g$ also form a heterodimensional cycle.
By the Kupka-Smale genericity theorem, every robust cycle  involves some {\em{nontrivial}} hyperbolic set (i.e., a set containing non-periodic orbits).
The presence of such cycles prevents hyperbolicity; thus, robust heterodimensional cycles give rise to open sets of non-hyperbolic diffeomorphisms.

Following \cite[Preface]{BonDiaVia:05}, examples of open sets of diffeomorphisms consisting entirely of non-hyperbolic maps are roughly divided into two categories: critical and non-critical.
Critical behavior is associated with the occurrence of homoclinic tangencies, whereas non-critical behavior is related to partial hyperbolicity, with typical bifurcations involving heterodimensional cycles. 
%In many cases, these two types of dynamics may overlap. 
The first category includes constructions such as those in \cite{New:70, New:74}, which deal with the persistence of homoclinic tangencies. 
The examples in \cite{AbrSma:68, Sim:72a} belong to the second category. In these constructions, 
the invariant manifolds of
two hyperbolic sets with different $\ut$-indices intersect cyclically, forming heterodimensional cycles, although this terminology was not used in the original works.
Indeed these cycles are robust.

A natural question in dynamics is to determine when the unfolding of a cycle leads to robust ones.
In the case of heterodimensional cycles, the co-index measures the dimension gap that must be bridged to produce ``transverse-like" intersections, those capable of generating cyclic and robust interactions between the invariant manifolds of the sets in the cycle.
To overcome this difficulty, additional geometric structures are required. This issue was resolved for co-index one cycles in \cite{BonDia:08} (in the $C^1$ setting) and \cite{LiTur:24} (for higher $C^r$ regularity), where the key ingredient is the construction of blenders. See Section~\ref{ss.bridging} for further discussion.

Within the terminology above, the dynamics that we consider  fall into the non-critical category. We study co-index two heterodimensional cycles within a partially hyperbolic setting that ensures the saddles in the cycle have well-defined strong stable and unstable manifolds. In this context, the generation of new cycles requires addressing two independent difficulties: bridging the two-dimensional gap to achieve ``transverse-like properties"  and overcoming the ``loss of recurrence"  along the central directions of the cycle. The latter is managed by assuming that the strong invariant manifolds are ``involved"  in the cycle. We call such cycles {\em{non-escaping,}} see Definition~\ref{d.nonescapinghc} and the heuristic discussion in Section~\ref{sss.heuristic}. We prove that these cycles simultaneously give rise to $C^1$ robust heterodimensional cycles of both co-index one and two.

  To state our main result, note that a co-index two cycle associated
with saddles of $\mathrm{u}$-indices $\ell$ and $(\ell+2)$ may generate hyperbolic sets of $\mathrm{u}$-indices 
$\ell$, $(\ell+1)$, and $(\ell+2)$,
which can in turn be involved in new cycles.
To capture the full complexity of the arising cycles, we introduce the following definition:

\begin{defi}[Cycle of type $(\ell,m)$]
\label{d.typesofcycles}
A heterodimensional cycle is a {\em{cycle of type $(\ell,m)$,}}
with $\ell<m$, if it is associated with hyperbolic sets of $\ut$-indices $\ell$ and $m$.
\end{defi}

%Here, $m=\ell+2$,
%and we may  find cycles of type $(\ell,\ell+2)$, $(\ell,\ell+1)$, and $(\ell+1,\ell+2)$.
%where the first type has co-index two, while the latter two types have co-index one.

\begin{mainthm}
\label{t.robustcycles}
Let $f$ be a $C^1$ diffeomorphism having  a non-escaping  cycle of type $(\ell, \ell+2)$ 
associated to a pair of saddles.
Then there is $g$ arbitrarily $C^1$ close to $f$ having
robust cycles of type $(\ell, \ell+2)$, $(\ell,\ell+1)$, and $(\ell+1,\ell+2)$.
\end{mainthm}

\begin{remark}
\label{r.LiTur} 
Theorem~\ref{t.robustcycles} applies to a broader class of cycles which, following the terminology of \cite{LiTur:24}, we call
cycles with a {\em contour of double type I},  defined in Section~\ref{ss.doubletypeone} involving the notion of a contour, see Section~\ref{ss.non-escaping}.
Informally, such cycles behave as if they carried two independent type I contours of a cycle of co-index one (in the sense of \cite{LiTur:24}), one associated with the strong stable manifold of the saddle of $\ut$-index $\ell$, and the other with the strong unstable manifold of the saddle of $\ut$-index $\ell+2$.

Viewing non-escaping cycles as cycles with a double type I contour, the coexistence of a pair of robust co-index one cycles, 
obtainable without destroying the original cycle, follows naturally from the approach in \cite{LiTur:24}.
The genuinely new and technically demanding phenomenon lies in the coexistence of robust co-index two cycles.
\end{remark}

\begin{Question}
\label{q.stabilization}
It remains an open problem whether there is a transitive set (for instance, a homoclinic class) that contains all the hyperbolic sets involved in the three types of cycles in the theorem.
A further question is whether such a set could also contain the continuations of the saddles in the cycle.
These questions are closely related to the problem of stabilization of cycles; see \cite{BonDiaKir:12,LiTur:24}.
A more precise version of  this question should involve contours. Let us observe that there are co-index one cycles
(precisely those having a contour of type I) that cannot be stabilizable, see \cite[Theorem C]{LiTur:24}.
\end{Question}

Our aim is to present the geometric ideas underlying robust cycles of co-index two, setting aside the regularity issues involved in the perturbations. The approach in \cite{LiTur:24} provides a natural starting point for addressing higher regularity, which remains a delicate and technically demanding aspect of the theory. In line with the study of co-index one cycles, our strategy is to first introduce the geometric framework (as in \cite{BonDia:08}).
%the development of the regularity aspects, along the lines of \cite{LiTur:24}, should then constitute a subsequent step.

%
%\textcolor{red}{
%\begin{remark}
%explain the eigth  configuration of a non-escaping cycle involving
%\vmargem{the correct form to classify is $(\pm, \pm, +)$ and $(\pm, \pm,-)$
%where the last symbol describe the type of transition between the cycle?}
%$(\pm, \ast)_P$ and $(\pm, \ast)_Q$, $\ast\in \pm$, involving
%the relative position of the heteroclinic
%points in the strong invariant manifolds and transition along the invariant manifolds of small dimension?.
%see figure. There is only one case, reminiscent of the twisted cycles? where we are not able to get robust cycles
%between $\Lambda_1$ and $\Lambda_2$. For details see Theorem~\ref{t.eightcases}.
%\end{remark}}
%
%
%
%\begin{mainthm}[what we actually prove]
%\label{t.qrobustcycles}
%Let $f$ be a $C^1$-diffeomorphism having  a non-escaping  heterodimensional cycle associated to saddles 
%of $\ut$-indices $\ell$ and $\ell+2$.
%Then there is $g$ arbitrarily $C^1$-close to $f$ having
%hyperbolic sets $\Lambda_1$ of $\ut$-index $\ell$,
%$\Gamma_1, \Gamma_2$ with $\ut$-index $\ell+1$, and 
% $\Lambda_2$ of $\ut$-index $\ell+2$ such that
% \begin{itemize}
% \item $\Lambda_i$ and $\Gamma_i$ have  a robust cycle, $i=1,2$, and
% \item $W^\ut (\Lambda_1) \cap W^\st (\Lambda_2)$ intersect robustly.
% \end{itemize}
% Moreover, if the cycle is not of type $(-,-,+)$ then $\Lambda_1$ and $\Lambda_2$ form a robust cycle of
% co-index two.
%\end{mainthm}
%

Before turning to the setting of high co-index cycles, 
we revisit the co-index one case. There are two (non-disjoint classes) of co-index one cycles: 
twisted and non-twisted (following the terminology of \cite{BonDiaKir:12}), or type I and type II (as in \cite{LiTur:24}). In the case of 
twisted  or type I cycles, the bifurcating diffeomorphism may exhibit very rich dynamics,
see Theorem~\ref{tp.finallyarobustcycle}.
 As shown in \cite[Theorems 4 and 5]{LiTur:24}, the saddles in the cycle may be accumulated by other hyperbolic sets, which, however, fail to be homoclinically or heteroclinically related to the saddles in the original cycle. This leads to a form of latent underlying dynamics\footnote{This is somewhat reminiscent of the so-called ``ghost dynamics" in saddle-node cycles, described via transition maps in \cite{NewPalTak:83}; see also \cite[Section 11.3]{DiaRioVia:01}. The key difference is that ghost dynamics are virtual, suppressed dynamics that only appear upon unfolding the cycle, whereas latent dynamics are actual and present in the system.}, whose study is essential for understanding the global dynamics of the cycle.
In the non-escaping setting, a similar form of latent dynamics arises.

Our interest in co-index two cycles arises from a class of skew-products with fiber dynamics on the two-sphere, a setting reminiscent of \cite{Goretal:05}; see Section~\ref{ss.toynonescaping}.  
This class leads us to introduce non-escaping cycles and to place them within a broader context.  

Within nonhyperbolic dynamics, non-escaping cycles already appear in the class of  diffeomorphisms constructed by Abraham-Smale \cite{AbrSma:68} and Shub~\cite{Shu:72}, as well as in Manning's versions of them in \cite{Man:72}; see Section~\ref{ss.shubandco}.  

Although our primary focus is on diffeomorphisms, an additional motivation arises from the study of matrix cocycles in $\mathrm{GL}(3, \mathbb{R})$, which are directly related to skew-products with fiber maps defined on the two-sphere.  
We note that heterodimensional cycles of co-index one have already appeared in the study of the boundary of hyperbolicity for matrix cocycles, see \cite[Theorem 4.1]{AviBocYoc:10}, even though the terminology of cycles is not explicitly used there (see also the discussion in \cite[Section 1.2]{DiaGelRam:22}).  
In the case of matrix cocycles in $\mathrm{GL}(3, \mathbb{R})$, we observe that the corresponding skew-product dynamics often exhibit non-escaping cycles; see Section~\ref{ss.cocyclesnonescaping} for further discussion.

\subsection{Bridging the dimension deficit using blenders}
\label{ss.bridging}
Blenders were introduced in \cite{BonDia:96} as a tool for constructing robustly transitive diffeomorphisms. Since then, various versions have been developed to serve different purposes. Rather than surveying all these variants, we focus on a key dynamical feature of them: a {\em{blender}} is a hyperbolic set with a special geometric superposition property that makes certain non-transverse intersections between invariant manifolds robust under perturbations, behaving  as transverse ones.
A bit more precisely, let $\Lambda$ be a blender, and denote by $d^\st$ and $d^\ut$ the dimensions of its stable and unstable bundles, respectively. Then $\Lambda$ is called a {\em{$\cu$-blender}} if its stable manifold ``behaves''  as if it had dimension $d^\st + j^\st$, for some $j^{\st} \geq 1$. Similarly, $\Lambda$ is a $\cs$-blender if its unstable manifold behaves as if it had dimension $d^\ut + j^{\ut}$, for some $j^{\ut} \geq 1$. 
For details see Definition~\ref{d.cublender}.
We call the number $j^\st$ {\em{(stable) dimension  jump of the 
$\cu$-blender,}} similarly for $j^\ut$.
For an informal presentation of blenders, see \cite{BonCroDiaWil:16}.
The robust cycles obtained in \cite{BonDia:08, LiTur:24} 
are associated to co-index one cycles and 
involve a saddle and a blender which provides a one-dimensional jump. To study cycles of higher co-index, a more refined analysis of blenders is required.

The first examples of robust cycles were provided in \cite{AbrSma:68} in dimension four and later extended to dimension three in \cite{Sim:72a}. 
 These configurations  are specific and not associated with bifurcations.
In \cite[Section 4.2]{Bon:11}, these constructions are revisited by embedding Plykin attractors on two disks
in a three-dimensional manifold as normally hyperbolic basic sets of $\ut$-index two. The resulting sets have a two-dimensional stable manifold, while their stable bundle is one-dimensional. Hence they can be regarded as $\cu$-blenders and used as a dynamical plug to generate robust cycles of co-index one. To complete the construction it is enough to relate the normally hyperbolic Plykin set with a saddle of $\ut$-index one.
We refer to this type of blender as a $\cu$-blender of Plykin type.
% in contrast to blenders that are conjugate to a horseshoe, called $\cu$-blender horseshoes; see \cite{BonDia:12} and the versions in
%\cite{NasPuj:12,BarKiRai:14}.

A crucial difference between the constructions
in \cite{AbrSma:68, Sim:72a, Bon:11}
 and those in \cite{BonDia:08, LiTur:24} lies in the mechanism by which robust cycles are obtained. The former relies on rigid, semi-global configurations in which blenders are present in the initial system. In contrast, the latter requires only the existence of a co-index one cycle
  associated to a pair of saddles, a flexible configuration that appears in many contexts. The key point is that in \cite{AbrSma:68, Sim:72a, Bon:11}, blenders are assumed and used as tools, whereas in \cite{BonDia:08, LiTur:24}, blenders 
are obtained
 by unfolding the 
 cycles and thereafter used
 as plugs to produce robust cycles.

We now turn to robust cycles of co-index $k \geq 2$.  
One approach, following \cite{Bon:11}, starts with a hyperbolic attractor of unstable dimension $k \geq 2$ and embeds it as a saddle-type basic set by adding a normally hyperbolic unstable direction.  
The resulting set is a $\cu$-blender with  jump $j^\st = k$, which serves as a dynamical plug to create robust cycles of co-index $k$. This, however, is an ad hoc construction.  

A different, more semi-global approach, proposed in \cite{BarKiRai:14}, perturbs the product of a diffeomorphism with a horseshoe and the identity on a $k$-dimensional manifold.  
This configuration also produces robust cycles of co-index $k$, involving a blender horseshoe with dimension jump $k$ and a saddle, but it is not intrinsically tied to cycles and requires starting from an already rich dynamics.

In the previous constructions, the dimension deficit required for transverse-like intersections is resolved by using a blender with a dimension jump equal to the full deficit. Our alternative approach,
 is to split the deficit $k$ into two jumps 
  $j^\st$ and $j^\ut$, corresponding to unstable and stable blenders, with
$j^\st+j^u=k$. This strategy was implemented in \cite{Asa:22}.
Our construction is inspired by the ideas developed there, although substantial modifications and new ingredients are 
required to overcome the difficulties specific to our setting.
This approach allows us to consider cycles involving just a pair of saddles.

The previous discussion 
%can be summarized as follows.
%Given $k \in \mathbb{N}$ and $\ell, m \in \mathbb{N}^\ast$ with $\ell + m = k$, there exist a $\cs$-blender and a $\cu$-blender with dimension jumps $j^\ut=\ell$ and $j^\st=m$, respectively, involved in a robust cycle of co-index $k$.
%The construction of such cycles typically requires certain global properties.
 leads to the question below, where the term ``simple''   is intentionally vague, referring to cycles involving saddles and specific conditions on the intersections of their invariant manifolds.
This excludes cases where, by assumption, the saddles belong to ``big''   hyperbolic sets (as in the first two cases discussed above).
For the second item in the question, have in mind the constructions in  \cite{SaiTakYor:23}, where the cycles are not associated to blenders
but to sets with large Hausdorff dimension. Also note that the question below can be posed in any $C^r$ topology.

\begin{Question}
\label{q.q1}
Consider any $k \in \mathbb{N}$.
\begin{enumerate}
\item Does there exist a ``simple"  class of (heterodimensional) cycles of co-index $k$ that generates robust cycles of co-index $k$?
\item Are such cycles associated with blenders? 
If so, what are the corresponding dimension jumps of these blenders?
\end{enumerate}
\end{Question}

In the case $k=1$, this question was positively answered in  \cite{BonDia:08, LiTur:24} obtaining a blender with 
dimension jump one.
We introduce the class of non-escaping cycles (of co-index two) and provide a positive answer to Question~\ref{q.q1} 
for them.
To fill the two-dimensional deficit in the indices of the sets, we get
 the simultaneous existence of a $\cs$-blender $\Lambda_1$ and a $\cu$-blender $\Lambda_2$ of $\ut$-indices $\ell$ and $\ell+2$, respectively, and both with dimension jump equal to one. 
% As the sum of the  dimensions of the stable bundle of $\Lambda_1$ and the unstable bundle of $\Lambda_2$ is $n+2$, where $n$ is the dimension of  
% ambience, intersections between $W^\st(\Lambda_1)$ and $W^\ut (\Lambda_2)$ may be done transverse and hence robust. 
% Here $W^\st(\Lambda)$ and $W^\ut(\Lambda)$ denote the stable and unstable manifolds of a hyperbolic set 
%$\Lambda$.
 The blender property implies that the sets
$W^\ut(\Lambda_1)$ and $W^\st(\Lambda_2)$
($W^\st(\cdot)$ and $W^\ut(\cdot)$ denote the stable and unstable manifolds)
 behave as manifolds of dimensions $\ell+1$ and $(n-\ell-2)+1$, respectively (here $n$ is the dimension of the ambience).
Hence the sum of the ``dimensions'' is $n$ and 
can provide intersections that behave as transverse ones and hence may be done robust.

Note that in our construction the blenders have a dimension jump equal to one. This leads to a second question, which refines item (2) of Question~\ref{q.q1}.

\begin{Question}
\label{q.q}
Given $j \geq 2$, does there exist a ``simple" class of heterodimensional cycles that generates blenders exhibiting a dimension jump equal to $j$? Are these blenders involved in robust cycles?
\end{Question}

A recent approach \cite{BarRai:17} provides robust homoclinic tangencies from robust heterodimensional cycles of the lifted dynamics in the Grassmannian manifold, in a partially hyperbolic setting with center dimension at least two.  
This strategy was further refined in \cite{Asa:22} to produce $C^2$ robust homoclinic tangencies of high co-dimension in a non-dominated setting, though not directly related to explicit bifurcations.  
In the same work, Asaoka also revisits the theory of blenders, introducing the notion of a {\em blending machine}, a mechanism for generating robust intersections of dynamically defined Cantor sets in higher dimensions.  
We adapt this concept to the dominated setting by introducing the {\em split blending machine}, a variant designed for partially hyperbolic systems with two- or higher-dimensional center direction, reflecting the co-index two cycles considered here (see Section~\ref{ss.blendermachines}).

 \section{Setting, toy model, and sketch of the proof} 
 \label{smanythings}
 
 In Section~\ref{ss.non-escaping}, we state precisely the setting where Theorem~\ref{t.robustcycles} applies.
In Section~\ref{ss.toynonescaping}, we present a toy non-escaping cycle.
 In Section~\ref{ss.techniquesandtools},
 we sketch the main ingredients of our constructions.
 
 Throughout the paper, we consider a diffeomorphism $f$ defined on a compact boundaryless  Riemannian manifold $M$ of dimension $n  \geqslant 4$,
having two saddles $P$ and $Q$ forming a heterodimensional cycle of co-index two
(in what follows we omit the heterodimensional), where the $\ut$-index of $P$ is smaller than that of $Q$.  
 
 \subsection{Non-escaping  cycles of co-index two}
 \label{ss.non-escaping}
\subsubsection{Heuristics} 
\label{sss.heuristic}
The dynamics associated with a cycle depend on several factors which are considered in the the definition of non-escaping ones
(see Remark~\ref{r.defnonesc}):
(i) domination or partial hyperbolicity in the cycle (if any);
(ii) the so-called central multipliers of the saddles;
(iii) the geometry of the intersections $W^\st (P) \cap W^\ut (Q)$ and $W^\ut (P) \cap W^\st (Q)$; and
(iv) the transitions along the cycle between their saddles.
These aspects are not entirely independent.

 Concerning (iii), the way the invariant manifolds of the saddles in the cycle intersect plays a fundamental role. A key problem in analyzing the unfolding of a cycle is the generation of ``recurrences'' (used here as a deliberately vague term) around the cycle. 
 The simplest precursor of this type of question appear in  surface diffeomorphisms with heteroclinic tangencies whose limit sets and non-wandering sets are different, see \cite[Figure 4.4]{New:78}. In such cases, the relative position of the tangencies is crucial in determining the recurrent behaviour.
In the specific case of heterodimensional cycles, recurrences always occur in the co-index one case, a  key fact
 in establishing their robustness. However, for higher co-indices, there may exist 
``escaping directions''  that obstruct the recurrence mechanism. The rough idea is that, in the unfolding of a cycle, orbits tend to follow the available center-strong stable direction to enter the cycle, and the unstable directions to exit it.  
In non-escaping cycles, these directions are intrinsically part of the cycle structure, which guarantees recurrence.  
We now formalize this heuristic.
  
 \subsubsection{Contours of a cycle}
 \label{sss.elements}
For simplicity, we assume that the saddles $P$ and $Q$ are fixed points. 
%Recall  
%that the $\ut$-index of $Q$ is bigger than the one of $P$.
The $f$-invariant  {\em{maximal contour of the cycle}} is the compact set
\begin{equation} 
\label{e.domaincycle}
\Theta \eqdef \{P\} \cup \{Q\} \cup \Big(W^\st (P) \cap W^\ut(Q) \Big) \cup  
\Big(W^\ut (P) \cap W^\st(Q) \Big).
\end{equation}
After unfolding the cycle, 
one aims to control the  orbits  that remain within a small neighbourhood of $\Theta$ (or some prescribed  part of it, since in many cases the set $\Theta$ may be very big). Hence  one may consider compact $f$-invariant subsets
$\Upsilon$ of $\Theta$ that contain the saddles $P$ and $Q$, along heteroclinic  points in both sets 
$W^\st (P) \cap W^\ut(Q)$ and  
$W^\ut (P) \cap W^\st(Q)$. We refer to such a set as a {\em{contour of the cycle.}}

To be more precise,
we suppose there exists a contour $\Upsilon$
%, which contains 
%the saddles $P$ and $Q$, along with heteroclinic points in $W^\st (P) \cap W^\ut(Q)$ and $W^\ut (P) \cap W^\st(Q)$,
having a partially hyperbolic splitting adapted to the co-index two 
\begin{equation} 
T_\Upsilon M = E^\sss \oplus E^\ct \oplus E^\uuu,
\end{equation}
where $ E^\sss$ and  $E^\uuu$ are uniformly contracting and expanding, respectively.
Adapted means that  $E^\ct$ is a central direction of dimension two such that
 the stable and unstable bundles of the saddles
of the cycle satisfy
\begin{equation}
\label{e.bundles}
\begin{split}
&E^\st(P)= E^\sss(P) \oplus E^\ct (P), \qquad
E^\ut(P)= E^\uuu(P),\\
&E^\st(Q)= E^\sss(Q),  \qquad  \qquad \quad \,\,\,\,
E^\ut (Q)= E^\ct (Q) \oplus E^\uuu (Q).
\end{split}
\end{equation}
 The bundle $E^\ct$ may split in a dominated way or not. We consider the first case,
assuming that
$E^\ct= E^\ct_1 \oplus E^\ct_2$, where $E^\ct_1$ is more contracting than $E^\ct_2$. These bundles play asymmetric roles, depending
on the saddle of the cycle as explained below.  Putting together the previous bundles we get a dominated splitting over 
$\Upsilon$,
\begin{equation} 
\label{e.domaincyclebis}
T_\Upsilon M = E^\sss \oplus E^\ct_1 \oplus E^\ct_2 \oplus E^\uuu.
\end{equation}

The results in \cite{Dia:95,DiaRoc:97} show that, for co-index one cycles, the resulting dynamics primarily depends on the central direction, which is one-dimensional and  
 further exploited in \cite{BonDia:08,LiTur:24} to get robust cycles.  
Co-index two cycles, however, require a more careful analysis.  
Note from \eqref{e.bundles} that the saddle $P$ admits several stable subbundles:
\[
E^\sss(P) \subset E^\ss(P) \subset E^\st(P), 
\quad \text{with} \quad 
E^\ss(P) \eqdef E^\sss(P) \oplus E^\ct_1(P).
\]
We refer to these as the {\em strongest stable}, {\em strong stable}, and {\em stable} bundles, respectively.  
Similarly, for the saddle $Q$ there are expanding bundles
\[
E^\uuu(Q) \subset E^\uu(Q) \subset E^\ut(Q), \quad 
\text{with} \quad 
E^\uu(Q) \eqdef E^\ct_2(Q) \oplus E^\uuu(Q),
\]
Accordingly, one can define the invariant manifolds tangent to them \cite{HirPugShu:77}:
\begin{equation}
\label{e.invariantstableunstablemanifolds}
W^\sss(P) \subset W^\ss(P) \subset W^\st(P) \qquad \mbox{and} \quad 
W^\uuu(Q) \subset W^\uu(Q) \subset W^\ut(Q).
\end{equation}

The orbits generated in the unfolding of the cycle split in segments as follows: 
first, iterations close to $Q$;
 then a transition from $Q$ to $P$ following some heteroclinic point in $\Upsilon \cap W^\st(P) \cap W^\ut (Q)$;
next, iterations close to $P$; and finally
a transition from $P$ to $Q$ 
 following some heteroclinic point in  
 $\Upsilon \cap W^\ut(P) \cap W^\st (Q)$.  
Typically, these orbits  exit from $Q$ along $E^\ct_2$ and approach to $P$ along
$E^\ct_1$. Therefore, both directions must be involved in the cycle, otherwise, points 
may leave the
neighbourhood of the cycle along $E^\ct_2$ and fail to return to the cycle along $E^\ct_1$.
This observation is done more precise
using  $W^\ss (P)$ and $W^\uu (Q)$:
the interaction of the central directions is guaranteed if 
$W^\ss (P)$ and $W^\uu (Q)$ both intersect the contour $\Upsilon$.
This  leads to the definition of non-escaping cycles.

%\lmargem{problems of stabilization, later...}

\subsubsection{Non-escaping contours and cycles} 
\label{sss.defnonesc}
 We introduce conditions (NE1)--(NE2) for a contour $\Upsilon$ of the cycle associated to  $P$ and $Q$
 to be {\em{non-escaping.}}
 Recall the definition of the maximal contour $\Theta$ in \eqref{e.domaincycle}.
\smallskip

 \noindent{\bf{(NE1)}}
 There are points $X_1, X_2, Z_1, Z_2  \in W^s(P) \pitchfork W^u(Q)$
 with pairwise disjoint orbits
 and $Y\in  W^\ut(P) \cap W^\st(Q)$ such that  
 \begin{equation}
 \label{e.contourupsilon}
 \Upsilon \eqdef
 \{P\} \cup \{Q\} \cup \bigcup _{i \in \mathbb{Z}}
 f^i (\{X_1, X_2, Z_1, Z_2,  Y\}) \subset \Theta
 \end{equation}
 has a dominated  splitting  
\begin{equation}
\label{e.fourbundles}
T_\Upsilon M =
E^\sss \oplus E^{\ct}_1  \oplus E^{\ct}_2 \oplus E^\uuu, \qquad E^\ct \eqdef E^{\ct}_1 \oplus E^\ct_2,
\end{equation}
with four non-trivial bundles satisfying \eqref{e.bundles} and
such that $\dim E^{\ct}_1= \dim E^{\ct}_2=1$.
%
%\begin{equation*}
%\begin{split}
%E^\st(P)&=E^\sss(P) \oplus E^{\ct}(P), \qquad E^\ut(P)=E^\uuu(P), 
%\\
%E^\st(Q)&=E^\sss(Q), \qquad \qquad \qquad
%E^\ut(Q)=E^{\ct}(Q)\oplus E^\uuu(Q).
%\end{split}
%\end{equation*}
%%\begin{itemize}
%\item
%$E^\st(P)=E^\sss(P) \oplus E^{\ct}(P) $ and $E^\ut(P)=E^\uuu(P)$, 
%and
%\item
%$E^\st(Q)=E^\sss(Q)$ and
%$E^\ut(Q)=E^{\ct}(Q)\oplus E^\uuu(Q)$.
%\end{itemize}

The eigenvalues of $Df$ at $P$ and $Q$ corresponding to $E^\ct_1$ and $E^\ct_2$ are the 
{\em{central eigenvalues}}
 in (ii) in Section~ \ref{sss.heuristic}.
By hypothesis, they are real and different in modulus. There is no assumption on their signals.

Consider the invariant manifolds
$W^\sss(P) \subset W^\ss(P) \subset W^\st(P)$ and
$W^\uuu(Q) \subset W^\uu(Q) \subset W^\ut(Q)$ in \eqref{e.invariantstableunstablemanifolds}.
Each of them has co-dimension one in the next one.
Thus, the sets
$$
W^\st(P)\setminus W^\ss(P), \quad W^\ss(P)\setminus W^\sss(P), \quad
W^\ut(Q)\setminus W^\uu(Q), \quad W^\uu(Q)\setminus W^\uuu(Q)
$$
each have two connected components, which we denote by
$W^\st_j(P)$, $W^\ss_j(P)$, $W^\ut_j(Q)$, and $W^\uu_j(Q)$,
for $j \in {\ell,r}$.

\smallskip

\noindent{\bf{(NE2)}}
The heteroclinic points 
$X_1,X_2, Z_1,Z_2$
in the contour $\Upsilon$ 
such that
$$
X_1,X_2 \in W^\uu (Q) \cap  (W^\st (P) \setminus W^\ss(P) )\quad
\mbox{and} \quad
Z_1, Z_2\in W^\ss (P) \cap (W^\ut (Q) \setminus W^\uu(Q))
$$
satisfy one of the following properties:
\begin{enumerate}
\item
 either (a) $X_1, X_2$ lie in different components of $W^\uu(Q)\setminus W^\uuu(Q)$, or (b)
they lie in the same component of $W^\uu(Q)\setminus W^\uuu(Q)$ but in different components of $W^\st(P)\setminus W^\ss(P)$;
\item
either (a) $Z_1, Z_2$ lie in different components of $W^\ss(P)\setminus W^\sss(P)$, or (b) they 
lie in the same component of $W^\ss(P)\setminus W^\sss(P)$ but in different components of $W^\ut(Q)\setminus W^\uu(Q)$.
\end{enumerate}

%
%\textcolor{blue}{{\bf{Accumulation from both sides}} We need to discuss this point. We need to assume that the points $X_1,X_2$ are in different
%sides of the strong unstable manifold. There are several forms to formalize it. None of them very satisfactory to me.
%The simplest one and more intuitive. Let $W^\ct_{\loc} (Q)$  be local central manifold of $Q$ tangent to
%the {\em{central bundle}} $E^\ct\eqdef  E^{\ct}_1  \oplus E^{\ct}_2$ containing the backward orbit of
%$X_1$ and $X_2$ (for this we replace by some backward iterate).
%Let $\gamma^\ct_1$ be any curve tangent to $E^\ct_1$ such that 
%$W^\ct_{\loc} (Q)\setminus \gamma_1^\ct$ gas two connected 
%components $W^{\ct, \pm}_{\loc} (Q)$. There are backwards iterates of the sets $\{X_1,X_2\}$ 
%contains points in both components  $W^{\ct, \pm}_{\loc} (Q)$. The con of this definition is that we do not now if
%such central manifold exists (certainly yes) and second we need to claim that this definition does not
%depend on the choice of the central manifold - which is not unique - and the central curve. This is probably simple.
%The whole point: we need a simple way to define approximation from both sides.}
%\lmargem{anyway, we need this discussion. Indeed we do not need that the four points $X_1,X_2, Z_1, Z_2$
%to be in different sides but we need some coupling to make concatenations of orbits}
%
%
%{\textcolor{red}{we need a symmetric or something like that, with similar arguments for both sides.}
% \begin{itemize}
%\item[{\bf{(NE3)}}]
%$Z_1, Z_2 \in W^\ss (P)$
% and their orbits
%are disjoint from $W^\uu(Q)$.
%\end{itemize}

\begin{defi}[Non-escaping contours and cycles]
\label{d.nonescapinghc}
A  contour $\Upsilon$ of a cycle of co-index two associated to saddles $P$ and $Q$ 
is called {\em{non-escaping}} if it 
satisfies conditions  (NE1) and (NE2). A co-index two cycle is {\em{non-escaping}} if it has some non-escaping contour.
\end{defi}

\begin{remark} 
\label{r.defnonesc}
Condition (NE2) directly addresses item (iii) in Section~\ref{sss.heuristic} and implicitly involves items (i) and (ii). The transverse intersection
$W^\st(P) \pitchfork W^\ut(Q)$ captures both $W^\ss(P)$ and $W^\uu (Q)$.
Regarding the transitions in item (iv), the forward transitions from $Q$ to $P$ along the orbits of $X_j$ avoid the strong stable manifold of $P$.
Similarly, the backward transitions from $P$ to $Q$ along the orbits of $Z_j$ avoid the strong unstable manifold of $Q$.
These conditions hold open and densely in the setting under consideration.
\end{remark}

\subsection{Toy non-escaping cycles}\
\label{ss.toynonescaping}
We introduce a class of symbolic skew-products, whose fiber maps are defined on $\mathbb{S}^2$, having non-escaping cycles. This construction is motivated by \cite{Goretal:05}.
For simplicity, and with cocycles in $\mathrm{GL}(3,\mathbb{R})$ in mind, we present the construction with fiber maps on $\mathbb{S}^2$, although it can be adapted to any surface.
These constructions can be translated to the differentiable setting by replacing the symbolic part with a horseshoe.

Consider  $C^1$ diffeomorphisms  $f_0,f_1\colon \mathbb{S}^2\to \mathbb{S}^2$
 defined on the two sphere satisfying conditions (a)-(e) below depicted in Figure~\ref{fig.cases-exp}.
 
\begin{figure}[h]
\begin{overpic}[scale=.33,
%grid,tics=5
]{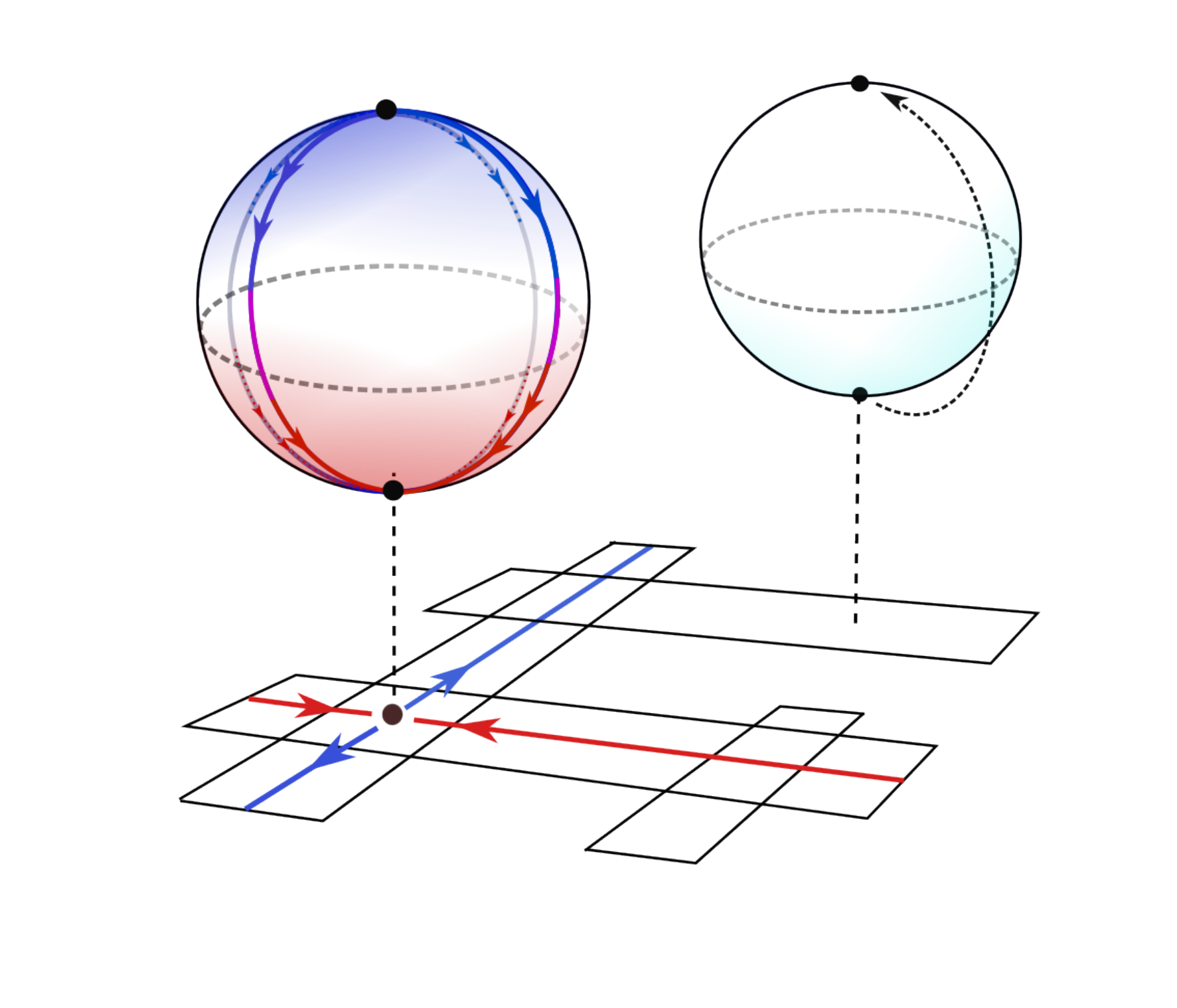}
	      	\put(30,38.5){$P$}	
		\put(30,80){$Q$}
	        \put(79,17){$0$}	
		\put(35,16){$0^{\mathbb{Z}}$}
		\put(15,75){$f_0$}	
		\put(69,46){$P$}	
		\put(71,82){$Q$}
		\put(88,29){$1$}	
	        \put(88,75){$f_1$}	
%	      	\put(32,31){\small$g(x)$}	
%	      	\put(72,31){\small$g^\ast(x)$}	
%	      	\put(50,31){\small (1)}	
%	      	\put(50,22){\small (2)}	
%	      	\put(50,17){\small (3)}	
                 \end{overpic}
		\vspace{-1.5cm}
		\caption{Toy non-escaping cycle}
		\label{fig.cases-exp}
\end{figure}
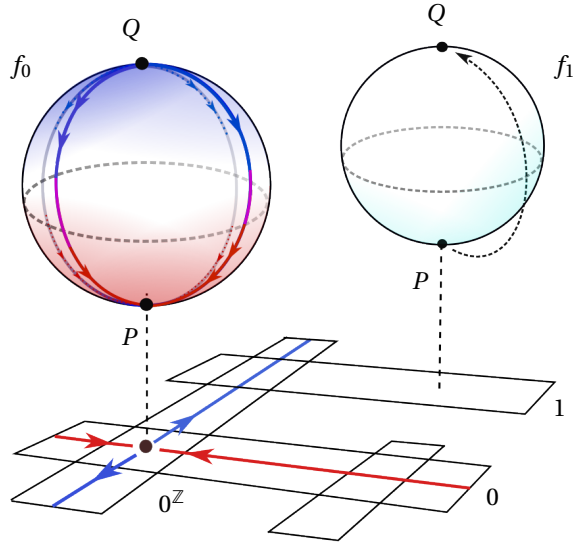

\begin{enumerate}
\item[(a)]
$f_0$ has two fixed points, a
global repeller $\norte$  and
global attractor
$\sul$,
\item[(b)]
$Df_0(P)$ and $Df_0(Q)$ have real eigenvalues with different modulus. 
\end{enumerate}
Let $E^\ss_P$ and $E^\cs_P$ be the eigenspaces of $Df_0(P)$ 
where $E^\ss_P$ is associated to the most contracting eigenvalue.
Similarly, we let 
$E^\cu_Q$ and $E^\uu_Q$ the eigenspaces of $Df_0(Q)$, 
where
$E^\uu_Q$ is associated to the most expanding eigenvalue. 
Hence,   there are  one-dimensional strong stable and unstable foliations, 
denoted by  $\mathcal{F}^\ss$ and  $\mathcal{F}^\uu$,
defined on 
$W^\st (P,f_0)$ and $W^\ut (Q,f_0)$, respectively. 
In particular, 
the one-dimensional strong stable manifold $W^{\ss}(P, f_0)\subset \mathbb{S}^2$ and 
$W^\uu (Q,f_0) \subset \mathbb{S}^2$ are defined. 
Note that every point in 
$$
W^\st(P,f_0) \cap W^\ut (Q,f_0) = \mathbb{S}^2 \setminus \{P,Q\}
$$ 
has strong stable and unstable leaves.
Note also that 
$W^{\ss}(P, f_0)\setminus \{P\}$ has two connected components
 $W^{\ss}_{\ell}(P, f_0)$  and   $W^{\ss}_{r}(P, f_0)$. Similarly, 
 we can write 
 $$
 W^{\uu}(Q, f_0)\setminus \{Q\}= W^{\uu}_{\ell}(Q, f_0)\cup W^{\uu}_{r}(Q, f_0).
 $$ 
\begin{enumerate}
%\item[(c)]
%$W^{\ss}_j (S, f_0)  \nsubseteq W^{\uu} (N, f_0)$ and 
%$W^{\uu}_j (N, f_0)  \nsubseteq   W^{\ss} (S, f_0)$, for $j\in \{\ell,r\}$.
\item[(c)]
For $j=1, 2$, there are points 
%\[
%\begin{split}
%&X_j \in W^\uu(Q,f_0) \pitchfork \mathcal{F}^{\ss} (X_j, f_0)\quad \mbox{with} \quad 
%X_j \in W^{\st}_j (P, f_0)  \setminus W^{\ss} (P, f_0),\\
%&Z_j \in W^\ss(P,f_0) \pitchfork \mathcal{F}^{\uu} (Z_j, f_0)
% \quad \mbox{with} \quad
%Z_j \in W^{\ut}_j (Q, f_0)  \setminus  W^{\uu} (Q, f_0)
%\end{split}
%\]
\[
X_j \in W^\uu(Q,f_0) \pitchfork \mathcal{F}^{\ss} (X_j, f_0)\quad
 \mbox{and} \quad Z_j \in W^\ss(P,f_0) \pitchfork \mathcal{F}^{\uu} (Z_j, f_0),
\]
such that 
\[
\begin{split}
&X_1 \in W^{\ss}_{\ell} (P, f_0)  \setminus W^{\uu} (Q, f_0), 
\qquad \,\,\,\,\,\,\,\,
X_2 \in W^{\ss}_r (P, f_0)  \setminus W^{\uu} (Q, f_0),
\\
&Z_1 \in W^{\uu}_{\ell}(Q, f_0)  \setminus  W^{\ss} (P, f_0),
\qquad
Z_2 \in W^{\uu}_r (Q, f_0)  \setminus  W^{\ss} (P, f_0).
\end{split}
\]
\item[(d)]
$f_1(\sul)=\norte.$
\end{enumerate}
 We finally assume the transversality condition
\begin{enumerate}
\item[(e)]
$Df_1 (E^{\ss}_P) \oplus E^{\uu}_{Q} = T_Q \mathbb{S}^2$.
\end{enumerate}

Let $\Sigma_2 = \{0,1\}^\mathbb{Z}$.
Given $\underline{\ti}\in \Sigma_2$, we write
$\underline{\ti}=\underline{\ti}^- . \underline{\ti}^+=(\ldots i_{-2} i_{-1} . i_0i_1\ldots)$, where ``$.$'' marks
the $0$th position, $\underline{\ti}^-=(i_j)_{j<0} $,
and $\underline{\ti}^+=(i_j)_{j\ge 0}$.
%Given words $\bi=(i_0\ldots i_n),\bj=(j_0\ldots j_m)\in I^\ast$, we denote by $\bi\bj\eqdef(i_0\ldots i_nj_0\ldots j_m)$  the concatenation of them.
%Note that any finite word $\bi\in I^\ast$ can be ``continued'' to an element in $I^\bZ$. For example, we can consider the bi-infinite concatenation $\bi^\bZ=(\ldots\bi|\bi\bi\ldots)\in I^\bZ$.
%
Consider the {\em{shift map,}}
$$
\sigma \colon \Sigma_2 \to \Sigma_2,
\quad
\sigma (\, \underline{\ti} \,)= \underline{\ti}', \quad \mbox{where} \quad i_{j}' = i_{j+1} \quad \mbox{for every} \quad  j \in \mathbb{Z}.
%\sigma(\ldots i_{-1}.i_0i_1\ldots)\eqdef(\ldots i_0.i_1i_2\ldots).
$$
We consider the transition matrix
$$
\mathfrak{M}=
\left(
\begin{matrix} 1 & 1\\ 1 & 0
\end{matrix}
\right),
$$
the subspace $\Sigma_{\mathfrak{M}}$ of $\Sigma_2$ formed by sequences without two consecutive symbols equal to
$1$, and the restriction $\sigma_{\mathfrak{M}}$ of $\sigma$ to  $\Sigma_{\mathfrak{M}}$.
The skew-product map $F=F_{\mathfrak{M},\mathcal{F}}$  
associated to $\sigma_\mathfrak{M}$ and the family $\mathcal{F}=\{f_0,f_1\}$ is defined by
\begin{equation}
\label{e.toyzinho}
F \colon \Sigma_\mathfrak{M} \times \mathbb{S}^2 \to
\Sigma_\mathfrak{M} \times \mathbb{S}^2,
\quad
F (\,\underline{\ti}, x\,)= ( \sigma_\mathfrak{M}(\,\underline{\ti}\,), f_{i_0} (x)),
\quad
\mbox{where}
\quad \underline{\ti} =(\ldots  i_{-1}.i_0i_1\ldots).
\end{equation}

This map can be thought as a partially hyperbolic one, where the subshift provides a hyperbolic part
with (arbitrarily) strong contraction and expansion and the fiber maps provide the central dynamics.
In this way, the points 
\begin{equation}
\label{e.hoy}
\mathbf{Q}=(0^\mathbb{Z}, \norte)\qquad \mbox{and} \qquad
\mathbf{P}=(0^\mathbb{Z}, \sul) 
\end{equation}
 are ``saddles" of $F$ with different ``types of
hyperbolicity" (their indices 
 differ by two).  
 We see that $F$ exhibits a non-escaping cycle. Although we refrain from giving a translation of this notion for skew-products. We outline the key steps.

  We first see  that $\mathbf{P}$ and $\mathbf{Q}$ form  ``heterodimensional'' cycle. 
For that,
fix any
$$
X\in \mathbb{S}^2\setminus
\{\norte, \sul\}=W^\ut(\norte, f_0) \cap W^\st(\sul, f_0)
$$ 
and consider 
\begin{equation}
\label{e.twopoints}
\mathbf{X}\eqdef (0^\mathbb{Z}, X)
\qquad
\mbox{and}
\qquad
\mathbf{Y}\eqdef  ( 0^{-\mathbb{N}}.1 0^\mathbb{N},\sul).
\end{equation}
It follows that
\begin{equation}
\label{eq.skewcycle}
\mathbf{X}\in
W^\ut (\mathbf{Q}, F ) \cap W^\st (\mathbf{P}, F )
\quad
\mbox{and}
\quad
\mathbf{Y}\in
W^\st (\mathbf{Q}, F) \cap W^\ut (\mathbf{P}, F),
\end{equation}
obtaining the claimed cycle.

%Note that in this construction we have not considered any assumption about these two maps.

The definition of a non-escaping cycle involves a dominated structure and
 invariant manifolds. We adapt them to the skew-product setting.
We first discuss the strongest and strong invariant  sets analogous to the ones in 
\eqref{e.invariantstableunstablemanifolds}.  
We focus on the point $\mathbf{P}$, the analysis for $\mathbf{Q}$ is similar.
We let
\[
%\begin{split}
W^\sss (\mathbf{P},F)\eqdef W^\st (0^\mathbb{Z}, \sigma_\mathfrak{M}) \times \{{P}\},
\qquad
W^\ss (\mathbf{P},F)\eqdef  W^\st (0^\mathbb{Z}, \sigma_\mathfrak{M}) \times  W^\ss ({P}, f_0) .
%\end{split}
\]
 Hence
$$
W^\ss (\mathbf{P},F) \setminus W^\sss (\mathbf{P},F) =  W^\ss_\ell (\mathbf{P},F) \cup W^\ss_r (\mathbf{P},F)
$$
where 
$$
W^\ss_j (\mathbf{P},F)\eqdef W^\st (0^\mathbb{Z}, \sigma_\mathfrak{M}) \times   W^\ss_j ({P}, f_0),  \qquad j=\ell,r.
$$
The sets $W^\uuu (\mathbf{Q},F)$,  $W^\uu (\mathbf{Q},F)$,  $W^\uu_j (\mathbf{Q},F)$, $W^{\uu}_j ({Q}, f_0)$ are analogously defined\footnote{Note that in the symbolic case these sets are not connected.
The corresponding translation to diffeomorphisms provides connected sets.}. 
This completes the discussion about invariant sets.

Given the points $X_1,X_2$ and $Z_1,Z_1$ in condition~(c), we define 
\begin{equation}
\label{e.choosingpoints}
\mathbf{X}_1 \eqdef (0^\mathbb{Z}, X_1),  \quad \mathbf{X}_2 \eqdef (0^\mathbb{Z}, X_2),  \quad
\mathbf{Z}_1 \eqdef (0^\mathbb{Z}, Z_1),  \quad \mathbf{Z}_2 \eqdef (0^\mathbb{Z}, Z_2).
\end{equation}
We consider the contour $\Upsilon$ as in~\eqref{e.contourupsilon} 
with elements
$\mathbf{Q}, \mathbf{P},  \mathbf{X}_1,\mathbf{X}_2,  \mathbf{Z}_1,  \mathbf{Z}_2$ and  $\mathbf{Y}$.

Transversality conditions (c) and (e) 
imply that
it satisfies (NE1).
Condition (NE2) is guaranteed by 
the choices of the points $X_1,X_2, Z_1,Z_2$
and equation
\eqref{e.choosingpoints}. 

The previous construction is summarised as follows:

 \begin{prop}
 \label{p.Fhasnonescaping}
 The map $F$ has a non-escaping cycle associated with $\mathbf{P}$ and $\mathbf{Q}$.
 \end{prop}

%
%\begin{remark}
%\label{r.sobrecociclos}
%The previous construction also holds taking the full shift  space $\Sigma_2$ and the shift map $\sigma$. 
%\end{remark}

%\begin{remark}
%\label{r.withtransitions}
%In the previous construction, every sequence is admissible 
%
%sequences do not depend on the fiber space. In Section~\ref{ss.admissibleiterates}, we will consider a setting where the admissible sequences also depend on the base point in the fiber
%space. \lmargem{editar}
%\end{remark}
%

\subsection{Organization and Ingredients of the proof}
\label{ss.techniquesandtools}

To prove our results, we work on three levels: quotient two-dimensional dynamics (corresponding to the the center direction), skew-products, and diffeomorphisms. As in \cite{BonDia:08}, the strategy is hierarchical, but here we place the results under a more conceptual framework. For the first two levels, we introduce preblending and split blending machines, while the extension to diffeomorphisms follows directly, since the considered skew-products reproduce the semi-local dynamics of a diffeomorphism and Asaoka's machinery in \cite{Asa:22} applies. We also follow the approach in \cite{LiTur:24}, considering  a contour of the cycle and the dynamics in neighbourhoods of it.

In Section~\ref{s.quotientdynamics}, we introduce a class of perturbations in a contour of the cycle leading to what we call simple contours. In a neighbourhood of them, the transverse dynamic to the two-dimensional center direction is affine and preserves the strongest contracting and expanding directions, allowing us to consider the quotient dynamics. 
The role of the quotient dynamics here is reminiscent of the quadratic limit dynamics arising in the renormalization of homoclinic bifurcations, see \cite{PalTak:93}.
Perturbations of the diffeomorphisms that preserve the transverse dynamics are called adapted. The dynamics around the contour is then formulated in terms of skew-products.

In Section~\ref{ss.skewassociated}, we introduce a dictionary relating the quotient dynamics of the adapted perturbations to the corresponding skew-product. These relations allow us to identify strong homoclinic intersections of points with neutral derivatives, a configuration that leads to robust co-index one cycles (see \cite{BonDia:08}).
A key property is that these cycles are obtained without destroying the initial co-index two cycle.

Using the quotient dynamics, in Section~\ref{ss.occurrence}, for each central direction, we construct periodic  points with neutral derivative in that direction. On the one hand, this leads to the occurrence of strong homoclinic intersections (which on its turn leads to robust cycles of co-index one). On the other hand, this allows to
construct the rectangles that will serve as the base for building first preblending machines and thereafter split blending machines (see Section~\ref{s.splitblendermachines}), which are the germ of the co-index two cycles 
and extend the blending machines in \cite{Asa:22} to contexts with splittings with more than two bundles. These structures are robust and provide an abstract framework capturing the mechanism behind blenders constructed using dynamical rectangles and covering properties.  In Section ~\ref{s.splitblendermachines}, we revisit and adapt tools in ~\cite{Asa:22} for this purpose. Our approach is more flexible and considers non-hyperbolic  covering properties (covering of some projections, see Remark~\ref{r.dc=0}) which may enable for further potential applications.

In Section~\ref{s.generationofsplitblendingmachines}, we consider one-step skew-products, whose base and fiber dynamics are given by an affine horseshoe and a family of two-dimensional local diffeomorphisms. These maps arise naturally in our context. We also introduce preblending machines, which induce split blending machines within the skew-product framework.

In Section~\ref{s.connectingpreblenders}, following \cite{Asa:22}, we introduce adapted transitions between preblending machines and show that they give rise to a pair of split blending machines
(one machine is central attracting and the other central repelling) having a robust cycle.
The proof of Theorem A  is completed in Section \ref{s.transitionmaps}. 

Finally, in Section \ref{s.examples}, we present two classes of dynamics that exhibit non-escaping cycles:
robustly transitive diffeomorphisms with two-dimensional center and dynamics related to matrix cocycles in 
$\mathrm{GL}(3,\mathbb{R})$.

\begin{caveat}
Throughout, all perturbations considered are $C^1$ and arbitrarily small.
\end{caveat}

\section{Simple cycles, skew-products, and quotient dynamics}
\label{s.quotientdynamics}
Throughout this section,  
$f \colon M \to M$ denotes a diffeomorphism with a non-escaping contour  $\Upsilon$ as in Section~\ref{sss.elements}.  
We analyse the dynamics of $f$ near $\Upsilon$, for that we introduce the  semi-local dynamics in Section~\ref{ss.severalthings}.  
In Section~\ref{ss.simplecycles}, we prove that $f$ admits  $C^1$ perturbations preserving the contour, whose dynamics near the contour is affine.  
Following \cite[Section 3.1]{BonDia:08}, we refer to such contours as {\em simple}.  
Section~\ref{ss.associatedquotient} introduces the associated two-dimensional quotient dynamics.

\subsection{Dynamics around the cycle contour}
\label{ss.severalthings}
%We now introduce a more general setting, motivated by the examples
%\lmargem{co-index or coindex}
%in Section~\ref{ss.toymodel}.

Recall that the contour $\Upsilon$ has elements
elements 
$Q,P,X_1, X_2, Z_1, Z_2,Y$
see equation~\eqref{e.contourupsilon}
and the splitting in \eqref{e.fourbundles}.
For notational simplicity, we write $A=X_1$, $B=X_2$, $D=Z_1$, $E=Z_2$, and $Y=C$.

\subsubsection{Local dynamics and itineraries}
\label{sss.localitineraries}
Consider the set
\begin{equation} 
\label{e.alphabets}
\mathcal{C}
\eqdef \{\tq, \tp, \ta,\tb, \tc, \td, \te\},
\end{equation}
and, for $\tx\in \mathcal{C}$,
neighbourhoods $U_\tx$ of $X$ and numbers
$n_\tx \in \mathbb{N}$, called {\em{transition times,}} with $n_\tp=n_\tq=1$.

Define the {\em{local strong stable manifold $W^\ss_{\loc} (P)$ of $P$}}   as the connected component of 
$W^\ss (P,f)\cap U_\tp$ that contains $P$. Similarly, 
the {\em{local strong unstable manifold $W^\uu_{\loc} (Q)$ of $Q$}}
 is the connected component of 
$W^\uu (Q,f)\cap U_\tq$ that contains $Q$.

\begin{remark}[Neighbourhoods and transition times]
\label{r.transitionsetc}
After shrinking these neighbourhoods, replacing the heteroclinic points by some iterate in their orbits, and increasing the transition times, recalling conditions (NE1)-(NE2) in Section~\ref{sss.defnonesc}, we can assume that (see Figure~\ref{f.10:45})
\begin{enumerate}
\item
$f(U_\tq) \cap U_\tp =\emptyset = f(U_\tp) \cap U_\tq$,
%&f_\tq (U_\ta \cup U_\tb) \cap U_\tq  = 
%\emptyset \quad \mbox{and} \quad
%f_\tq (U_\tb) \cap U_\tq =\emptyset,
\item
$\bigcup_{\tx \in \{\ta,\tb, \td, \te\}} U_\tx \subset U_\tq$ and
$U_\tc \subset U_\tp$,
\item
$f(\bigcup_{\tx \in \{\ta,\tb, \td, \te\}} U_\tx) \cap U_\tq =\emptyset$ and 
$f(U_\tc) \cap U_\tp=\emptyset$,
\item
$\bigcup_{\tx \in \{\ta,\tb, \td, \te\}} f^{n_\tx} (U_\tx) \subset U_\tp$ and
$f^{n_\tc} (U_\tc) \subset U_\tq$, 
\item
$\left( f^{n_\ta} (U_\ta) \cup f^{n_\tb} (U_\tb) \right) \cap
W^\ss_{\loc} (P) =\emptyset$ and 
$\left( U_\td \cup U_\te \right) \cap
W^\uu_{\loc} (Q)=\emptyset$, and
\item
the family of sets below is pairwise disjoint,
$$
U_\tp, \, U_\tq, \, 
\left\{f^i (U_\tj)\colon i\in \{1, \dots, n_\tj-1\}, \tj \in  \{\ta,\tb, \tc, \td, \te\} \right\}.
$$
\end{enumerate}
\end{remark}

\begin{figure}[h]
\begin{overpic}[scale=0.31,
%grid,tics=5
]{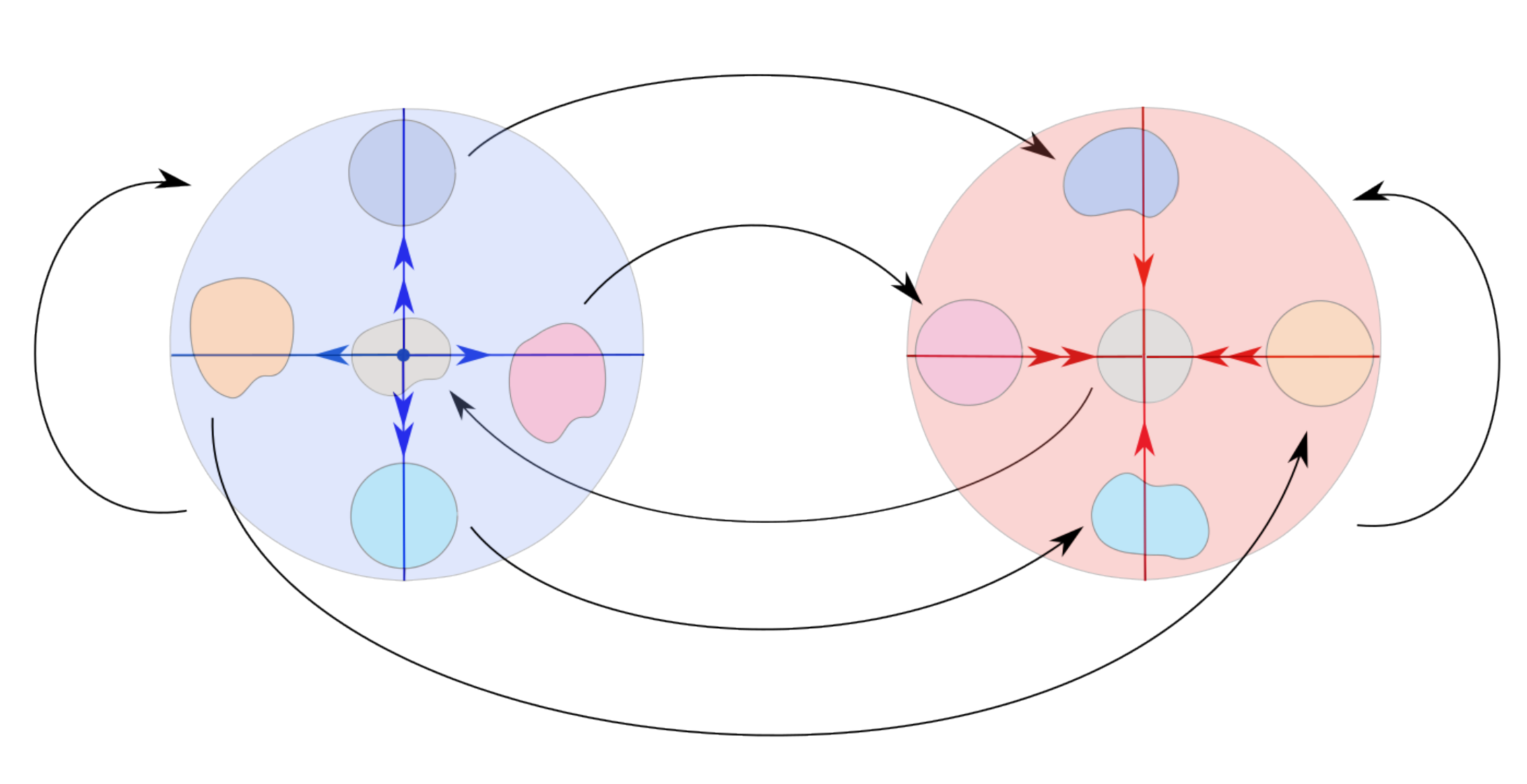}  
         \put(4,40){\large $f$}  
             \put(25.7,27.2){\large $\bullet$}  
                \put(21.5,23.5){\large $Q$}
                     \put(74.2,27){\large $\bullet$}  
                \put(77,24){\large $P$}    
              % \put(5,28){\large $f$} 
           \put(14,44){\Large $U_\tq$}  
         \put(14,29){\Large $U_\te$} 
            \put(34.5,25.5){\Large $U_\td$} 
               \put(24.5,39.5){\Large $U_\ta$} 
                 \put(24.5,16.5){\Large $U_\tb$} 
                    \put(70,30){\Large $U_\tc$} 
         \put(48,20){\Large  $f^{n_{\yc}}$}
             \put(48,39){\Large  $f^{n_{\yd}}$}
              \put(48,49){\Large  $f^{n_{\ya}}$}
         \put(47,12){\Large  $f^{n_{\yb}}$}
               \put(47,5){\Large  $f^{n_{\ye}}$}
         \put(94,40){\Large $f$}
           \put(83,44){\Large $U_\tp$} 
         %\put(60,13){\Large  $\psi_{\yb}$}
         %\put(78,15){\large $b<0$}
        % \put(78,40){\large $b>0$} 
         %\put(14,7){\large ${B}$}   
         \end{overpic}
         \vspace{-.3cm}
         \caption{Neighbourhoods and transition times}
         \label{f.10:45}
\end{figure}
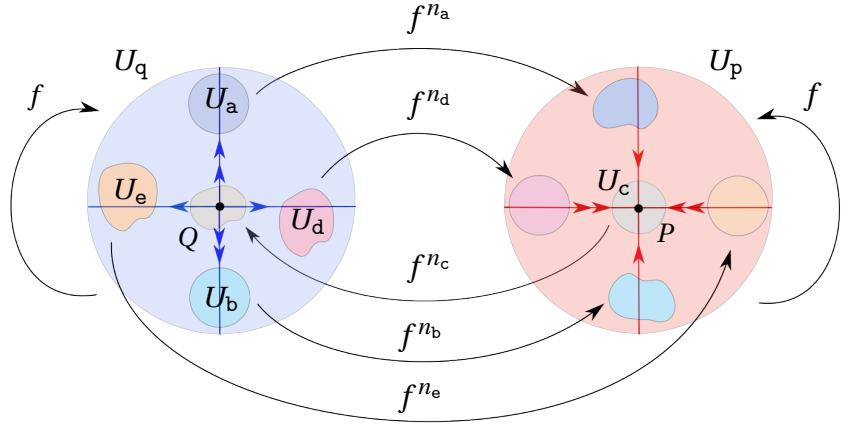

%
%
%\lmargem{esta condicion tiene que juntarse con la formula \eqref{e.severalconditions}
%para evitar repeticiones}
%
%
%\begin{enumerate}
%\item
%the neighbourhoods $U_\tx$, $\tx\in \mathcal{C}$ are pairwise disjoint,
%\item
%$U_\ta \cup U_\tb \cup U_\td \cup U_\te \subset f_\tq (U_\tq)$
%and 
%$U_\tc \subset f_\tp (U_\tp)$,
%\item
%$f^{n_\ta} (U_\ta)\cup f^{n_\tb} (U_\tb) \cup  f^{n_\td} (U_\td), f^{n_\te} (U_\te) \subset U_\tp$
%and 
%$f^{n_\tc} (U_\tc)\subset U_\tq$,
%\item
%$f(U_\tq) \cap U_\tp=  \emptyset = f(U_\tp) \cap U_\tq$, and
%\item
%for $\tx\in \mathcal{B}$, let $f^{-n_\tx}(\widehat U_\tx)\eqdef U_\tx$,
% the sets
%$
%\left\{ f^i (U_\tj): i\in \{0, \dots, n_\tj\}, \tj \in  \{\ta,\tb, \tc,\td, \te\} \right\}
%$
%are pairwise disjoint.
%\end{enumerate}
%

%
%\begin{enumerate}
%\item
%$U_\tq \cap U_\tp = \emptyset$,
%$U_\ta \cup U_\tb \cup U_\td \cup U_\te \subset U_\tq$, and
%$\widehat U_\td \cup \widehat U_\te \cup U_\tc \subset U_\tp$.
%\item the sets
%$$
%\left\{ f^i (U_\tj): i\in \{0, \dots, n_\tj\}, \tj \in  \{\ta,\tb, \tc,\td, \te\} \right\}
%$$
%are pairwise disjoint, and
%\item
%$f^{n_\ta} (U_\ta), f^{n_\tb} (U_\tb) \subset U_\tp$,
%$f^{n_\tc} (U_\tc)\subset U_\tq$,
%and 
%$ f^{-n_\td} (U_\td), f^{-n_\te} (U_\te) \subset U_\tq$.
%\end{enumerate}
%

We consider the maximal invariant set of $f$  in a neighbourhood $U_\Upsilon$ of the contour $\Upsilon$
 defined by
\begin{equation}
\label{e.maximalneighbofthecycle}
\Lambda_f (U_\Upsilon) \eqdef \bigcap_{i\in \mathbb{Z}} f^i (U_\Upsilon), \qquad
U_\Upsilon \eqdef  U_\tq \cup U_\tp \
\cup
\left(
\bigcup_{\tj \in \{\ta,\tb, \tc, \td, \te\}}  \left(  \bigcup_{i=1}^{n_\tj-1} f^i (U_\tj) \right) \right).
\end{equation}

\subsubsection{Symbolic spaces, shifts, and skew-products}
\label{sss.atoymodel}
To study  the dynamics in $\Lambda_f (U_\Upsilon)$,
we consider the set
$\mathcal{C}$ in \eqref{e.alphabets} as an
 alphabet,
 the symbolic space 
 $\Sigma_\mathcal{C} \eqdef \mathcal{C}^\mathbb{Z}$,
and 
the family of local diffeomorphisms, called {\em{local maps,}}
\begin{equation}
\label{e.thefamilyF}
\mathfrak{F}=\mathfrak{F}_\mathcal{C}\eqdef \left\{
f_\tx \eqdef f^{n_\tx}|_{U_\tx} \colon  
 \tx \in \mathcal{C} \right\}. 
 \end{equation}
 For these maps, we have properties similar to the ones in Remark~\ref{r.transitionsetc}.

\begin{notation}
\label{n.notationshifts}
Similarly to the notation in Section~\ref{ss.toynonescaping},
given $\underline{\ti} \in \Sigma_\mathcal{C}$, we write
$\underline{\ti} =\underline{\ti} ^-.\underline{\ti} ^+$, 
where ``$.$'' indicates
the $0$th position, $\underline{\ti}^-=(i_j)_{j<0} $,
and $\underline{\ti}^+=(i_j)_{j\ge 0}$.
We endow  $\Sigma_\mathcal{C}$ with  the 
standard metric that makes it a compact space
whose topology coincides with the one generated by cylinder sets.

We let
$\Sigma^\ast_\mathcal{C} \eqdef \bigcup_{k\geq 1} \mathcal{C}^k$.
An element $\ti=i_0\dots i_n\in \Sigma^\ast_\mathcal{C}$ is called a
\emph{{word.}} We denote by $|\ti|$ the {\emph{length}} of $\ti$ (the number $n+1$ of letters).
A word $\tj$ of the form $\tj=i_0\dots i_k$ with $k \leqslant n$ is called
a {\em{prefix of $\ti$.}}
\end{notation}

Consider the matrix  ${T}$ of possible transitions  between the sets $U_\tj$ by the maps in $\mathfrak{F}_\mathcal{C}$:
\begin{equation}
\label{e.matrixT}
 {T} \eqdef
\left(t_{\tx,\ty}
\right), \quad \mbox{where} \quad t_{\tx,\ty}=
\begin{cases}
 &1 \quad \mbox{if $f_\yx (U_\yx)\cap U_\yy \ne \emptyset$},\\
  &0 \quad \mbox{if $f_\yx (U_\yx)\cap U_\yy = \emptyset$}.
  \end{cases}
\end{equation}
Associated to $T$, 
consider the shift map $\sigma_{{T}}\colon \Sigma_{{T}}\to \Sigma_{{T}}$ and
define 
$\Sigma^\ast_{T}$ as the subset
$\Sigma^\ast_\mathcal{C}$ formed by the words that can be extended to  a sequence
in $\Sigma_{T}$. The words in $\Sigma^\ast_{T}$ are called {\em{admissible.}}

\begin{notation}%[Composition of maps]
\label{e.compositionsofmaps}
We use the following notation for the compositions of maps in $\mathfrak{F}$:
\begin{equation}
\label{e.notation}
f_{\ti}= f_{i_{0}  \cdots i_n} \eqdef
f_{i_{n}} \circ \cdots \circ f_{i_0}, \quad
\mbox{where} \quad 
 \ti = i_0\dots i_n\in \Sigma^\ast_{T},
\end{equation}
where the map $f_\ti$ is defined on its maximal domain (which may be an empty set)
\begin{equation}
\label{e.neighUi}
U_\ti = U_{\ti, \mathfrak{F}} \eqdef \{ X \in U_{i_0} \colon  f_{i_{0}  \cdots i_j} (X) \in U_{i_{j+1}}, \, j=0, \dots, n-1\}.
\end{equation}

%We let
%\begin{equation}
%\langle
%\mathfrak{F}
%\rangle^+ =\{ f_\ti \colon \ti \in \Sigma_T^\ast\}.
%\label{e.antestardequenunca}
%\end{equation}

Given $\ti \in \Sigma^\ast_{T}$,
the $\ti$-orbit of a point $X \in U_\ti$  is defined by
\begin{equation}
\label{e.orbitsegment}
\mathcal{O}_{\ti}(X) \eqdef \{ f_\tj(X) \colon \mbox{$\tj$ is a prefix of $\ti$}\}.
\end{equation}
We define similarly orbits of sets.
\end{notation}

We finally define the skew-product
\begin{equation}
\label{e.theskewproductfordiffeo}
F=F_{{T}, \mathfrak{F}} \colon \Lambda_{{T}, \mathfrak{F}} \to   \Lambda_{{T}, \mathfrak{F}},
\quad F(\,\underline{\ti}\,,X)=(\sigma_{T} (\,\underline{\ti}\,), f_{i_0} (X)),
\end{equation}
where
\begin{equation}
\label{e.setfordiffeo}
\Lambda_{{T}, \mathfrak{F}} \eqdef \Sigma_{T} \times U_\mathcal{C} \qquad \mbox{with} \qquad U_\mathcal{C} \eqdef \bigcup_{\ti \in \mathcal{C}} U_{\ti}.
\end{equation}

\subsection{Double type I contours}
\label{ss.doubletypeone}

The definition of type I contours of co-index one in \cite{LiTur:24} involves the so-called Shilnikov cross-form coordinates from \cite{Shi:69} and invokes \cite[Lemma 6]{GonShiTur:08} to obtain formulas for the composition of local maps. Here, we adapt this idea to express them in terms of local coordinates.

\begin{defi}[Double type I contour]
\label{d.doubletypeI} 
%The contour $\Upsilon$ of a diffeomorphism 
%$f$ having a non-escaping cycle associated with saddles $P$ and $Q$ of u-indices $\ell$ and 
%$\ell+2$ 
%as  in  Section~\ref{ss.severalthings} is of {\em double type I} if
Let $f$ be a diffeomorphism having a non-escaping cycle associated with saddles $P$ and $Q$ of u-indices $\ell$ and 
$\ell+2$ with contour $\Upsilon$ as  in  Section~\ref{ss.severalthings}. The contour $\Upsilon$  is of {\em double type I} if
 there are 
$\bar{A}, \bar{D} \in \Upsilon$ and
local coordinates
at $U_\tp$ and $U_\tq$ such that, for some $\rho>1$, 
\begin{enumerate}
\item
The local invariant manifolds of $P$ and $Q$ in \eqref{e.invariantstableunstablemanifolds} are of the form
%\[
%\begin{array}{rlrl}
%W_\loc^\sss(P) &= [-\rho,\rho]^{n-\ell-2} \times \{0^{\ell+2}\}, 
%&\quad 
%W_\loc^\ss(P) &= [-\rho, \rho]^{n-\ell-1} \times \{0^{\ell+1}\}, \\
%W_\loc^\st(P) &= [-\rho,\rho]^{n-\ell} \times \{0^{\ell}\}, 
%&\quad 
%W_\loc^\ut(P)&= \{0^{n-\ell-2}\} \times [-\rho,\rho]^{n-\ell-2}, \\
%W_\loc^\uuu(Q) &= \{0^{n-\ell}\} \times [-\rho,\rho]^{\ell}, 
%&\quad 
%W_\loc^\uu(Q) &= \{0^{n-\ell-1}\} \times [-\rho,\rho]^{\ell+1}, \\
%W_\loc^\ut(Q) &= \{0^{n-\ell-2}\} \times [-\rho,\rho]^{\ell+2}, 
%&\quad 
%W_\loc^\st(Q) &= [-\rho, \rho]^{n-\ell-2} \times \{0^{\ell+2}\};
%\end{array}
%\]
\[
\begin{array}{rlrl}
W_\loc^\ut(P)&=\{0^{n-\ell}\} \times [-\rho,\rho]^{\ell}, 
&\quad 
W_\loc^\st(Q) &=[-\rho, \rho]^{n-\ell-2} \times \{0^{\ell+2}\},
\\
W_\loc^\st(P) &= [-\rho,\rho]^{n-\ell} \times \{0^{\ell}\}, 
&\quad 
W_\loc^\ut(Q) &= \{0^{n-\ell-2}\} \times [-\rho,\rho]^{\ell+2},
\\
W_\loc^\ss(P) &= [-\rho, \rho]^{n-\ell-1} \times \{0^{\ell+1}\}, 
&\quad 
W_\loc^\uu(Q) &= \{0^{n-\ell-1}\} \times [-\rho,\rho]^{\ell+1},
\\
W_\loc^\sss(P) &= [-\rho,\rho]^{n-\ell-2} \times \{0^{\ell+2}\}, 
&\quad 
W_\loc^\uuu(Q) &= \{0^{n-\ell}\} \times [-\rho,\rho]^{\ell},
\end{array}
\]
\item
the points  $\bar{A}, \bar{D}\in U_\tq$ and $f^{n_\ya}(\bar{A}), f^{n_\yd} (\bar{D})\in U_\tp$ are of the form
\[
\begin{array}{rlrl}
\mathrm{(i)} \,\,\,\,\, \bar{A} &= (0^{\sss}, 0, a_2, a^{\uuu}), \quad a_2>0,  &\qquad
f^{n_\ya}(\bar{A}) &= (\bar a^{\sss}, \bar a_1, \bar a_2, 0^{\uuu}), \quad \bar a_2>0,
 \\
\mathrm{(ii)} \,\,\,\,\, \bar{D} &= (0^{\sss}, d_1, d_2, d^{\uuu}), \quad d_1>0,  &\qquad
f^{n_\ya}(\bar{D}) &= (\bar d^{\sss}, \bar d_1 , 0, 0^{\uuu}), \quad \bar d_1>0;
\end{array}
\]
\item
the derivative of the transition map $f^{n_\yc}$ is such that
\[
\begin{split}
\mathrm{(i)} \,\,\,\,\, &Df^{n_\yc} (C) (0^\sss, 1,0, 0^\uuu) \cdot (0^\sss, 1,0,0^\uuu)>0, \\
\mathrm{(ii)} \,\,\,\,\, &Df^{n_\yc} (C) (0^\sss, 0,1, 0^\uuu) \cdot (0^\sss, 0,1,0^\uuu)>0.
\end{split}
\]
\end{enumerate}
To $\Upsilon$ we associate the integer $\iota (\Upsilon)=\ell$ where $\ell$ is the $\ut$-index of $P$.
\end{defi}

\begin{lem} 
\label{l.everyneistypeI}
Given any neighbourhood $U$ of a non-escaping contour of a diffeomorphism $f$, there exists an arbitrarily small
perturbation $g$ of $f$
with a non-escaping cycle associated to the saddles of the initial one
 having a double type I contour contained in $U$.
\end{lem}

\begin{proof}
Recall the heteroclinic points $A,B,C,D$ of a non-escaping cycle as in Definition~\ref{d.nonescapinghc}.
%In the statement of the lemma, we take $A \in \{A,B\}$ and $D \in \{D,E\}$. 

We can select local charts $\eta_P$ and $\eta_Q$ satisfying (1).
These points fall into the cases (2)(i) ($A$ or $B$) 
or (2)(ii) ($D$ or  $E$), but their coordinate signs are not a priori yet determined.
We now show how to control them.

In these coordinates, consider the central involutions  
\[
\iota_1(x^\st, x_1, x_2, x^\ut) \mapsto (x^\st, -x_1, x_2, x^\ut), 
\qquad  
\iota_2(x^\st, x_1, x_2, x^\ut) \mapsto (x^\st, x_1, -x_2, x^\ut).
\]  
The conditions in (1) remain valid after replacing 
$\eta_P$ with $\eta_P \circ \iota_i$ and $\eta_Q$ with $\eta_Q \circ \iota_i$.  

The analysis of conditions (2)(i), (3)(i) (involving $A, B$ and the center coordinate $x_2$) 
and conditions (2)(ii), (3)(ii) (involving $D,E$ and the coordinate $x_1$) 
are independent, and can be solved using the involutions $\iota_2$ and $\iota_1$, respectively.  
Hence we focus only on case (i). In the chosen coordinates, write
\[
 X = (0^{\sss}, 0, x_{2}, x^{\uuu}), \qquad
 f^{n_\yx}(X) = (\bar x^{\sss}, \bar x_{1}, \bar x_{2}, 0^{\uuu}), \quad X \in \{A,B\}.
\]

The involutions $\iota_2$ allows us to assume that for instance
$a_2>0$, and that $Df^{n_\tc}$ satisfies (3)(i).  
There are four possible cases:
\begin{enumerate}
\item $\bar{a}_2 > 0$.
\item $\bar{a}_2 < 0$, $b_2 > 0$, $\bar b_2 > 0$.
\item $\bar{a}_2 < 0$, $b_2 < 0$, $\bar b_2 > 0$.
\item $\bar{a}_2 < 0$, $b_2< 0$, $\bar b_2 < 0$.
\end{enumerate}
In case (1), take $\bar{A}=A$.  
In case (2), take $\bar{A}=B$.  
In case (4), replace $\eta_P$ and $\eta_Q$ by $\eta_P \circ \iota_2$ and $\eta_Q \circ \iota_2$ and then take $\bar{A} = B$
(note that this does not modify (3)(i)).
It remains to consider case (3), where we need to consider a new contour and perhaps doing
a perturbation: 
it is enough to consider double round points (are depicted in Figure~\ref{f.doubleroundpoints}) 
in a small neighbourhood of $A$ having, in the local coordinates,
a return of the 
form
$$
f_\yb f_\tq^m f_\yc f_\tp^n f_\ya, 
$$
where $m,n$ are chosen such that admissible domain 
 $U_{\ta \tp^n \tc \tq^m\tb} \ne \emptyset$
 (hence  $U_{\ta \tp^n \tc \tq^m}\ne \emptyset$), which  may involve local perturbations $f_\tp$ and $f_\tq$.
This allows to take a new heteroclinic point $\bar{A}=A'$ having a return $n_{\ya'}=\ta \tp^n \tc \tq^m\tb$ satisfying
(2)(i).

\begin{figure}[h]
\begin{overpic}[scale=0.31,
%grid,tics=5
]{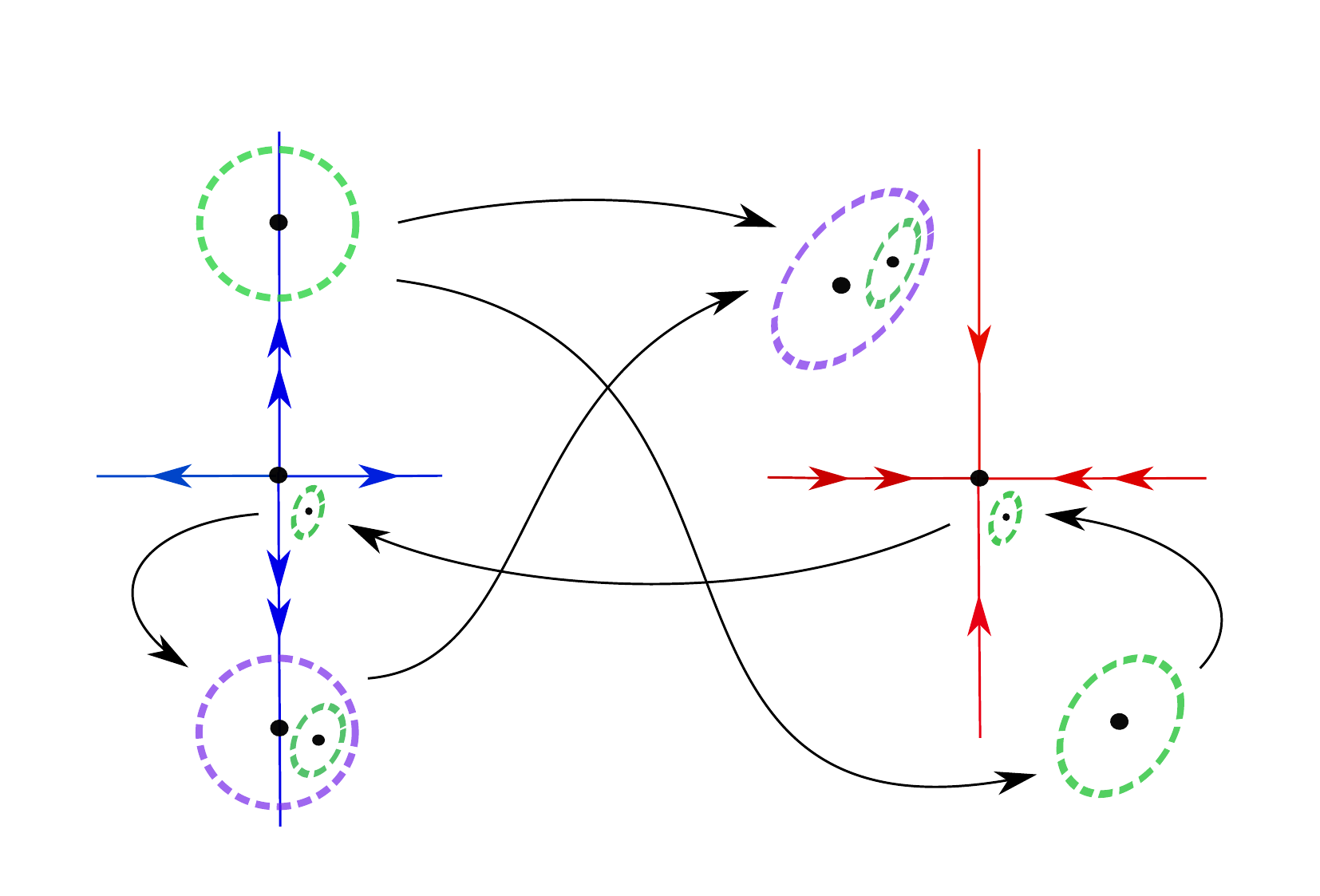}  
         \put(16.7,50){\large $A$}  
              \put(16.7,11.5){\large $B$}  
                \put(16,34){\large $Q$}
             \put(76,34){\large $P$}    
               \put(31.8,57){\Large  $f_\yb f_\tq^m f_\yc f_\tp^n f_\ya$}
         \put(3,25){\Large  $f^{m}_{\tq}$}
            \put(94,25){\Large  $f^{n}_{\tp}$}
             \put(36.7,35){\Large  $f_{\yb}$}
                       \put(45,18){\Large  $f_{\yc}$}
                                        \put(63,14){\Large  $f_{\ya}$}
         \end{overpic}
         \vspace{-.5cm}
         \caption{Case (3)}
         \label{f.doubleroundpoints}
\end{figure}

Note that all previous conditions do not change after introducing the involution $\iota_1$ in the coordinates.
This implies that the analysis for the points $D,E$ can be performed independently. The proof is now complete.
\end{proof}

In view of Lemma~\ref{l.everyneistypeI}, in what follows  we will  deal with double type I contours,
considering only the points $\bar{A}=A$ and $\bar{D}=D$ and the corresponding transitions. Then, we remove from 
the alphabet $\mathcal{C}$ the symbols $\tb$ and $\te$, as well as, the corresponding transition maps. We still denote this alphabet by $\mathcal{C}$ and the family of local maps by $\mathfrak{F}$.

\subsection{Simple contours}
\label{ss.simplecycles}
The arguments in this section are standard in the $C^1$ topology and  follow \cite[Section 3.1]{BonDia:08}, 
hence the details are omitted.
Our first result is a translation of \cite[Proposition 3.5]{BonDia:08}.
% and the proof is similar and just sketched.
The only difference is that here we consider two transverse heteroclinic points
in $W^\st(P) \pitchfork W^\ut (Q)$ whose orbits can be analysed independently and also deal with two central directions
(indeed, this is not an issue since in \cite[Proposition 3.5]{BonDia:08} this question is also addressed).

%For what is below, we recall in \eqref{e.thefamilyF}, the alphabet  $\mathcal{C}= \{\tq, \tp, \ta,\tb, \tc, \td, \te\}$
%and the local maps $f_\tx$, $\tx\in \mathcal{C}$.

\begin{lem}[Simple double type I contours]
\label{l.simple}
Let  $\Upsilon$ be a double type I contour of a  cycle of $f$
 whose
elements are
$Q,P,A, C,D$.
After replacing the heteroclinic  points $A,C,D$ for points in their orbits, 
shrinking the neighbourhoods
$U_\tx$, $\tx\in \mathcal{C}$, and increasing the numbers $n_\tx$, $\tx \in \mathcal{C}\setminus
\{\tp, \tq\}$, there is a perturbation  $g$ of $f$ having the same double type I contour
 $\Upsilon$ 
whose  local maps $g_\tx$, $\tx \in \mathcal{C}$, 
preserve the splitting $E^\sss\oplus E^\ct_1\oplus E^\ct_2 \oplus E^\uuu$ in \eqref{e.domaincyclebis} and such that,
in the local coordinates in $U_\tp$ and $U_\tq$:
\begin{enumerate}
\item[\emph{(1)}]
$g_\tp$ and $g_\tq$ are linear in $U_\tp$ and $U_\tq$;
\item[\emph{(2)}]
for $\tx=\ta,\tc,\td$, 
the maps
$g_\tx \colon U_\tx \to g_\tx (U_\tx)$ are affine and their
derivatives restricted to $E^\ct_1 \oplus E^\ct_2$ are of the form 
%\qquad \mbox{for} \quad \tx=\td \quad \mbox{where}\quad  \epsilon_\tx\ne 0,\\
$$
Dg_\ta (A)= \left(
\begin{matrix} \epsilon_\ta & 0\\ 0 & \pm 1
\end{matrix}
\right), \qquad
Dg_\td (D)
=
\left(
\begin{matrix} \pm 1 & 0\\ 0 & \epsilon_\td
\end{matrix}
\right), 
\qquad
Dg_\tc (C)=\left(
\begin{matrix}
a_{11} & a_{12}\\
a_{21} & a_{22}
\end{matrix}
\right),
$$
where $ \epsilon_\ta,  \epsilon_\td, a_{11}, a_{22} \ne 0$.
%\end{split}
%\]
\end{enumerate}
\end{lem}

Following \cite[Definition 3.4]{BonDia:08},
we call a cycle as in Lemma~\ref{l.simple} {\em{simple double type I.}} To avoid heavy notation, in what follows we call these cycles
just simple.

It follows the sketch of the proof of the lemma.
The first step is to linearize (after a perturbation) in neighbourhoods of $Q$ and $P$. By shrinking the neighbourhoods $U_\tq$ and $U_\tp$, we may assume that the maps $f_\tq$ and $f_\tp$ are both linear.

The second step uses domination. In the neighbourhoods $U_\tq$ and $U_\tp$, we consider the invariant splittings given by the derivatives. We replace the point $A$ by a point of the form $A' = f^{-i}(A)$ and redefine $n_\ta$ as $n_{\ta'} \eqdef i + n_\ta + j$, for suitably large $i$ and $j$. After a new perturbation, we can assume that the linearizing splitting in $U_\tq$ is mapped into the linearizing splitting in $U_\tp$.
Adjusting the values of $i$ and $j$, and taking into account that
\begin{itemize}
\item
    $Df^i$ expands $E^\ct_2$ in $U_\tq$,
\item
    $Df^{n_\ta}$ has derivative along $E^\ct_2$ bounded above and below, and
 \item
  $Df^j$ contracts $E^\ct_2$ in $U_\tp$,
\end{itemize}
we can choose $i$ and $j$ such that the total derivative along $E^\ct_2$ has modulus close to one. 
As the number of iterations can be done arbitrarily large, a final small perturbation yields derivative of modulus one.
From this point on, we replace $A$ by $A'$ and $n_\ta$ by $n_{\ta'}$. The argument for the other points proceeds in a similar way. 
The fact that in the derivative of $Dg_\yc(0,0)$ the numbers $a_{11}, a_{22}\ne 0$ follows from the transversality  conditions.
This completes our sketch.
 
 \begin{remark}
 In the case of diffeomorphisms, the matrix corresponding to $Dg_\tc$ can be chosen diagonal. Aiming applications to skew-products,
 we stated the previous, a bit more general, version. Note that in the skew-product,  the fiber component 
 of the
 orbit of the heteroclinic point going from $P$ to $Q$ is finite. Therefore, the previous argument can not applied.
 \end{remark}
 
\subsection{Quotient dynamics}
\label{ss.associatedquotient}
Recall the alphabet $\mathcal{C}$ in~\eqref{e.alphabets} and the local maps $f_{\tx}$ in~\eqref{e.thefamilyF}. 
In what follows, we assume the contour $\Upsilon$ of $f$ is simple cycle. For each $\tx\in \mathcal{C}$,
consider the quotient of the neighbourhood 
 $\qU_\tx$  of
$U_\tx$ by the $E^\ss\oplus E^\uu$,
that we identify with subsets of $\mathbb{R}^2$. 
The quotient points are denoted by $\qQ, \qP, \qA, \qC, \qD$. 
We denote by $\varphi_\tx$ the quotient maps induced by $f_\tx$ and consider the quotient family of $\mathfrak{F}$,
\begin{equation}
\label{e.thequotientfamily}
{\mathcal{F}} =\{ \varphi_{\tx} \colon \tx \in \mathcal{C} \}, \qquad
 \varphi_{\tx} \colon \qU_\tx \to \mathbb{R}^2.
 \end{equation}
By Lemma~\ref{l.simple}, the maps 
in $\mathcal{F}$
have the following form:

\noindent
(QD1) {\em{Local dynamics at $\qP$ and $\qQ$.}}  The points $\qQ$ and $\qP$ are identified with $(0,0)$ and
$\varphi_\tq \colon \qU_\tq  \to \mathbb{R}^2$ and $ \varphi_\tp \colon  \qU_\tp  \to \mathbb{R}^2$,
 are linear maps 
\begin{equation}
\label{e.fzero}
 \varphi_\tq(x,y)= (\beta_1\, x, \,\beta_2 \, y) 
\qquad \mbox{and} \qquad 
 \varphi_\tp(x,y)=
(\alpha_1\, x, \,\alpha_2 \, y),
\end{equation}
where $0< \vert \alpha_1 \vert< \vert\alpha_2 \vert< 1< \vert \beta_1 \vert< \vert\beta_2 \vert$.

\noindent
(QD2) \emph{Transverse heteroclinic transitions.}
We can assume that 
\begin{equation}
\label{e.fzerodos}
\qA=(0,1)\in \qU_\ta \subset  \qU_\tq  \qquad \mbox{and} \qquad
\varphi_\ya (\qA)\eqdef \qA' =(0,1) \in \qU_\tp ,
\end{equation}
and that $\varphi_\ya \colon \qU_\ta \to \qU_\tp$ is an affine map of the form
\begin{equation}
\label{e.fzerotres}
\varphi_\ya (x,1+y) = (\epsilon_\ta x, 1 \pm y),\quad \mbox{where}\quad \epsilon_\ta\ne 0.
\end{equation}

We let $\qD'= \varphi_\yd^{-1} (\qD)=(1,0)$ and $\qU_\td^\prime =\varphi_\yd^{-1} (\qU_\yd)$.
We assume similar conditions for the transition $\varphi_\yd^{-1} \colon \qU_\yd^\prime \to \qU_\tp$.

\smallskip

\noindent
(QD3)
\emph{Nontransverse heteroclinic transitions.}
\label{sss.cyclemap}
 In the local charts, 
the point $\qC$ is identified with $(0,0)$ and
the cycle map $\varphi_\yc\colon \qU_\yc\to \qU_\tq$ is given by 
\begin{equation}
\label{e.Fic}
\varphi_\yc (x,y)=  
\left(
\begin{smallmatrix}
a_{11} & a_{12}\\
a_{21} & a_{22}
\end{smallmatrix}\right) \bigl (\begin{smallmatrix}x\\ y\end{smallmatrix}\bigr),\quad \mbox{where}\quad a_{11}, a_{22}>0.
\end{equation}

\begin{rem}
\label{r.qtransitions}
The family
$\mathfrak{F}=\{f_\tx \colon \tx\in \mathcal{C}\}$  and its quotient family 
$\mathcal{F}=\{\varphi_\tx \colon  \tx\in \mathcal{C}\}$ have the same transition matrix $T$ defined in \eqref{e.matrixT}.
\end{rem}

\section{Skew-products and quotient dynamics}
\label{ss.skewassociated}
Throughout this section, $f$ denotes a diffeomorphism with a simple contour as in Section~\ref{ss.simplecycles}.  
We begin by defining a special type of perturbation of $f$ and its associated skew-product (see Section~\ref{ss.adapted}).  
In Section~\ref{ss.moreonquotient}, we derive properties of the dynamics of diffeomorphisms from those of the skew-product.  
Using these relations and strong homoclinic points, in Section~\ref{ss.strongesthomoclinic}, we state sufficient conditions for the occurrence of robust cycles of co-index one.

% and  of
%the local maps $\mathfrak{F}$ in \eqref{e.thefamilyF}.

\subsection{Adapted perturbations} 
\label{ss.adapted}
Recall the family $\mathfrak{F}$ in~\eqref{e.thefamilyF}, the point $C$ in Section~\ref{ss.severalthings}, and the alphabet $\mathcal{C}$ in~\eqref{e.alphabets}.

\begin{defi}
\label{d.adaptedperturbation}
An {\em adapted perturbation of $f$ with respect to $\mathfrak{F}$} is a $C^1$ perturbation $g$ of $f$ such that:
\begin{itemize}
  \item $g$ and $f$ coincide on the set $\Upsilon \setminus \bigcup_{i\in \mathbb{Z}} f^{i}(C)$;
  \item $g$ coincides with $f$ outside a neighbourhood of $\Upsilon$;
  \item $g$ preserves the foliations $\mathcal{F}^\sss$ and $\mathcal{F}^\uuu$ in a neighbourhood of $\Upsilon$; and
  \item 
   for every $\tx \in \mathcal{C}$, the quotient map $\psi_\tx$ of $g_\tx$ (defined in the obvious way) is a $C^1$ perturbation of  $\varphi_\tx$.
 \end{itemize} 
Moreover,  when $g$ coincides with $f$ in the whole set $\Upsilon$, we say that {\em{preserves the contour}}, and hence the cycle 
which is
also of double type I (but not simple, since affinity in the central direction may be lost).
\end{defi}

For an adapted perturbation $g$ of $f$ 
 with respect to $\mathfrak{F}$
(in what follows we will omit this dependence)
 we consider the family $\mathfrak{G}= \{ g_\tx \colon \tx \in \mathcal{C}\}$,
where $g_\tx$ is defined in the obvious way, the sets $\Lambda_{T, \mathfrak{G}}$, 
and the skew-product $F_{T, \mathfrak{G}}$ defined as in \eqref{e.setfordiffeo} and \eqref{e.theskewproductfordiffeo}.
We also consider the maximal $g$-invariant set  in $U_\Upsilon$,
\begin{equation} 
\label{e.gmaximal}
\Lambda_g (U_\Upsilon) \eqdef \bigcap_{i\in \mathbb{Z}} g^i (U_\Upsilon).
\end{equation}

Since $g$ preserves the splitting, 
it allows us to consider the corresponding quotient family
\begin{equation}
\label{e.thehappyfamilyG}
 \mathcal{G} = \mathcal{G}_g= \{\psi_\tx \colon \tx \in \mathcal{C}\},
 \quad 
  \mbox{with}
  \quad 
  \psi_{\tx} \colon\qU_{\tx,\mathcal{G}}\to \mathbb{R}^2.
 \end{equation}
By construction, the quotient maps satisfy
\[
\psi_\yx(\qX) = \varphi_\yx(\qX), \qquad \mbox{for every}
  \quad  \yx  \in \mathcal{C}.
\]
For simplicity, and with a slight abuse of notation, we say that the family 
{\em{$\mathcal{G}$  
is an adapted perturbation of the family $\mathcal{F}$ in~\eqref{e.thequotientfamily}.}}
Also, again with an abuse of notation,
we denote the sets $\qU_{\tx,\mathcal{G}}$ in \eqref{e.thehappyfamilyG} 
corresponding to $\mathcal{G}$
 by $\qU_\tx$.
%($\mathcal{G}$ is the quotient map of an adapted perturbation $g$ of a diffeomorphism $f$ whose quotient dynamics is $\mathcal{F}$.
%

Consider the matrix $T$ in \eqref{e.matrixT}, similarly as in
\eqref{e.setfordiffeo} and
 \eqref{e.theskewproductfordiffeo}, we consider the set
\begin{equation}
\label{e.theskewproductin}
\Gamma_{{T},\mathcal{G}} \eqdef \{ ( \,\underline{\ti}\,, \qX)\in \Sigma_{T}\times \qU_\mathcal{C}
\colon \qX\in \qU_{i_0}\}, \qquad \mbox{where} \qquad \qU_\mathcal{C} \eqdef \bigcup_{\ti \in \mathcal{C}} \qU_{\ti},
\end{equation}
and the skew-product
\begin{equation}
\label{e.theskewproductforquotien}
\Psi=\Psi_{{T}, \mathcal{G}} \colon \Gamma_{{T},\mathcal{G}} \to  \Sigma_{T}\times \qU_\mathcal{C},
\quad \Psi(\,\underline{\ti}\,,\qX)=(\sigma_{T} (\,\underline{\ti}\,), \psi_{i_0} (\qX)).
\end{equation}
The fiber dynamics of this skew-product corresponds to the model families in \cite[Section 3.2]{BonDia:08}\footnote{We chose to follow the skew-product approach, which is somewhat more abstract and general.}.
We will now focus on its study.

\subsection{A dictionary between quotient and global dynamics}
\label{ss.moreonquotient}
In what follows, $g$ stands for an adapted perturbation of $f$ as in the previous section.  
%We establish conditions under which $g$ exhibits cycles that are either non-escaping or robust of co-index one (see Sections~\ref{sss.nonescaping-secondary} and \ref{sss.strongesthomoclinic}).  
We derive relations between the dynamics of $g$ and those of its associated skew-product $\Psi$, whose fiber maps are denoted by $\psi_\tj$, with $\tj \in \Sigma^{\ast}_T$.

\begin{remark}
\label{r.firstword}
A periodic point $(\,\underline{\ti}\,,\qX)$ of $\Psi$ is \emph{hyperbolic} if 
$\underline{\ti}= \ti^\mathbb{Z}$ for some word $\ti \in \Sigma^{\ast}_T$, $\psi_{\ti}(\qX) = \qX$, and all eigenvalues of 
$D\psi_\ti (\qX)$ are different from $\pm 1$. 
Note that, due to domination, these eigenvalues are real and have distinct moduli.  
This point may be attracting, a saddle, or repelling. 
%We refer to the point 
%We denote by as a {\em{central attracting, saddle, or repelling point,}} 
We denote by
$W^\ast_{\loc}(\qA,\psi_\ti)$, $\ast = \st, \ut$, the corresponding local invariant manifolds of $\qA$. 
If $\qA$ is attracting, then $W^\st_{\loc}(\qA,\psi_\ti)$ is a disk and 
$W^\ut_{\loc}(\qA,\psi_\ti) = \{\qA\}$, and  
from the previous comments, $W^\ss_{\loc}(\qA,\psi_\ti)$ is a well defined curve. 
The analogous description holds for saddle and  repelling points.  
\end{remark}

We can also define heteroclinic and homoclinic intersections for hyperbolic periodic points. 
Specific examples of these relations were given in Section~\ref{ss.toynonescaping}. 
The remark bellow follows as in \cite[Section~2.1]{DiaEstRoc:16} and hence 
details are omitted.

\begin{remark}[Hetero and homoclinic relations]
\label{r.homohetero}
Let $(\ti^\mathbb{Z}, \qA)$ and $(\tj^\mathbb{Z}, \qB)$ be hyperbolic periodic points of $\Psi$ 
such that there exists an admissible word $\tr$ with 
\begin{equation}
\label{e.oneintersection}
\qA \in \qU_\tr \qquad \mbox{and} \qquad
\psi_\tr(\qA) \in W^\st_{\loc}(\qB,\psi_\tj).
\end{equation}
Then, exactly as in \eqref{e.twopoints} and \eqref{eq.skewcycle}, 
\begin{equation}
\label{e.heteroclinicskew}
(\ti^{-\mathbb{N}}. \tr \tj^{\mathbb{N}}, \qA)\in
W^\ut
\left(
(\ti^\mathbb{Z}, \qA), \Psi \right)
\cap W^\st
\left(
(\tj^\mathbb{Z}, \qB)
, \Psi \right).
\end{equation} 
This allows us to define homoclinic relations  
(again, no transversality is required, trans\-versality-like properties follow from the dominated splitting of the central bundle).  Note \eqref{e.heteroclinicskew} provides intersections in the set  $\Gamma_{{T}, \mathcal{G}}$ (see equation~\eqref{e.theskewproductin}). 

When the points have different $\ut$-indices this leads to heterodimensional cycles,
which may have co-index one or two.
\end{remark}

Given $g$ as above, the following two remarks establish a dictionary between its quotient, its skew-product, and its dynamics.

\begin{remark}
\label{r.secondarydynamics}
To each point in $(\,\underline{\ti}\,, \qX)\in \Gamma_{{T}, \mathcal{G}}$ corresponds a point in 
$X=X_{(\,\underline{\ti}\,, \qX)} \in \Lambda_g (U_\Upsilon)$ with the ``same itinerary''   (we refrain to formalize this).
If $(\,\underline{\ti}\,, \qX)$ is periodic and hyperbolic,  the same holds for $X$, where its $\ut$-index
is the sum of $\dim E^\uuu$ and the $\ut$-index of $\qX$. 
Intersections between the invariant sets of the skew-product translate to the same intersection for $g$, where the
intersection points are contained in $\Lambda_g (U_\Upsilon)$. 
\end{remark}

%
%
%Consider an adapted perturbation $g$ of $f$ \lmargem{adapted may preserve or not the cycle. work on this, as it is defined the map preserves the
%cycle.... perhaps two types of adapted....}
%with family of local maps $\mathcal{G}$ \lmargem{unify terminology before}
%Associated to $g$ there is defined the skew products $F_{T, \mathcal{G}}$ and
%its quoutient $\Psi_\mathcal{G}$. \lmargem{complete}
%
%
%
%
%Every periodic point $(\underline{\ti}, \qA)$ of $\Psi$ satisfies $\underline{\ti} = \ti^{\mathbb{Z}}$ for some word $\ti \in \Sigma_T$
%and $\psi_{\ti} (\qA)=\qA$. To this point corresponds a periodic point $A\in ??$ of $g$ of period $|\ti|$ and 
%$\ut$-index in $\{\ell, \ell+1, \ell+2\}$, where $\ell$ is the $\ut$-index of $P$. By domination hypotheses, the derivative
%$D\psi_{\ti}(\qA)$ has two real eigenvalues with different modulus. We are interested in the cases:
%\begin{enumerate}
%\item 
%Both eigenvalues of $D\psi_{\ti}(\qA)$ have modulus greater than one. In such case we can define the one dimensional local strong unstable 
%manifold $W^\uu(\qA, \psi_\ti)$. 
%\item 
%Both eigenvalues of $D\psi_{\ti}(\qA)$ have modulus less than one. In such case we can define the one dimensional local strong stable 
%manifold $W^\ss(\qA, \psi_\ti)$.
%\end{enumerate}
%We call the points above {\em{center expanding}} and {\em{center contracting,}} respectively.
%
%
%\lmargem{we need to emphasize the importance of this remark}

\begin{remark}[Strongest and strong manifolds]
\label{r.thegoodintersections}
Let  $A=A_{({\ti}^\mathbb{Z}, \qA)}$ be a periodic point of $g$
such that $({\ti}^\mathbb{Z}, \qA)$ is central repelling for $\Psi$. Then
each connected component of $W^\uu_{\loc}(A,g) \setminus W^\uuu_{\loc} (A,g)$ projects (by $\mathcal{F}^\uuu$)
into one component of $W^\uu_{\loc} (\qA, \psi_{\ti})\setminus \{\qA\}$. 
%Note that this apply to the points $Q$, $(0,0)$, and the map $\psi_\tq$.
Similarly, if $({\ti}^\mathbb{Z}, \qA)$ is central repelling
 then each component of $W^\ss_{\loc}(A,g) \setminus W^\sss_{\loc} (A,g)$ projects (by $\mathcal{F}^\sss$)
into one component of $W^\ss_{\loc} (\qA, \psi_{\ti})\setminus \{\qA\}$. 
\end{remark}

\subsection{Strong homoclinic intersections and robust cycles of co-index one}
\label{ss.strongesthomoclinic} 
We recall the definition of {\em{strong homoclinic points}} from \cite[Definition~3.2]{BonDia:08}. Given a nonhyperbolic periodic point 
$S$ with a splitting $E^\ss \oplus E^\ct \oplus E^\uu$, where $E^\ss$ and $E^\uu$ are uniformly contracting and expanding and $E^\ct$ is one-dimensional, we can consider 
the strong manifolds 
$W_{E^\ss}(S)$ and $W_{E^\uu}(S)$ of $S$ tangent to $E^\ss$ and $E^\uu$ (we use this notation to prevent overlap  with the notation
in \eqref{e.invariantstableunstablemanifolds}, this interference wil appear below). 
A point $X \in W_{E^\ss} (S) \cap W_{E^\uu}(S)$ such that $T_X W_{E^\ss} (S) \oplus T_X W_{E^\uu}(S)$ is called a 
{\em{quasi-transverse strong homoclinic point of $S$}}. 
The relevance of this configuration comes from  \cite[Theorem~2.4]{BonDia:08}:
   
  \begin{lem}[Robust co-index one cycles from strong homoclinic intersection] 
  \label{l.robustcycles}
Let $h$ be a diffeomorphism with a quasi-transverse strong homoclinic intersection $X$ associated with a nonhyperbolic periodic point $S$. Then, for every neighbourhood $V$ of the closure of the orbit of $X$, there exists a $C^1$ perturbation of $h$ with $C^1$ robust heterodimensional cycles of co-index one and type $(\ell, \ell+1)$ whose contours are contained in $V$, where $\ell = \dim E^\uu(S)$. 
\end{lem}

We now explain how strong homoclinic points  arise in our context
and are detected using the skew-product.
First, recall the dominated splitting  
$
E^\sss \oplus E^\ct_1 \oplus E^\ct_2 \oplus E^\uuu
$
from \eqref{e.domaincyclebis}, together with the strong invariant manifolds  
$W^\sss (S), W^\uuu(S),$ 
in \eqref{e.invariantstableunstablemanifolds}.  
Recall also the definition of $\iota (\Upsilon)$ in Definition~\ref{d.doubletypeI}.

\begin{remark}
\label{r.thefollowingcasesnh}
 There are the following cases for a nonhyperbolic periodic point $S$: 
\begin{enumerate}
\item
If $E^\ct_2(S)$ is nonhyperbolic (thus $E^\ct_1(S)$ is contracting), then  $E^\ss = E^\sss \oplus E^\ct_1$, 
$E^\ct = E^{\ct}_2$,
$E^\uu = E^\uuu$ and thus $ W_{E^{\uu}}(S)=W^\uuu (S)$.
Hence if $S$ has a quasi-transverse strong homoclinic point it generates robust cycles of
type $(\iota (\Upsilon), \iota (\Upsilon)+1)$.\item 
If $E^\ct_1(S)$ is nonhyperbolic (thus $E^\ct_2(S)$ is expanding), then  
$E^\ss = E^\sss$ and thus $ W_{E^{\ss}}(S)=W^\sss (S)$, $E^\ct = E^{\ct}_1$, and $E^\uu = E^\ct_2 \oplus E^\uuu$.
Hence if $S$ has a quasi-transverse strong homoclinic point it generates robust cycles of
type $(\iota (\Upsilon)+1, \iota (\Upsilon)+2)$. \end{enumerate}
\end{remark}

We now focus on points $A = A_{(\ti^\mathbb{Z}, \qA)}$ in case (1) above, case (2) is analogous.

\begin{lem}
\label{l.robustcyclesfromstrong}
%Let $f$ be a diffeomorphism with a simple non-escaping cycle \lmargem{say non-escaping includes co-index two, revise}
%and $g$ an adapted perturbation of $f$ preserving the cycle with quotient maps $\mathcal{G}$.
Let $(\ti^\mathbb{Z}, \qA)$ be a periodic point of $\Psi$ such that
 the eigenvalues of $D\psi_\ti (\qA)$ have modulus one and less than one
 (hence  $W^\ss_{\loc} (\qA, \psi_\ti)$ is well defined and one-dimensional)
  and 
  there is an admissible word $\tr$ such that 
$$
\psi_{\tr}(\qA)\in  W^\ss_{\loc} (\qA, \psi_\ti), \quad \mbox{$\tr$ is not a prefix of $\ti$}.
$$
Then there is a perturbation $h$ of $g$ having a robust heterodimensional cycle
of type $(\iota (\Upsilon), \iota (\Upsilon)+1)$ associated with sets contained in $U_\Upsilon$. 
\end{lem}

\begin{proof}
Let  $A = A_{(\ti^\mathbb{Z}, \qA)}$ and note that it is a nonhyperbolic periodic point of $g$.
Let
$$
X = X_{(\ti^{-\mathbb{N}}.\tr \ti^{\mathbb{N}}, \qA)} \in \Lambda_g (U_\Upsilon).
$$
Recalling that $W^\uuu (A,g)$ projects into points,
adapting Remarks~\ref{r.homohetero} and \ref{r.secondarydynamics},
and using Remark~\ref{r.thefollowingcasesnh}, we get that $X$ is a quasi-transverse strong homoclinic point of $A$.
Lemma~\ref{l.robustcycles} implies the assertion after noting that the final perturbation is local around the 
quasi-transverse point and hence can be done preserving the
co-index two cycle of $g$.  
\end{proof}

\section{Periodic points with neutral directions and consequences}
\label{ss.occurrence}
Using the quotient dynamics from the previous sections, we construct adapted perturbations that preserve the contour and generate periodic points with neutral derivative in the central direction $E^\ct_1$ (close to $\pm 1$) within a neighbourhood of the contour.
An analogous family with neutral derivative in the $E^\ct_2$ direction also arises; see Proposition~\ref{p.muitos}.
The proof of this proposition begins in Section~\ref{sss.neutral} and is completed in Section~\ref{ss.proofofpmuitos}.

For simplicity, we introduce a compact notation for words
\begin{equation}
\label{e.notationforwords}
\ti_{(m,n)}\eqdef  \ta \tp^{m} \yc \tq^{n}
 \qquad \mbox{and} \qquad
 \tj_{(r,t)} \eqdef \tp^{r} \yc \tq^{t} \yd.
\end{equation}

 \begin{prop}[Neutral periodic points  and strong homoclinic intersections]
\label{p.muitos}
Let $f$ be diffeomorphism with a simple  contour and  consider its  quotient family.
Then
for every
% $\varepsilon>0$ and 
$\ell_0\in \mathbb{N}$,
there exists an arbitrarily small
%$\varepsilon$ 
adapted perturbation $g$ of $f$, which preserves the contour, and
whose associated quotient family
 $\mathcal{G}$
 satisfies the following properties:
 
 For $k=1,\dots,4$ there
  are natural numbers  $m_k, n_k, r_k, t_k> \ell_0$ (increasing with $k$),
 words  $\ti_k=\ti_{(m_k, n_k)}, \tj_k= \tj_{(r_k, t_k)}\in  \Sigma_T^\ast$,
 and points  $\qA_{\ti_{k}}, \qD'_{\tj_{k}}$
such that, for $k=1,2, 3$
 \[
 \qA_{\ti_{k}}=(a_{\ti_{k},} 1) \to \qA =(0,1)\in \qU_\tq, \quad \mbox{and}\quad
 \qD'_{\tj_{k}} = (1, d_{\tj_{k}}) \to \qD=(1,0)
 \in \qU_\tp,
 \]
when $m_k,n_k,r_k,t_k \to \infty$, and
\begin{equation} 
\label{e.masmenos}
\begin{split}
&\psi_{\ti_k}(\qA_{\ti_k})=\qA_{\ti_k}\quad \mbox{with} \quad
D\psi_{\ti_k} (\qA_{\ti_k})=
\left(
\begin{smallmatrix}
\theta_{\ti_{k}} & 0\\
0 & 1
\end{smallmatrix}
\right)
\quad \mbox{and} \quad
 \theta_{\ti_{k}}
\to 0,
\\
&\psi_{\tj_{k}} (\qD'_{\tj_{k}})=\qD_{\tj_{k}}
\quad \mbox{with} \quad
D'\psi_{\tj_k} (\qD'_{\tj_k})=
\left(
\begin{smallmatrix}
1 & 0\\
0 & \kappa_{{\tj_{k} }}
\end{smallmatrix}
\right)
\quad \mbox{and} \quad
 \kappa_{\tj_{k}}\to \infty.
\end{split}
\end{equation}
Moreover,
\begin{equation}
\label{e.stronghomq}
\psi_{\ti_{4}} (\qA_{\ti_{1}})
\in W^\ss_\loc (\qA_{\ti_{1}}, \psi_{\ti_{1}} ),\quad \mbox{and} \quad
(\psi_{\tj_{4}})^{-1} (\qD'_{\tj_{{1}}}) 
\in W^\uu_\loc (\qD'_{\tj_{{1}}}, \psi_{\tj_{{1}}} ).
\end{equation}
\end{prop}

Observe that the conditions in \eqref{e.masmenos} allow us to consider the strong manifolds in \eqref{e.stronghomq}.

This proposition has two key consequences. First, in Section~\ref{ss.heteroandhomorobust}, using the points $A_{\ti_1}$ and $D'_{\tj_1}$ together with their strong homoclinic intersections obtained in \eqref{e.stronghomq}, we prove Theorem~\ref{tp.finallyarobustcycle}, which establishes the existence of the pair of robust co-index one cycles in Theorem~\ref{t.robustcycles}.
Second, in Section~\ref{ss.constructionofrectangles}, we derive a family of preliminary rectangles from the points $A_{\ti_2},
A_{\ti_3}, D'_{\tj_2},D'_{\tj_3}$. These rectangles will later serve as the basis for the preblending machines built in Section~\ref{ss.preblenders}.

\begin{remark}
\label{r.estimatesthetakappa}
There are the following estimates for $\theta_{\ti_k}$ and $\kappa_{\tj_k}$
in \eqref{e.masmenos}\footnote{The notation $a(n) \approx b(n)$ means  that two quantities are of the same order
as $n\to \infty$, that is, $\lim_{n\to \infty} \frac{a(n)}{b(n)}$ exists and is different from $0$.} 
$$
|\theta_{\ti_k}| \approx  | \alpha_1^{m_k} \beta_1^{n_k}| ,\qquad
 |\kappa_{\tj_k}| \approx   | \alpha_2^{r_k} \beta_2^{t_k}|.
$$
The estimate for $|\theta_{\ti_k}|$ follow from \eqref{e.derivativeofpsi} and \eqref{e.thingsgoingtozero}
in the proof of the proposition. The estimates for  $|\kappa_{\tj_k}|$ follows symmetrically.
We also have, 
$$
|a_{\ti_k}| \approx  \left(  \frac{|\beta_1|}{|\beta_2|}\right)^{n_k}, \qquad 
|d_{\tj_k}| \approx  \left(  \frac{|\alpha_1|}{|\alpha_2|}\right)^{r_k}. 
$$
The estimate for $a_{\ti_k}$ follows from \eqref{e.solamenteunavez}.  The estimates for  $|d_{\tj_k}|$ follows symmetrically.
\end{remark}

\subsection{Periodic points with neutral directions}
\label{sss.neutral}

The main step in the proof of Proposition~\ref{p.muitos} is the next lemma:

 \begin{lem}
 \label{lp.pp2di}
For every  $\ell_0\in \mathbb{N}$,
there exists an arbitrarily small adapted perturbation $g$ of $f$
such that
 there are  admissible words $\ti= \ti (m,n)$, 
with $|\ti|>\ell_0$,
and points $\qA_\ti=(a_{\ti},1)\to \qA=(0,1)$ when $|\ti|\to \infty$,  
 with
$
\psi_{\ti} (\qA_{\ti}) = \qA_\ti$. Moreover, the derivative satisfies
\begin{equation}
\label{e.goodmorning}
D\psi_{\ti} (\qA_{\ti})=
\left(
\begin{smallmatrix}
\delta^1_{\ti} & \delta^2_{\ti}\\
\delta^3_{\ti} & 1
\end{smallmatrix}
\right), \quad 
\mbox{where} 
\quad 
\delta^j_\ti \to 0
\quad 
\mbox{as} 
\quad 
 |\ti|\to \infty.
\end{equation}
\end{lem}

The proof of this lemma relies on the iterations of the point $\qA$  near $\qP$ and $\qQ$ under 
$\varphi_\tp$ and $\varphi_\tq$, see equation~\eqref{e.fzero}.   
The construction involves a sequence of adapted perturbations $g$ of $f$, standard in the $C^1$ topology (see \cite[Section~4.1.2]{BonDia:12} for details).  
We denote the quotient maps of these perturbations by $\mathcal{G}$, keeping the same notation after each perturbation.

%the last condition follows from  

%Before going into the details, we provide an heuristic proof of the lemma in the cases $(\pm,+)$.
%In these cases, we use the point $\qA$ as a ``base''  point and $\qB$ is not involved.
%To obtain points arbitrarily close to $\qA$ which are fixed for
%  $\psi_{ \ya \tp^m \yc \tq^n}$, for some large $m,n\in \mathbb{N}$,
%we study returns of boxes centered at $\qA$
%of the form
%\begin{equation}
%\label{e.boxes}
%\Delta_{\qA, \varrho}=\qA+\Delta_{\varrho} \subset \qU_\tq,
%\quad
%\Delta_\varrho=[-\varrho, \varrho]^2
%\end{equation}
% under the maps
%$\psi_{\ya \tp^m \yc \tq^n}$.
%We will see that, after adjusting the eigenvalues of the maps $\varphi_\tp$ and $\varphi_\tq$ 
%the resulting adapted perturbation satisfies
%$\psi_{\ya\tp^m \yc \tq^n}(\bar \qA)=\bar \qA$
% for  some $\bar \qA$ arbitrarily close to $\qA$ and suitable $m,n\geqslant 1$.  We will also control derivatives. 
 
 Before going into details, we sketch a heuristic proof of the lemma.
Here, the point $\qA$ serves as the ``base''.
To find fixed points of $\psi_{\ya \tp^m \yc \tq^n}$  close to $\qA$  we consider large $m,n\in \mathbb{N}$, 
and returns of boxes centered at $\qA$:
\begin{equation}
\label{e.boxes}
\Delta_{\qA, \varrho} = \qA + \Delta_\varrho \subset \qU_\tq,
\quad
\Delta_\varrho = [-\varrho, \varrho]^2.
\end{equation}
After adjusting the eigenvalues of $\psi_\tp$ and $\psi_\tq$, an adapted perturbation yields
$\psi_{\ya \tp^m \yc \tq^n}(\bar \qA) = \bar \qA$
for some $\bar \qA$ close to $\qA$.
We will also control the derivatives.

 To perform the perturbations, observe first  that  a direct calculation from equations
 \newline
 ~\eqref{e.fzero},~\eqref{e.fzerotres},\eqref{e.Fic}, gives 
$$
\varphi_{ \ya \tp^m \yc \tq^n}(0,1)=
\big(a_{12}\alpha^m_2\beta^n_1,
a_{22} \alpha^m_2\beta^n_2\big).
$$
We first select arbitrarily  large $m$ and the associated  $n$  
defined by
$$
| \alpha_2^m \beta_2^{n-1}| < 1 \qquad \mbox{and} \qquad |\beta_2| \geq | \alpha_2^m \beta_2^{n}| \geq 1.
$$
This choice  implies that if $m\to \infty$ then $n\to \infty$ and
\begin{equation}
\label{l.estimatesinextremis}
|\alpha^m_2\beta^n_1| 
= \left| \alpha^m_2 {\beta_2^n} \frac{\beta^n_1}{\beta_2^n} \right| \leq |\beta_2| \left| \left( \frac{\beta_1}{\beta_2} 
\right)^n \right|
\to 0 \quad \mbox{as}\quad m,n\to \infty.
\end{equation}

The idea now is to replace $\alpha_2$ and $\beta_2$ by close values
$\bar \alpha_2$ and $\bar \beta_2$ 
(these values converge to $\alpha_2$ and $\beta_2$ when $m,n\to \infty$)
such that 
$$
a_{22} \bar \alpha^m_2 \bar \beta^n_2=1.
$$
This choice implies that
$$
a_{12}\bar\alpha^m_2\beta^n_1= \frac{a_{12}}{a_{22}} \left(\frac{\beta_1}{\bar\beta_2}\right)^n \to 0
 \quad \mbox{as}\quad n\to \infty.
$$ 
Thus, $\qA$ is ``almost" a fixed point of  
the perturbation $\psi_{\ya\tp^m \yc \tq^n}$. The final step is to see that
the image of the square $\Delta_{\qA,\varrho}$ by $\psi_{\ya\tp^m \yc \tq^n}$ intersects $\Delta_{\qA,\varrho}$
in a Markovian way, providing the fixed point close to $\qA$.
We now go into the details.

%We begin with a general fact that will be used in all cases. In the other configurations we will need to either to change the base point of the construction and/or more
%complex itineraries.
%\lmargem{no entiendo nada}
%
 We begin with some simple facts. The first result is the following remark:

\begin{claim}
\label{cl.racionais}
Given $\alpha, \beta, \gamma \in \mathbb{R}$, $\varepsilon>0$, and $n_0 \in \mathbb{N}$  there are $\bar \alpha$, $\bar \beta$ with $\bar\alpha\in (\alpha-\varepsilon,\alpha+\varepsilon)$ and $\bar\beta \in (\beta-\varepsilon,\beta+\varepsilon)$ 
and natural numbers
$m_1,n_1>n_0$
  such that
  $$
\bar{\alpha}^{m_1}\bar{\beta}^{n_1} =|\gamma|.
$$
 Moreover, if either $\alpha<0$ or $\beta<0$ we may have  
$\bar{\alpha}^{m_1}\bar{\beta}^{n_1} =\gamma$. 
\end{claim}

%\begin{claim}
%\label{cl.racionais}
%Given $\gamma \in \mathbb{R}$, $\varepsilon>0$, and $n_0 \in \mathbb{N}$  there are
%$\bar \beta_2$, $\bar \alpha_2$ with 
%$$
%\max\{
%|\bar \beta_2-\beta_2|,
%|\bar \alpha_2-\alpha_2|\}
%< \varepsilon
%$$
%%
%%$\beta_2' \in (\beta_2-\varepsilon, \beta_2+\varepsilon)$,
%%$\alpha_2' \in (\alpha_2-\varepsilon, \alpha_2+\varepsilon)$,
%%$b' \in ({b} -\varepsilon, {b}+\varepsilon)$,
%%$\kappa_2' \in (\kappa_2-\varepsilon, \kappa_2+\varepsilon)$,
%and natural numbers
%$m_1,n_1>n_0$
%  %\smargem{son todos mayores que $\ell_0$? o solo $\ell_2$?}
%  such that
%  $$
%\bar{\alpha}^{m_1}_2\bar{\beta}^{n_1}_2 =|\gamma|.
%$$
% Moreover, if either $\alpha_2<0$ or $\beta_2<0$ we may have  
%$\bar{\alpha}^{m_1}_2\bar{\beta}^{n_1}_2 =\gamma$. 
%\end{claim}

\begin{proof}
It is enough to slightly  modify $\alpha, \beta$ to get large numbers $n_1, m_1$ with
$$
n_1 \, \frac{\log \bar \beta}{\log \bar \alpha} = -m_1 + \frac{\log |\gamma|}{\log \bar \alpha}.
$$
 The second part of the claim  follows considering odd natural numbers.
\end{proof}

%We denote by $B\big(\qX,r\big)\subset U_\tp\cup U_\tq$ the open ball of center $\qX$ and radius $r>0$.

\begin{claim} 
\label{cl.oxo}
Assume configuration $(+,+)$.
Given $\varepsilon, \delta>0$ and $n_0\in \mathbb{N}$ there are  $\bar{\alpha}_2\in (\alpha_2-\delta,\alpha_2+\delta)$, 
$\bar{\beta}_2\in (\beta_2-\delta,\beta_2+\delta)$, 
%$\bar d\in(1-\delta, 1+\delta)$, 
and $m_1$, $n_1\geqslant  n_0$ such that
$$
\mathrm{(i)} \qquad
%a_{22} \bar{\alpha}^{m_1}_2\bar{\beta}^{n_1}_2\approx 1\quad  \Rightarrow \quad 
\vert  \bar{\alpha}^{m_1}_2 {\beta}^{n_1}_1  \vert < \varepsilon/100 \qquad \mbox{and} \qquad
\mathrm{(ii)} \qquad a_{22} \bar{\alpha}^{m_1}_2\bar{\beta}^{n_1}_2=1.
$$
%
%$\big(a_{12} \bar{\alpha}^{m_1}_2 {\beta}^{n_1}_1 ,
%a_{22} \bar{\alpha}^{m_1}_2\bar{\beta}^{n_1}_2\big) \in B\big((0,1),\varepsilon\big)\subset U_\tq.$
\end{claim}

\begin{proof}
Let $\alpha=\alpha_2$, $\beta=\beta_2$, $\gamma=1/a_{22}>0$ and take  $\bar\alpha=\bar\alpha_2$, $\bar\beta=\bar\beta_2$ and $m_1, n_1$ as in Claim~\ref{cl.racionais} (associated to $\gamma$),
hence 
$$
 a_{22} \bar{\alpha}^{m_1}_2\bar{\beta}^{n_1}_2 =1.
$$
This implies (ii). 
Finally, noting that $|{\beta_1}|<|\bar{\beta_2}|$, 
we get
$$
\bar{\alpha}^{m_1}_2 {\beta}^{n_1}_1 =
\bar{\alpha}^{m_1}_2\bar{\beta}^{n_1}_2 \left(\frac{{\beta}_1}{\bar{\beta_2}}\right)^{n_1} =
\frac{1}{a_{22}}\,
\left(\frac{{\beta}_1}{\bar{\beta_2}}\right)^{n_1} \to 0, 
 \quad \mbox{as}\quad
n_1\to \infty,
$$
that implies item (i).
\end{proof}

We now consider an adapted perturbation $g$ of $f$ such that 
$\psi_\yj= \varphi_\yj$, for $\yj\ne \tq, \tp$, and such that
and $\psi_\tq$ and $\psi_\tp$ are (in the local coordinates) of the form 
$\psi_\tq (x,y) =(\beta_1x , \bar \beta_2 y)$ and
$\psi_\tp (x,y) =(\alpha_1x , \bar \alpha_2 y)$, where $\bar\alpha_2$ and $\bar\beta_2$ are as in 
Claim~\ref{cl.oxo}.

We now study 
the images  of the boxes $\Delta \eqdef \Delta_{\qA, \varrho}$ in \eqref{e.boxes}
by the maps 
$\psi_{\ya\tp^{m_1} \yc \tq^{n_1}}$.
In the following calculations, and without loss of generality, we will assume that $\epsilon_a$ in \eqref{e.fzerotres} are
positive.
%$\varphi^{n_1}_\tq \circ \varphi_\yc \circ \varphi^{m_1}_\tp \circ \varphi_\ya$
Note that
\[
\begin{split}
\psi_\ya (\Delta)&=[-\epsilon_a \varrho,\epsilon_a \varrho]\times [1-\varrho,1+\varrho]\subset \qU_\tp
\\
\psi_{\ya \tp^{m}}(\Delta)&=
[-\epsilon_a \varrho \alpha^m_1,\epsilon_a \varrho \alpha^m_1]\times [ \bar \alpha^m_2- \varrho \bar\alpha^m_2,\bar \alpha^m_2 + \varrho
\bar \alpha^m_2]\subset \qU_\tp,
\\
\psi_{\ya \tp^{m} \yc } (\Delta)&=
\{(a_{12}\bar\alpha^m_2+a_{11} x+a_{12} y, a_{22}\bar \alpha^m_2+a_{21}
x+a_{22} y) \colon \vert x \vert< \epsilon_a \varrho \alpha^m_1,\,\, \vert y \vert< \varrho \bar \alpha^m_2 \}.
\end{split}
\]

Consider the projections $p_1(x,y)=x$ and $p_2(x,y)=y$ on $\qU_\tq$. We have
\[
\begin{split}
p_1\big( \psi_{\ya \tp^{m} \yc }  (\Delta)\big) &=
\{a_{12}\bar \alpha^m_2+a_{11}x+a_{12} y \colon \vert x  \vert< \epsilon_a 
\varrho \alpha^m_1,\quad \vert y \vert< \varrho \bar {\alpha}^m_2\},\\
p_2\big( \psi_{\ya \tp^{m} \yc }  (\Delta)\big)&=
\{a_{22}\bar \alpha^m_2+a_{21}x+a_{22} y \colon \vert x \vert< \epsilon_a \varrho 
\alpha^m_1,\quad \vert y \vert< \varrho \bar \alpha^m_2\}
\end{split}
\]
and
$$
\psi_{\ya \tp^{m} \yc }  (\Delta) \subset \Delta_m, \quad \mbox{where} \quad
\Delta_{m}\eqdef p_1\big( \psi_{\ya \tp^{m} \yc }  (\Delta) \big)
\times p_2\big(  \psi_{\ya \tp^{m} \yc }  (\Delta) \big)\subset U_{\tq}.
$$
Denote by $\ell_i(m)$ the length of  $p_i\big( \psi_{\ya \tp^{m} \yc }  (\Delta)\big)$, $i=1,2$.

We now study the boxes $\psi^{n_1}_\tq \big( \Delta_{m_1}\big)$  for the values 
$\bar\alpha_2, \bar\beta_2, m_1,n_1$ in Claim~\ref{cl.oxo}.  Note that $n_1 =n_1(m_1)\to \infty $ as $m_1\to \infty$.
We see that  the lengths of the basis of these boxes
go to zero as $m_1\to \infty$ (Claim~\ref{cl.vamos}) 
 while the
heights have uniform size (Claim~\ref{cl.vamosquevamos}).

\begin{claim}
\label{cl.vamos}
$\beta_1^{n_1} \ell_1(m_1) \to 0$ as $m_1 \to \infty$.
\end{claim}

\begin{proof}
As $n_1\to \infty$ as $m_1\to \infty$, it is enough to check that for every  $\vert x \vert< \epsilon_a \varrho {\alpha}^{m_1}_1$ and
$\vert y \vert< \varrho \bar{\alpha}^{m_1}_2$ it holds
$$
{\beta}^{n_1}_1  ( a_{21} \bar\alpha^{m_1}_2+a_{11} x+a_{12} y) \to 0, \quad \mbox{as} \quad  m_1\to \infty .
$$
Using that $|\alpha_1|< |\bar \alpha_2|$ we get  $|\alpha_1^{m_1}\beta_1^{n_1} |< |\bar \alpha_2^{m_1}\beta_1^{n_1}  |$ and by 
Claim~\ref{cl.oxo}-(i) the later term can be chosen arbitrarily small if $m_1$ (and hence $n_1$) is large enough.  
Noting that
$$
\vert a_{12} \bar{\alpha}^{m_1}_2+a_{11}x+a_{12} y\vert
\leqslant
\vert a_{12} \bar{\alpha}^{m_1}_2\vert +\vert a_{11} \epsilon_a \varrho {\alpha}^{m_1}_1 \vert+\vert a_{12}   \varrho 
\bar{\alpha}^{m_1}_2\vert,
$$
the claim follows from the previous observations. 
\end{proof}

\begin{claim}
\label{cl.vamosquevamos}
$|\bar \beta_2^{n_1}| \ell_2(m_1) \to 2 \varrho$ as $m_1 \to \infty$.
\end{claim}
\begin{proof} 
Note that  the supremum and the infimum of $a_{21} x+a_{22} y$ restricted to the values 
$\vert x \vert< \epsilon_a \varrho  |{\alpha}_1|^{m_1}$ and $\vert y \vert< \varrho |\bar{\alpha}_2|^{m_1}$ are respectively
$$
\vert a_{21}  \epsilon_a \varrho {\alpha}^{m_1}_1 \vert +\vert a_{22}  \varrho \bar{\alpha}^{m_1}_2\vert \qquad\mbox{and}\qquad 
-\vert a_{21}  \epsilon_a \varrho {\alpha}^{m_1}_1 \vert -\vert a_{22}  \varrho \bar{\alpha}^{m_1}_2\vert .
$$
Using 
$$
\sup(I+J)=\sup(I)+\sup(J) \qquad \mbox{and} \qquad \inf(I+J)=\inf(I)+\inf(J)
$$
and that 
$|\bar\alpha_2^{m_1} \bar\beta_2^{n_1}|$ is bounded,
 it follows
%$$
%\bar{\beta}^{n_1}_2(\vert c\vert \varepsilon_1 u \bar{\alpha}^{m_1}_1 +\bar d v \bar{\alpha}^{m_1}_2 )=\vert c\vert \varepsilon_1 u \bar{\alpha}^{m_1}_1\bar{\beta}^{n_1}_2 +\bar d  v \bar{\alpha}^{m_1}_2\bar{\beta}^{n_1}_2\gtrsim 0+v.
%$$
$$
|\bar{\beta}^{n_1}_2 |\sup\{  a_{21}x+a_{22} y\}-|\bar{\beta}^{n_1}_2 |\inf\{ a_{21}x+a_{22} y \}= 2\, |\bar{\beta}^{n_1}_2|  \big(
\vert a_{21}  \epsilon_a \varrho {\alpha}^{m_1}_1 \vert +\vert a_{22}  \varrho \bar{\alpha}^{m_1}_2\vert).
$$
As $\bar{\beta}^{n_1}_2 \alpha_1^{m_1} \to 0$ and $\bar{\beta}^{n_1}_2 \bar\alpha_2^{m_1} =a_{22}^{-1}$
the claim follows.
\end{proof}

\subsubsection{End of the proof of Lemma~\ref{lp.pp2di}}
\label{ss.kmma}
We got an itinerary 
$\ti_{(m_1,n_1)}$ 
of $\qA=(0,1)$ such that
  $$
  \psi_{\ti}(0,1)= (a_{12}\bar\alpha^{m_1}_2\beta^{n_1}_1,1),
  \quad
  \mbox{where} \quad
  a_{12}\bar\alpha^{m_1}_2\beta^{n_1}_1=\frac{a_{12}}{a_{22}} \left(\frac{\beta_1}{\bar\beta_2}\right)^{n_1}.
  $$
  Note that the later term goes to $0$ as $n_1\to \infty$.
To get the assertion on the derivative (see~\eqref{e.goodmorning}), note that
\begin{equation}
\label{e.derivativeofpsi}
D\psi_{\ti} (\qA)=
\left(
\begin{matrix}
a_{11} \epsilon_a \,\beta_1^{n_1}\, \alpha_1^{m_1} & \pm a_{12} \,\beta_1^{n_1}\, \bar \alpha_2^{m_1}
\\
a_{21}\,\epsilon_a\,\bar \beta_2^{n_1}\, \alpha_1^{m_1} & 1
\end{matrix}
\right),
\end{equation}
where 
\begin{equation}
\label{e.thingsgoingtozero}
|a_{11} \epsilon_a \,\beta_1^{n_1}\, \alpha_1^{m_1}|, \,|a_{12} \,\beta_1^{n_1}\, \bar \alpha_2^{m_1}|,\,
|a_{21} \,\epsilon_a\,\bar \beta_2^{n_1}\, \alpha_1^{m_1} | \to 0
\quad
\mbox{as} \quad m_1, n_1\to \infty.
\end{equation}
 Thus, the eigenvalues
of $D\psi_{\ti} (\qA)$
 are close to $0$ and $1$.
By domination, the eigenspace associated to $1$ is almost vertical. Therefore, after a new perturbation we can assume that the vertical
direction is invariant.
We now modify the maps along the orbit of $\qA$ to see 
that $\psi_{\yi} (\Delta)$ intersects $\Delta$ in a Markovian way (see Figure~\ref{f.itineraries}). This provides a fixed point with derivative
close to one. As the number of iterates is arbitrarily large a new perturbation provides a point $\qA_\ti$ of the form  $\qA_\ti=(a_{\ti}, 1)$ with $\psi_\yi (\qA_\ti)=\qA_\ti$
and multiplier equal to $1$. Note that the other eigenvalue is arbitrarily small  (hence of contracting type).
We observe that by construction 
\begin{equation}
\label{e.solamenteunavez} 
\vert a_{\ti}\vert \approx
\left( \frac{ |\beta_1|}{|\bar \beta_2|}\right)^{n_1}. 
\end{equation}
 The proof
of the lemma is now complete. \hfill \qed

\begin{figure}[h]
\begin{overpic}[scale=0.38,
%grid,tics=5
]{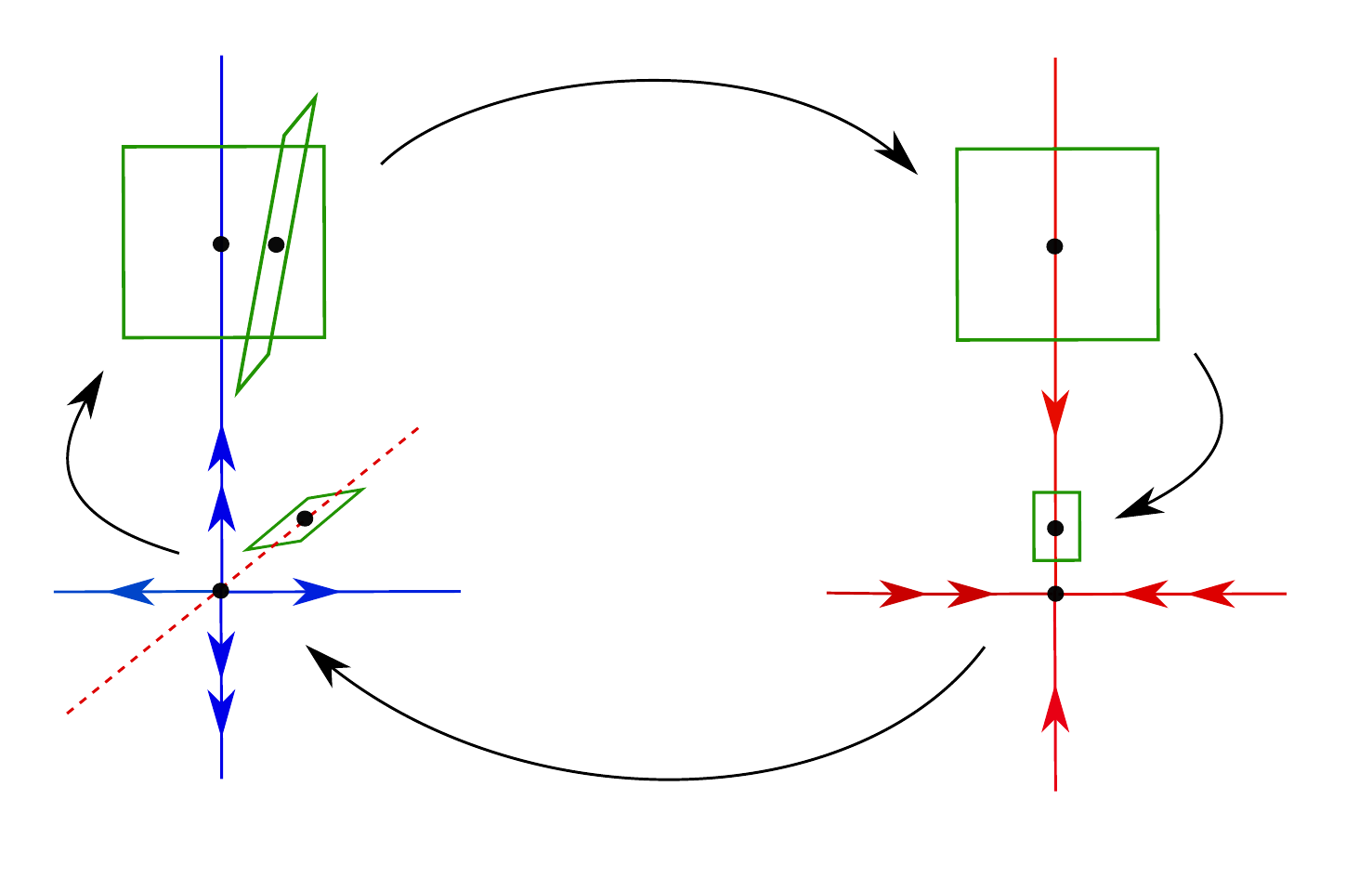}
         \put(7,55){\Large  $\Delta$}   
         \put(-3,30){\Large $\psi_{\tq}^n$}   
         \put(29,27){\large $a_{22}>0$} 
         \put(26,45){\Large  $\psi_{ \ya \tp^m \yc \tq^n}{\large (\Delta)}$}   
         \put(45,62){\Large  $\psi_{\ya}$}
         \put(45,11){\Large  $\psi_{\yc}$}
         \put(91,31){\Large $\psi_{\tp}^m$}
         %\put(60,13){\Large  $\psi_{\yb}$}
         %\put(78,15){\large $b<0$}
        % \put(78,40){\large $b>0$}
         \put(11,45){\large ${A}$}   
         %\put(14,7){\large ${B}$}   
         \end{overpic}
         \vspace{-.3cm}
         \caption{Itineraries $\ti= \ya \tp^m \yc \tq^n$}
         \label{f.itineraries}
\end{figure}

\subsection{End of the proof of Proposition~\ref{p.muitos}}
\label{ss.proofofpmuitos}
Suppose that $\ti_1$ and $A_{\ti_1}$ are defined.
We now explain  how to get large $m_2,n_2$ such that,
taking $\ti_2= \ti_{(m_2,n_2)}$,
 there is a fixed point
$\qA_{\ti_2}$ of  $\psi_{\ti_2}$ 
satisfying \eqref{e.masmenos}.
We observe that the orbit of $\qA_{\ti_1}$ is ``separated" from both $\qQ$ and $\qP$, and both, the cycle and transitions, are preserved.
As the cycle is not destroyed,
we can  repeat the construction choosing $n_2\gg n_1$ and $m_2 \gg m_1$.  This implies that a new perturbation can 
be made ``deeper" than the first one (done to get $\qA_{\ti_1}$), meaning there are points in the orbit of $\qA_{\ti_2}$ much closer to $\qQ$ and $\qP$. 
Consequently, a new itinerary with a fixed point can be obtained without altering the orbit of $\qA_{\ti_1}$.
See Figure~\ref{f.itineraries}.
This can be repeated to get a  new itinerary $\ti_3$ and a point $\qA_{\ti_3}$.

\begin{remark}
\label{r.uniformsizeofinvariantmanifolds1}
By construction, for large $m_k,n_k$, $k=1,2,3,$ (see Figure~\ref{f.relations})
$$
W^\ss (\qA_{\ti_{k}}, \psi_{{\ti_{k}}})\pitchfork W^\uu (\qQ, \psi_{\tq}) \ne\emptyset, \quad \ti_k =\ti_{(m_k,n_k)}.
$$
\end{remark}

To get the relations in~\eqref{e.stronghomq} for a word $\ti_4$,  we proceed similarly to the previous argument.  This time, we increase the number
$m_4$ so that we are closer to $\qP$ than the periodic orbits of $\qA_{\ti_1},\qA_{\ti_2},\qA_{\ti_3}$. This allows us to perform a deeper perturbation without modifying the orbits of $\qA_{\ti_1},\qA_{\ti_2},\qA_{\ti_3}$ in a such a way, that, for an appropriate and sufficiently large $n_4,m_4$, we have 
$\psi_{\ti_4} (\qA_{\ti_{1}})\in W^\ss_\loc (\qA_{\ti_1}, \psi_{\ti_{1}} )$, here we use Remark~\ref{r.uniformsizeofinvariantmanifolds1}.

%The argument to get the word $\ti_4$ is similar, but in this case we new to consider deeper perturbation that do not destroy the segments
%of orbits previously constructed. 

 Finally note that, by construction, to the words $\ti_2$ and $\ti_3$ there are associated
fixed points $\qA_{\ti_2}$ and $\qA_{\ti_3}$ satisfying \eqref{e.masmenos}. This completes the proof for the points $\qA_{\ti_k}$.
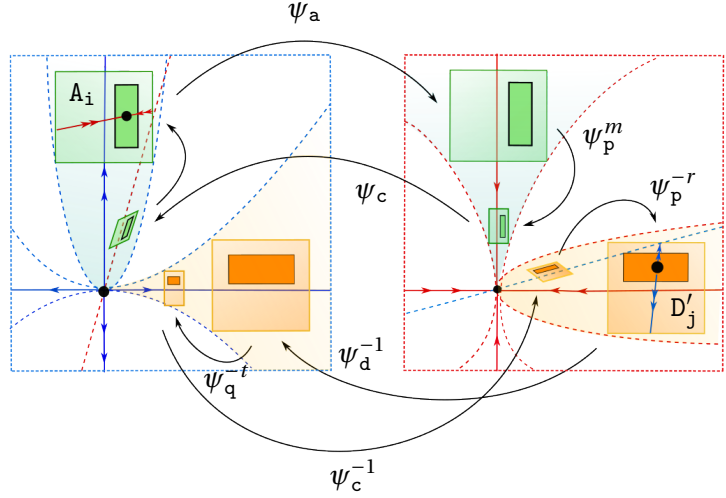
\begin{figure}[h]
\begin{overpic}[scale=0.35,
%grid,tics=5
]{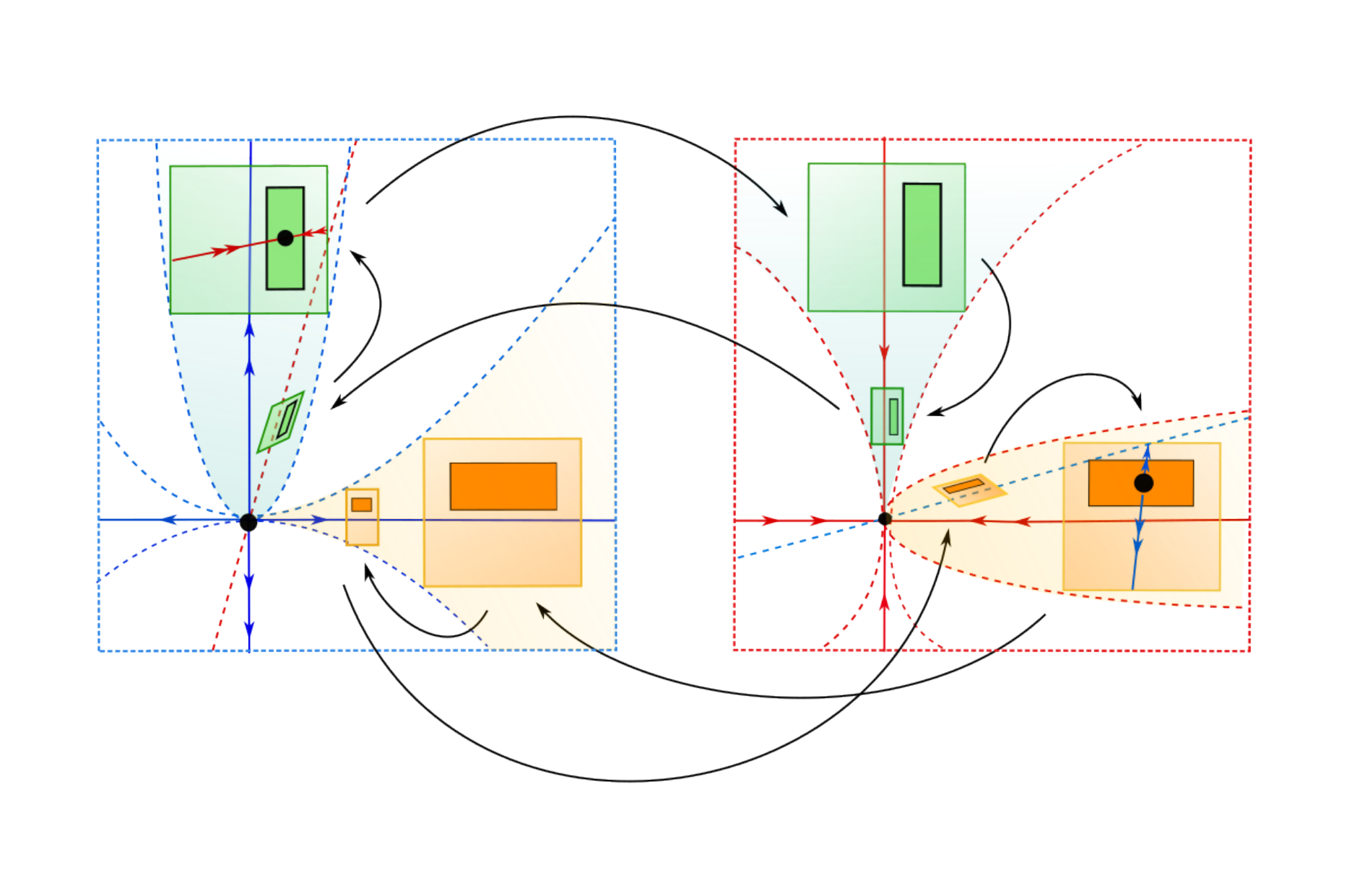}  
 \put(14,50){\large $\qA_{\ti}$}   
         \put(40,60){\large  $\psi_{ \ya}$}   
         \put(48,38.5){\large  $\psi_{\yc}$}
           \put(75,45){\large $\psi_{\tp}^m$}
\put(85.5,24.5){\large $\qD'_{\tj}$}   
         \put(46,19){\large  $\psi_{\yd}^{-1}$}
         \put(45,4.5){\large  $\psi_{\yc}^{-1}$}
         \put(83,39){\large  $\psi_{\tp}^{-r}$}
           \put(30,16){\large  $\psi_{\tq}^{-t}$}
         \end{overpic}
         \caption{Cuspidal regions for compatible perturbations}
         \label{f.Soportes}
\end{figure}
 
To get the points $\qD'_{\yj_k}$, with a neutral direction and strongly expanding eigenvalues, we proceed as above, but now considering
the transition  $\varphi^{-1}_\yd$. This construction can be done preserving the points $\qA_{\yi_k}$ and their homoclinic intersections.
Note that the perturbations in the construction of these points are done in the two disjoint cuspidal regions depicted in Figure~\ref{f.Soportes},
thus they can be performed independently. 
This ends the proof of the proposition. \hfill \qed

\subsection{Transverse heteroclinic intesections}
\label{ss.heteroclinicrelations}

The next remark will be used to get the transverse intersection
between the invariant manifolds of sets 
 in cycles of co-index two.

\begin{remark}
\label{r.uniformsizeofinvariantmanifolds}
By construction, for large $m_k,n_k, r_k, t_k$ and  $k=1,2,3$,
$$
W^\ss (\qA_{\ti_{k}}, \psi_{{\ti_{k}}})\pitchfork W^\uu (\qQ, \psi_{\tq}) \ne\emptyset,
\quad
W^\uu (\qD'_{\tj_{k}}, \psi_{\tj_{k,}})\pitchfork W^\ss (\qP, \psi_{\tp}) \ne\emptyset,
$$
see Figure~\ref{f.relations}. Noting  that, for large $n,m$,
$$
\mbox{if} \quad
|\beta_2^n \alpha_2^m|\approx 1 
\quad
\mbox{then}
\quad
|\beta_1^n \alpha_1^m|\approx 0.
$$
we obtain that
\begin{equation}
\label{e.22}
\psi_\tq^n \circ \psi_\tc \circ \psi_\tp^m 
\left(W^\uu (\qD'_{\tj_k}, \psi_{\tj_k})\right) \pitchfork
 W^\ss (\qA_{\ti_k}, \psi_{\ti_k}) \ne \emptyset.
\end{equation}
Thus, if $\qA_{\ti_k}$ and $\qD_{\tj_k}$ are hyperbolic points, then 
 $$
  W^\st \left( (\ti_{k}^\mathbb{Z}, \qA_{\ti_{k}}), \Psi \right)  
  \pitchfork   W^\ut \left( (\tj_{k}^\mathbb{Z}, \qD'_{\tj_{k}}), \Psi \right) 
  \ne \emptyset.
  %\quad   D'_{(\tj_\star^\mathbb{Z}, \qD'_\star)}
  $$
As a consequence of Remark~\ref{r.secondarydynamics}, the periodic points $A_{(\ti_k^\mathbb{Z}, \qA_{\ti_k})}$ and $D'_{(\tj_k^\mathbb{Z}, \qD'_{\ti_k})}$
of 
the diffeomorphism $g$ associated to $\Psi$ are hyperbolic and satisfy 
 $$
  W^\st ( A_{(\ti_k^\mathbb{Z}, \qA_{\ti_k})}, g)  
  \pitchfork   W^\ut ( D'_{(\tj_k^\mathbb{Z}, \qD'_{\ti_k})}, g ) 
  \ne \emptyset.
  $$
\end{remark}

\begin{figure}[h]
\begin{overpic}[scale=0.45,
%grid,tics=5
]{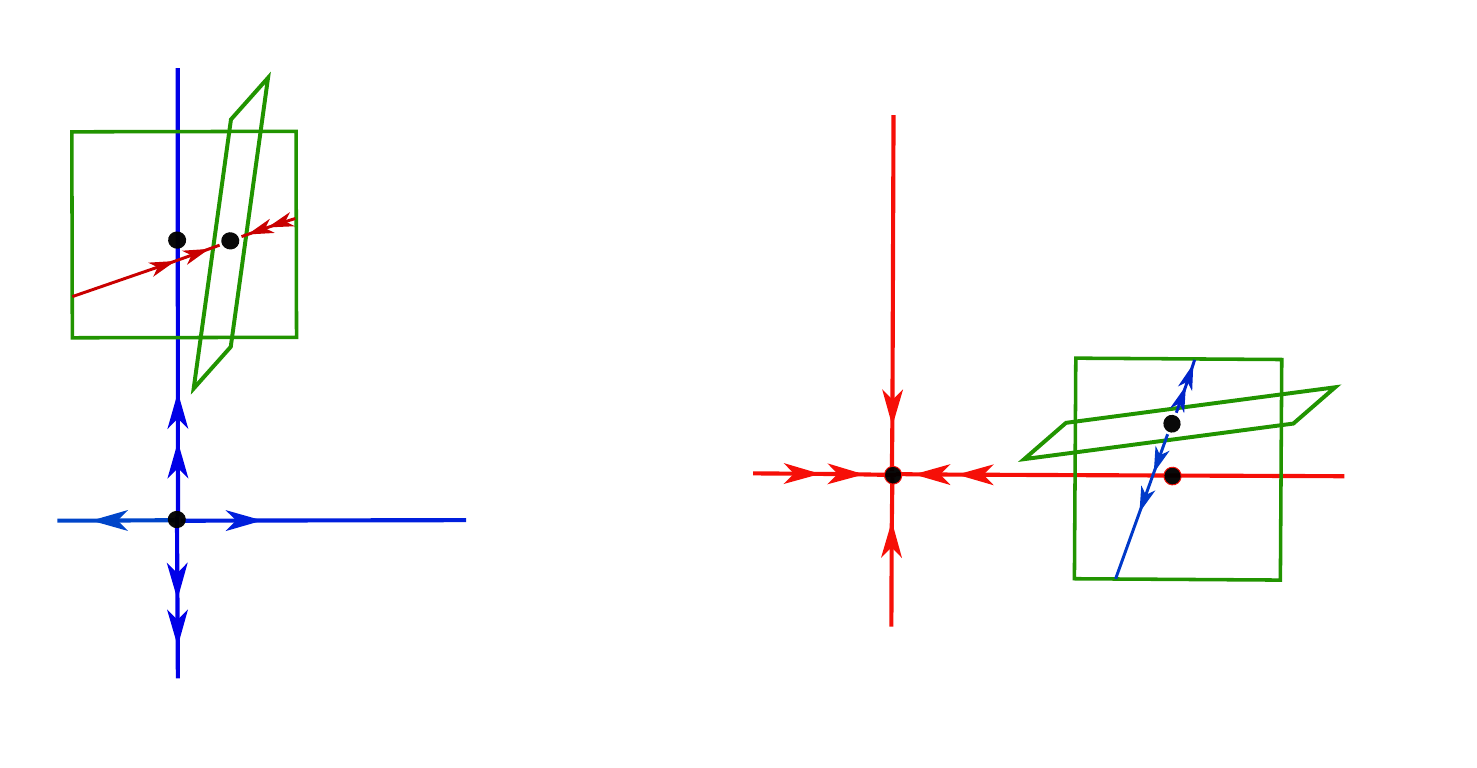}  
 \put(8,36.5){\large $\qA$}   
  \put(17,32){\large $\qA_{\ti_{k}}$}   
         \put(0,40){\Large  $\Delta_{\qA}$}   
                  \put(20,46.5){\large  $\psi_{{\ti_{k}}}(\Delta_{\qA})$} 
      \put(21,37){\large  $W^\ss (\qA_{\ti_{k}}, \psi_{{\ti_{k}}})$} 
           \put(2,50){\large $W^\uu (\qQ, \psi_{\tq})$}
         \put(14,12){\large  $\qQ$}
\put(80,16){\large $\qD'$}   
  \put(74,26){\large $\qD'_{\tj_{k}}$}   
         \put(88,15){\Large  $\Delta_{\qD'}$}   
                  \put(89,28){\large  $\psi_{{\tj_{k}}}(\Delta_{\qD'})$} 
      \put(64,8){\large  $W^\uu (\qD'_{\tj_{k}}, \psi_{{\tj_{k}}})$} 
           \put(44,22.5){\large $W^\ss (\qP, \psi_{\tp})$}
         \put(62,16){\large  $\qP$}
         \end{overpic}
         \caption{Transverse intersections $W^\ss (\qA_{\ti_{k}}, \psi_{{\ti_{k}}})\pitchfork W^\uu (\qQ, \psi_{\tq}) \ne\emptyset$
and 
$W^\uu (\qD'_{\tj_{k}}, \psi_{\tj_{k,}})\pitchfork W^\ss (\qP, \psi_{\tp}) \ne\emptyset$}
         \label{f.relations}
\end{figure}

\subsection{Co-index one robust cycles}
\label{ss.heteroandhomorobust}
We now derive the occurrence of robust cycles from Proposition~\ref{p.muitos}.
 
\begin{thm}\label{tp.finallyarobustcycle}
Let $f$ be a diffeomorphism with a simple contour $\Upsilon$.
Then, for every neighbourhood $U_\Upsilon$ of $\Upsilon$, there exists a diffeomorphism $h$ arbitrarily $C^1$ close to $f$ 
having simultaneously a contour of double type I  contained in $U_\Upsilon$
and  a pair of $C^1$ robust heterodimensional cycles of co-index one, of types $(\iota (\Upsilon),\iota (\Upsilon)+1)$ and 
$(\iota (\Upsilon)+1,\iota (\Upsilon)+2)$, associated to hyperbolic sets contained in $U_\Upsilon$.
\end{thm}

 \begin{proof}
  Consider the
  points $\qA_{\ti_{1}}$ and $\qD'_{\tj_{1}}$ and the words ${\ti_{4}}$ and ${\tj_{4}}$ 
  given by Proposition~\ref{p.muitos} such that
$\psi_{\ti_{4}} (\qA_{\ti_{1}})
\in W^\ss_\loc (\qA_{\ti_{1}}, \psi_{\ti_{1}})$ and  $(\psi_{\ti_{4}})^{-1}(\qD'_{\tj_{1}})
\in W^\uu_\loc (\qD'_{\tj_{1}}, \psi_{\tj_{1}})$.
Using the first condition,
 Lemma~\ref{l.robustcyclesfromstrong} provides the cycle of type $(\iota(\Upsilon), \iota (\Upsilon)+1)$.
 The second condition and the lemma provides the cycle of type $(\iota(\Upsilon)+1, \iota (\Upsilon)+2)$.
\end{proof}

%
%\begin{remark}[A pair of robust cycles of co-index one]
%\label{r.stronghomoclinic} 
%We consider the case $\lambda=\mu=1$, focusing on the points $A_k$ (similar arguments apply to the points $D_k$).
%Following the terminology of \cite[Definition 3.2]{BonDia:08}, the nonhyperbolic periodic points exhibit strong homoclinic intersections, that is, $W^\ut (A_k) \cap W^\ss (A_k)\ne\emptyset$, and this intersection is quasi-transverse (the sum of the tangent spaces at the intersection point has codimension one). This follows ipsis litteris from the argument in \cite[page 495]{BonDia:08}, where the corresponding points are denoted $A_{m,n}$.
%
%By \cite[Theorem 2.4]{BonDia:08}, there exists an arbitrarily small $C^1$ perturbation of the diffeomorphism $g$ that yields a robust heterodimensional cycle of co-index one. This cycle involves a saddle and a hyperbolic set and is of type $(\ell, \ell+1)$, where $\ell$ is the $\ut$-index of $P$. Importantly, this perturbation modifies only the orbits of two of the points $A_k$, while the remaining points are preserved and will later be used to construct cycles of co-index two.
%The co-index two cycle is also preserved.
%
%Analogously, using the points $D_k$, one obtains a robust heterodimensional cycle of co-index one of type $(\ell+1, \ell+2)$ involving a saddle and a hyperbolic set. Moreover, these perturbations can be performed independently.
%\end{remark}

\subsection{Selection of rectangles}
 \label{ss.constructionofrectangles}
  We 
 use the points $A_{\ti_2}, A_{\ti_3}, D'_{\tj_2},$ and $D'_{\tj_3}$  as base points of a family of rectangles that will be used in the
construction of preblending machines.
We assume that $\alpha_j, \beta_i$ in \eqref{e.fzero} are both positive.
We assume that  the numbers $m_k, n_k, r_k, t_k$ in the  words $\ti_k =\ti_{({m_k},{n_k})}$ 
and
$\tj_k =\tj_{({r_k}, {t_k})}$ are arbitrarily large. Recall that
$\qA_{\ti_k}=(a_{\ti_k},1)$, where $a_{\ti_k} \approx  \left(\frac{\beta_1}{\beta_2}\right)^{n_k} >0$ 
 (recall Remark~\ref{r.estimatesthetakappa}).
 
We denote by $|J|$ the lenght of an interval $J$.

 \begin{claim}
 \label{cl.rectangulos}
 Given $\varepsilon>0$ there is
 $\vartheta_0>0$ such that for every $\vartheta\in (0,\vartheta_0)$
 there is  an $\varepsilon$ adapted perturbation of $\mathcal{G}$ (whose quotient maps will be denoted as those of $\mathcal{G}$)
 and closed intervals
 $$
\qV(\vartheta) \eqdef [1-\vartheta, 1+\vartheta],
\quad 
\qH_{\ti_2}, 
\quad 
\qH_{\ti_3}, 
\quad 
\mbox{and}
\quad
\qH_{\ti_2,\ti_3}
$$ 
satisfying the following for $k=2,3$:
 \begin{enumerate}
 \item 
 $a_{\ti_k}\in \mathrm{int} (\qH_{\ti_k})$ and $\qH_{\ti_2}\cup \qH_{\ti_3} \subset \mathrm{int}(\qH_{\ti_2,\ti_3})$,
  \item
  $|\qH_{\ti_2,\ti_3}|\approx  |\qH_{\ti_2}|$
and $|\qH_{\ti_k}| \approx 
   |\alpha_1^{m_k} \beta_1^{n_k}| < \left(\frac{\beta_1}{\beta_2}\right)^{n_k}$, 
 % $\qD^k$ is in the interior of the rectangle $\qH(\delta)\times \qV_k$, where $\qH(\delta)\eqdef
%[1-\delta, 1+\delta]$, such
%that
\item
the restriction of $\psi_{\ti_k}$ to the rectangle $\qH_{\ti_k}  \times \qV(\vartheta)\subset \qU_\tq$
preserves the horizontal and vertical lines, it acts as the identity in the vertical direction, and
\item
$\psi_{\ti_k} (\qH_{\ti_2,\ti_3} \times \qV (\vartheta))= \qH_{\ti_k} \times \qV (\vartheta)$.
\end{enumerate}
\end{claim}

\begin{proof}
Recall that $\ti_2 =\ti_{(m_2, n_2)}$ and $\ti_3 =\ti_{(m_3, n_3)}$, the proof follows taking the $m_k,n_k$ sufficiently large.
Also recall $m_2<m_3$ and $n_2<n_3$.
We start with a pair of auxiliary intervals $\widehat{\qH}_{\ti_k}$ 
of size $\left(\frac{\beta_1}{\beta_2}\right)^{n_k}$ containing $a_{\ti_k}$.  
After performing an small adapted perturbation of $\mathcal{G}$, we may assume that the restriction of $\psi_{\ti_k}$ to  
$\qH_{\ti_k} \times \qV(\vartheta)$ preserves both horizontal and vertical directions, and acts as the identity along vertical lines. The set 
$\qH_{\ti_k}$
 in the claim will be a subset of 
$\widehat{\qH}_{\ti_k}$, and thus it satisfies condition (3). To get the intervals $\qH_{\ti_k}$ and  $\qH_{\ti_2,\ti_3}$,
note that the contraction of $\psi_{\ti_k}$ in the horizontal direction is of order of
$\alpha_1^{m_k} \beta_1^{n_k}$.
Therefore the horizontal size of  $\psi_{\ti_k} (\widehat{\qH}_{\ti_k} \times V(\vartheta) )$ is of order of
$$
\alpha_1^{m_k} \beta_1^{n_k}  \left( \frac{\beta_1}{\beta_2} \right)^{n_k}.
$$
%Since $|\alpha_2^{m_k} \beta_2^{m_k}|$ is bounded, 
The previous number can be take arbitrarily close to $0$ taking $m_k, n_k$ large enough.
This allows us to chose an  interval
$\qH_{\ti_2,\ti_3} \subset \widehat{\qH}_{\ti_2}\cap \widehat{\qH}_{\ti_3}$ and intervals
$\qH_{\ti_2}, \qH_{\ti_3} \subset \mathrm{int}(\qH_{\ti_2,\ti_3})$
such that $a_{\ti_k}\in \mathrm{int} (\qH_{\ti_k})$, 
with $|\qH_{\ti_k}| \approx  |\alpha_1^{m_k} \beta_1^{n_k}|  \leq |\qH_{\ti_2,\ti_3}| \approx |\alpha_1^{m_2} \beta_1^{n_2}|$, and 
$$
\psi_{\ti_k} (\qH_{\ti_2,\ti_3} \times \qV (\vartheta))= \qH_{\ti_k} \times \qV (\vartheta).
$$
Hence (1), (2), and (4) hold.
\end{proof}

We let $\widehat{{\qR}}_{\ti_k}\eqdef \qH_{\ti_k} \times \qV (\vartheta)$
and consider the canonical projections
 $$
 P_\yr,
   Q_\yr \colon \qU_\yr \to \mathbb{R}, \quad P_\yr(x,y)=x,\quad
   Q_\yr(x,y)=y, \quad \yr =\tq, \tp.
$$
Let $\mathcal{G}$ be a perturbation given in the Claim~\ref{cl.rectangulos}. After a new (standard)  perturbation one gets:

\begin{claim} \label{cl.pepinodecapri}
Given $\varepsilon>0$ there is
 $\vartheta_0>0$ such that for every $\vartheta\in (0,\vartheta_0)$, there is  an $\varepsilon$ perturbation of $\mathcal{G}$ (whose quotient maps will be denoted as in  $\mathcal{G}$)
 and rectangles
 \begin{equation}
 \label{e.irectangle}
 {\qR}_{\ti_k} \eqdef  \qH_{\ti_k} \times \qV_k(\vartheta), 
\quad\mbox{with}\quad \qV_k(\vartheta)\quad\mbox{close to}\quad \qV(\vartheta), 
 %\subset \qH_{\ti_1,\ti_2} \times \qV (\vartheta),
 \quad k=2,3
 \end{equation}
 containing $\qA_k$ in their interiors such that
\begin{enumerate}
\item
$\qV(\vartheta)   \subset \mathrm{int} ( Q_\tq ({\qR}_{\ti_2}) \cup Q_\tq ({\qR}_{\ti_3}))$,
\item
the restriction of
$\psi_{\ti_k}$ 
 to $\qR_{\ti_k}$ is affine with 
 derivative
 $\left(
 \begin{smallmatrix}
\rho_{\ti_k}& 0\\
0
& \theta_{\ti_k}
\end{smallmatrix} \right)$,
where 
$|\rho_{\ti_k}|$ is of order of $\alpha_1^{m_k} \beta_1^{n_k}$ (hence arbitrarily small)
and $|\theta_{\ti_k}| \in (0.99,1)$,
 \item
 $\qR_{\ti_k}\subset 
\psi_{\ti_k}^{-1}(\qH_{\ti_2,\ti_3} \times \mathrm{int} (\qV(\vartheta))$, where 
$\qH_{\ti_2,\ti_3}$  is the interval in Claim~\ref{cl.rectangulos} containing $\qH_{\ti_2} \cup \qH _{\ti_3}$ in its interior.  
 \end{enumerate}
  \end{claim}
      
  Arguing as above, replacing the points $\qA_{\ti_k}$ by  $\qD'_{\tj_k}$ and having again in mind the independence of the adapted perturbations to get the periodic points, it follows:
      
 \begin{claim} \label{cl.pepinodemodena}
 Given $\varepsilon>0$ there exist
 $\delta_0>0$ such that for every $\delta\in (0,\delta_0)$
there are an $\varepsilon$ adapted perturbation $\mathcal{G}$ of $\mathcal{F}$
satisfying Claim~\ref{cl.pepinodecapri}
and  rectangles
 \begin{equation}
 \label{e.jrectangle}
 {\qR}_{\tj_k} \eqdef \qH_k(\delta) \times \qV_{\tj_k}, 
 \quad\mbox{with}\quad \qH_k(\vartheta)\quad\mbox{close to} \quad
  \qH(\delta)  \eqdef [1-\delta,1+\delta],
 \quad k=2,3.
 \end{equation}
 containing $\qD_k$ in their interiors such that
\begin{enumerate}
\item
$\qH(\delta)  \subset \mathrm{int} ( P_\tp ({\qR}_{\tj_2}) \cup P_\tp ({\qR}_{\tj_3}))$,
\item
the restriction of
$\psi_{\tj_k}$ 
 to $R_{\tj_k}$ is affine with 
 derivative
 $\left(
 \begin{smallmatrix}
\lambda_{\tj_k}& 0\\
% b_{\tj_r}\\
%c_{\tj_r}&
0
& \eta_{\tj_k}
\end{smallmatrix} \right)$
%\qquad \qquad 
%\left(
%\begin{smallmatrix}
%\tau_{\tj_k}& 0\\
%% b_{\tj_r}\\
%%c_{\tj_r}&
%0
%& \varrho_{\tj_k}
%\end{smallmatrix}
%\right)
%%$\kappa \eqdef \min\{\kappa_{\tj_\ell}, \kappa_{\tj_r}\}$.}
%\end{equation}
where 
$|\lambda_{\tj_k}|\in (1,1.01)$ and 
 $|\eta_{\tj_k}|$ is of order of $\alpha_2^{t_k}\beta_2^{r_k}$ (hence arbitrarily large),
 \item
 $\psi_{\tj_k} (\qR_{\tj_k})\subset 
 \mathrm{int} (\qH(\delta)) \times \qV_{\tj_2,\tj_3}$, where  $\qV_{\tj_2,\tj_3}$ is an interval 
 whose interior contains
  $\qV_{\tj_2} \cup  \qV_{\tj_3}$.
 \end{enumerate}
  \end{claim}

\section{Split blending machines}
\label{s.splitblendermachines}
We introduce {\em split blending machines}, extending the concept of a blending machine from \cite[Section 2.4]{Asa:22}, with the key distinction of an additional ``center direction''. Accordingly, we revisit and adapt the tools developed in \cite{Asa:22}.
Remark~\ref{r.onthedefbm} provides a comparison of these two concepts.

These machines provide an abstract setting that captures the mechanism behind blenders constructed using the covering property (see Remark \ref{r.dc=0}). 
The setting involves local dynamics, invariant splittings, and systems of rectangles. Although the presence of splittings suggests partial hyperbolicity, no hyperbolicity is assumed. The system of rectangles displays intersection properties similar to those found in blenders, placing split blending machines as nonhyperbolic counterparts to blenders, see Section~\ref{preblendingmachinesintheworld}.
Before, we provide some preliminaries in Section \ref{ss.prerequisites}, followed by the definition and main properties in Section \ref{ss.blendermachines}.

\subsection{Splittings, cones, and rectangles}
\label{ss.prerequisites}
 For $n\geq 1$, denote  the box norm in $\mathbb{R}^n$ by
$$
\|x\|=
\|(x_1,\dots, x_n)\|\eqdef \max \{\,|x_1|,\dots, |x_n|\,\}, \qquad x=(x_1,\dots, x_n)\in \mathbb{R}^n.
$$ 
We identify the tangent space $T_x\mathbb{R}^n$ with $\mathbb{R}^n$.

Given manifolds $M_1$ and $M_2$, an open set  $U$ of $M_1$, and   a
$C^1$ map $f \colon U \to M_2$, we denote by $Df$ the tangent
map of $f$ between the tangent bundles $TU$ and $TM_2$. For notational simplicity, when no confusion can arise, given  $x\in U$ and $v\in T_xU$, we simply write $Df(v)$ instead $D_xf(v)$.

In what follows,
fix an open Riemannian manifold~$N$ and
positive integers $p$ and $q$.

\begin{defi}[Pairs, splittings, rectangles, and cones]
Given $C^1$ maps
$P \colon N \to \mathbb{R}^p$ and $Q \colon N \to \mathbb{R}^q$,
consider the map $\fS=(P,Q)$ defined by
$$
\fS \colon N\to \mathbb{R}^{p}\times \mathbb{R}^q, \quad \fS(x) \eqdef \big(P(x),Q(x)\big).
$$
We say that $\fS$ is a \emph{$(p,q)$-pair} of $N$.
Associated to $\fS=(P,Q)$, we consider the $(q,p)$-pair $\fS^t\eqdef (Q,P)$. 

 We say that $\fS$ is a {\em{$(p,q)$-splitting}} of $N$  if it
is a diffeomorphism onto its image. In particular,  the dimension of $N$ is $p+q$.

 A subset $R$ of $N$ is called a {\em{$\fS$-rectangle}} if the $(p,q)$-pair $\fS=(P,Q)$
 is a diffeomorphism between $R$   and $P(R) \times Q(R)$. For $\theta>0$ and $x\in N$, we define the \emph{$(\fS,\theta)$-cone} at $x$ by
\begin{equation}
\label{e.conocono}
C_\theta(x, \fS) \eqdef  \left\{ v\in T_x N \colon \, \|DQ (v)\|\leq \theta \|DP (v)\| \right\}.
\end{equation}
\end{defi}

Throughout the paper, we will use the terms ``splitting" and ``splitting pair" interchangeably. In particular, we will refer to $\fS$ as a splitting pair.

\begin{rem} We always assume that the $(p,q)$-pairs $\fS=(P,Q)$ are {\em{non-degenerate,}}
meaning that both kernels, $\mathrm{Ker}\,D_xP$ and $\mathrm{Ker}\,D_xQ$, are not equal to $T_xN$ for every $x\in N$.
\end{rem}

\begin{defi}[Invariance, dilatation, and domination]
\label{d.invariance}
Let %and $N_2$
 $\fS_\tau=(P_\tau,Q_\tau)$ be a $(p,q)$-pair of an open manifold $N_\tau$, with $\tau=1,2$. %nd  $N_2$, respectively.
%Consider  open manifolds $N_1$ and $N_2$ and $(\ell,m)$-pairs $\fS_1=(P_1,Q_1)$ and  $\fS_2=(P_2,Q_2)$ of $N_1$ and  $N_2$, respectively.
Consider an open set $U$ of $N_1$,
a $C^1$ embedding  $f \colon U \to N_2$,  a subset $R\subset U$, and constants
 $\mu, \gamma, \theta>0$.
\begin{enumerate}[leftmargin=1cm]
\item
The map $f$ satisfies the
\emph{$(\fS_1, \fS_2,\theta)$-cone invariance condition} on  $R$  if there exists $0 < \theta' < \theta$ such that for every $x \in R$ it holds
$$
 Df^{-1}\big(C_\theta(f(x), \fS_2)\big) \subset C_{\theta'}(x,\fS_1) \quad \text{and} \quad
Df\big(C_\theta(x,\fS^t_1)\big) \subset C_{\theta'}(f(x),\fS^t_2).
$$ 
\item
The map $f$  satisfies
{\em{$(\fS_1,\fS_2,\theta, \mu,\gamma)$-dilatation condition on $R$}} if the following holds
\begin{align*}
\|D(P_1\circ f^{-1})(w)\| &> \mu \|DP_2(w)\| \quad  \text{for every $w\in C_{\theta}\big(f(x),\fS_2\big)\setminus\{0\}$ and $x\in R$}, \\
%$$
%\|D(P_2\circ f )(v)\| < \lamhypbda \|DP_1(v)\| \quad  \text{for all }
%$$
 \|D(Q_2\circ f)(v)\| &> \gamma \|DQ_1(v)\|  \quad   \text{for every $v\in C_{\theta}(x,\fS^t_1)\setminus\{0\}$ and $x\in R$.}
\end{align*} 
\item
The map $f$ satisfies the \emph{$(\fS_1,\fS_2, \theta, \mu, \gamma)$-cone  condition} on $R$ if\
it satisfies the $(\fS_1, \fS_2,\theta)$-cone invariance  and the 
$(\fS_1,\fS_2, \theta, \mu,\gamma)$-dilatation conditions on $R$.
\item
The map $f$ is
 \emph{$(\fS_1,\fS_2, \mu, \gamma)$-dominated} on $R$
 if it satisfies the $(\fS_1, \fS_2,\theta,\mu,\gamma)$-cone condition on $R$ for every $\theta>0$.
\end{enumerate}
When $N=N_1 = N_2$
and $\fS=\fS_1=\fS_2$, we refer to these conditions as $(\fS,\theta)$-cone invariance, $(\fS,\theta, \mu, \gamma)$-cone, and  $(\fS,\theta, \mu, \gamma)$-dilatation and $(\fS, \mu, \gamma)$-domination. % and the $(\fS, \mu, \gamma)$-dilatation condition.
% \end{enumerate}
\end{defi}

Using  the notation and  terminology  in Definition~\ref{d.invariance},  we have the following:

\begin{lem} \label{l.new-dilatation}
Suppose that $f$ is $(\fS_1,\fS_2, \mu, \gamma)$-dominated on $R$. Then
 for every $x\in R$ and every  $v, w\in T_xU$ such that
 $v\not\in \mathrm{Ker}\, D_x(P_2\circ f)$ and
$w\not\in  \mathrm{Ker}\, D_xQ_1$ it holds
 \begin{equation}\label{e.lambda-gamma-dilatation}
\|D(P_2\circ f )(v)\|  < \mu^{-1}\, \|DP_1(v)\| \quad \text{and} \quad \|D(Q_2\circ f)(w)\| >
{\gamma} \, \|DQ_1(w)\|.
\end{equation}
When $\mu \gamma >1$,  the inequalities in \eqref{e.lambda-gamma-dilatation}
imply that
$f$ is $(\fS_1,\fS_2, \mu, \gamma)$-dominated on $R$.
\end{lem}

\begin{proof}
We prove the first inequality in \eqref{e.lambda-gamma-dilatation},
the other one follows analogously.
By hypothesis,  the map
 $f$ satisfies the $(\fS_1, \fS_2,\theta,\mu,\gamma)$-cone  condition on $R$ for every
 $\theta>0$. In particular, 
 $f$ satisfies the $(\fS_1, \fS_2,\theta,\mu,\gamma)$-dilatation  condition on $R$,  therefore, %that
 \begin{equation} 
 \label{e.lambda} 
 \|DP_1(v)\| > \mu \|D(P_2\circ f)(v)\|  
 \quad 
 \text{for every $x\in R$ and $v\in Df^{-1}\big(C_{\theta}(f(x),\fS_2)\big)\setminus\{0\}$}. 
\end{equation}
Since $Df^{-1}$ is an isomorphism  and 
\begin{equation} 
\label{eq.kernel}
\bigcup_{\theta>0}  C_{\theta}\big(f(x),\fS_2\big)\setminus\{0\} =T_{f(x)}f(U) \setminus \mathrm{Ker} \,D_{f(x)}P_2
\end{equation}
condition
\eqref{e.lambda} holds for every $v\in T_x U \setminus Df^{-1}(\mathrm{Ker} \, D_{f(x)}P_2)$,
implying the first inequality and proving the first part of the lemma.

Assuming now that $\mu\gamma>1$ and equation \eqref{e.lambda-gamma-dilatation} holds,
 we prove that   $f$ satisfies the $(\fS_1,\fS_2,\theta,\mu,\gamma)$-cone condition on ${R}$
 for every $\theta>0$.
 To obtain the $(\fS_1, \fS_2,\theta)$-cone invariance condition on $R$ take
 \begin{equation}
\label{e.thetalinha}
 \theta'\eqdef \theta  \, (\mu\gamma)^{-1} < \theta.
  \end{equation}
To get $Df\big(C_\theta(x,\fS^t_1)\big) \subset C_{\theta'}\big(f(x),\fS^t_2\big)$,
note that
\begin{equation}
\label{e.semluz}
\|DP_1(v)\| \leq   \theta \, \|DQ_1(v)\|,
\end{equation}
for every $x\in R$ and  every $v\in C_\theta(x,\fS^t_1)$. Hence, 
 \begin{align*}
\|D({P}_2\circ f)(v)\| 
&\overset{\eqref{e.lambda-gamma-dilatation}}{<}   \mu^{-1} \|DP_1({v})\| \overset{\eqref{e.semluz}}{\leq} \mu^{-1} \theta \, \|DQ_1(v)\|  \\ 
 &\overset{\eqref{e.lambda-gamma-dilatation}}{<}  \mu^{-1} \gamma^{-1} \theta \, \|D({Q}_2\circ f)({v})\| \overset{\eqref{e.thetalinha}}{=} \theta' \, \|D({Q}_2\circ f)({v})\|.
     \end{align*}
Thus, $Df(v)\in C_{\theta'}\big(f(x),\fS_2^t\big)$, proving the inclusion.
To get  $Df^{-1}(C_\theta\big(f(x), \fS_2)\big) \subset C_{\theta'}(x,\fS_1)$, observe that
  \begin{equation}
  \label{e.chuvaa}
 \|DQ_2({w})\| \leq \theta \, \|DP_2({w})\| \,\,\mbox{for every $w \in C_{\theta}\big(f(x),\fS_2\big)$}.
  \end{equation}

 Take any ${w} \in C_{\theta}\big(f(x),\fS_2\big)$  and consider
 ${v} =Df^{-1}({w}) \in T_xN_1$. Then
 %From \eqref{e.lambda-gamma-dilatation} follows
  \begin{align*}
       %\|D({Q}_{u}_1\circ f^{-1})({w})\| =
       \|D{Q}_1 ( v)\|  &\overset{\eqref{e.lambda-gamma-dilatation}}{\leq} 
       \gamma^{-1} \|D(Q_2\circ f)({v})\| = \gamma^{-1} \|DQ_2({w})\|   \overset{\eqref{e.chuvaa}}{\leq}  \gamma^{-1} \theta \, \|DP_2( w)\|
       \\
       &= \gamma^{-1} \theta \, \|D(P_2\circ f)( v)\| \overset{\eqref{e.lambda-gamma-dilatation}}{\leq}   \gamma^{-1} \mu^{-1}  \theta \, \|D{P}_1({v})\| \overset{\eqref{e.thetalinha}}{=}  \theta' \, \|D{P}_1({v})\|. % \|(D{P}_1_{cs}\circ f^{-1})({w})\|.
     \end{align*}
 This implies that $Df^{-1}( w) \in C_{\theta'}(x,\fS_1)\setminus\{0\}$. Thus,  the $(\fS_1,\fS_2,\theta)$-cone invariance condition on $R$ holds for $f$.

 It remains to prove the dilatation condition in the cones. For
 $w\in C_{\theta}\big(f(x),\fS_2\big)\setminus\{0\}$, letting  ${v} =Df^{-1}({w})$,
 equation ~\eqref{eq.kernel} implies that $\|D(P_2\circ f)(v)\|=\|DP_2(w)\|\not =0$. Thus, $v\in T_xU\setminus \mathrm{Ker}\,D_x(P_2\circ f)$ and
 %But, by the cone condition we have that $0<\theta'<\theta$ such that $$Df^{-1}(C_{\theta}(f(x),\fS_2))\subset C_{\theta'}(x,\fS_1) \subset T_{x}U \setminus \mathrm{Ker}\,D_{x}P_1.$$ Thus,
%we have also $\|D(P_1\circ f^{-1})(w)\|=\|DP_1(v)\| \not = 0$ where ${v} =Df^{-1}({w})$.  Then,
equation ~\eqref{e.lambda-gamma-dilatation}  implies
\begin{align*}
\|D(P_1\circ f^{-1})(w)\| &> \mu \|DP_2(w)\| \quad  \text{for every $w\in C_{\theta}(f(x),\fS_2)\setminus\{0\}$}.
\end{align*}
Similarly, if  $v\in C_{\theta}(x,\fS_1^t) \subset T_xU \setminus \mathrm{Ker}\, D_xQ_1$, then $\|DQ_1(v)\|\not =0$ and
%$$Df(v) \in C_{\theta}(x,\fS_1^t) \subset C_{\theta'}(f(x),\fS_2^t)\subset T_{f(x)}f(U)\setminus \mathrm{Ker}\, D_{f(x)}Q_2.$$ Hence $\|D(Q_2\circ f)(v)\|\not  = 0$ and thus,
from~\eqref{e.lambda-gamma-dilatation}  follows that
$$
\|D(Q_2\circ f)(v)\| > \gamma \|DQ_1(v)\|  \quad   \text{for every $v\in C_{\theta}(x,\fS^t_1)\setminus\{0\}$.}
$$
Therefore, $f$ satisfies the $(\fS_1,\fS_2,\theta,\mu,\gamma)$-cone  condition on $R$
for every $\theta>0$, ending the proof.
\end{proof}

%\begin{rem} The condition $\mu^{-1}<\gamma$ is not a requirement for the direct implication in the above lemma. That is, it always holds that if $f$ satisfies the $(\fS_1,\fS_2, \mu, \gamma)$-dilatation condition on $R$ then it holds~\eqref{e.lambda-gamma-dilatation-corrected} and consequently~\eqref{e.lambda-gamma-dilatation}.
%\end{rem}

\begin{rem}
\label{r.robustness}
 In general, $(\fS_1,\fS_2,\mu,\gamma)$-domination of $f$ does not hold
for small $C^1$ perturbations of $f$. In that sense, $(\fS_1,\fS_2,\mu,\gamma)$-domination is not  a robust property. However,  fixed any $\theta>0$,
the $(\fS_1,\fS_2,\theta,\mu,\gamma)$-cone  condition on a compact set  is a $C^1$ robust property. This sort
of robustness is enough in our setting.
\end{rem}

The next two lemmas
 will play a key role in
Section~\ref{s.connectingsplitblenders}.
The first one is a version of \cite[Corollary 2.7]{Asa:22} in our setting, it
 deals with diameters of images of rectangles by compositions of pairs and maps.

Given a  set $S\subset \mathbb{R}^n$, we denote by $\diam(S)$  its  diameter (with respect
to the metric induced by the box norm).

\begin{lem} \label{lem.new.asaoka}
Let $N_\tau$ be open manifolds and
$\fS_\tau=(P_\tau,Q_\tau)$ a $(p,q)$-splitting of $N_\tau$, where $\tau=1,2$. Let
 $U$ be an open subset of $N_1$, $f \colon U\to N_2$ a $C^1$ embedding, and 
 $R$ a
$(P_1,Q_2 \circ f)$-rectangle in $U$.

Consider $C^1$ maps $P\colon \mathbb{R}^p \to \mathbb{R}^\ell$ and $Q \colon\mathbb{R}^q \to \mathbb{R}^{m}$, where  $1\leq \ell < p$ and $1\leq m < q$, and
the $(\ell,m)$-pairs $\tfS_\tau=(P\circ P_\tau,{Q} \circ Q_\tau)$ of $N_\tau$ with $\tau=1,2$.

Suppose that $f$ satisfies %both, the $(\fS_1,\fS_2, \theta, \mu, \gamma)$-cone and
the $(\tfS_1,\tfS_2, \theta, \mu, \gamma)$-cone condition  on $R$. Then,
\begin{equation}
 \label{meq:4}
\begin{split}
\diam \big( Q\circ Q_1(R)\big) & \leq \theta\, \diam  \big( P\circ P_1(R)  \big)+ \gamma^{-1} \,
\diam  \big( Q\circ Q_2(f(R)) \big), \\
\diam  \big(P\circ P_2(f(R)) \big) & \leq   \theta\, \diam  \big(Q\circ Q_2(f(R))  \big)+ \mu^{-1}\,
\diam  \big(P\circ P_1(R) \big).
\end{split}
\end{equation}
\end{lem}
\begin{proof}
Fix any pair of points $x, x'\in R$ and  the point
$\big(P_1(x), Q_2(f (x'))\big)\in P_1(R) \times Q_2\big(f(R)\big)$.
Since $R$ is a $(P_1,Q_2\circ f)$-rectangle, there exists $x_*\in R$
such that
$$
P_1(x_*) = P_1(x) \quad \text{and} \quad Q_2 \big( f(x_*)\big) = Q_2 \big( f(x')\big).
$$
Define the following submanifolds
$$
R^- \eqdef \{y \in R : \,  Q_2(f(y)) = Q_2(f(x_*))\} \quad \text{and} \quad
R^+ \eqdef \{y \in R : \, P_1(y) = P_1(x_*)\}.
$$
Note that  $\{x', x_*\} \subset  R^-$ and $\{x, x_*\} \subset  R^+$. In particular,
\begin{equation}
\label{e.igualdad}
Q \circ Q_2 \big( f(y)\big) %=({Q}_c \circ Q_2 \circ f)(y)= {Q}_c((Q_2 \circ f)(x_*))
=Q \circ Q_2 \big( f(x_*)\big)
\quad \text{for every $y\in R^-$}.
\end{equation}

Given any vector $v\in T_yR^-$,
equation \eqref{e.igualdad} implies that $D(Q\circ Q_2\circ f)(v)=0$. Thus
 $Df(v)\in C_\theta \big(f(y),\tfS_2\big)$. Since $f$ satisfies the $(\tfS_1,\tfS_2, \theta)$-cone invariance condition on $R$, it holds
$$
v \in Df^{-1}\big(C_\theta(f(y),\tfS_2)\big)\subset C_\theta(y,\tfS_1).$$
Therefore, by the definition of $(\tfS_1,\theta)$-cone (recall equation \eqref{e.conocono}),  we have
$$
   \|D (Q\circ Q_1) (v)\| \leq  \theta \, \|D (P\circ P_1) (v)\| \quad \text{for every $y\in R^-$ and $v\in T_yR^-$.}
$$
We claim that this
 implies
 that
\begin{equation} \label{eq:QRmenos}
    \diam  \big(Q\circ Q_1(R^-) \big) \leq \theta\, \diam  \big(P\circ P_1(R^-) \big).
\end{equation}
In fact, 
 take  any regular curve $t\in [0,1] \mapsto \rho(t)\in R^-$ and consider  $\alpha(t)\eqdef
 Q\circ Q_1
 \big(\rho(t)\big)$ and  $\beta(t)\eqdef P\circ P_1\big(\rho(t)\big)$. The inequality in~\eqref{eq:QRmenos}
 is obtained from the following relation between the lengths of $\alpha$ and $\beta$:
$$
  \int_{0}^1 \|{\alpha}'(t)\|\, dt = \int_0^1 \|D
  (Q\circ Q_1)
  ( {\rho}'(t))\|\, dt \leq \theta \int_0^1 \|D
  (P\circ P_1)({\rho}'(t))\| \, dt =
  \theta \int_{0}^1 \|{\beta}'(t)\|\, dt.
$$
%\pmargem{la notacion de "'" no esta buena. se congunde con derivada. Usar $\dot{\beta}$ para la derivada}

Similarly, for every $y' \in R^+$ and every vector $v'\in T_{y'}R^+$ it holds $D(P\circ P_1)(v') = 0$,
and hence $v'\in C_\theta(y',\tfS^t_1)$. The $(\tfS_1,\tfS_2,\theta,\mu,\gamma)$-dilatation condition implies that
$$
\|D(Q\circ Q_1) (v')\| \leq \gamma^{-1}\,\| D(Q\circ Q_2 \circ f)(v')\|.
$$
As above, this implies that
\begin{equation} \label{eq:QRmas}
\diam  \big(Q\circ Q_1(R^+) \big) \leq  \gamma^{-1} \diam  \big(Q\circ Q_2(f(R^+)) \big).
\end{equation}
Since $\{x', x_*\}\subset R^-$ and $\{x, x_*\} \subset R^+$, it follows  that
%\lmargem{I guess that now... take any point $x_\ast$.... if this point does not exist there is no problem...}
%\textcolor{red}{
%\[
%\begin{split}
%\|\mathscr{Q}_1(x) - \mathscr{Q}_1(x')\| &\leq  \theta\, \diam ( \mathscr{Q}_1(R^-)) + \gamma^{-1}\,\diam (\mathscr{Q}_1(f(R^+))),
%\end{split}
%\]}
%\lmargem{I do not get the previous equation, obtaining a somewhat different.... that provides the same estimate.... first line}
\[
\begin{split}
\|(Q\circ Q_1)(x') - (Q\circ Q_1)(x)\| 
&\leq
\|  Q\circ Q_1 (x')-Q\circ Q_1 (x_*)  \| + \|  Q\circ Q_1(x_*)  -Q\circ Q_1 (x) \|
\\
& \leq
\diam  \big(Q\circ Q_1 (R^-) \big) + \diam  \big(Q\circ Q_1(R^+) \big)
 \\
&%\overset{\eqref{eq:QRmenos}, \eqref{eq:QRmas}}{\leq}
\leq 
\theta\, \diam  \big( P\circ P_1(R^-) \big) + \gamma^{-1}\,\diam  \big(Q\circ Q_2 (f(R^+)) \big)
 \\
&\,\,\leq
\theta\, \diam  \big( P\circ P_1(R) \big) + \gamma^{-1}\,\diam  \big(Q\circ Q_2 (f(R)) \big).
\end{split}
\]
where the third inequality follows from~\eqref{eq:QRmenos} and~\eqref{eq:QRmas} and
the last inequality follows from  $R^-, R^+ \subset R$. As the points $x,x'\in R$ are arbitrary, the first inequality in~\eqref{meq:4} follows.

The second
inequality in~\eqref{meq:4} follows similarly. This ends the proof of the lemma.
\end{proof}

The following result is  a version of  Lemma 2.5 in \cite{Asa:22} in our setting, it  deals with dynamically defined rectangles under compositions of embeddings and pairs.

\begin{lem}[Lemma 2.5 in \cite{Asa:22} with $N=3$]
\label{l.simpleasa}
Consider open manifolds $N_i$, open subsets $U_i$ of $N_i$, and
$(p, q)$-splittings $\fL_i=(P_i,Q_i)$
%\lmargem{no podemos llamar a estos splittings $\fS$, cuando los usamos el $\fS$ esta usado}
of $N_i$, where $i=1,\dots, 4$.
Consider, for $i=1,2,3$, embeddings $g_i\colon U_i \to N_{i+1}$ and
compact
 $(P_i, Q_{i+1} \circ g_i)$-rectangles $R_i$ in $U_i$ such that
$$
P_{j+1} (g_j (R_j)) \subset \mathrm{int} ( P_{j+1} (R_{j+1}))
\quad
\text{and}
\quad
 Q_{j+1} (R_{j+1}) \subset \mathrm{int} (Q_{j+1} (g_j (R_j)), \quad j=1,2.
$$
Let
$$
 G_3\eqdef g_3 \circ g_2\circ  g_1,
 \quad G_2\eqdef g_2\circ  g_1,
\quad
 G_1\eqdef g_1
$$
be defined on the maximal open subset of $U_1$ where the composition is well defined.
Suppose that there is $0<\theta\leq1$ such that $g_i$ satisfies the $(\fL_i,\fL_{i+1},\theta)$-cone invariance condition on $R_i$ for $i=1,2,3$.
Then the set
$$
R_\ast \eqdef R_1 \cap G_1^{-1}(R_2) \cap G_2^{-1} (R_3)
$$
is a $(P_1,Q_4\circ G_3)$-rectangle
such that
$$
P_1 (R_\ast)=P_1(R_1) \quad \text{and} \quad
Q_4 \circ G_3 (R_\ast) = Q_4 \circ g_3 (R_3).
$$
In particular, the set $R_\ast$ is non-empty.
\end{lem}

\subsection{Split blending machines}
\label{ss.blendermachines}
We now extend the definition of a blending machine from \cite[Section~2.4]{Asa:22} by adding a central direction.  
To this end, we introduce conditions (BM3) and (BM4) in Definition~\ref{def:cu-blender}.  
For further discussion, see Remark~\ref{r.onthedefbm}.

Given two sets $A$ and $B$, the notation $A\Subset B$ means that the closure of ${A}$ is contained in the interior
of $B$.

\begin{defi}[Split blending machine] \label{def:cu-blender}
Let  $f\colon U \to N$ be a $C^1$ map defined on an open subset $U$ of a manifold $N$. Let $d_s$, $d_c$, and $d_u$ be positive integers and set $d_{cs}\eqdef d_s+d_c$. Consider (see Figure
~\ref{fig:Machine})
\begin{enumerate}[label=$\bullet$, leftmargin=0.5cm ]
\item
 a family
  $\blR=\{R_i\}_{i\in I}$
  %\fmargem{explicite $I\ge 2$, ver que eso se deduce no es tan simple.... y una def debe ayudar}
 of compact subsets of $U$, where $I$ is a finite set with at least two elements,
 \item
 a $(d_{cs},d_{u})$-splitting $\fS=(P_{cs},Q_u)$ of {$N$},
 \item
a $(d_s,d_c)$-splitting $\fC=(P_{s},P_{c})$ of {$\mathbb{R}^{d_{cs}}$},
%where
%$d_s\eqdef d_{cs}-d_c$ and
% $1\leq d_{c} < d_{cs}$,
%\PB{No deberia ser  $U$?.} %\PB{No. Debe ser $N$ porque la condicion de cono requiere splitting al menos em $f(U)$. Lo mastht facil es decir que $f(U)\subset N$ con $N$ variedad abierta con splitting. }
\item
a   $\fS$-rectangle $Z\subset U$, and
\item a real number {$0<\theta<1$}.
\end{enumerate}
The tuple
$\cB_f= (
\blR, Z, \fS, \fC, \theta)$ is a {\em{split blending machine for $f$}}
 if the following holds:
\begin{enumerate}[itemsep=0.2cm,leftmargin=1.2cm]
\item[(BM1)]
 for every $i\in I$, the restriction of $f$ to some neighbourhood of $R_i$ is a $C^1$ embedding satisfying the
$(\fS,\theta)$-cone invariance condition,
\item[(BM2)]
 for every $i \in I$, the set $R_i$ is a $(P_{cs},Q_u \circ f)$-rectangle such that
$$
Q_u(R_i) \Subset Q_u(Z), \quad (Q_{u}\circ f)(R_i) = Q_u(Z), \quad \text{and} \quad
(P_{cs}\circ f)(R_i) \Subset P_{cs}(Z),
%;% \quad \text{and} \quad P_s(R_i)=P_s(Z);
$$
%\begin{align*}
%P_u(R_i) &\Subset P_u(Z)  && \text{and} &   P_s(R_i)&=P_s(Z), \\
%P_{cs}(f(R_i)) &\Subset P_{cs}(Z)  && \text{and} &   P_{u}(f(R_i)) &= P_u(Z);
%\end{align*}
\item[(BM3)]
 for every $i \in I$, the set $P_{cs}(R_i)$ is a $\fC$-rectangle such that
%\textcolor{red}{denoting ${P}'_s\eqdef P_s\circ P_{cs}$}
%\begin{equation}
%\label{e.above}
%{P}'_s(R_i) = {P}'_s (Z)
%\quad
$P_s\circ  P_{cs} (Z)\Subset P_s\circ P_{cs}(R_i)$,
%\quad  \text{and} \quad
%$$
%\overline{(P_c\circ  P_{cs}) (Z)} \subset \bigcup_{j\in I} \mathrm{Int}( (P_c\circ  P_{cs}) (R_j)).
%$$
%\end{equation}
%{\sout{Moreover, the number
%$d_c$ is the minimum positive integer with such a property }}
%\textcolor{red}{(i.e.,
%assuming that $d_c\ge 2$,
%for every $d\in [1, d_c)$ there is no
%$(d_{cs}-d,d)$-splitting  of $\mathbb{R}^{d_{cs}}$ satisfying \eqref{e.above}).}
\item[(BM4)]
 $\{\mathrm{int}\big((P_c\circ  P_{cs})(R_i)\big) \}_{i\in I}$ is a covering
  of  $\overline{P_c\circ  P_{cs}(Z)}$
  having  a  Lebesgue number $\Delta$
  such that
 $$
 \theta\, {\|DP_c\|} \, \mathrm{diam} \, Q_u(Z)<\Delta.
 $$
\end{enumerate}
We call     $\fS$ the
 {\em{{stable center splitting}}},
   $\fC$ the
   {\em{center splitting,}} {$P_c$ the \emph{central projection}},
     $Z$
%\lmargem{esto ha sido editado} \pmargem{llevar para fuera de la definicion}
   the {\em{superposition domain,}}
   $\theta$  the {\em{cone width,}}
   and $\Delta$
   {the} {\em{Lebesgue number}}
   of the split
   blending machine.
% \textcolor{red}{The number $d_{c}$ is called the  \emph{dimension of the central projection.}}
 We say that the split blending machine $\cB_f$ has dimensions  $(d_s, d_c, d_u)$.
\end{defi}

\begin{figure}
\centering
\begin{overpic}[scale=0.3,
%grid,tics=5
]{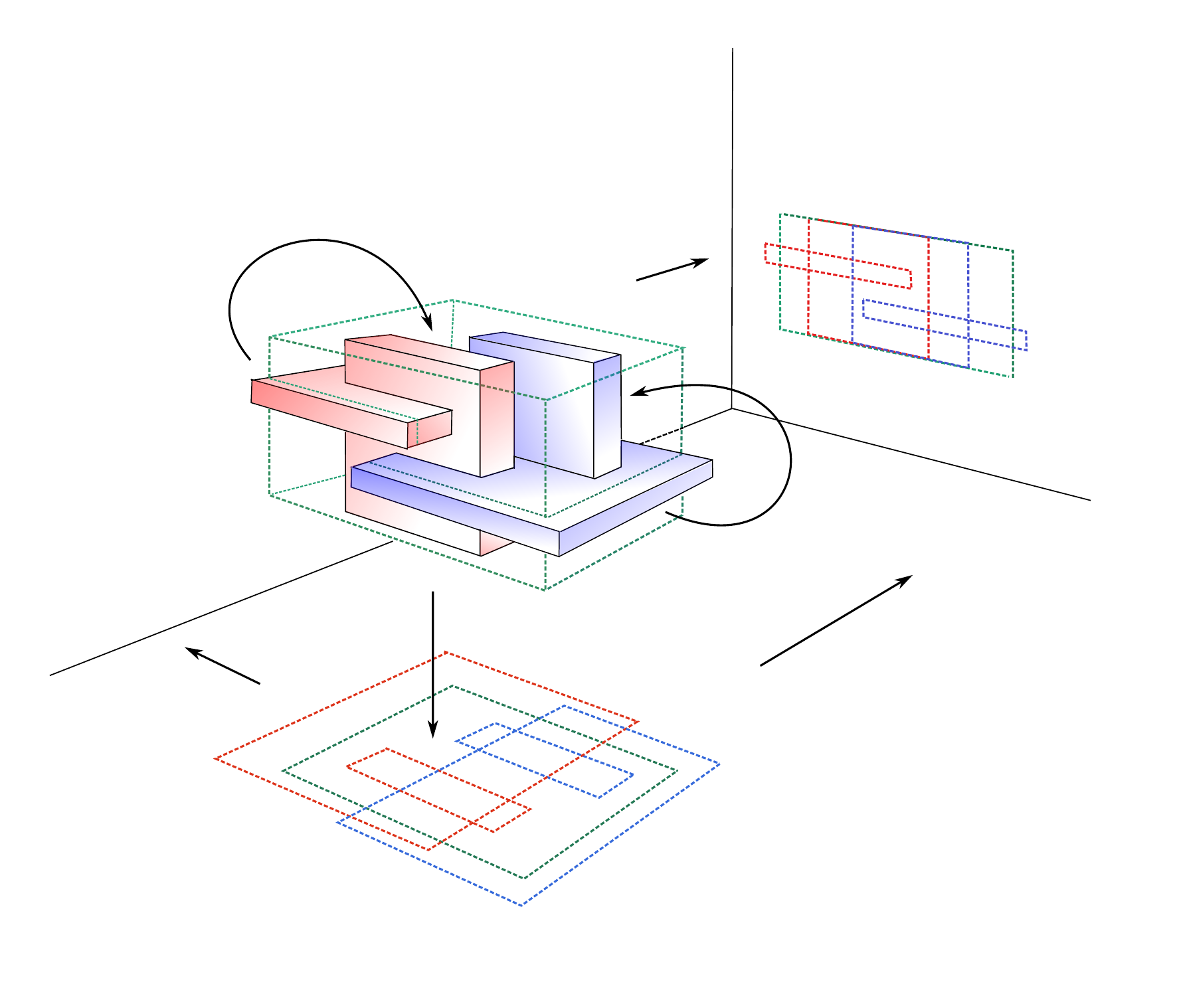}
    \put(03,31){  \Large {$\mathbb{R}^{d_s}$}}
     \put(21,30){\Large {$P_{s}$}} 
    \put(63,78){ \Large {$\mathbb{R}^{d_u}$}}
      \put(15,52){\Large {$R_1$}}
          \put(50,36){\Large{$R_2$}}
      \put(38,32){\Large {$P_{cs}$}} 
        \put(53,65){{$Q_{u}$}}
         \put(24,68){\large{$f$}}
          \put(68,43){\large{$f$}}
    \put(90,45){\Large {$\mathbb{R}^{d_c}$}}
    \put(71,28){\Large {$P_{c}$}}
    \put(130,195){\large{$Z$}}
    \put(159,117){{$R_2$}}
 \end{overpic}
 \vspace{-0.9cm}
 \caption{Split blending machine}
 \label{fig:Machine}
\end{figure}
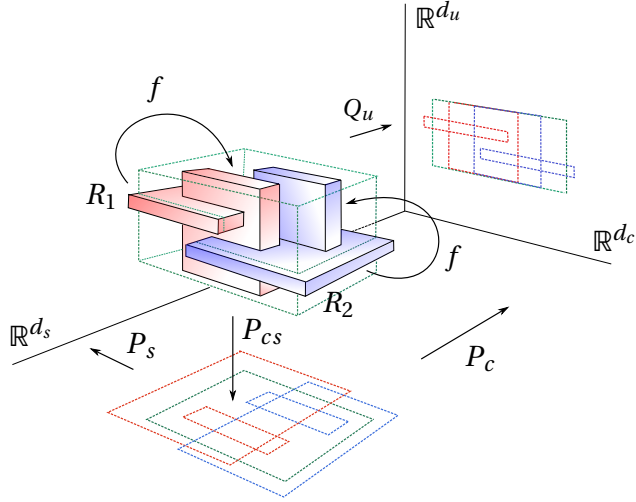

\begin{rem}[Split blending  machines versus blending machines]~
\label{r.onthedefbm}
\begin{enumerate}[label=(\roman*), itemsep=0.2cm,leftmargin=0.8cm]
\item
As  $d_s, d_c, d_u\geq 1$ then $\dim (N) \geq 3$.
\item
%Recalling that in a $(0,m)$-splitting $\fC=(P,Q)$ the projetion $P$ is the trivial projection on $\{0^m\}$. Hence,
The definition of a blending machine in \cite{Asa:22} is recovered from 
 Definition~\ref{def:cu-blender} taking $d_{cs}=d_c$  and $\fC=(P_{s},P_{c})$,
with $P_s$ the trivial projection onto $\{\bar 0\}$ and $P_{c}$ the identity map in $\mathbb{R}^{d_{cs}}$.

\item
Blending machines in \cite{Asa:22} are defined in dimension at least two.
In particular, any horseshoe of a surface diffeomorphism satisfies the definition of
blending machine. To avoid this case, for that we assume that $d_s\geq 1$ and hence $d_{c}<d_{cs}$.
\item
Condition $\theta<1$ is implicitly used in \cite{Asa:22}.
Here this condition is explicitly used in several  proofs along the paper.
% of  Lemma \ref{l.}. \lmargem{tal vez decir donde se usa en Asaoka... y ver
%nuestra ref}
%{\color{red}
%\item
%In \cite{Asa:} blending machines were introduced to generate $cu$-blenders.
%\lmargem{completar con referencias}
%\cite[Proposition 2.10]{Asa:} claims that this is so
%when the maximal invariant set in
%$R=\cup_{i\in I} R_i$ is contained in a hyperbolic set.}
%\pmargem{Las cosas son mas dificiles que esto}
\end{enumerate}
\end{rem}

%\begin{figure}[h]
%  \centering
%  \begin{minipage}[b]{0.4\textwidth}
%%   \includegraphics[width=\textwidth]{../fig/BD-blender}
%  \caption{Split blender machine. If we assume that $E^c$ is expanding, the
%   Bonatti-Diaz cu-blender is recovered}
%    \label{fig:Machine}
%  \end{minipage}
%  \hfill
%  \begin{minipage}[b]{0.4\textwidth}
%%    \includegraphics[width=\textwidth]{../fig/A-Blender}
% \caption{Asaoka's blender machine ($E^c$ is not necessarily decomposable and could be assumed expanding or contracting or nonhyperbolic).}
%    \label{fig:figura2}
%  \end{minipage}
% \caption{Blender machines}
%  \label{fig:figuras}
%\end{figure}

Next remark states the robustness of split blending machines, see  \cite[Remark~2.9]{Asa:22}.

\begin{rem}[Robustness of split blending machines]
Let $\cB_f = (\blR, Z, \fS, \fC, \theta)$ be a split blending machine, where $\blR = \{R_i\}_{i\in I}$.
For every $i\in I$, consider a
 neighbourhood $U_i$ of $R_i$ and a compact subset $K_i \subset \mathrm{int}(R_i)$. 
 Then  there is a $C^1$ neighbourhood $\mathcal{U}$ of $f$ such that every $g \in \mathcal{U}$ has
 a split blending machine $\cB_g = (\blR', Z, \fS, \fC, \theta)$,  where  $\blR' = \{R'_i\}_{i\in I}$ is a family of sets with $K_i \subset  \mathrm{int} (R'_i) \subset U_i$ for every $i\in I$.
\end{rem}

In what follows, if no confusion arises, we will refer to split blending machines simply as blending machines.

\subsection{Superposition properties of split blending machines}
\label{preblendingmachinesintheworld}
Blending machines were introduced to generate $\cu$-blenders (see \cite[Proposition 2.10]{Asa:22}). In Proposition~\ref{p.alsogenerates}, we show a similar for split blending machines. To this end, we first recall the notion of a {\em{$\mathrm{cu}$-blender}} and introduce some  terminology.

\begin{defi}[cu-blender, Definition~3.1 in \cite{BocBonDia:16}]
\label{d.cublender}
Let  $\Gamma$ be a hyperbolic compact set of a diffeomorphism $f$ with
a $Df$-invariant dominated splitting $E^s\oplus E^c\oplus E^{uu}$ into three non-trivial bundles, where $E^s$ and $E^u = E^c \oplus E^{uu}$ are the stable and unstable bundles of $\Gamma$.
Let $\mathscr{D}$ be an open set of
  disks\footnote{Here the elements of the set
 $\mathscr{D}$  are images of embedded disks
 and a suitable  metric is considered in the space of these disks, see \cite[Section 3.1]{BocBonDia:16} for details.}
 of dimension $d_{uu}=\dim (E^{uu})$.
The pair  $(\Gamma, \mathscr{D})$ is a cu-\emph{blender}  if
there is a $C^1$ neighbourhood $\mathcal{U}$ of $f$
such that
\begin{equation}
\label{e.intersectionestable}
D\cap W^s_{\mathrm{loc}}(\Gamma_g) \not=\emptyset \quad  \text{for every $g\in \mathcal{U}$ and every $D\in \mathscr{D}$,}
\end{equation}
where $\Gamma_g$ is the continuation of $\Gamma$ for $g$.
The number $d_c=\dim (E^c)>0$
is the \emph{central dimension} 
and  the set $\mathscr{D}$ is the \emph{superposition region} of the cu-blender.  With a slight abuse of notation (when the role of 
$\mathscr{D}$ is omitted), we also call the set $\Gamma$ a blender.
\end{defi}

Let $X$ be a topological space, $Y$ an open set of $X$, 
 $K$ a compact subset of $Y$,
and $f \colon Y\to X$ a continuous map.
The {\em{maximal forward invariant set}} $\Lambda^s(K, f)$ of $f$ in $K$ is 
\begin{equation}
\label{e.setLambda}
\Lambda^s(K,f)
\eqdef
\bigcap_{n\geq 0} f^{-n}(K).
\end{equation}
For invertible maps $f$, we also define
$$
\Lambda^u(K,f) \eqdef \Lambda^s(K,f^{-1}) \quad \mbox{and} \quad
\Lambda (K,f) \eqdef \Lambda^s(K,f) \cap \Lambda^u (K,f).
$$
% \quad \text{and} \quad \Lambda^u(f,K)=\bigcap_{n\geq 0} f^n(K).$$
Given a split blending machine $\fB_{f}=(\mathcal{R}, Z, \fS, \fC, \theta)$ with $\mathcal{R}=\{R_{i}\}_{i\in I}$, its
{\em{maximal forward invariant set}} is the set given by
\begin{equation}
\label{e.stablesetblender}
\Lambda^s(\fB_{f})\eqdef \Lambda^s\big(\bigcup_{i \in I} R_{i} ,f\big).
\end{equation}

The next result is a variation of \cite[Proposition~2.10]{Asa:22} for split blending machines, where the intersection property in~\eqref{e.intersectionestable} is reformulated: the family of disks is replaced by a family of graphs, and the local stable set of the blender by the forward invariant set of the split blending machine.  
This proposition also gives a sufficient condition for a split blending machine to generate a blender.

\begin{prop}
\label{p.alsogenerates} Let $f \colon U\to N$ be a $C^1$ embedding defined on an open subset $U$ of $N$ and
%\lmargem{talvez tengamos que decir antes que $N$ y $N_i$ denotaran open manifolds... y evitar masacrar con eso}
$\cB_f= (\mathcal{R}, Z, \fS, \fC, \theta)$ a split blending machine
for $f$ with dimensions $(d_s, d_c, d_u)$.
Let $\fS=(P_{cs},Q_u)$ and $\fC=(P_s,P_c)$.
Consider the set  $\mathscr{S}$ of $C^1$ maps $\sigma \colon Q_u(Z)\to \mathrm{int} \big(P_{cs}(Z)\big)$ with $\|D\sigma\| < \theta$. 
Then, for every $\sigma \in \mathscr{S}$, it holds
$$
  \mathrm{graph}(\sigma) \cap \Lambda^s(\cB_f) \ne \emptyset
\qquad
\text{where} \qquad  \mathrm{graph}(\sigma) \eqdef \{x\in Z \colon P_{cs}(x) = \sigma\big(Q_u(x)\big)\}.
$$
Moreover, assume that
\begin{itemize}
\item
 $\mathscr{D}$ is a set of
 $C^1$ embedded compact disks of dimension $d_{u}$ such that for every $D\in \mathscr{D}$ there is
 $\sigma \in \mathscr{S}$ with
$\mathrm{graph}(\sigma) \subset \mathrm{int} (D)$, and
\item
 $\Gamma$ is a hyperbolic set of $f$ such that
\begin{enumerate}[label=$\bullet$, leftmargin=0.5cm]
\item[\emph{(a)}]
there is a neighbourhood $K$ of\, 
$\bigcup{R_i}$, $R_i\in \mathcal{R}$,
%\lmargem{para Cris: dejamos $K$ pues queremos tener la ecuacion \eqref{e.setLambda}}
with $\Lambda (K, f) \subset \Gamma$, and
\item[\emph{(b)}]
 $\Gamma$ has a  partially  hyperbolic splitting
  $E^s\oplus E^c \oplus E^{uu}$, where $E^s$ and $E^u=E^c\oplus E^{uu}$ are the stable and unstable bundles of $\Gamma$ and
  $\dim (E^s)=d_s$,
   $\dim (E^c)=d_c$, and $\dim (E^{uu})=d_{u}$.
\end{enumerate}
\end{itemize}
Then $(\Gamma, \mathscr{D})$ is a $cu$-blender with central dimension $d_c$.
\end{prop}

The proof of this proposition follows the ideas in \cite[Proposition~2.10]{Asa:22}
adapting the argument of graph transform in ~\cite{BarKiRai:14}. We proceed to the details.

Take any $\sigma\in \mathscr{S}$ and note that $\mathrm{graph} (\sigma) \subset Z$. Then
\begin{equation}
\label{e.diameterofthegraph}
\begin{split}
%\diam (P_c(K_0) )
\diam \big( (P_{c} \circ P_{cs}) (\mathrm{graph} (\sigma) ) \big) 
&\leq \|DP_c\|\,\, \diam \big(P_{cs}(\mathrm{graph} (\sigma)) \big) =\\
&= \|DP_c\|\,\, \diam \big(\sigma( Q_{u}(\mathrm{graph}(\sigma)))\big) \\
&\leq \|DP_c\|\,\|D\sigma\|\,\, \diam \big(Q_{u}(\mathrm{graph}(\sigma))\big) \\
&< \|DP_c\|\, \theta \, \,\diam \big(Q_u(Z) \big)
\overset{\text{(BM4)}}{<} \Delta ,
\end{split}
\end{equation}
where the first equality follows from the definition
of $\mathrm{graph} (\sigma)$ and  $\Delta$ is the Lebesgue number of the split blending machine $\cB_f$, see Definition~\ref{def:cu-blender}.

We need the following lemma
used instead  of  the Lebesgue number of the covering  $\{\mathrm{int} \big(P_{cs}(R_i)\big)\}_{i\in I}$ of $\overline{P_{cs}(Z)}$ in ~\cite{Asa:22}.

\begin{lem} \label{lclaim:b3}
%\textcolor{red}{\sout{Under assumption (B3) in Definition~\ref{def:cu-blender}, it holds the following.}}
%Let $\Delta$ be a Lebesgue number  of $\fB_f$.
 Consider any  set $K\subset \overline{P_{cs}(Z)}$
with $\diam \big(P_c(K)\big)<\Delta$. Then there is $i\in I$ such that $K \subset \mathrm{int} \big(P_{cs}(R_i)\big)$.
%\textcolor{red}{In particular,
%$$
%\overline{P_{cs}(Z)} \subset \bigcup_{i\in I} \mathrm{int}  (P_{cs}(R_i)).
%$$}
%\lsmargem{la verdad esto esta al pedo... lo sacariamos}
\end{lem}
\begin{proof}
Note that $P_c(K) \subset P_c \big(\overline{( P_{cs})(Z)}  \big)\subset \overline{(P_c \circ P_{cs})(Z)}$ and (recall (BM4)) $\diam \big(P_c(K)\big) < \Delta$. As $\Delta$ is a Lebesgue number of the covering $\{\mathrm{int} (P_c \circ P_{cs})(R_i)\}_{i\in I}$ of
$\overline{(P_c \circ P_{cs})(Z)}$, there exists $i\in I$ such
that 
\begin{equation}
\label{e.sala}
P_c(K) \subset \mathrm{int} \big(P_c \circ P_{cs} (R_i)\big).
\end{equation}
Moreover, %from (BM3) holds that
\begin{equation}
\label{e.chuva}
{P_{s}(K)} \subset \overline{(P_{s} \circ P_{cs})(Z)}
 \overset{\text{(BM3)}}{\subset} 
\mathrm{int} (P_{s} \circ P_{cs}(R_i)).
\end{equation}
As $P_{cs}(R_i)$  is a $\fC$-rectangle (BM3) and $\fC=(P_s,P_c)$, these inclusions imply that
\begin{align*}
    K  & \subset \fC^{-1}\big( P_s(K)\times P_c(K) \big) 
\overset{\eqref{e.chuva},
\eqref{e.sala}}{\subset}    
     \fC^{-1}\big(\mathrm{int}\, (P_{s} \circ P_{cs})(R_i)\times \mathrm{int}\, (P_{c} \circ P_{cs})(R_i) \big) \\
    &= \mathrm{int} \big( \fC^{-1} \big( (P_s \circ P_{cs} )(R_i)\times
    (P_c \circ P_{cs})(R_i) \big) \big) \\
    &= \mathrm{int} \big( \fC^{-1}(\fC(P_{cs}(R_i))\big) = \mathrm{int} (P_{cs}(R_i)),
\end{align*}
proving the lemma.
\end{proof}

Equation \eqref{e.diameterofthegraph} allows us to apply
Lemma~\ref{lclaim:b3} to the set
$P_{cs}\big(\mathrm{graph} (\sigma)\big) \subset P_{cs}(Z)$ to get $i_0\in I$ such that $P_{cs}\big(\mathrm{graph} (\sigma)\big) \subset \mathrm{int} \big( P_{cs}(R_{i_0})\big)$.
For each $t \in Q_u(Z)$ consider  the set
$$
D(t) \eqdef \{x \in R_{i_0}\colon Q_u \circ f(x) = t \}.
$$

\begin{lem}
\label{l.isapoint}
The intersection
$\mathrm{graph}(\sigma)\cap D(t)$ is a singleton for every $t \in Q_u(Z)$.
\end{lem}
\begin{proof}
Fix $t \in Q_u(Z)$.
As $R_{i_0}$ is a $(P_{cs},Q_u\circ f)$-rectangle (BM2) we can consider the $C^1$ map
$g \colon P_{cs}(R_{i_0}) \to Q_u(R_{i_0})$
defined by $g=Q_u \circ (P_{cs}|_{D(t)})^{-1}$.

\begin{claim}
\label{c.Dgmenorque}
$\|Dg\|<\theta$.
\end{claim}

We are postponing the proof of the claim for now to conclude the proof of the lemma.

Recall that $\fS=(P_{cs}, Q_u)$ and observe that $\fS(D(t))$ is the graph of
$g$,
\begin{equation}\label{eq:graf}
   \fS(D(t))=\{(s,g(s)): s \in P_{cs}(R_{i_0}) \}.
\end{equation}

From $P_{cs}\big(\mathrm{graph} (\sigma)\big) \subset \mathrm{int} \big(P_{cs}(R_{i_0})\big)$ and the definition of $g$, we define the map
$$
G \colon P_{cs}(R_{i_0}) \times Q_u(R_{i_0}) \to
P_{cs}(R_{i_0}) \times Q_u(R_{i_0}), \quad G(s, t) \eqdef (\sigma(t),g(s)).
$$

Since $\|Dg\|$ and $\|D\sigma\|$ are less than $\theta<1$, the map $G$ is a uniform contraction. Hence, by the compactness of its domain, $G$ has a unique fixed point.% $(s_*, t_*)$.

Note that $x\in \mathrm{graph} (\sigma) \cap D(t)$ if and only if $\big(P_{cs}(x),Q_u(x)\big)$ can be simultaneously written as $(\sigma(t),t)$ and $(s,g(s))$, for some $s$ and $t$. This implies that $\big(P_{cs}(x),Q_u(x)\big)$ is a fixed point of $G$. As $G$ has a unique fixed point, this intersection is a singleton, proving the lemma. It remains to prove the claim.

\begin{proof}[Proof of Claim~\ref{c.Dgmenorque}]
Consider  $x\in D(t)$  and any vector $v\in T_xD(t)$. Since $Q_u\circ f (x)=t$, it follows that
$D({Q}_u\circ f)(v)=0$ and thus $Df(v)\in C_\theta\big(f(x),\fS\big)$. As $f$ satisfies the $(\fS, \theta)$-cone invariance condition on $R_{i_0}$, we get $v \in Df^{-1}\big(C_\theta(f(x),\fS)\big)\subset C_\theta(x,\fS)$.
Therefore, by the definition of  $C_\theta(x,\fS)$ in \eqref{e.conocono},  it holds that
$$
   \|D{Q}_u(v)\| \leq  \theta \|D{P}_{cs}(v)\| \quad \text{for every $x\in D(t)$ and $v\in T_x \, D(t)$.}
$$
This inequality and $g=Q_u \circ (P_{cs}|_{D(t)})^{-1}$  imply that $\|Dg\|<\theta$, proving the claim.
\end{proof}
The proof of the lemma is now complete.
\end{proof}

We are now ready to complete the proof of  Proposition~\ref{p.alsogenerates}.

\begin{proof}[Proof of  Proposition~\ref{p.alsogenerates}]
Fix any $\sigma_0\in \mathscr{S}$.
Lemma~\ref{l.isapoint}
 allows us to define a $C^1$ map 
 $$
 \sigma_1 \colon Q_u(Z)\to P_{cs}(Z)
 \quad \mbox{such that} \quad 
 f\big(\mathrm{graph}(\sigma_0) \cap R_{i_0}\big) = \mathrm{graph}(\sigma_1).
 $$ 
 Using that  
$P_{cs}\big(f(R_{i_0} )\big) \Subset P_{cs}(Z)$, (BM2),  and that $f$ satisfies the
$(\fS,\theta)$-cone invariance
condition on $R_{i_0}$, (BM1), one obtain that $\sigma_1\in \mathscr{S}$.
Iterating this process, one gets  sequences $\{i_n\}_{n}$ in $I$ and $\{\sigma_n\}_{n}$
in $\mathscr{S}$ such that
$$
f\big(\mathrm{graph}(\sigma_n) \cap R_{i_n}\big) = \mathrm{graph}(\sigma_{n+1}) \quad   \text{for every $n \geq  0$}.
$$
By construction,
$$
\emptyset \not= \bigcap_{n\geq 0} f^{-n}\big(\mathrm{graph}(\sigma_{n})\big) \subset \bigcap_{n\geq 0} f^{-n}(R_{i_n}) \cap \mathrm{graph}(\sigma_0)  \subset
 \Lambda^s(\fB_f)\cap \mathrm{graph}(\sigma_0),
$$
which proves the first part of the proposition.

To prove the second part of the proposition, it suffices to observe that, by definition of $\Gamma$, it holds $\Lambda^s(\fB_f) \subset W^s_{\mathrm{loc}}(\Gamma)$, ending the proof.
\end{proof}

\section{Split blending machines for skew-products}
\label{s.generationofsplitblendingmachines}

In the spirit of Sections~\ref{s.quotientdynamics}-\ref{ss.occurrence},
we consider skew-products whose base and fiber  dynamics are given by an affine horseshoe map and
a family of  local diffeomorphisms, see Section~\ref{ss.affine}. These systems arise from diffeomorphisms with simple contours and their
adapted perturbations. In Section~\ref{ss:preblendermachines}, we define \emph{preblending machines}, which serve as germs of split blending machines. Section~\ref{ss.rumboalasblendingmachines} contains Theorem~\ref{thm:blending-yield-blenders}, showing how preblending machines yield split blending machines for skew-products.

\subsection{Skew-products over affine horseshoes}
\label{ss.affine}

%We consider  skew products whose base dynamics is an affine horseshoe. 

\begin{defi}[Locally affine horseshoe maps] \label{ld.affine}  Fix integers
$k, d_s, d_u\geq 1$ and $\nu\in (0,1)$. Let
\begin{equation}
\label{e:setD}
\mathsf{D}\eqdef H\times V,
\quad
\text{where  $H\eqdef [0,1]^{d_s}$ and  $V\eqdef [0,1]^{d_u}$.}
\end{equation}
 A  \emph{locally affine $(k,d_s,d_u,\nu)$-horseshoe map}  (for short a
 {\em{$(k,d_s,d_u,\nu)$- or a $(k,\nu)$-horseshoe map}} if the roles of  $d_s$ and
 $d_u$ are not relevant)  %on $\mathsf{D}=[0,1]^{\ell+m}$
is a diffeomorphism \mbox{${A}\colon \mathbb{R}^{d_s+d_u} \to
\mathbb{R}^{d_s+d_u}$}  such that
there are families:
\begin{enumerate}[label=$\bullet$, leftmargin=0.5cm]
\item
$(V_i)_{i=1}^k$ of
pairwise disjoint compact disks of dimension $d_u$  contained in $\mathrm{int}(V)$,
\item
$(H_i)_{i=1}^k$ of pairwise disjoint of compact disks of dimension $d_s$ contained in  $\mathrm{int}(H)$,
\item
$(A_i^{s})_{i=1}^k$  and  $(A_i^{u})_{i=1}^k$ of affine maps
$A_i^{s} \colon \mathbb{R}^{d_s} \to  \mathbb{R}^{d_s}$
and
$
A_i^{u} \colon \mathbb{R}^{d_s} \to  \mathbb{R}^{d_s},
$
with 
\begin{equation} \label{eq:nu}
\max \left\{  \|DA^{s}_i\|, \|(DA^{u}_i)^{-1}\|, \, i\in\{1, \dots, k\}\right\}< \nu
\end{equation}
\end{enumerate}
such that for every  $i=1,\dots, k$ it holds:
\begin{enumerate}[label=(\alph*), leftmargin=0.8cm, itemsep=0.2cm]
\item
$A(H \times V_i)= H_i \times V$,
\item
the restriction $A|_{\mathsf{H}_i}$ of $A$  to $\mathsf{H}_i\eqdef H\times V_i$
coincides with 
$
 {A}_i \eqdef A^{s}_i\times A^{u}_i.
$
\end{enumerate}
The number $\nu$ is called a {\em{stretching}} of $A$.
\end{defi}

%
%For the following, let us introduce some notation and terminology. 
% 
%\begin{notation} (Composition of maps)
%\label{n.revcompositionandnotation}
%Fix   $k\in \mathbb{N}$ %, %$k \geqslant 2$,
%and let
%$\Sigma^\ast_k\eqdef \bigcup_{n\geq 0} \{1,\dots,k\}^n$.
%The elements $\ti=i_0\dots i_n$ in  $\Sigma^\ast_k$ are called
%\emph{{words.}} We denote by $|\ti|$ the {\emph{length}} of $\ti$ (the number $n+1$ of letters).
%A word $\tj$ of the form $\tj=i_0\dots i_k$ with $k \leqslant n$ is called
%a {\em{prefix of $\ti$.}}
%Throughout this section, $M$ denotes a Riemannian manifold of dimension $d=\dim (M)\geq 1$.
%Given any family of endomorphisms  $\{g_1, \dots g_k\}$ of $M$ we use the following notation for their compositions:
%\begin{equation}
%\label{e.notation}
%g_{\ti}= g_{i_{0}  \cdots i_n} \eqdef
%g_{i_{n}} \circ \cdots \circ g_{i_0}, \quad \ti = i_0\dots i_n\in \Sigma^\ast_k.
%\end{equation}
%
%Given a word $\ti \in \Sigma^\ast$,
%the $(\ti)$-orbit of a point $x\in M$, denoted by $\mathcal{O}_{\ti} (x)$, is defined by
%\begin{equation}\label{e.orbitsegment}
%\mathcal{O}_{\ti}(x) \eqdef \{ g_\tj(x) \colon \mbox{$\tj$ is a prefix of $\ti$}\}.
%\end{equation}
%We define similarly orbits of sets.
%\end{notation}
%

\begin{rem}[Families of legs]\label{r.horseshoeafin}
We let $\Sigma_k=\{1,\dots,k\}^\mathbb{Z}$ and $\Sigma_k^\ast$ the family of finite words. 
Given $\ti \in \Sigma^\ast_k$,
using the notation for composition of maps in \eqref{e.notation},
define the sets
\begin{equation}
\label{e.HiVi}
H_\ti \eqdef A_\ti^{s} (H)
\quad
\mbox{and}
\quad
V_\ti \eqdef  (A_{\ti}^{u})^{-1}(V)
\end{equation}
and 
\begin{equation} 
\label{e.otrosHi}
\mathsf{H}_\ti \eqdef H \times V_{\ti}
\quad
\mbox{and}
\quad
\mathsf V_\ti \eqdef H_\ti \times V
=A_\ti (\mathsf{H}_\ti).
\end{equation}
Note that for $\ti=i_0\dots i_n$ it holds
$$
%\mathsf{H}_\ti \subset \mathsf{H}_{i_0} \quad  \text{and} \quad
A^j (\mathsf{H}_{\ti} )\subset \mathsf{H}_{i_j} \quad \text{for every $j=0,1,\dots,n-1$.}
$$
Furthermore,  $A^j$ is a $(k^j,d_s,d_u,\nu^j)$-horseshoe map whose restriction to $\mathsf{H}_{\ti}$,  $A^j|_{\mathsf{H}_{\ti}}$,  with $|\ti|=j$, coincides with
$A_\ti=A^s_\ti \times A^u_\ti$. We refer to the family  $(\mathsf{H}_{\ti})_{\ti \in \Sigma^*_k}$ of $\mathsf{D}$ as the \emph{legs} of the  horseshoe map $A$.
\end{rem}

\begin{remark}[Inverse horseshoe maps]
\label{rem:inversehmap}
Given a   $(k,d_s,d_u,\nu)$-horseshoe map $A$, its inverse
$B=A^{-1}$ is the
 $(k,d_u,d_s,\nu)$-horseshoe map obtained considering the
 maps
 $(B_i^{s})_{i=1}^k$  and  $(B_i^{u})_{i=1}^k$
 given by $B_i^s= (A_i^{u})^{-1}$ and
$B_i^u= (A_i^{s})^{-1}$.
\end{remark}

%Given a family $\mathcal{F}=\{\phi_1,\dots, \phi_k\}$ of local diffeomorphisms of $M$
%define $\langle \mathcal{F} \rangle^+$ as in \eqref{e.antestardequenunca}.

Consider a manifold $M$, pairwise disjoint open sets $(U_i)_{i=1}^k$ of $M$, and 
a family of local diffeomorphisms  
$\mathcal{F}=\{\phi_i \colon U_i \to M, \, i=1, \dots, k\}$.
Consider the compositions $\phi_\ti$, $\ti \in \Sigma_k^\ast$, 
and their domains $U_\ti$ as in \eqref{e.neighUi}. We also consider the corresponding transition matrix
$T$ as in \eqref{e.matrixT} and the shift space $\Sigma_T$ and the set $\Sigma_T^\ast$ of admissible words.

\begin{defi}[Skew-product over an affine horseshoe]
\label{d.skew}
Let $A$ be a horseshoe map as in Definition \ref{ld.affine} 
and 
$\mathcal{F}=\{\phi_i\colon U_i \to M, \, i=1, \dots, k\}$ a familiy of local $C^1$ diffeomorphisms of $M$.
We define the skew-product with base map $A$ and fiber maps $\mathcal{F}$ by
\begin{equation}
\label{e.AtimesF}
\Phi \eqdef A\ltimes \mathcal{F} \colon
\bigcup_{i=1}^k (H_i \times U_i)
\to \mathbb{R}^{d_{s}+d_{u}}\times M,
\qquad
\Phi_{|_{\mathsf{H}_i\times U_i}} \eqdef {A}_i\times \phi_i.
\end{equation}
\end{defi}

\begin{rem} \label{rem:Phij}
Note that $\Phi^j|_{\mathsf{H}_{\ti}\times U_\ti}= A_\ti \times \phi_\ti$ for every $\ti \in \Sigma_k^*$ with $|\ti|=j$.
\end{rem}

\subsubsection{Adapted stretching and partial hyperbolicity}
\label{ss.adaptedstretching}
We now introduce mild assumptions on a skew-product $\Phi = A\ltimes \mathcal{F}$ and on the pair  $(\tfS_1, \tfS_2)$ to guarantee that the locally affine horseshoe map $A^{|\ti|}$ is $(\tfS_1,\tfS_2)$-adapted to $\phi_\ti$ for every $\ti\in \Sigma_k^*$.

\begin{defi}[Adapted stretching]
\label{d.stretching}
Let $\phi$ be a $(\tfS_1,\tfS_2,\mu,\gamma)$-dominated  $C^1$ local diffeomorphism on a subset $R$ of $N$. %and let $A$ be a $(k,\nu)$-horseshoe map.
A $(k,\nu)$-horseshoe map $A$ (or its stretching $\nu$)  is \emph{$(\tfS_1,\tfS_2)$-adapted to $\phi$ on $R$} if
$\nu < \mu^{-1} < \gamma < \nu^{-1}$.
If $\tfS = \tfS_1 = \tfS_2$, we say it is $\tfS$-adapted.
\end{defi}

\begin{lem} \label{c.pp} Consider a family
$\mathcal{F}=\{\phi_1,\dots, \phi_k\}$
 of $(\lambda,\rho)$-Lipschitz  local
diffeomorphisms of $M$ (see \eqref{e.nu}). Let $\tfS_\tau=(P_\tau,Q_\tau)$,
$\tau=1,2$, be a splitting of an open set $N_\tau\subset M$  such that
there is a constants $C_\tau\geq 1$ with
\begin{equation} \label{e.C}
    \frac{1}{C_\tau}\, \|v\| \leq \|D\tfS_\tau(v)\| \leq  C_\tau \, \|v\| \quad \text{for every $x\in N_{\tau}$ and  $v\in T_x N_{\tau}$.}
\end{equation}
%and $\ti\in \Sigma^\ast_k$.
Suppose that $\phi_\ti$, $\ti \in \Sigma^\ast_k$, is $(\tfS_1,\tfS_2,\mu,\gamma)$-dominated
 on a subset $R$ of $N_1$. Then
$$
  \frac{\lambda^{|\ti|}}{C} \leq \mu^{-1}  \quad\text{and} \quad \gamma \leq C \, \rho^{|\ti|} \quad \text{where} \ \ C\eqdef C_1C_2.
$$
Furthermore, let $A$ be a $(k,\nu)$-horseshoe map such that
\begin{equation}
 \label{e.stv}
\nu^{|\ti|} < \frac{\lambda^{|\ti|}}{C} \quad \text{and} \quad C \,\rho^{|\ti|} < \nu^{-|\ti|}
\end{equation}
then  $A^{|\ti|}$ is $(\tfS_1,\tfS_2)$-adapted to $\phi_{\ti}$ on $R$.
\end{lem}

\begin{proof}
The $(\lambda,\rho)$-Lipschitz of the family $\mathcal{F}$ means  
\begin{equation} \label{e.nu}
   \lambda^{|\ti|} < m(D_x\phi_\ti) \leq  \|D_x\phi_\ti\| < \rho^{|\ti|} \quad \text{for every $x\in M$ and
    $\ti\in \Sigma^\ast_k$.}
\end{equation}

Fix  $x\in  R$ and (recall \eqref{e.conocono}) consider any vector in
$v\in C_1(x,\tfS^t_1)\setminus
\mathrm{Ker}(D_x Q_{\color{violet} 1})$ (note that this vector always exist).
 Then $\|DP_1(v)\|\leq \|DQ_1(v)\|$ and thus 
 \begin{equation}
\label{e.copo}
 \|D\tfS_1(v)\|=\|DQ_{1}(v)\|.
  \end{equation}
 %\PB{La pruve usa la implicacion directa del Lemma~\ref{l.new-dilatation} que no necesita de la condicion $\mu^{-1}_{\ti}<\lambda_{\ti}$. Sin embargo, nunca usaremos este lema sin esa condición. Posiblemente no vale la pena señalar estas cosas y lo mas facil sea asumir desde partida (definicion de dilatación) esta condicion (suplerflua).}
Since $\phi_\ti$ is $(\tfS_1\tfS_2,\mu,\gamma)$-dominated on $R$, we get 
% Lemma~\ref{l.new-dilatation} together with equations \eqref{e.C} and \eqref{e.nu} imply
\begin{align*}
{\gamma\, \|DQ_1(v)\|} &\overset{\eqref{e.lambda-gamma-dilatation}}{\leq} {\|D(Q_2\circ \phi_\ti)(v)\|} \leq
\|D(\tfS_2\circ \phi_\ti)(v)\| \overset{\eqref{e.C}}{\leq} C_2\, \|D\phi_\ti (v)\|  \\ 
&\overset{\eqref{e.nu}}{\leq} C_2\, \rho^{|\ti|}\, \| v\|\overset{\eqref{e.C}}{\leq} C_2C_1\, \rho^{|\ti|}\, \|D\tfS_1(v)\| \overset{\eqref{e.copo}}{=} C\, \rho^{|\ti|} \|DQ_1(v)\|.
\end{align*}
As $\|DQ_1(v)\|\ne 0$ we get $\gamma \leq C \rho^{|\ti|}$ and hence the second inequality in the lemma.

Similarly, take any $w \in C_1(\phi_\ti(x),\tfS_2)\setminus
\mathrm{Ker}(D_{\phi_\ti(x)} P_2)$.
 Hence, $\|DQ_2(w)\|\leq \|DP_2(w)\|$ and $\|D\tfS_2(w)\|=\|DP_2(w)\|$.   As $\phi_\ti$ satisfies $(\tfS_1,\tfS_2,\mu,\gamma)$-dominated condition on $R$,  by definition, it  satisfies the
$(\tfS_1,\tfS_2,1)$-cone invariance condition on $R$.
Thus, $v=D\phi_\ti^{-1}(w) \in C_{\theta'}(x,\tfS_1)$ for some $0<\theta'<1$. This implies that $\|DQ_1(v)\|\leq \|DP_1(v)\|$ and hence $\|D\tfS_1(v)\|=\|DP_1(v)\|$.  As above, Lemma~\ref{l.new-dilatation} and equations \eqref{e.C} and \eqref{e.nu} imply that
\begin{align*}
\mu^{-1} \, \|DP_1(v)\| &\geq  \|D(P_2\circ \phi_\ti)(v)\| = \|DP_2(w)\|= \|D\tfS_2(w)\| \overset{\eqref{e.C}}{\geq} \frac{1}{C_2}\,\|w\| \\  
&\geq  \frac{1}{C_2}\, \|D\phi_\ti(v)\|  
\overset{\eqref{e.nu}}{\geq}
 \frac{\lambda^{|\ti|}}{C_2}\, \|v\| 
\overset{\eqref{e.C}}{\geq}
 \frac{\lambda^{|\ti|}}{C_1C_2} \,\|D\tfS_1(v)\| = \frac{\lambda^{|\ti|}}{C}\,\|DP_1(v)\|.
\end{align*}
As $\|DP_1(v)\|\ne 0$, this implies that $C\mu^{-1} \geq \lambda^{|\ti|}$, which concludes the first part of the lemma. The second part of the lemma follows immediately, concluding the proof.
\end{proof}

\begin{defi}
Given a
$(k,\nu)$-horseshoe map $A$ and
%with refined family of refined legs $\{\mathsf{H}_{\ti}\}_{\ti\in \Sigma^\ast_k}$.
a finite family $\mathcal{F}$ of $(\lambda,\rho)$-Lipschitz local
diffeomorphisms,
the  map $\Phi=A\ltimes \mathcal{F}$  in~\eqref{e.AtimesF} is
{\em{$(\nu,\lambda,\rho)$-partially hyperbolic}} if
\begin{equation} \label{eq:PH}
\nu < \lambda <  \rho <\nu^{-1}.
\end{equation}
\end{defi}

\begin{rem} \label{rem.pp} 
Consider a family $\mathcal{F}$, a pair $(\tfS_1,\tfS_2)$, and a constant $C$
as in  Lemma~\ref{c.pp}. Let $A$ be a $(k, \nu)$-horseshoe
such that the skew-product $\Phi=A\ltimes \mathcal{F}$ is $(\nu, \lambda, \rho)$-partially hyperbolic and
consider $\ti \in  \Sigma^\ast_k$ such that
$$
|\ti| > \max\left \{ \frac{2\,\log C}{\log (\lambda\, \nu^{-1})}, \frac{2\,\log C}{-\log (\rho\,\nu)}   \right\}.
$$
Since this inequality implies~\eqref{e.stv} in Lemma~\ref{c.pp}, the horseshoe map $A^{|\ti|}$ is $(\tfS_1,\tfS_2)$-adapted to $\phi_{\ti}$ on $R$. Furthermore, if we have $C=1$ (as in the case of canonical projections), then it suffices to take $|\ti|\geq 1$. In particular, if $\Phi=A\ltimes \mathcal{F}$ is partially hyperbolic, then $A$ is adapted to
any preblending  machine $\tfB_{\mathcal{F}}=(\scR,B, \tfS)$ for $\mathcal{F}$ (see Definition \ref{def:cu-blending}) with splitting $\tfS=(P,Q)$ given by canonical projections $P$ and $Q$.
\end{rem}

\subsection{Preblending machines}
\label{ss:preblendermachines}
We now introduce preblending machines and state their robustness, see Remark~\ref{r.robustpreblender}.

\begin{defi}[Preblending machine] \label{def:cu-blending}
Let
$\mathcal{F}=\{\phi_i \colon U_i \to M, \, i=1, \dots, k\}$ be a family  of local
 diffeomorphisms of a manifold $M$, an open subset $N$ of $M$, and
a positive integer $d_c < d=\dim (M)$. Consider
\begin{enumerate}[label=$\bullet$, leftmargin=0.5cm]
\item
a family
$\scR=\{R_{\ti}\}_{\ti \in I}$
 of (nonempty)  compact subsets of $N$,
 where $R_\ti \subset U_\ti$ and $I$ is  a finite subset
of $\Sigma^\ast_k$ with at least two elements,
\item
a $(d_{c},d-d_{c})$-splitting $\tfS=(P,Q)$  of $N$, and
\item
a  $\tfS$-rectangle $B \subset N$.
\end{enumerate}
The three-tuple $ \tfB_{\mathcal{F}}=(\scR,B, \tfS)$ is a
{\em{preblending machine}} for $\mathcal{F}$  if:
\begin{enumerate}[itemsep=0.2cm,leftmargin=1.2cm, label=(PB\arabic*)]
    \item
    for every $\ti \in I$ there are
    $\mu_{\ti}, \gamma_{\ti}$, with
    $0<\mu_{\ti}^{-1}<\gamma_{\ti}$, such that
    the map $\phi_{\ti}$ is a $C^1$ embedding from an open neighbourhood  of $R_{\ti}$ into
 $N$
    which is $(\tfS,\mu_{\ti},\gamma_{\ti})$-dominated on $R_{\ti}$, 
    \item
for every $\ti \in I$,  the set $R_{\ti}$ is a $(P,Q\circ \phi_{\ti})$-rectangle such that
$$
Q(R_{\ti}) \Subset Q(B), \quad (Q\circ \phi_{\ti})(R_{\ti}) = Q(B), \quad \text{and} \quad
(P\circ \phi_{\ti}) (R_{\ti}) \Subset P(B),
$$
\item
the family $\{\mathrm{int}(P(R_{\ti})) \}_{\ti\in I}$ covers $\overline{P(B)}$.
% as follows
%$$
%    \overline{P(B)} \subset \bigcup_{\ti\in I}
%    \mathrm{int}(
%    P(R_{\ti})).
%$$
%and its Lebesgue number $\Delta>0$ is greater
%than $\theta\cdot \mathrm{diam} QB$.
%\item[(P4)]
%\textcolor{blue}{$d_{c}>0$ and $d_{u}\geq 0$. When $d_{u}=0$, $Q$ is understood as the trivial projection.}
%\lmargem{esto ahora sobra}
\end{enumerate}
%\lmargem{poner esto en la def de blending machine}
%We call
% the sets $B$, $I$, $d_c$, and the numbers $(\mu_\ti)_{\ti \in I}$ and $(\gamma_\ti)_{\ti \in I}$
%the {\em{preblending region,}}
%the {\em{symbols,}}
% {\em{the central dimension,}}
%and
%  the {\em{dominated constants}} of  $\tfB_{\mathcal{F}}$.
%Any Lebesgue number of the covering  $\{\mathrm{int}(P(R_{\ti}))\}_{\ti \in I}$ is 
%called a {\em{$P$-Lebesgue number of  $\tfB_{\mathcal{F}}$.}} \lmargem{check notation}
%%%%%
We call $B$ the {\em{preblending region}},
$I$ the {\em{symbols}},
$d_c$  {\em{the central dimension}}, and 
 $(\mu_\ti)_{\ti \in I}$, $(\gamma_\ti)_{\ti \in I}$
  the {\em{dominated constants}} of  $\tfB_{\mathcal{F}}$.
A {\em{Lebesgue number}} of  $\tfB_{\mathcal{F}}$ is any
Lebesgue number of the covering  $\{\mathrm{int}(P(R_{\ti}))\}_{\ti \in I}$. 
The preblending machine is {\em{expanding}} if  $\mu_\ti>1$ and $\gamma_\ti>1$ for every
$\ti\in I$.  
\end{defi}

\begin{rem}[Robustness of preblending machines]
\label{r.robustpreblender}
Let $ \tfB_{\mathcal{F}}=(\scR,B, \tfS)$ be a preblending machine for the family 
$\mathcal{F}=\{ \phi_i\colon U_i \to M, i=1,\dots, k\}$, where $\scR = \{R_{\ti}\}_{\ti\in I}$.
For each $\ti \in I$, consider a
 small neighbourhood $V_{\ti}$ of $R_{\ti}$, $\overline V_\ti \subset U_\ti$, and a compact subset $K_{\ti} \subset \mathrm{int}(R_{\ti})$. 
 Then, for every  $i=1,2,\dots, k$, there exist a $C^1$ neighbourhood $\mathcal{U}_{i}$ of $\phi_i$ such that every family $\mathcal{G}=\{ \varphi_i\colon U_i \colon M,  i=1,\dots , k\}$ with $\varphi_i \in \mathcal{U}_i$, has
 a preblending machine 
 $ \tfB_{\mathcal{G}}=(\scR',B, \tfS)$ where  $\scR' = \{ R'_{\ti} \}_{\ti \in I}$ is a family of sets with $K_{\ti}\subset R'_{\ti} \subset V_{\ti}$, $\ti \in I$, where $V_{\ti}$ is contained in $U_{\mathcal{G}, \ti}$ (the set $U_\ti$ corresponding to the
 family $\mathcal{G}$).
\end{rem}

\begin{rem}\label{r.dc=0}
%Considering  $d_c=d$, $P=\mathrm{id}$, and $Q$ as the trivial projection, the Definition~\ref{def:cu-blending} can be restated as follows:
%\begin{enumerate}[label=(PB'\arabic*), leftmargin=1.2cm, itemsep=0.2cm]
%\item For every $\ti\in I$, the map $\phi_\ti$ is a $C^1$ embedding on a neighborhood $U$ of $R_{\ti}$ satisfying $\|D\phi_\ti(v)\|\leq \mu^{-1}_{\ti} \|v\|$ for every $v\in T_xU$.
%\item For every $\ti\in I$, it holds that $\phi_{\ti}(R_\ti)\Subset B$.
%\item \label{e.coverPB3} $\overline{B} \subset \bigcup_{\ti\in I} \mathrm{int}(R_\ti)$.
%\end{enumerate}
%When the family  $\mathcal{F}$ is invertible, the above conditions recover the definition of a blending region for ~$\mathcal{F}^{-1}\eqdef \{\phi_i^{-1}\}$ (note that we do not assume contraction for $\mathcal{F}^{-1}$).
%
% Contracting blending regions were introduced in~\cite{NasPuj:12,BarKiRai:14} (see also~\cite{BaFaMaSa:16}). Later,~\cite{BarRai:17} considered blending regions that contract in some directions and expand in others, while still satisfying a covering property analogous to (PB'3).
%
%In contrast, Definition~\ref{def:cu-blending} describes a "non-hyperbolic covering property", as it requires neither contraction nor expansion of $\mathcal{F}$. Thus, the covering condition~(PB3) is of a different nature, since it only demands covering the projection of the set $B$. These distinctive features are illustrated in Figure~\ref{fig:blending-regions}.
Considering  $d_c=d$, $P=\mathrm{id}$, and $Q$ as the trivial projection, Definition~\ref{def:cu-blending} can be restated as follows:
\begin{enumerate}[label=(PB'\arabic*), leftmargin=1.2cm, itemsep=0.2cm]
\item For every $\ti\in I$, the map $\phi_\ti$ is a $C^1$ embedding on a neighborhood $U$ of $R_{\ti}$ satisfying $\|D\phi_\ti(v)\|\leq \mu^{-1}_{\ti} \|v\|$ for every $v\in T_xU$.
\item For every $\ti\in I$, it holds that $\phi_{\ti}(R_\ti)\Subset B$.
\item \label{e.coverPB3} $\overline{B} \subset \bigcup_{\ti\in I} \mathrm{int}(R_\ti)$.
\end{enumerate}
If the family $\mathcal{F}$ is invertible, the conditions above recover the definition of a blending region for~$\mathcal{F}^{-1}\eqdef \{\phi_1^{-1},\dots,\phi_k^{-1}\}$ without assuming contraction for~$\mathcal{F}^{-1}$ (i.e., expansion for $\mathcal{F}$). Contracting blending regions were introduced in~\cite{NasPuj:12,BarKiRai:14} (see also~\cite[Definition 6.1]{BaFaMaSa:17}). Later, in \cite[Definition~5.2]{BarRai:17}, the authors considered blending regions that contract in some directions and expand in others, while still providing a cover as in (PB'3). Note that Definition~\ref{def:cu-blending} requires neither contraction nor expansion. Furthermore, the covering condition~(PB3) is of a different nature, as it only requires covering the projection of the set $B$. These differing features are depicted in Figure~\ref{fig:blending-regions}.
\end{rem}

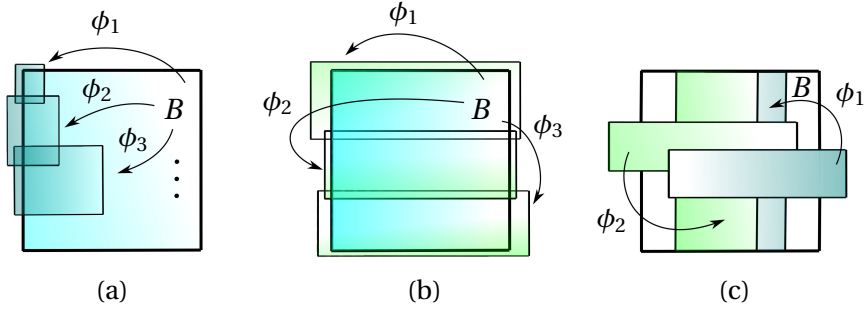
\begin{figure}[h]
\begin{overpic}[scale=.4,
%grid,tics=5
]{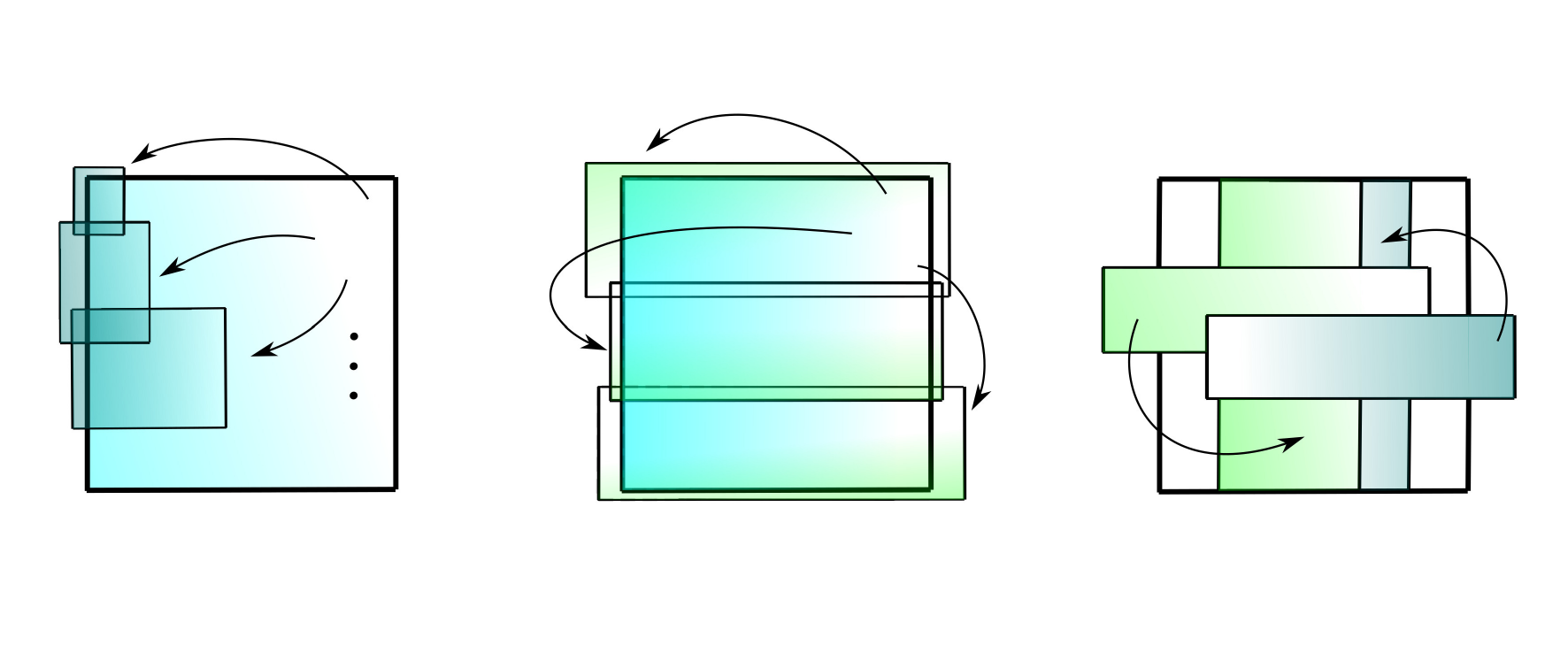}
		\put(13,35){\large $\phi_1$}
		\put(12,27.5){\large $\phi_2$}	
		\put(16,22){\large $\phi_3$}	
		\put(21,25){\Large $B$}
	        \put(13.5,5){\large (a)}
		%%%%%
		\put(46,36){\large $\phi_1$}
		\put(32,26){\large $\phi_2$}	
		\put(62,23.5){\large $\phi_3$}	
	        \put(55,25){\Large $B$}
	        \put(48,5){\large (b)}
		%%%%%
		\put(95.5,27){\large $\phi_1$}
		\put(69,13){\large $\phi_2$}	
		\put(90.5,27.5){\Large $B$}
		  \put(82.5,5){\large (c)}
                 \end{overpic}
		\vspace{-.6cm}
  \caption{(a) Contracting blending region. (b) Non-contractive blending region.  (c) Preblender region}\label{fig:blending-regions}
		\label{fig:blending-regions}
\end{figure}

\subsection{Preblending machines leading to split blending machines}
\label{ss.rumboalasblendingmachines}
%We consider skew products over horseshoes. 
Theorem~\ref{thm:blending-yield-blenders} states how 
 preblending machines associated to the fiber maps of skew-product over horseshoes  lead to 
split blending machines of tower maps. 

We start by defining the latter element.
%
%\begin{defi}[Tower map]
%\label{d.towermap}
%A  {\em{tower}} is a tuple  $\torT=(g,\mathcal{Y}, I, \tln )$, where 
% $g\colon X\to X$ is a continuous map,  
%  $\mathcal{Y}=\{Y_i\}_{i\in I}$ is a finite family
% of pairwise disjoint subsets of $X$, 
% and
%$\tln=\{r_i\}_{i\in I}$ is a set of positive integers numbers (called {\em{return times}}).
%The {\em{tower map}} $G=G_{\torT}$ associated to $\torT$ is defined by %\{Y_i\}_{i\in I},\{n_i\}_{i\in I})$}}
%$$
%G  \colon \bigcup_{i \in I} Y_i \to X,\qquad G|_{Y_i}  \eqdef g ^{r_i}|_{Y_i}.
%$$
%\end{defi}
Given a skew-product map $\Phi={A}\ltimes \mathcal{F}$, 
where $\mathcal{F}= \{\phi_i\colon U_i \to M, i=1, \dots, k\}$,
as in Definition~\ref{d.skew},
consider  the tupla
$\torT=(\Phi,
\mathcal{R}, I)$, where
$I$ is a finite subset of $\Sigma^\ast_k$,
$\mathcal{R}=\{R_{\mathtt{i}}\}_{\mathtt{i}\in I}$ is a family  of
(nonempty)
 compact subsets of $\mathsf{D}\times M$
(recall \eqref{e:setD})
 with
$
{R}_{\mathtt{i}} \subset \mathsf{H}_{\ti}\times U_\ti$ for every $\ti \in I$.
%and the family of returns
%$\tln=\{|\ti|\}_{\ti\in I}$. 
%Since in this case the return times $\tln$ are determined by the set $I$, we will use the
%notation
% $\torT=(\Phi,
%\mathcal{R}, I)$.
%\textcolor{red}{Recall also the family of maps 
%$\{\phi_{\ti}\}_{\ti \in I}\subset \langle \mathcal{F}\rangle^+$.}

\begin{defi}[Skew-product tower map]
\label{d.r.towermap}
The {\em{skew-product tower map $F=F_\torT$ associated to $\torT=(\Phi,
\mathcal{R}, I )$}} is defined by
\begin{equation}
\label{e.towermap}
F  \colon
\bigcup_{\ti \in I} {R}_{\ti} \to \mathsf{D}  \times M,
\qquad
F|_{{R}_{\ti}}  \eqdef \Phi^{|\ti|}|_{{R}_\ti} =
{A}_{\ti} \times \phi_{\ti},
\end{equation}
\end{defi}
By Remark~\ref{rem:Phij}, this map is well defined.

\begin{thm} 
[Split blending machines for tower maps]
\label{thm:blending-yield-blenders}
Let $\mathcal{F} =\{\phi_1,\dots, \phi_k\}$ be
%=\{\phi_1,\dots, \phi_k\}$ be
a family of
$C^1$ local diffeomorphisms of a manifold $M$ 
and $A$ a $(k,d_{s},d_{u},\nu)$-horseshoe map with domain
 $\mathsf{D}=[0,1]^{d_{s}}\times [0,1]^{d_{u}}= \mathsf{H} \times \mathsf{V}$
 with elements as in Definition~\ref{ld.affine}.
 Consider
\begin{enumerate}[label=$\bullet$, leftmargin=0.5cm]
\item
a preblending  machine $\tfB_{\mathcal{F}}=(\scR,B,\tfS)$
 with central dimension
$d_c\leq d= \dim (M)$  for
$\mathcal{F}$, where $\scR=\{R_{\ti}\}_{\ti\in I}$, $I \subset \Sigma^\ast_k$, and $\tfS=(P,Q)$,
\item  the skew-product tower map $F$ associated to
$({A} \ltimes \mathcal{F}, \mathcal{R},I)$,  with
$$
\mathcal{R}=\{\bfR_{\ti}\}_{\ti \in I}, \qquad
\bfR_{\ti}\eqdef \mathsf{H}_{\ti}\times R_{\ti},  \quad  \mbox{where $\mathsf{H}_\ti$ as in \eqref{e.otrosHi}.}
$$
\end{enumerate}
 
 Assume that
the stretching of $A$ is  adapted to $\tfB_{\mathcal{F}}$. Then there is $\theta_0>0$ such that for every $0<\theta\leq \theta_0$ the tuple
$\fB_{\tfB_{\mathcal{F}}, F}=(\mathcal{R},
%\{\mathcal{R}_{\ti}\}_{\ti\in I},
Z, \fS, \fC, \theta)$
 is a split blending machine for  $F$ with  dimension $(d_s,d_{c},d-d_c+d_u)$,
where
\begin{enumerate}
\item [\emph{(1)}]
$Z=\mathsf{D} \times B$, 
%where $\mathsf{D}=H\times V$, $H=[0,1]^{d_{s}}$ and $V= [0,1]^{d_{u}}$,
\item [\emph{(2)}]
$\fS=({P}_{cs},Q_u)$
is the $(d_{s}+d_c, d-d_c+ d_{u})$-splitting of
$\mathbb{R}^{d_{s}} \times \mathbb{R}^{d_{u}}\times M$
given by
\begin{equation}
\label{eq.nopair}
\begin{split}
P_{cs}(x^{s},x^{u},x)
&\eqdef \big(\Pi_{s}(x^{s},x^{u}),P(x)\big)
=\big(x^{s},P(x)\big)
\in \mathbb{R}^{d_{s}} \times \mathbb{R}^{d_c}  \quad \text{and} \quad  \\
Q_{u}(x^{s},x^{u},x)
&\eqdef \big(Q(x),\Pi_{u}(x^{s},x^{u})\big)=
\big(Q(x),x^{u}\big)
\in \mathbb{R}^{d-d_c} \times \mathbb{R}^{d_{u}},
\end{split}
\end{equation}
 where $(x^{s},x^{u},x)\in \mathbb{R}^{d_{s}} \times \mathbb{R}^{d_{u}}\times M$ and
 \item[\emph{(3)}]
 $\fC=(\pi_s,\pi_c)$ is the
  $(d_{s},d_{c})$-splitting  of $\mathbb{R}^{d_{cs}}$, $d_{cs}\eqdef d_{c}+d_s$,  given by the standard projections
 from $\mathbb{R}^{d_{cs}}= \mathbb{R}^{d_s} \times  \mathbb{R}^{d_c}$
 into $\mathbb{R}^{d_{s}}$ and $\mathbb{R}^{d_c}$ respectively.
\end{enumerate}
 \end{thm}

We say that the split blending machine $\fB_{\tfB_{\mathcal{F}}, F}$ is associated 
to $\tfB_{\mathcal{F}}$ and $F$.

\begin{proof}
%[Proof of Theorem~\ref{thm:blending-yield-blenders}]
We begin with a technical lemma.

\begin{lem} \label{l.cone}
Consider  $(p,q)$-splittings $\tfS_\tau=(P_\tau,Q_\tau)$ of open manifolds $N_\tau$, $\tau=1,2$,
an open set $U$ of $N_1$, a subset
$R\subset U$, and  a $C^1$ embedding  $\phi \colon U \to N_2$.
  Assume that:
\begin{enumerate}[label=$\bullet$, leftmargin=0.5cm]
\item $\phi$ is $(\tfS_1,\tfS_2,\mu,\gamma)$-dominated on $R$,
\item there exists a $(k,d_s,d_u,\nu)$-horseshoe map $A$  with legs  $\{\mathsf{H}_{\ti}\}_{\ti\in \Sigma^\ast_k}$,
\item for some $j\geq 1$, the $(k^j,d_s,d_u,\nu^j)$-horseshoe map $A^j$ is $(\tfS_1,\tfS_2)$-adapted to $\phi$ on $R$.
\end{enumerate}
Let  $\Pi_{s}$ and $\Pi_{u}$ be the canonical
projections from $\mathbb{R}^{d_{s}} \times \mathbb{R}^{d_{u}}$ into $\mathbb{R}^{d_{s}}$ and $\mathbb{R}^{d_{u}}$, respectively. Define the maps
\begin{align*}
P_{cs,\tau}(x^{s},x^{u},x)
&\eqdef\big(\Pi_{s}(x^{s},x^{u}),P_\tau(x)\big)
=\big(x^{s},P_\tau(x)\big)
\in \mathbb{R}^{d_{s}} \times \mathbb{R}^{p}  \quad \text{and} \quad  \\
Q_{u,\tau}(x^{s},x^{u},x)
&\eqdef\big(Q_{\tau}(x),\Pi_{u}(x^{s},x^{u})\big)=
\big(Q_{\tau}(x),x^{u}\big)
\in \mathbb{R}^{q} \times \mathbb{R}^{d_{u}},
\end{align*}
 where $(x^{s},x^{u},x)\in \mathbb{R}^{d_{s}} \times \mathbb{R}^{d_{u}}\times N_\tau$ with  $\tau=1,2$.

 Then for every $\ti \in \Sigma^*_k$ with $|\ti|=j$, the map $A_{\ti} \times \phi$ is $(\fS_1,\fS_2,\mu,\gamma)$-dominated  on ${R}_\ti\eqdef \mathsf{H}_\ti \times R$, where
$\fS_\tau\eqdef ({P}_{cs,\tau},Q_{u,\tau})$
%\quad \text{and} \quad  \fS_2\eqdef ({P}_{cs,2},Q_{u,2})
%$$
is a $(d_{s}+p,q+d_{u})$-splitting of  $\mathbb{R}^{d_{s}}\times \mathbb{R}^{d_{u}}\times N_\tau$, $\tau=1,2$.
\end{lem}
%\lmargem{pasar estas ecuaciones para dentro del lema}
%With the notation in Lemma~\ref{l.cone}, given $(x^{ss},x^{uu},x)\in \mathbb{R}^{d_{ss}} \times \mathbb{R}^{d_{uu}}\times N_\tau$ and
%$\tau\in\{1,2\}$, it holds
%\begin{align*}
%P^\tau_{cs}(x^{ss},x^{uu},x)&=(P_{ss}(x^{ss},x^{uu}),P_\tau(x))=(x^{ss},P_\tau(x))
%\in \mathbb{R}^{d_{ss}} \times \mathbb{R}^{\ell}  \quad \text{and} \quad  \\ Q^\tau_{cs}(x^{ss},x^{uu},x)&=(Q_{\tau}(x),Q_{uu}(x^{ss},x^{uu}))=
%(Q_{\tau}(x),x^{uu})
%\in \mathbb{R}^{m} \times \mathbb{R}^{d_{uu}} .
%\end{align*}

\begin{proof} Fix $\ti \in \Sigma^*_k$ with $|\ti|=j$.
 Write  $\bar x=(x^{s}, x^{u},x)\in \mathbb{R}^{d_{s}} \times \mathbb{R}^{d_{u}} \times N_1$
 and  $\mathbf{v}=(v^{s},v^{u},v)\in \mathbb{R}^{d_{s}}\times \mathbb{R}^{d_{u}}  \times  T_xN_1 $.
 %Consider
 % $\theta>0$
 % and $\bar{v}\in C_\theta(\bar{x},\fS^t_1)$
 %where $\bar x\in \mathcal{R}=\mathsf{H}_i\times R$. Then, we have that
 % \begin{equation}
 % \label{e.PQ}
%%   \|(v^{ss},DP_1(v))\| =
% \|DP^1_{cs}(\bar{v})\| \leqslant
%  \theta \, \|DQ^1_{u}(\bar{v})\|.
%  %=\theta  \, \|(DQ_1(v),v^{uu})\|.
%  \end{equation}
Since $\phi$ is $(\tfS_1,\tfS_2,\mu,\gamma)$-dominated  on $R$, %and $\mu^{-1}<\gamma_{\ti}$,
Lemma~\ref{l.new-dilatation} implies that
\begin{equation}
\label{e.first-second}
\|D(P_2\circ \phi )(v)\| \leq \mu^{-1} \|DP_1(v)\| \quad \text{and} \quad \|D(Q_2\circ \phi)(v)\| \geq  \gamma \|DQ_1(v)\|.
\end{equation}
Using  these inequalities and $\nu^{|\ti|} < \mu^{-1} < \gamma < \nu^{-|\ti|}$, recalling Definition \ref{d.stretching}, and setting $f=A_\ti\times \phi$, we get
% \[ \begin{split}
%   \|D(\mathcal{P}\circ F)(v)\| \leqslant& \|(\nu^{n_{\ti}} v^{ss}, D(P\circ \phi_{\ti})(v^c) )\| \leqslant \|(\nu^{n_{\ti}} v^{ss}, \gamma^{cu}_{\ti} DP(v^c) )\|
%     \leq  \gamma^{cu}_{\ti} \|(v^{ss},DP(v^c))\| \\ \leq& \theta  \gamma^{cu}_{\ti} \|(DQ(v^c),v^{uu})\| = \theta  \gamma^{cu}_{\ti} (\gamma^u_{\ti})^{-1} \|(D(Q\circ \phi_{\ti})(v^c),\gamma^u_{\ti} v^{uu})\| \\
%     \leq&   \theta  \gamma^{cu}_{\ti} (\gamma^u_{\ti})^{-1} \|(D(Q\circ \phi_{\ti})(v^c),\nu^{-n_{\ti}} v^{uu})\| =\theta ' \|D(\mathcal{Q}\circ F)(v)\|
% \end{split}
%\]
%
\begin{equation}\label{e.tercera}
\begin{aligned}
   \|D({P}_{cs,2}\circ f)(\mathbf{v})\|  &= \|\big(\nu^{|{\ti}|} v^{s}, D(P_2\circ \phi)(v) \big )\| 
   \overset{\eqref{e.first-second}}{\leq} 
   \|
   \big(\nu^{|{\ti}|} v^{s}, \mu^{-1} DP_1(v)\big )\| \\
     &\leq   \mu^{-1} \|\big(v^{s},DP_1(v)\big)\| = \mu^{-1} \|DP_{cs,1}(\mathbf{v})\|
\end{aligned}
\end{equation}
and
\begin{equation}\label{e.quarta}
\begin{aligned}
   \|D({Q}_{u,2}\circ f)(\mathbf{v})\| &= \|\big(D(Q_2\circ \phi)(v),\nu^{-|{\ti}|} v^{u} \big)\| \overset{\eqref{e.first-second}}{\geq}  \|
   \big(\gamma DQ_1(v), \nu^{-|{\ti}|} v^{u} \big)\| \\
     &\geq  \gamma \|\big(DQ_1(v),v^{u}\big)\| = \gamma \|DQ_{u,1}(\mathbf{v})\|.
\end{aligned}
\end{equation}

 Take $\mathbf{v} =(v^{s}, v^{u},v) \not\in \mathrm{Ker}\, D_{\bar{x}}(P_{cs,2}\circ f)$.
 We claim that the inequality in~\eqref{e.tercera} is strict for $\mathbf{v}$.
%Indeed, assume that $\bar{v}\not\in \mathrm{Ker}\, D_{\bar{x}}(P^2_{cs}\circ f)$.
Note first that if $\|\nu^{|{\ti}|} v^{s}\|\leq \|D(P_2\circ \phi)(v)\|$, then $\|D(P_{cs,2}\circ f)(\mathbf{v})\|=\|D(P_2\circ \phi)(v)\|$
and consequently   $v\not \in \mathrm{Ker}\, D_x(P_2\circ\phi)$.  Thus, by Lemma~\ref{l.new-dilatation}, the first inequality in~\eqref{e.first-second} is strict and hence, the first inequality in~\eqref{e.tercera} is also strict. On the other hand, if $\|D(P_2\circ \phi)(v)\| < \|\nu^{|\ti|}v^{s}\|$, then $v^{s} \not =0$, thus the second inequality in~\eqref{e.tercera} is strict since $\nu^{|\ti|} < \mu^{-1}$.
Arguing analogously, we get that if $\mathbf{v}\not\in \mathrm{Ker}\, D_{\bar x} Q_{u,1}$ then the inequality in~\eqref{e.quarta} is strict.
Lemma~\ref{l.new-dilatation} now implies that   $f=A_{\ti}\times \phi$ is $(\fS_1,\fS_2,\mu,\gamma)$-dominated on ${R}_\ti$.
 \end{proof}

%\begin{lem} Let $\mathfrak{S}=(P,Q)$ be a $(\ell,m)$-splitting of $M$. Consider a word $\ti \in \Sigma^\ast_k$ and compact set $R_{\ti}$ of $M$. If $\phi_{\ti}$ satisfies the $(\fS,\mu_\ti^{-1}, \gamma_{\ti})$-dilatation condition on $R_{\ti}$ for some positive real numbers $\mu_\ti^{-1}$ and $\gamma_{\ti}$. Then
%$\lambda_{\ti}, \gamma_{\ti}$ can be taken satisfying
%$$
%\nu^{|\ti|}<\lambda^{|\ti|}< \lambda_{\ti} \leq \gamma_{\ti} < \gamma^{|\ti|}<\nu^{-|\ti|}.
%$$
%\lmargem{perhaps add partial hyperbolicity}
%%In particular $\nu^{|\ti|}< \lambda_{\ti} < \gamma_{\ti} < \nu^{-|\ti|}$.
%\end{lem}

To prove the theorem,
  we wiil see that (PB1)-(PB3) in
  Definition~\ref{def:cu-blending}
 for $\tfB_{\mathcal{F}}$ imply (BM1)-(BM4) in
  Definition~\ref{def:cu-blender} for $\fB_{F}$ for  every $\theta>0$
  small enough. 
  %For the horseshoe map $A$ we consider the sets $H_i$ and  $V_i$ as in Definition~\ref{ld.affine}.

  \noindent{{\em{Proof of \emph{(BM1)}:}}} As the stretching $\nu$ of $A$ is adapted to $\tfB_{\mathcal{F}}$, it follows that
$\phi_\ti$ is $(\tfS,\mu_\ti,\gamma_\ti)$-dominated  on $R_\ti$ with $\nu^{|\ti|}<\mu^{-1}_\ti<\gamma_\ti<\nu^{-|\ti|}$   for every $\ti \in I$, (PB1).
 Lemma~\ref{l.cone} implies that $A_{\ti} \times \phi_\ti$ satisfies the $(\fS,\theta)$-cone invariance condition on $\bfR_\ti$ for 
 every $\theta>0$. As
 $F|_{\bfR_\ti}=A_\ti\times \phi_\ti$, it follows that $F$ satisfies the $(\fS,\theta)$-cone invariance condition on $\bfR_\ti$.

  \noindent{\em{Proof of \emph{(BM2)}:}}
  Recall $H_\ti$ and $V_\ti$ in \eqref{e.HiVi}.
  Note that $R_{\ti}$ is a $(P,Q\circ \phi_{\ti})$-rectangle for every $\ti\in I$, (PB2).
As
$F|_{\bfR_{\ti}} =
{A}_{\ti} \times \phi_{\ti}$,
it follows that  every $\bfR_{\ti}$ is a $({P}_{cs},{Q}_u\circ F)$-rectangle. Similarly, 
%from (PB2), 
we get
 \begin{align*}
 {Q}_u(\bfR_{\ti}) &= Q(R_{\ti})\times  V_{{\ti}} \overset{\text{(PB2)}}{\Subset} Q(B)  \times  V \overset{\eqref{eq.nopair}}{=} {Q}_u(Z), \\
 ({Q}_u\circ F)(\bfR_{\ti}) &=(Q\circ \phi_\ti)({{R_{\ti}}}) \times  V  = Q(B) \times V \overset{\eqref{eq.nopair}}{=}  {Q}_u(Z), \\
 ({P}_{cs}\circ F)(\bfR_{\ti}) &= H_{\ti} \times (P\circ \phi_{\ti})(R_{\ti}) \Subset H \times P(B) \overset{\eqref{eq.nopair}}{=}  {P}_{cs}(Z),
 \end{align*}
 proving  (BM2).
 % in Definition~\ref{def:cu-blender}.

\noindent{{\em{Proof of \emph{(BM3)}:}}}
 Note that ${P}_{cs}\colon \mathbb{R}^{d_{s}+d_{u}}\times  M \to \mathbb{R}^{d_{s}+d_c}= \mathbb{R}^{d_{cs}} $.
 Thus we can consider the composition $\pi_s\circ {P}_{cs}$ to get
\begin{align*}
(\pi_s\circ {P}_{cs})(\bfR_{\ti}) =
\pi_s (H \times R_\ti)
= H= \pi_s( H \times B)= (\pi_s \circ P_{cs})(Z),
\end{align*}
obtaining (BM3).

\noindent{{\em{Proof of \emph{(BM4)}:}}}
%From (PB3), we have
Note that we have
\begin{align*}
(\pi_c\circ {P}_{cs})(Z) &\overset{\eqref{eq.nopair}}{=} P(B) \overset{\text{(PB3)}}{\Subset} \bigcup_{i\in I} P (R_{\ti}) = \bigcup_{i\in I} (\pi_c\circ {P}_{cs})(\bfR_{\ti}).
\end{align*}

Let  $\Delta$ be a Lebesgue number of the
covering $\{\mathrm{int}\,(\pi_c\circ {P}_{cs})(\bfR_{\ti})\}_{\ti \in I}$ of $(\pi_c\circ P_{cs})(Z)$.
Taking $\theta_0$ small enough such that $\Delta > \theta_0 \, \mathrm{diam}\, Q_{u}(Z)$,
condition
(BM4) holds for every $\theta \in (0,\theta_0]$. This completes the proof of the theorem.
 \end{proof}

\section{Connecting split blending machines}
\label{s.connectingsplitblenders}
We present a refined version of \cite[Theorem 2.13]{Asa:22}, establishing a robust intersection between the stable sets 
of two split blending machines\footnote{We follow the approach in \cite{Asa:22}, where these machines are defined and connected in an abstract way for two different dynamics, which is why their stable sets may intersect.}, see Theorem~\ref{t.Asaoka} stated in Section~\ref{ss.csbm} and whose proof is completed  in Section~\ref{ss.keylemma}. These two machines model, respectively, a $\cs$-blender  and a $\cu$-blender. The construction relies on a transition map connecting the two machines, see  Section~\ref{ss.transitionsbetweenmachines}. The conditions satisfied by this transition map (see Definition~\ref{d.adaptedt}) resemble those in \cite[Theorem 2.13]{Asa:22}. The core of the argument
 is Lemma~\ref{l.asaokamodificado}, proved in Section~\ref{ss.keylemma}.
 % which completes the proof of Theorem~\ref{t.Asaoka}.

 \subsection{Transitions between split blending machines}
 \label{ss.transitionsbetweenmachines}
 For what follows, 
 recall the notation for a split blending machine $\fB = (\mathcal{R}, Z, \fS, \fC, \theta)$,
 where $\mathcal{R}= \{ R_{\ti}\}_{\ti \in I}$,
  in Definition~\ref{def:cu-blender}. 
%  Note that this notation implicitly includes the set of symbols $I$ associated with $\fB$
%  labeling the rectangles.

\begin{defi}[Adapted transitions between split blending machines]
\label{d.adaptedt}
Let $p,q, d_{c,1}, d_{c,2}$ be positive integers with $d_{c,1}<p$  and $d_{c,2} < q$. For $\tau=1,2$, consider open 
 manifolds
$N_\tau$, open subsets $U_\tau$ of $N_\tau$, and split blending machines $\fB_{\tau}=(\mathcal{R}_\tau,
 %\{R_{\tau,\ti}\}_{\ti\in I_{\tau}},
Z_\tau, \fS_\tau, \fC_\tau,  \theta)$
  for  $C^1$ maps $f_\tau \colon U_\tau \to N_\tau$, where
 \begin{enumerate}[label=$\bullet$, leftmargin=0.5cm]
\item
$\fS_1=(P_{cs,1},Q_{u,1})$
is  a $(p,q)$-splitting  of $N_1$,
\item
$\fS_2=(Q_{cs,2},P_{u,2})$  is a
$(q, p)$-splitting  of $N_2$,
%of open manifolds $N_1$ and $N_2$, respectively.
%Consider open manifolds $N_1$ and $N_2$, an $(\ell,m)$-splitting $\fS_1=(P^1_{cs},Q^1_{u})$ defined on $N_1$
%and an $(m, \ell)$-splitting $\fS_2=(Q^2_{cs},P^2_{u})$ defined on $N_2$.
\item
$\fC_1=(P_s,P_c)$  is a $(p-d_{c,1},d_{c,1})$-splitting  of $\mathbb{R}^p$,
and
\item
$\fC_2=(Q_s,Q_c)$ is a $(q-d_{c,2},d_{c,2})$-splitting of $\mathbb{R}^q$.
\end{enumerate}

Consider
a compact set $R_\sharp$,
 an open neighbourhood $U_\sharp$ of $R_\sharp$,
with $U_\sharp \subset U_2$,
 and a $C^1$ embedding  $f_\sharp \colon U_\sharp \to N_1$.
We say that the 
triple $(R_\sharp, U_\sharp, f_\sharp)$
%the quintuple $(R_\sharp, U_\sharp, f_\sharp, P,Q)$ 
is a 
{\em{adapted transition between the split blending machines $\fB_{1}$ and $\fB_{2}$}}
if there are
 positive numbers $\mu_\tau$, $\gamma_\tau$, $\tau \in \{ 1, 2, \sharp\}$,
satisfying conditions (AT1)--(AT5) below: 
\begin{enumerate}[label=(AT\arabic*),  leftmargin=1cm] %, %leftmargin=1cm]
\item
$\mu_1\gamma_2 >1$ and $\mu_2\gamma_1 > 1$,
\item for $\tau=1,2$, the map $f_\tau$ satisfies %the $(\fS_\tau,\theta, \mu_\tau, \gamma_\tau)$-cone and
the $(\tfS_\tau,\theta,\mu_\tau,\gamma_\tau)$-cone condition  on every set $R_{\tau,\ti}$ of  $\mathcal{R}_\tau\eqdef \{R_{\tau,\ti}\}_{\ti\in I_{\tau}}$,
where
\begin{equation}\label{e.bh}
\begin{split}
&\mbox{$\tfS_1 \eqdef (P_c\circ P_{cs,1},Q_{c}\circ Q_{u,1})$ is a $(d_{c,1},d_{c,2})$-pair of $N_1$ and}\\
&\mbox{%$\tfS_1 \eqdef (\mathscr{P}_1,\mathscr{Q}_1)$ and $\tfS_2\eqdef (\mathscr{Q}_2,\mathscr{P}_2)$
%where
%$$ \mathscr{P}_1 = P_c\circ P^1_{cs}$$
$\tfS_2 \eqdef (Q_c \circ Q_{cs,2},P_{c}\circ P_{u,2})$ is a $(d_{c,2},d_{c,1})$-pair of $N_2$,
}
\end{split}
\end{equation}
\item
$R_\sharp$ is a ($P_{u,2},Q_{u,1} \circ f_\sharp)$-rectangle such that
\begin{align*}
&\mathrm{(a)}\,\, Q_{cs,2}(R_\sharp) \Subset Q_{cs,2}(Z_2), & & \mathrm{(b)}\,\,P_{u,2}(R_\sharp) = P_{u,2}(Z_2), \\
 & \mathrm{(c)}\,\,(P_{cs,1}\circ f_\sharp)(R_\sharp)\Subset P_{cs,1}(Z_1),
&   &\mathrm{(d)}\,\, (Q_{u,1}\circ f_\sharp)(R_\sharp) = Q_{u,1}(Z_1),
\end{align*}
\item $f_\sharp$ satisfies the %$(\fS_2^t,\fS^{}_1,\theta, \mu_\sharp, \gamma_\sharp)$-cone and the
$(\tfS_2^t,\tfS^{}_1,\theta, \mu_\sharp, \gamma_\sharp)$-cone condition~on~$R_\sharp$,
%$P_s\colon\mathbb{R}^{\ell}\to \mathbb{R}^{\ell-d^1_c}$ and $Q_s\colon \mathbb{R}^{m}\to \mathbb{R}^{m-%d^2_c}$
%\item $f_\sharp$ satisfies  on $R_\sharp$;
\item
 the Lebesgue number
 $\Delta_\tau$  of $\fB_\tau$  satisfy
%Let $\Delta_\tau$ be the Lebesgue number of $\fB_\tau$
%the cover $\{P_\tau(R_{\tau,\ti})\}_{\ti \in I_{\tau}}$ of $P_\tau(Z_\tau)$
%for $\tau = 1, 2$. Then,
\begin{align*}
\theta\, \diam ( Q_{c}\circ Q_{u,1})(Z_1) &+ \mu_{\sharp}^{-1} \diam (P_{c}\circ P_{u,2})(Z_2) < \Delta_1, \\
\theta \,\diam (P_{c}\circ P_{u,2})(Z_2) &+ \gamma^{-1}_\sharp \diam (Q_{c}\circ Q_{u,1})(Z_1) < \Delta_2.
\end{align*}
\end{enumerate}
\end{defi}

\begin{remark}
\label{r.dimensionsofmachines}
The split blender machines $\fB_{1}$ and $\fB_{2}$ in Definition~\ref{d.adaptedt} 
have dimensions   $(p-d_{c,1}, d_{c,1}, q)$  and
 $(q-d_{c,2}, d_{c,2}, p)$, respectively.
\end{remark}

\begin{remark}[Definition~\ref{d.adaptedt} versus  Theorem 2.13 in \cite{Asa:22}]
\label{r.adapted}
Each condition (ATi) mimics hypothesis (i) in \cite{Asa:22}.
The main difference occurs between (AT5) and (5).
Condition (5)   in \cite{Asa:22} deals with the diameter of the unstable projection of the superposition region of the blending machine.
This condition is not suitable for our purposes (as these diameters cannot be small enough). 
In contrast, condition (AT5)
depends only on the central projection, which
in our setting can be done as small as required.
\end{remark}

\begin{remark}
\label{r.compareasaoka}
Bearing in mind item (ii) in Remark~\ref{r.onthedefbm},
if in Definition~\ref{d.adaptedt} we take
$P_c=\mathrm{Id}$, $Q_c=\mathrm{Id}$ and $P_s$ and $Q_s$ the trivial projections, we recover
the conditions in \cite{Asa:22}. 
\end{remark}

\subsection{Connecting split blending machines} 
\label{ss.csbm}
 Recall the definition of the stable set $ \Lambda^s(\fB_f)$ in
\eqref{e.stablesetblender}. The theorem below does not explicitly involve 
the dimensions of the machines.
The fact that these dimensions match is implicitly contained in the definition of the transition map, see also
Remark~\ref{r.dimensionsofmachines}.

\begin{thm}\label{t.Asaoka}
For $\tau=1,2$, consider  split blender machines $\fB_{\tau}$
  for  $C^1$ maps $f_\tau \colon U_\tau \to N_\tau$.
  % where
%\begin{enumerate}[label=$\bullet$, leftmargin=0.5cm]
%\item
%$\fS_1=(P_{cs,1},Q_{u,1})$
%is  a $(p,q)$-splitting  of $N_1$ and
%$\fS_2=(Q_{cs,2},P_{u,2})$  is a
%$(q, p)$-splitting  of $N_2$,
%%of open manifolds $N_1$ and $N_2$, respectively.
%%Consider open manifolds $N_1$ and $N_2$, an $(\ell,m)$-splitting $\fS_1=(P^1_{cs},Q^1_{u})$ defined on $N_1$
%%and an $(m, \ell)$-splitting $\fS_2=(Q^2_{cs},P^2_{u})$ defined on $N_2$.
%\item
%$\fC_1=(P_s,P_c)$  is a $(p-d_{c,1},d_{c,1})$-splitting  of $\mathbb{R}^p$
%and
%$\fC_2=(Q_s,Q_c)$ is a $(q-d_{c,2},d_{c,2})$-splitting of $\mathbb{R}^q$.
%\end{enumerate}
Let $(R_\sharp, U_\sharp, f_\sharp)$ be an adapted transition between $\fB_{1}$ and 
$\fB_{2}$. Then
 $$
 \Lambda^s(\fB_1) \cap f_\sharp(R_\sharp \cap \Lambda^s(\fB_2))\not = \emptyset.
 $$
Moreover, this intersection is robust:
given a compact neighbourhoods $K_\tau$ of\,
$\bigcup_{\ti \in I_\tau}R_{\tau,\ti}$,
$\tau=1,2$,
there exist $C^1$ neighbourhoods $\mathcal{U}_\tau$ of $f_\tau$,
 $\tau=1,2,\sharp$,
 such that
$$
\Lambda^s(K_1,g_1) \cap g_\sharp(\Lambda^s(K_2,g_2)) \not=\emptyset \quad
\text{for every $g_\tau \in \mathcal{U}_\tau$.}
$$
\end{thm}

\begin{proof}
The proof follows modifying the inductive process in the proof of~\cite[Theorem~2.13]{Asa:22}.
For completeness, we will repeat the steps of this proof and explain the changes.
Write
\begin{enumerate}[label=$\bullet$, leftmargin=0.5cm]
\item
$\fS_1=(P_{cs,1},Q_{u,1})$
is  a $(p,q)$-splitting  and
$\fS_2=(Q_{cs,2},P_{u,2})$  is a
$(q, p)$-splitting,
%of open manifolds $N_1$ and $N_2$, respectively.
%Consider open manifolds $N_1$ and $N_2$, an $(\ell,m)$-splitting $\fS_1=(P^1_{cs},Q^1_{u})$ defined on $N_1$
%and an $(m, \ell)$-splitting $\fS_2=(Q^2_{cs},P^2_{u})$ defined on $N_2$.
\item
$\fC_1=(P_s,P_c)$  is a $(p-d_{c,1},d_{c,1})$-splitting 
and
$\fC_2=(Q_s,Q_c)$ is a $(q-d_{c,2},d_{c,2})$-splitting.
\end{enumerate}

Let $R$ be a
compact subset of $U_2$ and
$h$ a
$C^1$ embedding 
from an open neighbourhood of
 $R$ to $U_1$. We say that the pair $(R,h)$
satisfies the {\em{condition $(\sharp)$}}  if:
\begin{enumerate}
%[itemsep=0.2cm,leftmargin=1.2cm, label=($\#$\arabic*)]
  \item [$(\sharp1)$]$R$ is a $(P_{u,2},Q_{u,1} \circ h)$-rectangle;
  \item [$(\sharp2)$] $h$ satisfies %the $(\fS^t_2,\fS_1, \theta, \mu_\sharp, \gamma_\sharp)$-cone and
  the $(\tfS^t_2,\tfS_1, \theta , \mu_\sharp,\gamma_\sharp)$-cone condition~on~$R$, where $\tfS_1$ and $\tfS_2$ are as in~\eqref{e.bh};
%  \item $h$ satisfies the  on $R$;
  \item [$(\sharp3)$] it holds
  \begin{align*}
& \mathrm{(a)}\,\, Q_{cs,2}(R) \Subset Q_{cs,2}(Z_2), &  &\mathrm{(b)}\,\,  P_{u,2}(R) = P_{u,2}(Z_2),
\\
&   \mathrm{(c)}\,\,(P_{cs,1}\circ h)(R) \Subset P_{cs,1}(Z_1), &
& \mathrm{(d)}\,\,(Q_{u,1}\circ h)(R) = Q_{u,1}(Z_1).
\end{align*}
\end{enumerate}

\begin{remark} 
\label{r.satisfiessharp}
By (AT3)--(AT4), the pair
$(R_\sharp, f_\sharp)$ satisfies
condition $(\sharp)$.
\end{remark}

\begin{lem}
\label{l.asaokamodificado}
Consider a pair $(R, h)$ satisfying condition ($\sharp$).
Then there are $\ti \in I_1$ and $\tk \in I_2$ such that the
pair $(R',h')$  given by
$$
R' \eqdef f_2\big(
 R_{2,\tk} \cap
 R
 \cap
h^{-1}(R_{1,\ti})
\big) \quad \text{and} \quad  h'\eqdef f_1 \circ h \circ f^{-1}_2
$$
satisfies  condition ($\sharp$).
\end{lem}

%\begin{lem}
%\label{l.asaokamodificado}
%There exist $\ti \in I_1$ and $\tk \in I_2$ such that the
%pair $(R',h')$  given by
%$$
%R' \eqdef f_2\big(
% R_{2,\tk} \cap
% R
% \cap
%f_\sharp^{-1}(R_{1,\ti})
%\big) \quad \text{and} \quad h'\eqdef f_1 \circ f_\sharp \circ f^{-1}_2
%$$
%satisfies  condition ($\sharp$).
%\end{lem}

We postpone the proof of the lemma and
complete the proof of the theorem assuming it.  By Remark~\ref{r.satisfiessharp},
the pair $(R_0, h_0) =(R_\sharp, f_\sharp)$ satisfies condition ($\sharp$). Let $(R_1, h_1)$ be the pair
 obtained from $(R_0, h_0)$ using Lemma~\ref{l.asaokamodificado}. Since this pair also satisfies
 condition ($\sharp$), we can argue inductively to get a sequence of pairs $(R_n, h_n)$, where each
 $(R_n, h_n)$ is obtained from $(R_{n-1}, h_{n-1})$ using Lemma~\ref{l.asaokamodificado}.
 In this construction we have
 $$
 h_{n} \eqdef  f_1 \circ h_{n-1} \circ f_{2}^{-1} =f_1^n \circ f_\sharp \circ f_2^{-n} 
 \quad \mbox{and} \quad  
 R_{n} \eqdef  f_2\big(R_{2,\tk_{n-1}}\cap R_{n-1} \cap h_{n-1}^{-1}(R_{1,\ti_{n-1}})  \big).
 $$
 where $\ti_{n} \in I_1$ and $\tk_n \in I_2$.
 Each set $R_n$ is a non-empty compact subset of $U_2$, then $\bigcap_{n\geq 0} f_\sharp \circ f_2^{-n} (R_n)\neq \emptyset$. 
 Finally, we have that
 $$
  \bigcap_{n\geq 0}  f_1^{-n} (R_{1,\ti_n}) \cap f_\sharp \left(R_\sharp \cap
   \bigcap_{n\geq 0} f_2^{-n} (R_{2,\tk_n})
   \right) 
=
\bigcap_{n\geq 0} f_\sharp \circ f_2^{-n} (R_n)   
   \ne \emptyset .
 $$
Recalling \eqref{e.stablesetblender},   note that
 $$
  \bigcap_{n\geq 0}  f_1^{-n} (R_{1,\ti_n}) \subset  \Lambda^s(\fB_1)
  \quad \mbox{and} \quad  \bigcap_{n\geq 0} f_2^{-n} (R_{2,\tk_n})
  \subset \Lambda^s(\fB_2)
$$
implies that 
$$
 \Lambda^s(\fB_1) \cap f_\sharp(R_\sharp \cap \Lambda^s(\fB_2))\not = \emptyset,
 $$
proving the first part of the theorem.

The robustness statement is standard and follows exactly as in \cite{Asa:22}.
\end{proof}

To complete the proof the theorem we prove the above lemma.

\subsection{Proof of Lemma \ref{l.asaokamodificado}}
\label{ss.keylemma}
According to  (b) and (d) in condition $(\sharp3)$, we have that 
\begin{equation}
\label{e.agua}
P_{u,2}(R) = P_{u,2}(Z_2)\qquad\mbox{and}\qquad (Q_{u,1}\circ h)(R) = Q_{u,1}(Z_1).
\end{equation}
Since 
$h$ satisfies %the $(\fS^t_2,\fS_1, \theta, \mu_\sharp, \gamma_\sharp)$-cone and
  the $(\tfS^t_2,\tfS_1, \theta , \mu_\sharp,\gamma_\sharp)$-cone condition~on~$R$ (condition $(\sharp2)$), then 
Lemma~\ref{lem.new.asaoka},  \eqref{e.agua}, and condition (AT5) 
imply that
\begin{equation*}
\begin{split}
\diam (P_{c}\circ P_{cs,1})(h(R)) \leqslant \theta\, \diam ( Q_{c}\circ Q_{u,1})(Z_1) &+ \mu_{\sharp}^{-1} \diam (P_{c}\circ P_{u,2})(Z_2) \overset{\text{(AT5)}}{<}  \Delta_1, \\
\diam (Q_{c}\circ Q_{cs,2})(R) \leqslant\theta \,\diam (P_{c}\circ P_{u,2})(Z_2) &+ \gamma^{-1}_\sharp \diam (Q_{c}\circ Q_{u,1})(Z_1) \overset{\text{(AT5)}}{<} \Delta_2,
\end{split}
\end{equation*}
where 
%$\Delta_\tau$ is the Lebesgue number of $\fB_{\tau}$. By definition, 
$\Delta_1$ and $\Delta_2$ are
 Lebesgue numbers of the coverings $\{\mathrm{int}\, \big(P_{c}\circ P_{cs,1} (R_{1,\ti})\big)\}_{\ti\in I_1}$ 
 of $\overline{P_{c}\circ P_{cs,1}(Z_1)}$ and 
  $\{\mathrm{int}\,\big(Q_{c}\circ Q_{cs,2}(R_{2,\tk})\big)\}_{\tk\in I_2}$ of $\overline{Q_{c}\circ Q_{cs,2} (Z_2)}$, respectively. 
  Thus, there exist 
$\ti\in I_{1}$ and $\tk \in  I_2$ such that
\begin{equation} \label{eq.1rev}
%\begin{split}
%&
P_{c} \circ P_{cs,1}(h(R))
% \subset  P_{c} \circ P_{cs}^1(Z_1)
\subset \mathrm{int}\,\big(P_{c}  \circ P_{cs,1}(R_{1,\ti}) \big)
\quad
\text{and}
\quad
 Q_{c} \circ Q_{cs,2}(R)
%\subset Q_{c} \circ Q_{cs}^2(Z_2)
\subset \mathrm{int}\, \big(Q_{c}\circ Q_{cs,2}(R_{2,\tk})\big).
%\end{split}
\end{equation}

Now we will verify that $(R',h')$ satisfy the condition ($\sharp$) for the rectangles $R_{1,\ti}, R_{2,\tk}$ obtained above.  We begin with condition ($\sharp$1).
Consider
\begin{itemize}
\label{e.splittingsytodoeso}
\item the $(p,q)$-splittings $ \fL_1, \fL_2$ of $N_2$ define by 
$$\fL_i= (P_i, Q_i)\eqdef (P_{u,2},Q_{cs,2})= \fS_2^t, \quad i=1,2,$$
\item the $(p,q)$-splittings $ \fL_3, \fL_4$ of $N_1$ defined by
$$
\fL_j=(P_j,Q_j)\eqdef (P_{cs,1}, Q_{u,1})=\fS_1, \quad j=3,4,
$$ 
\item the compact sets $R_1\eqdef f_2(R_{2,\tk})$,  $R_2\eqdef R$ and $R_3=R_{1,\ti}$,
\item the $C^1$ embeddings 
\begin{itemize}
\item $g_1=\mbox{the restriction of $f_2^{-1}$ to an open neighbourhood of $R_1$}$,
\item $g_2=h$,
\item $g_3=\mbox{the restriction of $f_1$ to an open neighbourhood of $R_3$}$,
\end{itemize}
and define $G_1\eqdef g_1$, $G_2\eqdef g_2\circ g_1$, and $G_3\eqdef  g_3\circ g_2\circ g_1=h'$.
\end{itemize}
If for $i=1,2,3$, the set $R_i$ is a $(P_i, Q_{i+1} \circ g_i)$-rectangle such that
$$
P_{i+1}\big(g_i(R_i)\big) \subset \mathrm{int}\big(P_{i+1}(R_{i+1})\big)\quad\mbox{and}\quad 
Q_{i+1}(R_{i+1}) \subset \mathrm{int}\big(Q_{i+1}(g_i(R_{i}))\big),
$$
then Lemma~\ref{l.simpleasa} guarantees that 
$$
R_*\eqdef R_1\cap G_1^{-1}(R_2)\cap G_2^{-1}(R_3)=f_2\big(
 R_{2,\tk} \cap
 R
 \cap
h^{-1}(R_{1,\ti})
\big)=R'
$$
is a $(P_{u,2},Q_{u,1}\circ h')$-rectangle with
$$
P_{u,2}(R') =  P_{u,2}\big(f_2(R_{2,\tk})\big)\quad\mbox{and}\quad (Q_{u,1}\circ h')(R') = Q_{u,1}\big(f_1(R_{1,\ti})\big),
$$
 obtaining condition ($\sharp$1), as desired. We will now check the part ``if'' above.
%\begin{equation}
%\label{e.splittingsytodoeso}
%\begin{split}
%&
%\mbox{\hspace{-4cm}the
%$(p,q)$-splittings}
%%$$
%%(P_1,Q_1)= (P_2,Q_2) \eqdef (P^2_u,Q^2_{cs})=\fS_2^t \quad \text{and} \quad  (P_3,Q_3)=(P_4,Q_4)\eqdef (P^1_u,Q^1_{cs})=\fS_1,
%%$$
%%
%%\textcolor{blue}{
%%$$
%%(P_1,Q_1)= (P_2,Q_2) \eqdef (P^2_u,Q^2_{cs})=\fS_2^t \quad \text{and} \quad  (P_3,Q_3)=(P_4,Q_4)\eqdef (P^1_u,Q^1_{cs})=\fS_1,
%%$$}
%%\lmargem{en el enunciado no hay transposiciones....}
%%
%\\
%\quad \quad \quad \fL_i&= (P_i, Q_i)=(P_{u,2},Q_{cs,2})= \fS_2^t, \quad \text{for $i=1,2$},\\
%%\fL_2&=(P_2,Q_2)= (P_u^2, Q_{cs}^2)=\fS_2^t, \\
%\quad \quad \quad
%\fL_j&=(P_j,Q_j)= (P_{cs,1}, Q_{u,1})=\fS_1, \quad \text{for $j=3,4$},\\
%%\fL_4&=(P_4,Q_4)= (P_{cs}^1 ,Q^u_1)=\fS_1.
%&
%\mbox{\hspace{-4cm}the embeddings
%$g_1=f_2^{-1}$, \textcolor{red}{$g_2= h$} and $g_3=f_1$,}
%\\
%&\mbox{\hspace{-4cm}
%the compact sets
%$R_1=f_2(R_{2,\tk})$,
%$R_2=
%R$, and $R_3=R_{1,\ti}$.}
%\end{split}
%\end{equation}
%We need to verify that these elements satisfy the assumptions of the lemma.
%in \eqref{e.splittingsytodoeso},

Conditions ``$R_i$ is a $(P_i, Q_{i+1} \circ g_i)$-rectangle" translate to:
\begin{enumerate}[itemsep=0.2cm,leftmargin=1.2cm, label=(r\arabic*)]
 \item
 %$R_1$ is a $(P_1, Q_2 \circ g_1)$-rectangle translates to
 $f_2(R_{2,\tk})$ is a $(P_{u,2}, Q_{cs,2} \circ f_2^{-1})$-rectangle,
 \item
% $R_2$ is a $(P_2, Q_3 \circ g_2)$-rectangle translates to
 $R$ is a $(P_{u,2}, Q_{u,1} \circ h)$-rectangle, and
  \item
 %$R_3$ is a $(P_3, Q_4 \circ g_3)$-rectangle translates to
 $R_{1,\ti}$ is a
 $(P_{cs,1}, Q_{u,1} \circ f_1)$-rectangle.
\end{enumerate}
The inclusion conditions are translated as follows:
\begin{enumerate}[itemsep=0.2cm,leftmargin=1.2cm, label=(i\arabic*)]
\item
$P_2 \big(g_1(R_1)\big) \subset \mathrm{int} \big(P_2 (R_2)\big) \mapsto P_{u,2} \big(f_2^{-1}(f_2 (R_{2,\tk}) )\big)= P_{u,2} (R_{2,\tk}) \subset
\mathrm{int} \big(P_{u,2}(R)\big)$,
\item
$P_3 \big(g_2(R_2)\big) \subset \mathrm{int} \big(P_3 (R_3)\big) \mapsto P_{cs,1} \circ h (R) \subset \mathrm{int} \big(P_{cs,1} (R_{1,\ti})\big)$,
\item
$Q_2 (R_2) \subset \mathrm{int} \big(Q_2 (g_1 (R_1))\big) \mapsto
Q_{cs,2} (R)  \subset \mathrm{int} \big(Q_{cs,2}(f_2^{-1} (f_2 (R_{2,\tk}) ))\big)=
\mathrm{int} \big(Q_{cs,2} (R_{2,\tk}) \big),
$
\item
$Q_3 (R_3) \subset \mathrm{int} \big(Q_3 (g_2 (R_2))\big) \mapsto Q_{u,1} (R_{1, \ti}) \subset \mathrm{int}\big (Q_{u,1} (h(R))\big)$.
\end{enumerate}

%To prove the lemma it is enough to check that conditions (t1)-(t4) are satisfied
%and Conditions ($\#$2) and ($\#$3) hold for $(R', h')$. This will be done in Claims~\ref{f.claim} and \ref{s.claim},
%respectively.
%\lmargem{no se para que necesitamos repetir esto de nuevo}
%\textcolor{red}{Note that $\tfS_1=(P_{c}\circ P^1_{cs},Q\circ Q^1_u)$ and $\tfS_2=( Q_c\circ Q^2_{cs},P\circ P^2_u)$.}
%

\begin{claim}
\label{f.claim}
%With the choices in \eqref{e.splittingsytodoeso},
Conditions $\mathrm{(r1)-(r3)}$ and $\mathrm{(i1)-(i4)}$ hold.
\end{claim}

\begin{proof}
We just check condition (r1), conditions (r2)-(r3) follow similarly and their proofs omitted.
Note that $f_2 (R_{2, \tk})$ is diffeomorphic to $R_{2,\tk}$, which is (by definition) a
$(Q_{cs,2}, P_{u,2} \circ f_2)$-rectangle. Therefore
$R_{2,\tk}$ (and hence $f_2 (R_{2, \tk})$) is diffeomorphic to
$$
\big(Q_{cs,2} (R_{2,\tk})\big)\times \big(P_{u,2} \circ f_2 (R_{2,\tk}) \big)=
 \big(Q_{cs,2} \circ f_2^{-1} (f_2(R_{2,\tk}))\big)\times 
\big(P_{u,2} \circ f_2 (R_{2,\tk}) \big).
$$
This precisely means that  $f_2 (R_{2, \tk})$ is a $(P_{u,2}, Q_{cs,2} \circ f_2^{-1})$-rectangle.

We now check the inclusion properties. We begin proving (i2) and (i3).
%  Since $h$ satisfies the %$(\fS^t_2,\fS_1, \theta, \mu_\sharp, \gamma_\sharp)$-cone and
% $(\tfS^t_2,\tfS_1, \theta,\mu_\sharp, \gamma_\sharp)$-cone condition on $R$,
% Lemma~\ref{lem.new.asaoka} and (b5) imply that
% \begin{equation}
% \label{e:delta12}
% \begin{split}
%\diam ({Q}_c \circ Q_{cs,2}(R)) & \leq \theta\, \diam ({P}_c \circ P_{cs,2}(Z_2)) + \gamma^{-1}_\sharp\,
%\diam ({Q}_c \circ Q_{cs,1} (Z_1))< \Delta_2, \\
%\diam ({P}_c \circ P_{cs,1} (h(R))) & \leq   \theta \, \diam ({Q}_c \circ Q_{cs,1}(Z_1)) + \mu^{-1}_\sharp\,
%\diam ({P}_c \circ P_{cs,2} (Z_2))< \Delta_1.
%\end{split}
%\end{equation}
%Note that by ($\#$3a)
%  and ($\#$3c)
%  it follows
%$$
%Q_{c}\circ Q_{cs,2}(R)\subset Q_{c}\circ Q_{cs,2}(Z_2) \quad \mbox{and} \quad
%P_c \circ P_{cs,1}(h(R)) \subset P_c \circ P_{cs,1}(Z_1).
%$$
%As $\Delta_1$ and $\Delta_2$ are
% Lebesgue numbers of the coverings $\{\mathrm{int}\, (P_{c}\circ P_{cs,1} (R_{1,\ti}))\}_{\ti\in I_1}$ and
%  $\{\mathrm{int}\,(Q_{c}\circ Q_{cs,2}(R_{2,\tk}))\}_{\tk\in I_2}$ of
%  $\overline{P_{c}\circ P_{cs,1}(Z_1)}$ and  $\overline{Q_{c}\circ Q_{cs,2} (Z_2)}$, respectively,
%by \eqref{e:delta12} there are
%$\ti\in I_{1}$ and $\tk \in  I_2$ such that
%\begin{equation} \label{eq.1rev}
%%\begin{split}
%%&
%P_c \circ P_{cs,1}(h(R))
%% \subset  P_c \circ P_{cs}^1(Z_1)
%\subset \mathrm{int}\,(P_c \circ P_{cs,1}(R_{1,\ti}) )
%\qquad
%\text{and}
%\qquad
% Q_c \circ Q_{cs,2}(R)
%%\subset Q_c \circ Q_{cs}^2(Z_2)
%\subset \mathrm{int}\, (Q_c \circ Q_{cs,2}(R_{1,\tk})).
%%\end{split}
%\end{equation}
For that, the condition ($\sharp$3)  part (a) and (c) implies that
$$
P_{cs,1}(h(R)) \Subset P_{cs,1}(Z_1)
\quad
\mbox{and} \quad Q_{cs,2}(R)\Subset
Q_{cs,2} (Z_2)
$$
and  condition (BM3) in Definition \ref{def:cu-blender} of a blender machine
implies that
$$
P_s\circ P_{cs,1}(Z_1)
\Subset
(P_s\circ P_{cs,1})(R_{1,\ti})
 \qquad \text{and} \qquad
Q_s\circ Q_{cs,2}(Z_2)\Subset (Q_s\circ Q_{cs,2})(R_{2,\tk}).
$$
Therefore
\begin{equation}\label{eq.2rev}
(P_s\circ P_{cs,1})(h(R)) \subset \mathrm{int}\, (P_s\circ P_{cs,1})(R_{1,\ti}) \quad
\text{and} \quad (Q_s\circ Q_{cs,2})(R) \subset \mathrm{int}\, (Q_s\circ Q_{cs,2})(R_{2,\tk}).
\end{equation}
Recalling the diffeomorphims $\fC_1=(P_s,P_c)$ and $\fC_2=(Q_s,Q_c)$, it follows that the two inclusions 
~\eqref{eq.1rev} and~\eqref{eq.2rev} are equivalent to 
\[
\begin{split}
 \fC_1 \big(P_{cs,1}(h(R)\big)    &\subset \mathrm{int} \big(\fC_1(P_{cs,1}(R_{1,\ti}))\big) =
 \fC_1\big(\mathrm{int}\, P_{cs,1}(R_{1,\ti})\big)  \quad  \text{and} \\
 \fC_2 \big(Q_{cs,2}(R)\big)& \subset \mathrm{int} \big(\fC_2 (Q_{cs,2}(R_{2,\tk}))\big) =
 \fC_2 \big(\mathrm{int} (  Q_{cs,2}(R_{2,\tk}))\big).
\end{split}
\]
Hence
\begin{equation} \label{eq:pp1}
{P}_{cs,1}(h(R)) \subset  \mathrm{int} \big({P}_{cs,1}(R_{1,\ti} )\big) \quad \text{and}
\quad {Q}_{cs,2}(R) \subset \mathrm{int} \big( {Q}_{cs,2}(R_{2,\tk})\big)
\end{equation}
and inclusions (i2) and (i3) follow.

For inclusions in (i4) and (i1), by ($\sharp$3) part (b) and (d) and the condition (BM2) in  Definition \ref{def:cu-blender}  of a blender machine, it follows
\begin{equation} \label{eq:pp2}
\begin{split}
Q_{u,1}(R_{1,\ti} ) &\subset  \mathrm{int} (Q_{u,1}(Z_1)) \overset{\text{(d)}}{=} \mathrm{int}
(Q_{u,1}(h(R))) \quad \text{and} \\
 P_{u,2}(R_{2,\tk}) &\subset \mathrm{int} (P_{u,2}(Z_2) ) \overset{\text{(b)}}{=} \mathrm{int} (P_{u,2}(R))
 \end{split}
\end{equation}
and thus (i4) and (i1) hold. Thus inclusions (i1)-(i4) hold, proving the claim.
\end{proof}
Therefore, Claim~\ref{f.claim} imply that condition ($\sharp$1) holds for $(R', h')$, that is, 
$$
R' =
f_2\big(R_{2,\tk}\cap R \cap h^{-1}(R_{1,\ti})   \big) %\  \text{is a $(P^2_u,Q^1_u\circ h')$-rectangle where $h'=f_1 \circ h \circ f^{-1}_2$}
$$
is a $(P_{u,2},Q_{u,1}\circ h')$-rectangle.
Moreover, 
 \begin{equation}
 \label{e.tobeused}
 P_{u,2}(R') =  P_{u,2}\big(f_2(R_{2,\tk})\big) \quad \text{and} \quad (Q_{u,1}\circ h')(R') = Q_{u,1}\big(f_1(R_{1,\ti})\big).
 \end{equation}

 Note also that, by the definitions of $R'$ and $h'=f_1\circ h \circ f_2^{-1}$ it holds,
 \begin{equation}
 \label{e.reused}
 R'\subset f_2(R_{2,\tk})
\qquad \mbox{and} \qquad
h'(R') \subset  f_1 (R_{1,\ti}).
 \end{equation}
 From \eqref{e.tobeused}, \eqref{e.reused}, and (BM2) in  Definition \ref{def:cu-blender}  of blender machine, it follows
\begin{itemize}
\item[(a')]
$Q_{cs,2}(R') \overset{\eqref{e.reused}}{\subset}  Q_{cs,2}\big(f_2(R_{2,\tk})\big)
\overset{\text{(BM2)}}{\Subset} Q_{cs,2}(Z_2)$,
\item [(b')]
$P_{u,2}(R')  \overset{\eqref{e.tobeused}}{=} P_{u,2}\big(f_2(R_{2,\tk})\big)
\overset{\text{(BM2)}}{=}
 P_{u,2}(Z_2)$,
\item [(c')]
$(P_{cs,1}\circ h')(R') \overset{\eqref{e.reused}}{\subset} P_{cs,1}\big(f_1(R_{1,\ti})\big)
\overset{\text{(BM2)}}{\Subset}
 P_{cs,1}(Z_1)$, and
\item[(d')]
$(Q_{u,1}\circ h')(R') 
\overset{\eqref{e.tobeused}}{=}
 Q_{u,1}\big(f_1(R_{1,\ti})\big)
\overset{\text{(BM2)}}{=}
 Q_{u,1}(Z_1)$ 
\end{itemize}
Note that (a'),(b'), (c'), and (d') correspond, respectively, to  conditions (a), (b), (c), and (d) in  ($\sharp$3).
Therefore, condition ($\sharp$3) holds for $(R', h')$.
%\[
%  \begin{split}
%& \mathrm{(a)}\,\, Q^2_{cs}(R') \Subset Q^2_{cs}(Z_2), \qquad  \mathrm{(b)}\,\,  P^2_u(R') = P^2_u(Z_2),
%\\
%&   \mathrm{(c)}\,\,(P^1_{cs}\circ h)(R') \Subset P^1_{cs}(Z_1), \quad \mbox{and} \quad
% \mathrm{(d)}\,\,(Q^1_u\circ h')(R' ) = Q^1_u(Z_1).
%\end{split}
%\]

It remains to check condition ($\sharp$2) ($h'$ satisfies %the $(\fS^t_2,\fS_1, \theta, \mu_\sharp, \gamma_\sharp)$-cone and
  the $(\tfS^t_2,\tfS_1, \theta , \mu_\sharp,\gamma_\sharp)$-cone condition~on~$R'$).
Note that the cone conditions for $f_1$, $f_2$, and $h$ in (AT2) in Definition \ref{d.adaptedt} and ($\sharp$2)
imply that  $h'=f_1 \circ h \circ f_2^{-1}$
satisfies the $(\tfS_2^t,\tfS_1,\theta, \mu_1 \mu_{\sharp}\gamma_2,\gamma_1 \gamma_\sharp \mu_2)$-cone condition on $R'$.
As $\mu_1\gamma_2 >1$ and $\mu_2\gamma_1 > 1$
(condition (AT1)  in Definition \ref{d.adaptedt}) we have  that  %$\mu_1\gamma_2 > 1$ and $\gamma_1\mu_2 > 1$,
$h'$ satisfies the $(\tfS_2^t,\tfS_1,\theta, \mu_{\sharp},\gamma_\sharp)$-cone condition on $R'$.
This ends the proof of the lemma. \hfill \qed

%The rest of the proof of the theorem follows verbatim
%the proof in \cite[Theorem.~2.13]{Asa:22}. For completeness, we recall  this construction.

\section{Connecting preblending machines}
\label{s.connectingpreblenders}
We introduce adapted transitions for preblending machines and  show that they give rise to
two connected
 split blending machines associated with skew-product tower  maps, see Theorem~\ref{tct.conecting-preblender}.

\subsection{Transitions between preblending machines}
\label{ss.transitionspreblending}

We introduce transitions between preblender machines, following the spirit of Definition~\ref{d.adaptedt}.
Conditions (at1)-(at5) below mimic conditions (AT1)-(AT5) in that definition.

\begin{defi}[Adapted transitions between preblender machines]
\label{d.adpreblender}
Let $M$ be a manifold of dimension $d\geq 2$ and
$N_1$ and $N_2$ open subsets of $M$.
Let $d_c$ be a positive integer with 
$d_c< d$. 
For $\tau=1,2$, consider
\begin{enumerate}[label=$\bullet$, leftmargin=0.5cm]
  \item
  families $\mathcal{F}_\tau=\{{\phi_{\tau,i} } \colon  i=1,\dots, k_{\tau}\}$, $k_\tau \geq 2$, of
 $C^1$ local diffeomorphisms of $M$, and
 \item
 preblending machines $\tfB_\tau=(\mathscr{R}_\tau, B_\tau,\tfS_\tau)$  for $\mathcal{F}_\tau$,
 where  
 $\tfS_1=(P_1,Q_1)$ is a $(d_c,d-d_c)$-splitting of $N_1$ and
 $\tfS_2=(Q_2,P_2)$ is a $(d-d_c, d_c)$-splitting of $N_2$.
\end{enumerate}
Let
$R_\sharp \subset N_2$ be a compact set, $U_\sharp$ an open neighbourhood of $R_\sharp$,
and  $\phi_{\sharp}\colon U_\sharp \to N_1$ a $C^1$ embedding.
The triple $(R_\sharp, U_\sharp, \phi_\sharp)$ is a {\em{adapted transition to
the preblending machines $\tfB_1$ and $\tfB_2$}}
if there are
positive numbers $\mu_\tau$, $\gamma_\tau$, $\tau = 1, 2, \sharp$,
satisfying conditions (at1)-(at5) below:
\begin{enumerate}[label=(at\arabic*), leftmargin=1cm]
\item
 $\mu_1\gamma_2 >1$ and $\mu_2\gamma_1 > 1$,
\item
for $\tau=1,2$ the map
$\phi_{\tau,\ti}$ is $(\tfS_\tau,\mu_\tau, \gamma_\tau)$-dominated on
every set $R_{\tau,\ti}$ of  $\mathscr{R}_\tau\eqdef \{R_{\tau,\ti}\}_{\ti\in I_{\tau}}$,
  \item
 $R_\sharp$ is a $(P_2,Q_1\circ \phi_{\sharp})$-rectangle satisfying
     \begin{align*}
    &Q_2(R_\sharp) \Subset Q_2(B_2), &  &P_2(R_\sharp) = P_2(B_2), \\
    &(P_1\circ \phi_{\sharp})(R_\sharp) \Subset P_1(B_1), &   &(Q_1\circ \phi_{\sharp})(R_\sharp)=Q_1(B_1),
    \end{align*}
\item
$\phi_{\sharp}$ is
$(\tfS_2^t,\tfS^{}_1,\mu_\sharp, \gamma_\sharp)$-dominated on $R_\sharp$,
 \item
 the Lebesgue numbers $\Delta_1$ and $\Delta_2$  of $\tfB_1$ and $\tfB_2$
satisfy
$$
\mu_{\sharp}^{-1} \diam (P_2(B_2)) < {\Delta_1} \quad \mbox{and} \quad
\gamma^{-1}_\sharp \diam (Q_1(B_1)) < {\Delta_2}.
$$
\end{enumerate}
%To emphasise the role of the constants, we say that
% $(\mu_\tau, \gamma_\tau)$ 
% are the elements of $(R_\sharp, U_\sharp, \phi_\sharp)$. 
%  \lmargem{I do not understand this that is repeated above}
\end{defi}

\subsection{Connecting split blending machines through preblending machines}
\label{ss.connetcingpreblendingandskew}

The next result is a direct consequece of Theorem~\ref{t.Asaoka}.
To state it, recall Definitions~\ref{ld.affine}, ~\ref{d.stretching},
and ~\ref{d.r.towermap} of a horseshoe map,
adapted stretching, and a skew-product tower map.

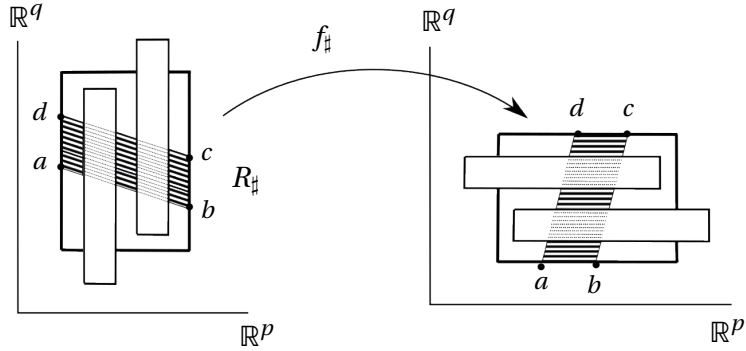
\begin{figure}[h]
\begin{overpic}[scale=0.4,
%grid,tics=3
]{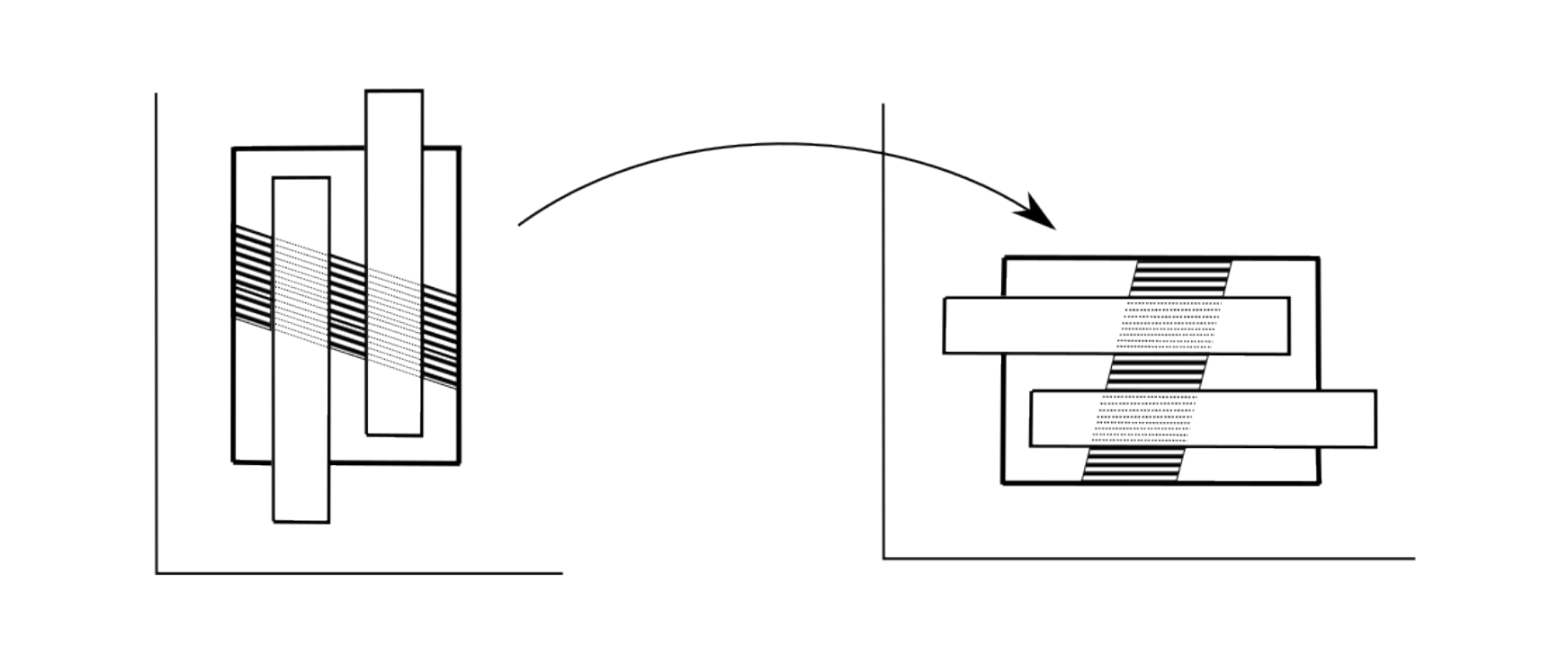}
     %\put(45,150){\Large{$f_1$}}    
    \put(35,2){\Large{$\mathbb{R}^p$}}     
    \put(9,38){\Large{$\mathbb{R}^q$}}                                    
     \put(43,36){\large{$f_{\sharp}$}}       
     \put(88,2.5){\Large{$\mathbb{R}^p$}}     
     \put(55.5,37.5){\Large{$\mathbb{R}^q$}} 
    \put(34,20){\Large{$R_{\sharp}$}}    
   \put(14.2,21.5){\small{$\bullet$}}  
     \put(11.5,22){$a$}    
      \put(14.2,27.2){\small{$\bullet$}}  
        \put(11.5,27.5){$d$} 
     \put(28.7,17){\small{$\bullet$}}  
     \put(30.5,16.5){$b$}   
     \put(28.7,22.5){\small{$\bullet$}}   
     \put(30.5,23){$c$}      
     \put(68.3,10.2){\small{$\bullet$}}  
     \put(68,8.5){$a$}    
     \put(74.5,10.5){\small{$\bullet$}}  
     \put(74,8){$b$}    
     \put(78,25.2){\small{$\bullet$}}  
     \put(78,28){$c$}   
     \put(72.5,25.2){\small{$\bullet$}}  
     \put(72,28){$d$}      
 \end{overpic}
\vspace{-0.4cm}
\caption{Connected preblending machines}
\label{fig:AsaOka}
\end{figure}

%\smargem{este teorema vale Cr}
\begin{thm} \label{tct.conecting-preblender}
Let $M$ be a manifold of dimension $d\geq 2$ and
$N_1$ and $N_2$ open subsets of $M$.
Let $d_c$ be a positive integer with 
$d_c< d$. 
Consider:
 \begin{itemize}[leftmargin=0.5cm]
 \item
 Families  
 $\mathcal{F}_\tau=\{\phi_{\tau,i} \colon  i=1,\dots, k_{\tau}\}$,
  $\tau=1,2$ and
   $k_\tau\geq 2$,
  and $\mathcal{F}_\sharp=\{\phi_{\sharp}\}$,
  of
  $C^1$ local  diffeomorphisms of $M$.
  \item
   Preblending machines $\tfB_\tau=(\mathscr{R}_\tau, B_\tau,\tfS_\tau)$  for $\mathcal{F}_\tau$,
 $\tau=1,2$,
 where
 $\tfS_1$ is a $(d_c,d-d_c)$-splitting of $N_1$,
 $\tfS_2$ is a $(d-d_c, d_c)$-splitting of $N_2$, and  $\mathscr{R}_\tau=\{{R_{\tau,\ti}\}_{\ti \in I_\tau}}$.
  \item
  A compact set $R_\sharp$ and a neighbourhood $U_\sharp$ of it such that
 $(R_\sharp, U_\sharp, \phi_\sharp)$ is an adapted transition between $\tfB_1$ and $\tfB_2$.
\item 
Horseshoe maps
 $A_1, A_2, A_\sharp$ of type
 $(k_1,d_{s},d_{u},\nu_1)$,  $(k_2,d_{u},d_{s},\nu_2)$, and
  $(1,d_{s},d_{u},\nu_\sharp)$, whose legs are $\{\mathsf{H}_{1,\ti}\}_{\ti \in \Sigma^\ast_{k_{1}}}$, $\{\mathsf{H}_{2,\ti}\}_{\ti \in \Sigma^\ast_{k_{2}}}$, and 
  $\{\mathsf{H}_{\sharp}\}$.
  \end{itemize}
  
  Assume that 
 the stretching $\nu_1, \nu_2$  of  $A_1, A_2$
    are adapted to $\tfB_{1}$, $\tfB_{2}$ and the stretching $\nu_\sharp$ of $A_\sharp$ is $(\tfS_2^t,\tfS^{}_1)$-adapted to
 $\phi_{\sharp}$ on $R_\sharp$.

Consider  the skew-product tower maps
 $F_\tau$ associated to
 $\torT_\tau= ({A}_\tau \ltimes \mathcal{F}_\tau,
 \mathcal{R}_\tau, I_\tau)$, $\tau=1,2,\sharp$, where:
  \begin{itemize}
   \item 
   $ \mathcal{R}_\tau= \{\bfR_{\tau,\ti}\}_{\ti \in I_\tau}$, where
   $\bfR_{\tau,\ti}=\mathsf{H}_{\tau,\ti}\times R_{\tau,\ti}$, $\tau=1,2$, and 
\item
   $ \mathcal{R}_\sharp= \{\bfR_{\sharp}\}$, where
$\bfR_{\sharp}=\mathsf{H}_{\sharp}\times R_{\sharp}$,
  \end{itemize}
and
the split blending  machines
$\fB_{\tau}=  \fB_{\tfB_{\mathcal{F}_\tau}, F_\tau}$ 
given by Theorem~\ref{thm:blending-yield-blenders}.

 Then for every sufficiently small
$\theta>0$ holds
$$
   \Lambda^s(\fB_1)  \cap F_{\sharp}(\bfR_{\sharp} \cap \Lambda^s(\fB_2) )  \ne \emptyset,
   \quad \mbox{where} \quad
F_{\sharp} = A_{\sharp} \times \phi_{\sharp} \quad \mbox{and} \quad
\bfR_{\sharp}= \mathsf{H}_\sharp \times R_\sharp.
$$
Moreover, 
for every compact neighbourhood $K_\tau$ of
$
\mathbf{R}_{\torT_\tau}=
\bigcup_{\ti \in I_\tau} \bfR_{\tau,\ti}
$
there is a $C^1$ neighbourhood $\mathcal{U}_\tau$ of $F_\tau$, $\tau=1,2, \sharp$,
such that
$$
\Lambda^s(K_1,G_1) \cap G_\sharp(\Lambda^s(K_2,G_2)) \not=\emptyset,
\quad \mbox{for every $G_\tau \in \mathcal{U}_\tau, \, \tau=1,2,\sharp$.}
$$
\end{thm}

The split blending  machines $\fB_{1}$ and $\fB_{2}$  in Theorem~\ref{tct.conecting-preblender}
have dimensions $(d_s,d_c,d-d_c+d_u)$  and $(d_u,d-d_c,d_c+d_s)$, respectively.

%\textcolor{red}{
%\begin{remark}
%\label{r.ourapplication}
%We will apply Theorem~\ref{tct.conecting-preblender} in Section~\ref{s.transitionmaps} in a semi-local setting of 
%heterodimensional cycles of co-index two.
%In this case, the ingredients of the corollary are the following:
%\begin{itemize}
%\item
%A family $\mathcal{F}=\{\phi_1,\dots,\phi_k\}$ of
%$C^1$ diffeomorphisms and its inverse family $\mathcal{F}^{-1}\eqdef \{\phi_1^{-1},\dots\phi_k^{-1}\}$.
%\item
%A horseshoe map $A$ and its inverse $A^{-1}$.
%\end{itemize}
%Let
%$$
%\mathcal{F}_1=\mathcal{F}, \quad \mathcal{F}_2=\mathcal{F}^{-1},\quad
%A_1=A,\quad \mbox{and} \quad A_2=A^{-1}.
%$$ 
%For $\tau=1,2$,
%consider the
%preblending  machines  $\tfB_\tau=(\mathscr{R}_\tau, B_\tau,\tfS_\tau)$  for $\mathcal{F}_\tau$,
%whose splittings consist of natural projections (in some local charts). Using these ingredients, we first
%define
%the corresponding tower maps and thereafter the preblending  machines. The adapted transition between 
%these preblender is given by a map $\phi_\sharp \in \langle \mathcal{F}_1 \rangle^+$ corresponding
%to a transition of the cycle. 
%\end{remark}}

\begin{proof}[Proof of Theorem~\ref{tct.conecting-preblender}]
We will apply Theorem~\ref{t.Asaoka} to a pair of split blending machines
with dimensions $(p-d_{c,1}, d_{c,1},q)$ and $(q-d_{c,2}, d_{c,2},p)$
where
\begin{itemize}
\item
$p=d_s+d_c$, $q=d_u+d-d_c$, $d_{c,1}=d_c$, and $d_{c,2}=d-d_{c}$,
\item
the standard projections
\begin{equation}\label{eq.defPQ}
P \colon \mathbb{R}^p=\mathbb{R}^{d_c} \times \mathbb{R}^{d_s} \to \mathbb{R}^{d_{c}}=
\mathbb{R}^{d_{c,1}} \ \ \mbox{and} \   \
Q \colon \mathbb{R}^q=\mathbb{R}^{d-d_c} \times \mathbb{R}^{d_u} \to \mathbb{R}^{d-d_c}=
\mathbb{R}^{d_{c,2}},
\end{equation}
\item
 $\mathbf{N}_1 =\mathbb{R}^{d_{s}+ d_u} \times N_1$
 and
 $\mathbf{N}_2 =\mathbb{R}^{d_{s}+d_u} \times N_2$,
 \end{itemize}
 to the split blending machines
  $\fB_{\tau}=(\mathcal{R}_\tau,Z_\tau,\fS_\tau, \fC_\tau,\theta)$ for the tower maps
  $F_\tau\colon \bigcup_{R\in \mathcal{R}_\tau} R \to \mathbf{N}_\tau$, $\tau=1,2,$
    provided by
  Theorem~\ref{thm:blending-yield-blenders},
  the  map
  $F_\sharp \colon \mathbf{R}_{\sharp} \to \mathbf{N}_1,$
  the set $\mathbf{R}_{\sharp}\subset \mathbf{N}_2$, and the positive numbers
  $\mu_\tau, \gamma_\tau$, $\tau=1,2,\sharp$.

Let us now recall the elements
 %definitions of
%$ Z_\tau,\fS_\tau, \fC_\tau$
  of the split blending machines
  $\fB_{\tau}$:
 \begin{itemize}
 \item
 $Z_\tau=\mathsf{D}\times B_\tau$,
 $\tau=1,2$,
where $\mathsf{D}=[0,1]^{d_s+d_u}$.
\item
  The splitting $\fS_1=({P}_{cs,1},Q_{u,1})$ of  $\mathbf{N}_1$
  is given (in coordinates) by
  $$
  {P}_{cs,1}= (\Pi_{s},P_1) \quad \mbox{and} \quad  {Q}_{u,1}= (Q_1,\Pi_{u}),
  $$
  where $\Pi_{s}$ and $\Pi_{u}$  are the canonical projections
 from $\mathbb{R}^{d_{s}+d_u}=\mathbb{R}^{d_{s}} \times \mathbb{R}^{d_{u}}$
   into $\mathbb{R}^{d_{s}}$ and $\mathbb{R}^{d_{u}}$.\footnote{There is a slight abuse of notation here, as the domain of the coordinate maps are different. We use simplified notation similar to the one in \eqref{eq.nopair} for the sake of brevity. }

% \textcolor{red}{This splitting has dimensions $(d_{s}+p,q+d_{u})$.} \fmargem{sobra?}
 \item
  The splitting
 $\fS_2=({Q}_{cs,2},P_{u,2})$
 of $\mathbf{N}_2$ is given (in coordinates) by
$$
{Q}_{cs,2}= (\Pi_{u},Q_2) \quad \mbox{and} \quad {P}_{u,2}= (P_2, \Pi_{s}).
$$
%  where $\Pi_{s,2}$ and $\Pi_{u,2}$  are the canonical projections
% from $\mathbb{R}^{d_{s}+d_u}=\mathbb{R}^{d_{s}} \times \mathbb{R}^{d_{u}}$
%   into $\mathbb{R}^{d_{s}}$ and $\mathbb{R}^{d_{u}}$.
%\textcolor{red}{and has  dimensions
%$(d_{u}+q,p+d_{s})$, } \fmargem{sobra?}
\item
The splittings
$\fC_1=(\pi_s, P)$ and $\fC_2 =(\pi_{u}, Q)$, where
$\pi_{s}$ and $\pi_{u}$
    are the standard projections
\[
\pi_{s}\colon\mathbb{R}^{d_{s}}\times \mathbb{R}^{d_c} \to  \mathbb{R}^{d_s}  \qquad
\mbox{and}
\qquad
\pi_{u}\colon\mathbb{R}^{d_{u}}\times \mathbb{R}^{d-d_c} \to  \mathbb{R}^{d_u}.
\]
\end{itemize}

We now are ready to check that conditions (AT1)-(AT5) hold. Note that (AT1) and (at1) are the same.

To obtain (AT2) consider the pairs  $\widehat{\tfS}_1$ and $\widehat{\tfS}_2$
given  by
%$\tfS_1 =(\mathscr{P}_1,\mathscr{Q}_1)$ and $\tfS_2=(\mathscr{Q}_2,\mathscr{P}_2)$
%where
%$$ \mathscr{P}_1 = P_c\circ P^1_{cs}$$
$$
\widehat{\tfS}_1 =(P\circ P_{cs,1},Q \circ Q_{u,1}) \quad \text{and} \quad
\widehat{\tfS}_2 = (Q \circ Q_{cs,2},P \circ P_{u,2}).
$$
%\textcolor{red}{With the notation in Theorem~\ref{thm:blending-yield-blenders}, we have $P=\Pi_p$ and $Q=\Pi_q$.}
Denote by $\varpi_\tau \colon \mathbf{N}_\tau \to  N_\tau$ the standard projection and note that
$\widehat{\tfS}_\tau = {\tfS}_\tau \circ \varpi_\tau$, $\tau=1,2$.  %\quad \text{and} \quad  \widehat{\tfS}_2 = {\tfS}_2 \circ \pi
%where  is the standard projection.
From this relation and  $\varpi_\tau \circ  (F_\tau|_{\mathbf{R}_{\tau,\ti}}) = (\phi_{\tau,\ti}|_{R_{\tau,\ti}}) \circ \varpi_\tau$,
 it follows from (at2) that $F_\tau$ satisfies the
$(\widehat{\tfS}_\tau,\theta,\mu_\tau,\gamma_\tau)$-cone condition for $\tau=1,2$. Hence (AT2) holds.
Similarly, we get that (at4) implies (AT4).

To get (AT5) note that from the definition of $P$ in \eqref{eq.defPQ} it follows that
%where $\mathsf{D}=[0,1]^{d_{ss}}\times [0,1]^{d_{uu}}$,
$$
P \circ P_{u,2}(Z_2)= P \big(P_2(B_2), \Pi_{s} (\mathsf{D}_{2})\big)= P_2(B_2).
$$
Similarly,
$Q \circ Q_{u,1}(Z_1)=
% Q_c (P_{uu} (D), Q_1(B_1))=
Q_1(B_1)$,
 it follows from (at5)  that
\begin{align*}
 \mu_{\sharp}^{-1} \diam (P\circ P_{u,2})(Z_2) &= \mu^{-1}_{\sharp} \diam \big(P_2(B_2)\big) < \Delta_1, \\
\gamma^{-1}_\sharp \diam (Q\circ Q_{u,1})(Z_1) &= \gamma_\sharp^{-1} \diam \big(Q_1(B_1)\big) <  \Delta_2.
\end{align*}
Therefore condition  (AT5) follows for  every $\theta>0$ small enough.

It remains to show (AT3), for that recall the notation in Section~\ref{ss.affine} for horseshoe maps.
Observe that $\mathsf{H}_\sharp$ is a leg of the
 $(1,d_{s},d_{u},\nu_\sharp)$-horseshoe map $A_\sharp$, therefore
\begin{equation}
\label{Rsharp}
\begin{split}
&  \mathsf{H}_{\sharp}= [0,1]^{d_{s}}\times V_{\sharp}^{u},  \quad \text{where  $V_{\sharp}^{u}\Subset [0,1]^{d_{u}}$,}\\
 &   \mathsf{V}_{\sharp}=A_{\sharp}(\mathsf{H}_{\sharp})= H_{\sharp}^{s} \times [0,1]^{d_{u}},
    \quad \text{where $H_{\sharp}^{s}\Subset [0,1]^{d_{s}}$}.
    \end{split}
\end{equation}
Recalling that $\mathbf{R}_{\sharp}=\mathsf{H}_{\sharp} \times R_{\sharp}$, it follows that
$F_\sharp(\mathbf{R}_{\sharp})= \mathsf{V}_{\sharp} \times \phi_{\sharp}(R_{\sharp})$.
Using (at3) we get the following relations:
%\lmargem{el $\Subset$ es un igual...}
\[
 \begin{split}
 {Q}_{cs,2}(\mathbf{R}_{\sharp}) &= V_{\sharp}^{u} \times Q_2(R_{\sharp})  \overset{\text{(at3)}}{\Subset}  [0,1]^{d_{u}} \times Q_2(B_2) = {Q}_{cs,2}(Z_2), \\
  P_{u,2}(\mathbf{R}_{\sharp}) &= P_2(R_\sharp) \times [0,1]^{d_{s}} \overset{\text{(at3)}}{=}  P_2(B_2) \times [0,1]^{d_{s}} = P_{u,2}(Z_2), \\
 (P_{cs,1}\circ F_\sharp)(\mathbf{R}_{\sharp}) &= H_{\sharp}^{s} \times (P_1\circ \phi_{\sharp})(R_{\sharp})  \overset{\text{(at3)}}{\Subset}
  [0,1]^{d_{s}} \times P_1(B_1) = P_{cs,1}(Z_1), \\
(Q_{u,1}\circ F_\sharp)(\mathbf{R}_{\sharp}) &=  (Q_1 \circ \phi_{\sharp, \tum})(R_{\sharp})
\times [0,1]^{d_{u}} = \, Q_1(B_1) \times [0,1]^{d_{u}}
= {Q}_{u,1}(Z_1).
 \end{split}
\]
As $R_{\sharp}$ is a $(P_2,Q_1\circ \phi_{\sharp})$-rectangle, it  follows from
the equation above that $\mathbf{R}_{\sharp}$ is a $({P}_{u,2},{Q}_{u,1}\circ F_\sharp)$-rectangle
satisfying the claimed properties. This proves that (AT3) holds, concluding the proof of the theorem.
\end{proof}

\section{Proof of Theorem~\ref{t.robustcycles}}
\label{s.transitionmaps}

In view of Lemmas~\ref{l.everyneistypeI} and \ref{l.simple}, it is enough to prove Theorem~\ref{t.robustcycles} for diffeomorphisms
$f$ having 
simple contours $\Upsilon$.  Theorem~\ref{tp.finallyarobustcycle}
provides  a pair of robust cycles of type $(\iota (\Upsilon) , \iota (\Upsilon)+1)$ and $(\iota (\Upsilon)+1, \iota (\Upsilon)+2)$
for a perturbation $g$ of $f$ preserving the contour
was obtained. Thus it remains to obtain robust cycles of co-index two.

\begin{remark}[The hyperbolic sets in the robust cycle]
\label{r.preparacion}
Consider the points $\qA_{\ti_k}$ and $\qD'_{\tj_k}$, $k=2,3$, in Proposition~\ref{p.muitos}, and the corresponding periodic points
(that we assume to be  hyperbolic) given by Remark~\ref{r.secondarydynamics},
$A_k = A_{({\ti_k}^\mathbb{Z}, \qA_{\ti_k})}$ and $D'_k = D'_{({\tj_k}^\mathbb{Z}, \qD'_{\tj_k})}$
for a diffeomorphism $g$ with a simple contour. 
We construct a pair of hyperbolic sets $\Lambda_\tq$ and $\Lambda_\tp$ with $\ut$-indices $\iota(\Upsilon)$ and $\iota(\Upsilon)+2$, such that 
\begin{equation}
\label{e.preparation}
A_k \in \Lambda_\tq, \,\, D'_k \in \Lambda_\tp, 
\quad \mbox{thus, by Remark~\ref{r.uniformsizeofinvariantmanifolds},}
\quad
 W^\st (\Lambda_\tq,g) \pitchfork W^\ut (\Lambda_\tp,g)\ne\emptyset. 
 \end{equation}
 Thus, to prove the theorem, 
 it remains 
to obtain perturbations $h$ of $g$ such that
$$
W^\ut (\Lambda_\tq,h) \cap W^\st (\Lambda_\tp,h)\ne\emptyset
$$ 
 is robust one (see Section~\ref{ss.robustcycle1}).
The construction of the sets $\Lambda_\tq$ and $\Lambda_\tp$
relies on the preblending machines
constructed in Section~\ref{ss.connectingpreblendingmachines}.
\end{remark}

\subsection{Connecting expanding preblending machines} \label{ss.connectingpreblendingmachines}
Consider the family 
$\mathfrak{F}=\{f_{\tx} \colon \tx \in \mathcal{C}\}$ of local maps of $f$ in \eqref{e.thefamilyF} and
its quotient family $\mathcal{F}$ in \eqref{e.thequotientfamily}.
Proposition~\ref{p.t.occonpreblenders-vA} below 
provides an adapted perturbation 
$g$ of $f$ (with respect to $\mathfrak{F}$) 
whose family of local maps $\mathcal{G}$
have a
 pair of expanding preblending machines (Definition~\ref{def:cu-blending}), 
one for $\mathcal{G}$ and the other for $\mathcal{G}^{-1}$, that are ``connected'' in the sense below.
Recall also the adapted transitions between preblending machines in Definition~\ref{d.adpreblender}.

%
%
%
%\begin{thm}
%\label{trep.robustcycles}
%Consider
% a family 
% $\mathcal{F}$ as in \eqref{e.thequotientfamily}
% satisfying conditions (QD1), (QD2), (QD3) in Section~\ref{ss.associatedquotient}.
%Then for every $\varepsilon > 0$  there is 
%a family 
% $\mathcal{G}$
%which is
%an $\varepsilon$ $C^1$ perturbation  
% $\mathcal{F}$ such that
%the 
%skew product $F_{\mathcal{G}}$ associated to $\mathcal{G}$ in~\eqref{e.theskewproductfordiffeo} has a  $C^1$
%robust heterodimensional cycle of co-index two and a pair of cycles of co-index one
%associated to \vmargem{completar}
%
%
%Moreover, if $\mathcal{F}$ is a quotient family of a diffeomorphism $f$ with a simple cycle
%then there is an adapted perturbation $g$ of $f$ whose associated 
%family is $\mathcal{G}$.
%\end{thm}

\begin{defi}[Connected preblending machines]
\label{d.connectedpreblendermachines}
Let $\mathcal{G}_1=\mathcal{G}$, $\mathcal{G}_2=\mathcal{G}^{-1}$
with expanding
preblending machines $\tfB_\tau=(\scR_\tau, B_\tau,\tfS_\tau)$,
$\tau=1,2$, defined on disjoint open sets $N_1, N_2$. We say that 
$\tfB_1$ and $\tfB_2$ are {\em{connected}} if there are $\sharp \in \Sigma_T^\ast$, a
 compact set $R_\sharp$, and an open
set  $U_\sharp$ such that
 $R_\sharp \subset U_\sharp\subset N_2$ and $(R_\sharp, U_\sharp, \psi_\sharp)$ is an adapted transition
from $\tfB_2$ to $\tfB_1$.
 \end{defi}

\begin{prop}
\label{p.t.occonpreblenders-vA}
Consider a diffeomorphism $f$ with a simple contour
and its quotient family $\mathcal{F}$.
Then  there exist an arbitrarily small adapted perturbation $g$ of $f$ preserving the contour
whose quotient family
 $\mathcal{G}$
has
expanding preblending machines $ \tfB_1$
 (for $\mathcal{G}$)
  and $\tfB_2$ (for $\mathcal{G}^{-1}$) 
  which are connected. 
  \end{prop}

The proof of this proposition has two steps: the construction of the preblending machines (Section~\ref{ss.preblenders})  and 
their connection (Section~\ref{sss.choiceofparameters}).

\subsubsection{Choice of the preblending machines}
\label{ss.preblenders}

Associated to $f$, consider the quotient dynamics in Section~\ref{ss.associatedquotient}.
For simplicity, we will assume that the numbers $\alpha_1,\alpha_2, \beta_1, \beta_2$ 
in~\eqref{e.fzero} are all positive.

%\subsubsection{Existence of preblending machines}
%\label{sss.existence}
The preblending machines are $\tfB_1=\tfB_{\tp}=(\scR_\tp,B_\tp, \tfS_\tp)$ and
$\tfB_2=\tfB_{\tq}=(\scR_\tq,B_\tq, \tfS_\tq)$
are associated with a perturbation $\mathcal{G}$
 of $\mathcal{F}$ and defined  as follows. 
%\begin{itemize}
%\item
%$\tfB_{\tq}$  is associated to $\mathcal{G}$ and the sets $\qB_\tq$ and $\scR_\tq$
%are contained in
%$\qU_\tq$,
%\item
%$\tfB_{\tp}$ is associated to $\mathcal{G}^{-1}$ and the sets $\qB_\tp$  and $\scR_\tp$
%are contained in
%$\qU_\tp$.
%\end{itemize}
First, $\tfS_\tq,\tfS_\tp$ are the projections
given by
\begin{equation}
\label{e.PQproj}
\tfS_\yr=( P_\yr, Q_\yr),\quad
   P_\yr,
   Q_\yr \colon \qU_\yr \to \mathbb{R}, \quad P_\yr(x,y)=x,\quad
   Q_\yr(x,y)=y, \quad \yr =\tq, \tp.
\end{equation}
The other elements  are obtained from the rectangles in
Section~\ref{ss.constructionofrectangles} as follows. 
Consider
$\vartheta_0,\delta_0>0$,
 the adapted perturbation
$\mathcal{G}$ of $\mathcal{F}$, the words $\ti_2,\ti_3, \tj_2, \tj_3$, and
the rectangles 
$$
\qH(\delta) \times \qV_{\tj_2, \tj_3}, \qquad
\qH_{\ti_2, \ti_3}  \times \qV (\vartheta),
\qquad \mbox{and} \qquad 
{\qR}_{\ti_k}, \quad  {\qR}_{\tj_k}, \quad k=2,3
$$
in  Claims~\ref{cl.pepinodecapri}
and \ref{cl.pepinodemodena}.
We let
\begin{equation}
\label{e.ingredients}
\begin{split}
&\scR_{\tp,\tj_2,\tj_3} \eqdef\{{\qR}_{\tj_2}, {\qR}_{\tj_3}\},\qquad
B_{\tp, \tj_2,\tj_3} (\delta) \eqdef \qH(\delta) \times \qV_{\tj_2, \tj_3},  \qquad
\tfS_\tp= \{P_{_\tp},Q_{_\tp}\},\\
&\scR_{\tq,\ti_2,\ti_3} \eqdef\{{\qR}_{\ti_2}, {\qR}_{\ti_3}\},\qquad
B_{\tq, \ti_2,\ti_3} (\vartheta) \eqdef \qH_{\ti_2, \ti_3}  \times \qV (\vartheta),  \qquad
\tfS_\tq= \{P_{_\tq},Q_{_\tq}\},
\end{split}
\end{equation}
and denote
\begin{equation}
\label{e.belndingmachinepq}
\begin{split}
\tfB_{\tp} &= \tfB_{\tp}(\delta)=
\tfB_{\tp} (\delta, \tj_2, \tj_3)\eqdef (\scR_{\tp,\tj_2,\tj_3} , B_{\tp, \tj_2,\tj_3} (\delta) , \tfS_\tp),\\
\tfB_{\tq} &= \tfB_{\tq} (\vartheta)=
\tfB_{\tq} (\vartheta, \ti_2, \ti_3)\eqdef (\scR_{\tq,\ti_2,\ti_3} , B_{\tq, \ti_2,\ti_3} (\vartheta) , \tfS_\tq).
\end{split}
\end{equation}
In Lemma~\ref{l.Ds} below, we see that 
$\tfB_{\tp}$ and $\tfB_{\tq}$
are indeed preblending machines.

With this terminology in  \eqref{e.notationforwords}, we have
$$
 \ti_k=   \ti_{(m_k, n_k)}, \qquad
 \tj_k=\tj_{(r_k,t_k)}, \qquad \mbox{where} \quad
 m_3>m_2,\, n_3>n_2,\,
 r_3>r_2,\, t_3>t_2.
 $$

\begin{lem}
\label{l.Ds}
Given the words $\tj_k=\tj_{(r_k,t_k)}$,
and
 $\ti_k=   \ti_{(m_k, n_k)}$, $k=2,3$,  in Proposition~\ref{p.muitos},
there  are $\delta_0, \vartheta_0>0$ 
such that for every $\delta \in (0,\delta_0)$ 
and every  $\vartheta \in (0,\vartheta_0)$ the following holds:
\begin{itemize}
\item 
$\tfB_{\tp}=\tfB_{\tp} (\delta, \tj_2, \tj_3)$ is an expanding preblender machine for $\mathcal{G}$
with constants $\mu_\tp \in (1,2)$ and $\gamma_\tp \approx \alpha_2^{r_2}\beta_2^{t_2}$, and Lebesgue number $\Delta_\tp \in (\delta/2, \delta)$;
\item
$\tfB_{\tq}=\tfB_{\tq} (\vartheta, \ti_2, \ti_3)$ is an expanding preblending machine for $\mathcal{G}^{-1}$
with constants $\mu_\tq\in (1,2)$ and $\gamma_\tq \approx (\alpha_1^{m_2}\beta_1^{n_2})^{-1}$, and Lebesgue number $\Delta_\tq \in (\vartheta/2, \vartheta)$.
\end{itemize}
\end{lem}

\begin{proof}
We prove the lemma for $\tfB_{\tp}$, the proof for $\tfB_{\tq}$ is similar and hence omitted.
We
need to check conditions (PB1)--(PB3) in Definition~\ref{def:cu-blending}.
To check the domination
 condition (PB1), it is enough to take
$\mu_{\tj_2}=\mu_{\tj_3} =1/2$
and
$\gamma_{\tj_2}=\gamma_{\tj_3} =\kappa$, for some $\kappa$ with 
$5 < \kappa <
 \min\{ \eta_{\tj_2}, \eta_{\tj_3}\}$, where $\eta_{\tj_k}$ is as in Claim~\ref{cl.pepinodemodena}.
Thus, condition (PB1) follows from
item (2) in Claim~\ref{cl.pepinodemodena}.

To check (PB2), note   that, for $k=2,3$ it holds
\begin{itemize}
\item ${\qR}_{\tj_k}$ is a $(P_\tp,Q_\tp \circ \psi_{\tj_k} )$-rectangle  by  construction,
\item
$Q_\tp (\qR_{\tj_k})\Subset Q_\tp (B_\tp)$
follows from  $\qV_{\tj_k} \subset \qV_{\tj_2, \tj_3} $,
\item
$(Q_\tp\circ \psi_{\tj_k})(\qR_{\tj_k}) = Q_\tp(B_\tp)$ 
follows from \ref{cl.pepinodemodena}, and
 \item
 $(P_\tp \circ  \psi_{\tj_k}) (\qR_{\tj_k}) \Subset P_\tp(B_\tq)$
 follows from Claim~\ref{cl.pepinodemodena}.  
 \end{itemize}

Condition (PB3) follows from the definitions of the elements of
$\tfB_{\tp}$ and item (1) in Claim~\ref{cl.pepinodemodena} that can be read as $\overline{P_\tp(B_\tp)} \subset
\mathrm{int} (P_\tp ({\qR}_{\tj_2})) \cup \mathrm{int} (P_\tp ({\qR}_{\tj_3}))$.

The assertion about the expanding constants follows again by Claim~\ref{cl.pepinodemodena}.

Recalling the construction of the rectangles 
$\qR_{\ti_k}$ in \eqref{e.irectangle}
and having in mind the
inclusion in item (1) in Claim~\ref{cl.pepinodemodena}, one deduces that
the Lebesgue number $\Delta_\tp$ of $\tfB_{\tp}$ satisfies $\Delta_\tp \in (\delta/2,\delta)$, recall 
Definition~\ref{def:cu-blending}, ending the
proof of the lemma.
\end{proof}

\begin{remark}
\label{r.Dnecoexistence}
In the construction  of the preblending machines above:
% (\vartheta, \ti_1, \ti_2)$ 
\begin{itemize}
\item
The fixed points of $\psi_\tq$ and $\psi_\tp$ are not modified and
the condition $\psi_\yc(0,0)=(0,0)$ (in local charts) is preserved. Thus
the cycle structure kept unchanged.
\item
The machines
%  preblending machines $\tfB_{\tp}$
%(\delta, \tj_1, \tj_2)$ 
%and  $\tfB_{\tq}$
% (\vartheta, \ti_1, \ti_2)$ 
coexist independently and
persist after perturbations, see Remark~\ref{r.robustpreblender}.
Note that $\delta$ and $\vartheta$ can be
chosen independently and arbitrarily small. The words 
$ \tj_k$ and  $\ti_k$
can be chosen independently and of arbitrarily large length.
\item
The points
$\qA_{\ti_k}$ and $\qD'_{\tj_k}$ belong to the corresponding rectangles of the machines.
\end{itemize}
\end{remark}

\subsubsection{Connecting preblending machines}
\label{sss.choiceofparameters}
To connect
 the preblending machines 
  $ \tfB_{\tp} =\tfB_{\tp} (\delta, \tj_1, \tj_2)$ and 
$ \tfB_{\tq}  =\tfB_{\tq} (\vartheta, \ti_1, \ti_2)$ in Lemma~\ref{l.Ds}
we unfold the cycle map $\psi_\yc$.

\begin{lem}
\label{l.item2cor}
 There exists an adapted perturbation $g$ of $f$ 
such that, for the corresponding  quotient family $\mathcal{G}$,
the preblending machines $\tfB_{\tq}$ and $\tfB_{\tp}$ are connected.
\end{lem}

\begin{proof}
Recalling Notation~\ref{n.notationshifts} for a prefix of a word, consider 
the orbits of the preblending machines given by:
\begin{equation}
\label{e.orbitsmachines}
\begin{split}
\mathcal{O} ( \tfB_{\tp} (\delta, \tj_2, \tj_3)) &\eqdef
\{ \psi_{\tj} (\qR_{\tj_2}), \, \mbox{$\tj$ a prefix of $\tj_2$} \}
\cup \{ \psi_{\tj} (\qR_{\tj_3}), \, \mbox{$\tj$ a prefix of $\tj_3$} \},\\
\mathcal{O} (\tfB_{\tq} (\vartheta, \ti_2, \ti_3)) &\eqdef
\{ \psi_{\ti} (\qR_{\ti_2}), \, \mbox{$\ti$ a prefix of $\ti_3$} \}
\cup \{ \psi_{\ti} (\qR_{\ti_3}), \, \mbox{$\ti$ a prefix of $\ti_3$} \}.
\end{split}
\end{equation}
These sets are compact and disjoint from $\{\qP,\qQ\}$.
This allows to
select small neighbourhoods
 $\qV_\yc$ and  $\qO_\yc$ of $\qP$ contained in $\qU_\tc\subset \qU_\tp$
which are
disjoint from the orbits
$\mathcal{O} (\tfB_\tp)$  and $\mathcal{O} (\tfB_\tq)$ 
such that $\overline \qV_\yc \subset \qO_c \subset 
 \overline{\qO_\yc} \subset
\qU_\yc$. 
Then there is small  $\tau>0$ such that for every $\bar \omega$ 
with $|\bar \omega|<\tau$
there is a  local diffeomorphism, called a {\em{translations of $\psi_\yc$,}} 
such that
\begin{equation} 
\label{e.translationofcyclemap}
\psi_{\yc,\bar \omega} \colon
\qU_{\yc} \to \qU_{\tq},\qquad
\psi_{\yc,\bar \omega }(x,y)=
\begin{cases}
&\psi_\yc (x,y)+\bar \omega , \quad \mbox{if $(x,y) \in \qV_\yc$},\\
&\psi_\yc (x,y), \quad \mbox{if $(x,y) \not\in \overline{\qO}_\yc$}.
\end{cases}
\end{equation}

To connect 
the machines  $\tfB_\tq$ and  $\tfB_\tp$, we consider compositions of the form
$$
\psi_\# \eqdef \psi_\td  \circ \psi_\tq^{w} \circ \psi_{\tc,\bar \omega_{v, w}}
\circ \psi_\tp^v  \circ \psi_\ta,
$$
for arbitrary large $w, v$ 
and a small vector $\bar \omega_{v,w}$.
We begin by choosing $w$ and $v$.
% we consider the following
%local perturbations of $\varphi_\yc$ (obtained by perturbing $f_1$ at $S$)
Fixed large $w$,
the vertical $\Delta_\tp^{\mathrm{ver}} (w)$ 
and  horizontal $\Delta_\tp^{\mathrm{hor}} (w)$ sizes 
of  $\psi_\tq^{-w} (\psi_\td^{-1} (B_\tp))$  satisfies
\begin{equation}
\label{e.tamanosp}
 \Delta_\tp^{\mathrm{ver}} (w)\approx \beta_2^{-w} |\qV_{\tj_2, \tj_3}| \qquad \mbox{and} \qquad  \Delta_\tp^{\mathrm{hor}} (w) \approx
 \beta_1^{-w} \delta.
\end{equation}
Similarly, the vertical $\Delta_\tq^{\mathrm{ver}} (v)$ 
and  horizontal $\Delta_\tq^{\mathrm{hor}} (v)$ sizes 
of  $\psi_\tp^{v} ( \psi_\ta (B_\tq))$  satisfy
\begin{equation}
\label{e.tamanosq}
 \Delta_\tq^{\mathrm{ver}}(v)\approx \alpha_2^v \vartheta 
   \qquad \mbox{and} \qquad  \Delta_\tq^{\mathrm{hor}} (v)\approx  \alpha_1^v |\qH_{\ti_2, \ti_3}|. 
\end{equation}
We will chose $v, w$ arbitrarily large such that
$$
 \Delta_\tq^{\mathrm{ver}}(v) >  \Delta_\tp^{\mathrm{ver}}(w) 
 \qquad 
 \mbox{and}
 \qquad
 \Delta_{\tq}^{\mathrm{hor}}(v) <  \Delta_\tp^{\mathrm{hor}}(w).
 $$
In view of \eqref{e.tamanosp} and \eqref{e.tamanosq},
to get these inequalities  it is enough to find large $v,w$  with
\begin{equation}\label{eq.fromabove}
\alpha_2^v \beta_2^w \geq \frac{ |\qV_{\tj_2, \tj_3}|}{\vartheta}
\qquad
\mbox{and}
\qquad
\alpha_1^v\beta_1^w \leq \frac{\delta}{ |\qH_{\ti_2, \ti_3}|}.
\end{equation}
Note that there are arbitrarily large $v, w$ as in \eqref{eq.fromabove} such that:
\begin{equation}
\label{e.uniformlybounded}
\frac{4\vartheta}{\delta}
\leq \alpha_2^v \beta_2^w \leq \frac{4\vartheta}{\delta} \alpha_2^{-2} \beta_2^2 \eqdef \kappa
\end{equation}
Therefore 
\begin{equation}
\label{e.itgoestozero}
\alpha_1^v \beta_1^w =  \left( \frac{\alpha_1}{\alpha_2}\right)^v \alpha_2^v  \beta_2^w
\left( \frac{\beta_1}{\beta_2}\right)^w \approx
 \left( \frac{\alpha_1}{\alpha_2} \right)^v \left( \frac{\beta_1}{\beta_2}\right)^w 
\kappa \to 0
 \quad \mbox{as}\quad
 v, w\to \infty.
\end{equation}
Moreover, for $v, w$ big enough and every  $|\bar \omega| \leq \tau$ it holds
\begin{equation}
\label{e.inclusiontranslation}
 \psi_\tp^{v} ( \psi_\ta (B_\tq)) \subset \qV_\yc \quad \mbox{and} \quad
 \psi_\tq^{-w} (\psi_\td^{-1} (B_\tp)) \subset \psi_{\yc} (\qO_\yc) =\psi_{\yc, \bar \omega} 
  (\qO_\yc). 
\end{equation}

Note that there are sequences $(v_k), (w_k)\to \infty$ satisfying 
\eqref{e.uniformlybounded}, \eqref{e.itgoestozero}, and \eqref{e.inclusiontranslation}. 
This choice of $v_k$ and $w_k$ implies that there are vectors $\bar \omega_k \eqdef \bar \omega_{v_k, w_k}$,
with $|| \bar \omega_{k}||\approx \beta_1^{-v_k} \alpha_2^{w_k}$,
whose associated 
translations $\psi_{\tc,k}\eqdef \psi_{\tc, \bar \omega_{v_k,w_k}}$ of
$\psi_\tc$ are such that
 $\psi_\tq^{-w_k} (\psi_\td^{-1} (B_\tp))$ and 
  $\psi_{\tc, \bar \omega_{v_k, w_k}}(\psi_\tp^{v_k} (\psi_\ta (B_\tq)))$ intersect in a Markovian way, as depicted in Figure~\ref{fig.markovianint}. Call this intersection $\widetilde \qR_{\sharp, k}$. Let 
   \begin{equation}
   \label{e.sharptodojunto}
   \begin{split}
   \qR_{\sharp,k} &\eqdef
  \psi_\ya^{-1} \circ
  \psi_\tp^{-v_k} \circ
  \psi_{\yc, k}^{-1}
    (\widetilde \qR_{\sharp, k}) \subset \qB_\tq, \\
    U_{\sharp}& \eqdef \mathrm{int}(B_\tp),\\
     \psi_{\sharp, k} & \eqdef 
\psi_\td  \circ \psi_\tq^{w_k} \circ \psi_{\tc, k } \circ \psi_\tp^{v_k}  \circ \psi_\ta.
   \end{split} 
   \end{equation}
%      \begin{equation}
%   \label{e.sharptodojunto}
%    R_{\sharp,k} \eqdef
%  \psi_\ya^{-1} \circ
%  \psi_\tp^{-v_k} \circ
%  \psi_{\yc, k}^{-1}
%    (\widetilde R_{\sharp, k}) \subset \qB_\tq, \quad
%    U_{\sharp} \eqdef \mathrm{int}(B_\tp),\quad
%     \psi_{\sharp, k}  \eqdef 
%\psi_\td  \circ \psi_\tq^{w_k} \circ \psi_{\tc, k } \circ \psi_\tp^{v_k}  \circ \psi_\ta.
%   \end{equation}

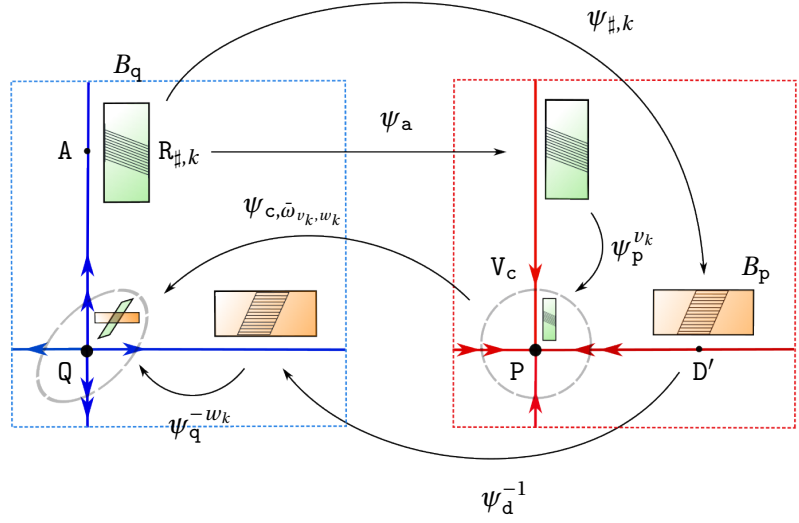
\begin{figure}[h]
\begin{overpic}[scale=0.4,
%grid,tics=5
]{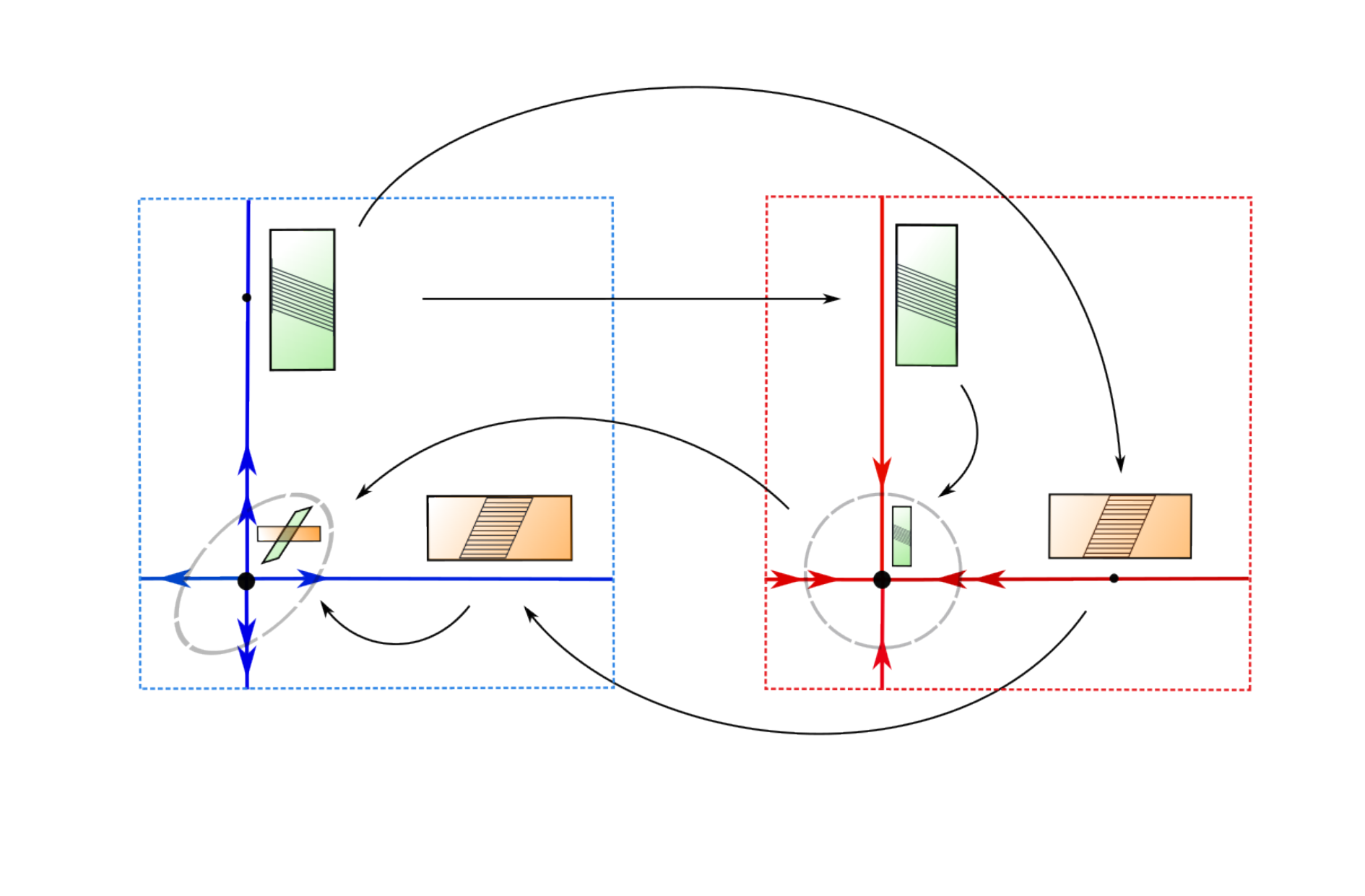}
\put(15,20){\large $\qQ$}   
\put(62,20){\large $\qP$}  
\put(15,43){\large $\qA$}   
     \put(20.7,52){\large $B_\tq$}   
         \put(48.5,46.5){\large  $\psi_{ \ya}$}   
         \put(34.3,37.5){\large  $\psi_{\tc, \bar \omega_{v_k, w_k}}$}
               \put(60,31){\large $\qV_\yc$} 
           \put(72.5,33){\large $\psi_{\tp}^{v_k}$}
\put(81,20){\large $\qD'$}   
  \put(86,31){\large $B_\tp$}  
         \put(59,7){\large  $\psi_{\yd}^{-1}$}
           \put(26.5,14.5){\large  $\psi_{\tq}^{-w_k}$}
         \put(70,57){\large $\psi_{\sharp, k}$}   
           \put(25.5,43){\large $\qR_{\sharp,k} $}              
         \end{overpic}
         \caption{Markovian intersection between $\psi_\tq^{-w_k} (\psi_\td^{-1} (B_\tp))$ and 
  $\psi_{\tc, \bar \omega_{v_k, w_k}}(\psi_\tp^{v_k} (\psi_\ta (B_\tq)))$}
         \label{fig.markovianint}
\end{figure}

The following claim implies the lemma.

\begin{claim}
\label{cl.isadated}
For every $k$ large enough,
the triple
$(\qR_{\sharp,k}, U_{\sharp},\psi_{\sharp, k})$
is an adapted transition between
$\tfB_{\tq}$ and $\tfB_{\tp}$.
\end{claim}

\begin{proof}
We check that conditions (at1)-(at5)  in Definition~\ref{d.adpreblender} hold. We will omit the subscript $k$.
Observe first that, by construction, the pair verifies condition (at3). 
Let $\mu_\tp, \gamma_\tp$ and  $\mu_\tq, \gamma_\tq$ be
the expanding constants of the preblender machines $\tfB_{\tp}$ and $\tfB_{\tq}$.
Taking 
$$
\mu_1=\mu_\tp,  \quad \gamma_1=\gamma_\tp \quad \mbox{and} \quad
 \mu_2=\mu_\tq, \quad \gamma_2=\gamma_\tq,
 $$
  it follows from
 Lemma~\ref{l.Ds} 
 that 
 $\mu_1\gamma_2>1$ and $\mu_2 \gamma_1>1$, proving (at1).
 The domination condition in (at2) follows similarly.

It remains to check conditions (at4) and (at5) hold with the constants
$$
\mu_\sharp=\mu_{\sharp,k}  \approx (\alpha_1^{v_k} \,  \beta_1^{w_k})^{-1},
\qquad
\gamma_\sharp=
\gamma_{\sharp,k} \approx
\alpha_2^{v_k} \, \beta_2^{w_k}.
$$
These choices are done
to guarantee that $\psi_{\sharp,k}$ is 
$\big((P_\tq,Q_\tq), (P_\tp,Q_\tp), \mu_{\sharp,k}, \gamma_{\sharp, k}\big)$-dominated
on $R_{\sharp,k}$, providing (at4).

 To get (at5), recall that the Lebesgue numbers $\Delta_{\tp}$ and  $\Delta_{\tq}$ of $\tfB_{\tp}(\delta)$ and  $\tfB_{\tq}(\vartheta)$ satisfies $\Delta_{\tp}\in (\delta/2,\delta)$ and 
 $\Delta_{\tq}\in (\vartheta/2,\vartheta)$, see Lemma~\ref{l.Ds}.
Moreover, 
$$
\diam P_\tp(B_\tp (\delta))=\diam (\qH(\delta))=2\delta \qquad \mbox{and} \qquad
\diam Q_\tq(B_\tq (\vartheta))= \diam (\qV(\vartheta))=2\vartheta.
$$
By \eqref{e.uniformlybounded}, the values  $\gamma_{\sharp,k}$ are uniformly bounded and by \eqref{e.itgoestozero}, 
 $\mu_{\sharp, k}^{-1}\to 0$ as $k\to \infty$. Therefore, for sufficiently large $k$,
$$
\mu_{\sharp,k}^{-1}\, \diam P_\tp(B_\tp (\delta))
\approx
\mu_{\sharp, k}^{-1}  \, 2\delta
 < {\Delta_\tq}\quad \mbox{and}\quad 
 \gamma^{-1}_{\sharp,k}\, \diam Q_\tq(B_\tq (\vartheta))
=
\gamma^{-1}_{\sharp,k} \, 2\vartheta
 < {\Delta_\tp},
$$
where the inequalities follow from \eqref{e.uniformlybounded}.
This completes the proof of the claim.
\end{proof}
The proof of the lemma is now complete.
\end{proof} 

The proof of the propostion is now complete.
\hfill \qed

\subsection{End of the proof of Theorem~\ref{t.robustcycles}}
\label{ss.robustcycle1}
Consider an adapted perturbation $g$ of $f$ as in Proposition~\ref{p.t.occonpreblenders-vA} and its quotient family 
$\mathcal{G}$  having connected expanding preblending machines $\tfB_{2}=\tfB_{\tq}$ for $\mathcal{G}^{-1}$
  and $\tfB_{1}=\tfB_{\tp}$ for $\mathcal{G}$. 
 First, by Theorem~\ref{thm:blending-yield-blenders},  the tower map
  $F$ associated to $\mathcal{G}$ in~\eqref{e.towermap} 
  has a pair of split blending machines associated with these preblending machines.
  These machines provides the hyperbolic sets
  $\Lambda_\tp$ and $\Lambda_\tq$ in Remark~\ref{r.preparacion} 
  satisfying \eqref{e.preparation}. 
  
    We now see that these split blending machines satisfy the hypotheses  Theorem~\ref{tct.conecting-preblender}, hence
  the tower map $F$  
   has a $C^1$ robust cycle of co-index to these split blending machines and  
the same holds for the sets $\Lambda_\tp$ and $\Lambda_\tq$. 
 
 We now explain the objects involved in our application of Theorem~\ref{tct.conecting-preblender},
 Given $g$ as above  and its local dynamics $\{g_\tx \colon \tx \in \mathcal{C}\}$,
   we consider the family of affine maps 
   \begin{equation}
   \label{e.thetwolegmaps}
   A_\tx=g_\tx|_{{E^\sss \oplus E^\uuu}}.
   \end{equation}
  As $g$ is an adapted perturbation, these maps are constant in the sets $U_\tx$.
 Given a word $\ti \in \Sigma_T^\ast$, we define compositions  $A_\ti$ as in \eqref{e.notation}.
     
 Associated with the preblending machines 
   \begin{equation}
  \label{e.thefinalpreblending}
  \begin{split}
  \tfB_1&\eqdef \tfB_{\tp}=(\scR_\tp, B_\tp, \tfS_\tp)=(\scR_{\tp,\tj_2,\tj_3} , B_{\tp, \tj_2,\tj_3} (\delta) , \tfS_\tp),\\
  \tfB_2&\eqdef \tfB_{\tq}=(\scR_\tq, B_\tq, \tfS_\tq)=(\scR_{\tq,\ti_2,\ti_3} , B_{\tq, \ti_2,\ti_3} (\vartheta) , \tfS_\tq),
  \end{split}
  \end{equation} 
  whose ingredients  are specified in~\eqref{e.ingredients}, 
  we consider the two legs horseshoe maps $\mathbb{A}_1$ and $\mathbb{A}_2$ whose maps
  are  $(A_{\ti_k})_{k=2,3}$ and $(A_{\tj_k})_{k=2,3}$, respectively, defined by \eqref{e.thetwolegmaps} on the corresponding rectangles.  
The indices of these maps are $I_1=\{\ti_2, \ti_3\}$ and $I_2=\{\tj_2, \tj_3\}$. 

We also note that $\mathbb{A}_1$ is $\tfS_\tp$-adapted to $\tfB_{\tp}$, and that $\mathbb{A}_2^{-1}$ is  $\tfS_\tq$-adapted to $\tfB_{\tq}$.
All the assertions about adapted stretching follows from domination and partial hyperbolicity.

Finally,
   consider the triple $(R_{\sharp,k},U_{\sharp,k},\psi_{\sharp,k})$ in Claim~\ref{cl.isadated}, see also~\eqref{e.sharptodojunto}.
   In what follows, $k$ is fixed and hence omitted. 
   We also consider the one-leg horseshoe map
$\mathbb{A}_\sharp=(A_\sharp)$ with index $I_\sharp=\{\sharp\}$.
   
We now detail  the ingredients in our application of  Theorem~\ref{tct.conecting-preblender}:
 \begin{enumerate}[label=$\bullet$, leftmargin=0.5cm]
 \item the open sets $N_1,N_2\subset M$ where $N_1\eqdef \mathrm{int}(U_\tp)$ and 
 $N_2\eqdef\mathrm{int}(U_\tq)$;
 \item the $(1,1)$-splitting $\tfS_1\eqdef\tfS_\tp$ and $\tfS_2\eqdef \tfS_\tq$ of $N_1$ and $N_2$;
 \item the families $\mathcal{F}_1\eqdef \mathcal{G}$, $\mathcal{F}_2\eqdef \mathcal{G}^{-1}$ and
  $\mathcal{F}_\sharp\eqdef \{\psi_{\sharp} \}$; 
     \item  
  the preblenders in  $\tfB_1$ and  $\tfB_2$ in \eqref{e.thefinalpreblending},
  \item
the adapted transition 
 $(R_\sharp, U_\sharp, \psi_\sharp)$ between $\tfB_1$ and $\tfB_2$ in~\eqref{e.sharptodojunto},
\item 
the two-legs horseshoe maps $\mathbb{A}_1$ ($\tfS_\tp$-adapted to $\tfB_{\tp}$) and 
 $\mathbb{A}_2^{-1}$ ($\tfS_\tq$-adapted to $\tfB_{\tq}$),
  \item  
the one-leg horseshoe map
 $\mathbb{A}_\sharp$, 
 with domain $\bfR_{\sharp}= \mathsf{H}_\sharp \times R_\sharp$,
 whose stretching is 
 $(\tfS_2^t,\tfS_1)$-adapted to  $\psi_{\sharp}$ on $R_{\sharp}$.
\end{enumerate}

Consider the tower maps  $F_\rho$ associated to
 $(\mathbb{A}_\rho \ltimes \mathcal{F}_\rho,
 \mathcal{R}_\rho, I_\rho)$, $\rho=1,2,\sharp$, see~\eqref{e.towermap},
and the split blending machines
$\fB_{\tau}$ associated to $F_\tau$ and $\tfB_\tau$ for $\tau=1,2$,
given by Theorem~\ref{thm:blending-yield-blenders}. Now,  by Theorem~\ref{tct.conecting-preblender},
$$
   \Lambda^\st(\fB_1, F_1)  \cap F_{\sharp}(\bfR_{\sharp} \cap \Lambda^\st(\fB_2, F_2) )  \neq \emptyset.
$$
Moreover, this intersection holds for $C^1$ neighbourhoods of $F_1$ and $F_2$. Noting that $F_2$ is associated
to a local inverse map, 
there exist a $C^1$ neighbourhood $\mathcal{U}$ of $F_{\mathcal{G}}$ such that for every $G\in \mathcal{U}$ we get
$$
\Lambda^\st(\fB_{1,G}, G)  \cap  \Lambda^\ut(\fB_{2,G}, G)  \neq \emptyset
$$
where $\fB_{\tau,G}$ is the continuation of $\fB_{\tau}$, $\tau=1,2$.

Finally, we observe that there exists a \( C^1 \)-neighbourhood \( \mathcal{V} \) of \( g \) and a correspondence \( h \in \mathcal{V} \to G_h \in \mathcal{U} \), such that for each \( h \in \mathcal{V} \), there exist hyperbolic sets \( \Lambda_{1,h} \) and \( \Lambda_{2,h} \) of u-indices \( \ell+2 \) and \( \ell \), respectively, corresponding to the continuations of the sets \( \mathfrak{B}_{1,G_h} \) and \( \mathfrak{B}_{2,G_h} \)This complete the proof of the theorem.

\section{Examples of non-escaping heterodimensional cycles} 
\label{s.examples}
%We introduce two  classes of examples exhibiting  non-escaping cycles. A first class is obtained  
%by adapting the examples of nonhyperbolic diffeomorphisms by Abraham-Smale~\cite{AbrSma:68}  and Shub~\cite{Shu:72} following the
% Manning's construction in \cite{Man:72}, that we call
% Abraham-Smale-Manning and Shub-Manning's examples.
%We also exhibit examples of 
% non-escaping cycles in the context of matrix cocycles of $\mathrm{GL}(3,\mathbb{R})$.
% These constructions are just sketched.

We present two classes of examples exhibiting non-escaping cycles.  
The first is obtained by adapting the nonhyperbolic diffeomorphisms of Abraham-Smale~\cite{AbrSma:68} and Shub\newline~\cite{Shu:72}, following 
Manning's construction~\cite{Man:72}.
The second arises in the context of matrix cocycles in $\mathrm{GL}(3,\mathbb{R})$.  
These constructions are outlined only briefly.

\subsection{Shub-Manning and Abraham-Smale-Manning's maps} 
\label{ss.shubandco}

We begin with a variation of the  {\em{double derived from Anosov (DA)}} diffeomorphism in \cite[Example 3, pag 36]{Man:72}. We closely follow the construction in \cite[Section 2]{CP20}. 

Consider a
  (linear) Anosov diffeomorphism $L\colon \mathbb{T}^2 \to \mathbb{T}^2$ 
having two  fixed points $\theta_1\neq\theta_2$, a
hyperbolic splitting,  $T\mathbb{T}^2=E^\st_L\oplus E^\ut_L$,  and the corresponding stable 
and unstable foliations ${W}^\st(\,\cdot,L)$ and ${W}^\ut(\,\cdot,L)$.
For $i=1,2$, denote by ${W}^\ut_{\pm}(\theta_i,L)$ the connected component of 
${W}^\ut(\theta_i,L)\setminus\{\theta_i\}$. 
Similarly, for the components ${W}^\st_{\pm}(\theta_i,L)$
of  ${W}^\st_L(\theta_i,L)\setminus\{\theta_i\}$. We observe that for any combination $(i,j) \in (\pm,\pm)$ it holds
\begin{equation}
\label{e.theintersection}
{W}^\ut_{i}(\theta_1,L) \pitchfork
W^\st_{j}(\theta_2,L) \neq \emptyset.
\end{equation}

The DA map $g$ is obtained by perturbing $L$ in two steps (these perturbations are not $C^1$ small).
First, we get a classical DA map $g_1$ with a nontrivial transitive hyperbolic attractor $\Lambda_1$.
This maps  is obtained by a homotopic deformation of $L$ around of the fixed point $\theta_1$,  breaking it into three fixed points $\theta_1^\ell$, $\theta_1^\ct=\theta_1$, 
and $\theta_1^r$, where $\theta_1$ is a source and $\theta_1^{\ell,r}$ are saddles. This perturbation preserves the stable foliation of $L$. The attractor of $g_1$ is $\Lambda_1=\mathbb{T}^2\setminus W^\ut(\theta_1,g_1)$ which contains $\theta_1^{\ell}$,  $\theta_1^{r}$. 

In this  construction, the map $g_1$  has a dominating splitting  
$T\mathbb{T}^2=E^\ct_{g_1}\oplus E^\ut_{g_1}$, where $E^\ct_{g_1}=E^\st_L$, and 
there are invariant foliations 
$W^\ct(\,\cdot, g_1)$ and $W^\ut(\,\cdot, g_1)$  tangent to $E^\ct_{g_1}$ and $E^\ut_{g_1}$,
respectively, such that $W^\ut(\,\cdot, g_1)$ is uniformly expanding and $W^\ct(\,\cdot, g_1)=W^\st(\,\cdot, L)$. 
This construction implies that equation \eqref{e.theintersection} reads as follows
\begin{equation}
\label{e.otherintersection1}
W^\ut_{i}(\theta_1,g_1) \pitchfork
W^\ct_{j}(\theta_2,g_1) \neq \emptyset.
\end{equation}
Note that $W^\ut_{\pm}(\theta_1,g_1)$ are the separatrices of the strong unstable manifold of $\theta_1$ and 
$W^\ct_{\pm}(\theta_2,g_1)$ are the separatrices of the stable manifold of $\theta_2$.

 The map $g$ is obtained from $g_1$, in a similar way, by a homotopic deformation 
around of the fixed point $\theta_2$ giving rise to three fixed points $\theta_2^\ell$, $\theta_2^\ct=\theta_2$, where
$\theta_2$ is a sink and $\theta_2^{\ell,r}$ are saddles. 
As above, this perturbation preserves the unstable foliation of
$g_1$. Therefore there exists an invariant (dominated) splitting
$T\mathbb{T}^2=E^{\ct_1}_{g}\oplus E^{\ct_2}_{g}$ and invariant foliations 
 $W^{\ct_1}(\,\cdot, g)$ and $W^{\ct_2}(\,\cdot, g)$  tangent to $E^{\ct_1}_{g}$ and $E^{\ct_2}_{g}$.
Equation \eqref{e.otherintersection1} now reads as follows
\begin{equation}
\label{e.otherintersection2}
W^{\ct_2}_{i}(\theta_1,g) \pitchfork
W^{\ct_1}_{j}(\theta_2,g) \neq \emptyset.
\end{equation}
Note that $W^{\ct_2}_{\pm}(\theta_1,g)$ and $W^{\ct_1}_{\pm}(\theta_2,g)$ are the separatrices of the strong unstable and stable manifolds of $\theta_1$ and 
  $\theta_2$, respectively.

Moreover, any diffeomorphisms $C^1$ close to $g$ satisfies the corresponding equation~\eqref{e.otherintersection2}.
This ends the definition of the DA map $g$.

Consider now an Axiom A diffeomorphism
$h \colon M \to M$ of a compact manifold $M$ with two different fixed points $p$ and $q$ 
in a basic set $\Lambda$ of its spectral decomposition.
Consider the skew-product 
\begin{equation}\label{e.19}
\Phi \colon M\times \mathbb{T}^2\to M \times \mathbb{T}^2,
\quad
\Phi (x,y)=(h(x), f_x(y)),
\end{equation}
where $(f_x)_{x\in M}$ is a smooth family 
 of diffeomorphisms such that
\begin{equation}
\label{e.family}
f_x \colon  \mathbb{T}^2 \to \mathbb{T}^2, \qquad
f_x = \left\{
\begin{array}{ll}
L, & \quad \hbox{if $x \notin B_{\varrho}(q)$,} \\
g, & \quad \hbox{if $x \in B_{\varrho/2}(q)$,}
\end{array}
\right.
\end{equation}
where $\varrho>0$ is small and $B_\varrho(q)$ denotes the open ball centred at $q$ with radius $\varrho$.
On the set $B_{\varrho}(q)\setminus B_{\varrho/2}(q)$, the family $f_x$ is an homotopic deformation between
the maps
$L$ to $g$
with one-dimensional foliations tangent to a dominated splitting, see Figure~\ref{fig:Sh-M}.

\begin{figure}[h]
\centering
\begin{overpic}[scale=.35,
%grid,tics=5
]{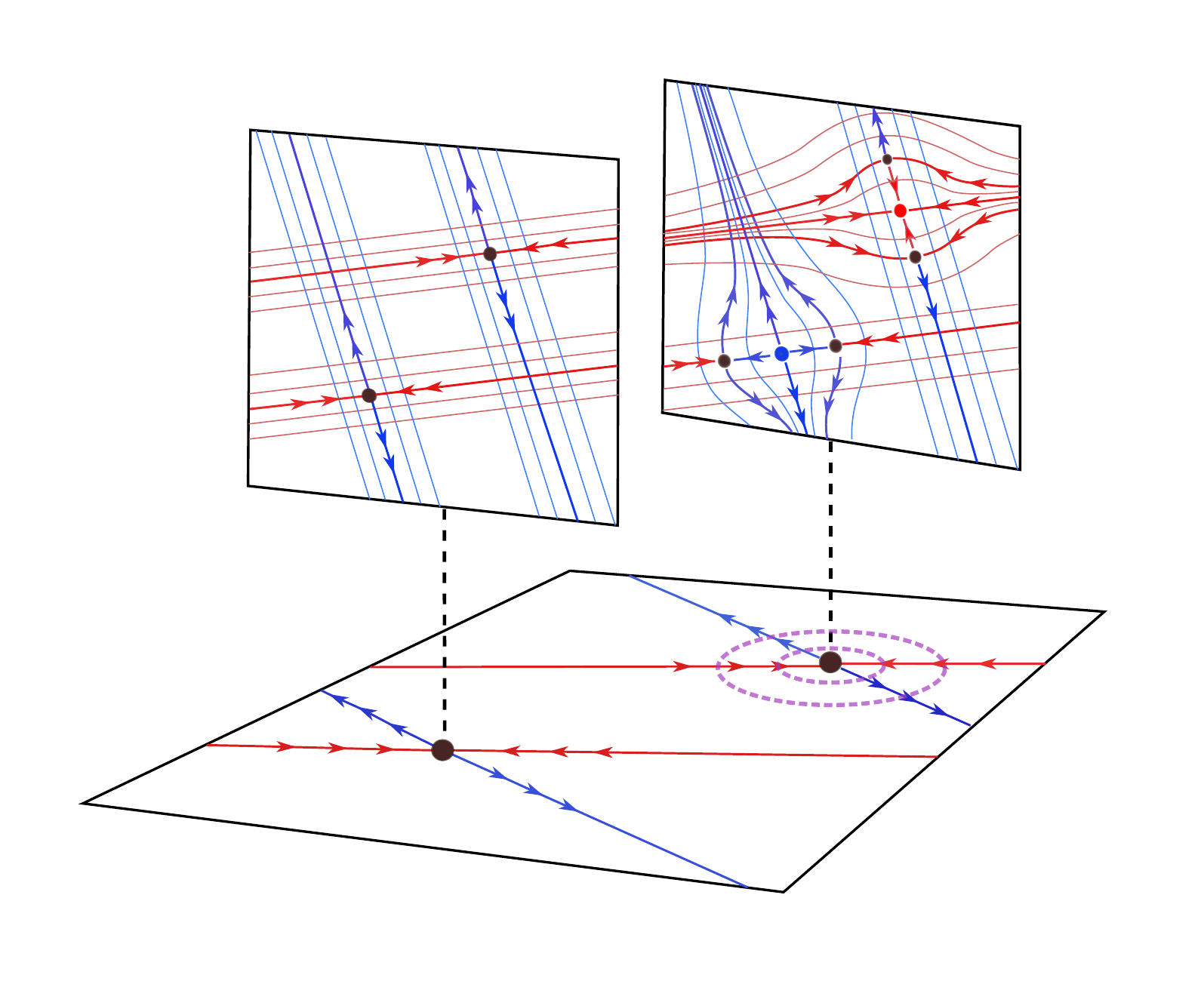}
            \put(30,73){\large{Anosov}}
           \put(26,42){{$\theta_1$}}
           \put(43,65){{$\theta_2$}}
           \put(35.5,13){{$p$}}
           \put(68,21){$q$}
          % \put(100,275){$f_p=L$}
           %\put(250,305){$f_q=M$}
           \put(60,48){{$\theta_1$}}
           \put(77,68){{$\theta_2$}}
           \put(67,76){\large{DA}}
\end{overpic}
\caption{Homotopic deformation from $\sigma\times L$ to $\Phi$}
\label{fig:Sh-M}
\end{figure}
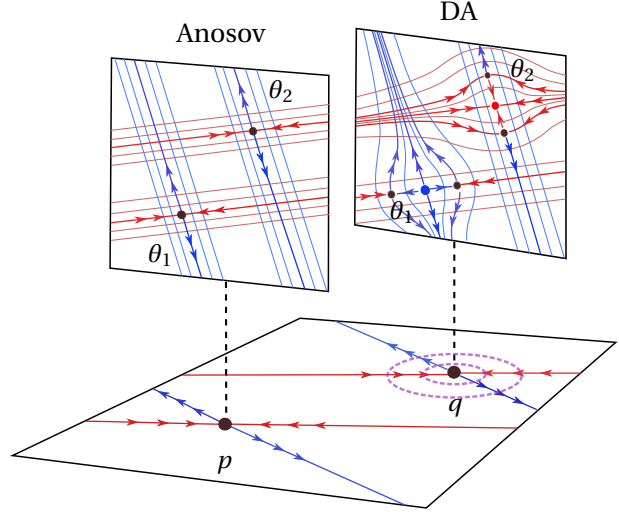
  We assume that the contraction/expansion in the fibers are dominated by ones on the base, 
 as  a consequence the set $\Lambda\times\mathbb{T}^2$ is a partially hyperbolic 
with a splitting $T_{\Lambda\times \mathbb{T}^2} M\times \mathbb{T}^2=E^\st\oplus E^{\ct_1} \oplus E^{\ct_2} \oplus E^\ut$,
%\begin{equation}\label{e.SMdominated}
%T_{\Lambda\times \mathbb{T}^2} M\times \mathbb{T}^2=E^\st\oplus E^{\ct_1} \oplus E^{\ct_2} \oplus E^\ut, 
%\end{equation}
where
the fibers of the bundle
 $E^\ct = E^\ct_1 \oplus E^\ct_2$ are tangent to the normally hyperbolic invariant foliation
 $W^\ct(x,y) \eqdef \{x\}\times \mathbb{T}^2$, $x\in \mathbb{T}^2$. Moreover, each central leaf is sub foliated by invariant foliations  
 $W^{\ct_1}, W^{\ct_2}$  tangent to  $E^{\ct_1}, E^{\ct_2}$, respectively.

 As a consequence, there exists a $C^1$ neighbourhood $\mathcal{U}$ of $\Phi$ such that every map in $\mathcal{U}$  is conjugate to skew-product  diffeomorphisms as in ~\eqref{e.19} with the same basis $h$ and fiber dynamics $\tilde{f}_x$ $C^1$ close to $f_x$, see~\cite[Chapter 8]{HirPugShu:77}. Moreover, $\Phi$ (restricted to $\Lambda\times\mathbb{T}^2$) is robustly transitive, see \cite[Proposition 7.2]{BonDiaVia:05}.
  In what follows, we use the perturbations in $\mathcal{U}$ and the associated skew-products interchangeably.

Consider the fixed points of saddle type
$Q \eqdef \{(q,\theta_1)\}$ and   $P \eqdef \{(q,\theta_2)\}$
of $\ut$-indices three and one, respectively. We claim that, after a small perturbation, these points form a non-escaping cycle.

By robust transitivity, and after an arbitrarily small perturbation, the points \( Q \) and \( P \) form a cycle.  
This is an standard consequence of the Connecting Lemma in \cite{Hay:97}.
Condition~\eqref{e.otherintersection2}, involving the separatrices of the fixed saddles \( \theta_1 \) and \( \theta_2 \) of the map \( g \), ensures that items (NE1)-(NE2) in Definition~\ref{d.nonescapinghc} are satisfied for the intersection \( W^\ut(Q, \Phi) \pitchfork W^\st(P, \Phi) \). These separatrices behaves as  those of the map \( f_0 \) in the toy model 
in Section~\ref{ss.toynonescaping}; see Figure~\ref{fig.cases-exp}.)

The previous construction applies to the \textit{Shub-Manning maps}, 
where 
 $h$ in \eqref{e.19} is a  (transitive) Anosov map in  $\mathbb{T}^2$.
It also applies to the \textit{Abraham-Smale-Manning maps},
taking $h\colon \mathbb{S}^2 \to \mathbb{S}^2$ a (global) horseshoe map and $\Lambda$ is a horseshoe\footnote{This map can be defined as a \textit{one-step skew product}. Indeed, if $R_1$,$R_2\subset M$ is a Markov partition of $\Lambda$, such that $p\in R_1$ and $q\in R_2$, then replacing $B_{\varrho}(q)$ in~\eqref{e.family} by a small neighbourhood of $R_2$ disjoint from $R_1$, the restriction of $\Phi$ to $\Lambda\times \mathbb{T}^2$  is a one-step skew product.}.

% 
%\subsubsection{Occurrence of non-escaping cycles}
%For the diffeomorphism $\Phi$ in~\eqref{e.19}, consider the hyperbolic sets
%$$
%Q \eqdef \{(q,\theta_1)\}, \quad 
%\Lambda\eqdef \{p\}\times \mathbb{T}^2,
%\quad
%  \mbox{and} \quad 
% P \eqdef  \{(q,\theta_2)\}
% $$ 
%of $\ut$-indices three, two, and one, respectively. 
%The sets 
%$\Lambda$ and $Q$ form a
%$C^1$ robust heterodimensional cycle of co-index one, see~\cite{AbrSma:68}. 
%Similarly, for $P$ and $\Lambda$. In particular, $Q$, $\Lambda$, and $P$ form a ``chain'' of co-index one heterodimensional cycles. Thus, performing a $C^1$ perturbation if necessary, we can assume that $P$ and $Q$ form a cycle of co-index two.

%\begin{claim}
%\label{cl.nonescaping}
%The saddles $P$ and $Q$ form a non-escaping heterodimensional cycle.
%\end{claim}
%
%\begin{proof}
%The domination condition follows from the fact that $\Phi$ has a dominated splitting as in 
%\eqref{e.SMdominated}. We observe also that $P$ and $Q$ are in the same central fixed fiber $\{q\}\times \mathbb{T}^2$. Thus, 
%we can identify the restriction of  $\Phi$ to $\{q\}\times \mathbb{T}^2$ with the double derived from Anosov map $g$ and the saddles $P$ and $Q$ with $\theta_2$ and $\theta_1$, respectively. In particular,  (ESF) is satisfied.
% Thus, equation~\eqref{e.otherintersection}  imply condition 
% 

\subsection{Non-escaping cycles arising from $\mathrm{GL}(3,\mathbb{R})$-cocycles}
\label{ss.cocyclesnonescaping}
%We see how non-escaping cycles may appear in the context of matrix cocycles. 
In \cite{AviBocYoc:10}, the set of uniformly hyperbolic $\mathrm{SL}(2,\mathbb{R})$-cocycles is characterized, and its boundary is described either by the presence of parabolic elements or by the existence of two maps in the cocycle, one of which sends the unstable direction of the first into the stable direction of the second; see \cite[Theorem 4.1]{AviBocYoc:10}.
Considering the projectivization of the cocycles,
this property can be viewed as a heterodimensional cycle
 of the corresponding
skew-product, see  \cite[Section 1.2]{DiaGelRam:22}.
Note that to obtain co-index two cycles, one needs to consider dimensions higher than two. We  
focus on $\mathrm{GL}(3,\mathbb{R})$-cocycles (with $\mathrm{SL}(3, \mathbb{R})$ being a particular case).
%Our goal is not generality but to present a setting for non-escaping cycles.

%
%Such heterodimensional cycles appear in many nonhyperbolic contexts. One yet less explored context is, for instance, the study of ${\rm SL}(2,\bR)$-matrix cocycles, following the approach in \cite[Section 11]{DiaGelRam:19}. To see how such a plug appears, consider the projective action of two $2\times 2$ real matrices, one of them hyperbolic giving rise to a map similar to $f_0$ and another one producing $f_1$. The case when there is some (admissible) concatenation sending $1$ to $0$, when a heterodimensional cycle occurs, is precisely  the situation studied in the boundary case in \cite[Theorem 4.1]{AviBocYoc:10}. Our analysis includes such boundary situations, but also goes beyond.
%
%

  Consider $A,B\in \mathrm{GL}(3,\mathbb{R})$ such that $A$ has  
  basis of unitary eigenvectors $\ve_1,\ve_2,\ve_3$ with
  \begin{equation}
  \label{e.cyeigen}
%  \begin{split}
%&  
A(\ve_i)= \lambda_i \ve_i,  \quad \mbox{where} \quad 0<\lambda_1 < \lambda_2< \lambda_3,
\qquad \mbox{and}   \qquad 
 %&
  B(\ve_3)= \ve_1, \quad B(\ve_1)\neq \ve_3.
 %\end{split}
  \end{equation}  
  
  Given $L\in \mathrm{GL}(3,\mathbb{R})$, denoting $[\vv]$ the class of a nonzero vector
  $\vv \in  \mathbb{R}^3$
  in $ \mathbb{P}(\mathbb{R}^3)$,
  consider its projectivization 
 $$
 f_L \colon 
 \mathbb{P}(\mathbb{R}^3)\to \mathbb{P}(\mathbb{R}^3), \quad 
 f_L([\vv]) \eqdef [L(\vv)].
 %\left[ \frac{L(\vv)}{\Vert L(\vv)\Vert} \right], 
 $$ 
Condition \eqref{e.cyeigen} implies that (see Figure~\ref{fig:Cocliclos})
 \begin{itemize}
 \item
 $[\ve_3]$ is a sink of $f_A$ and has a strong stable manifold,
 \item
   $[\ve_2]$ is a saddle $f_A$, and
\item
  $[\ve_1]$ is a source of $f_A$ and has a strong unstable manifold.
\end{itemize}

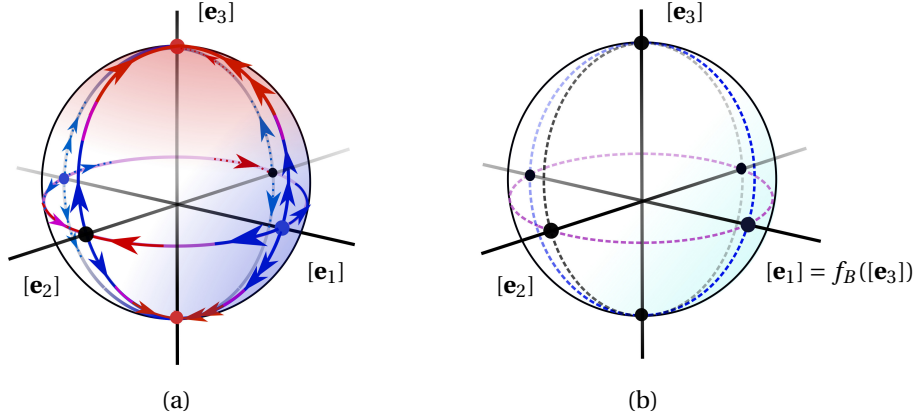
\begin{figure}[h]
\centering
\begin{overpic}[scale=.34,
%grid,tics=5
]{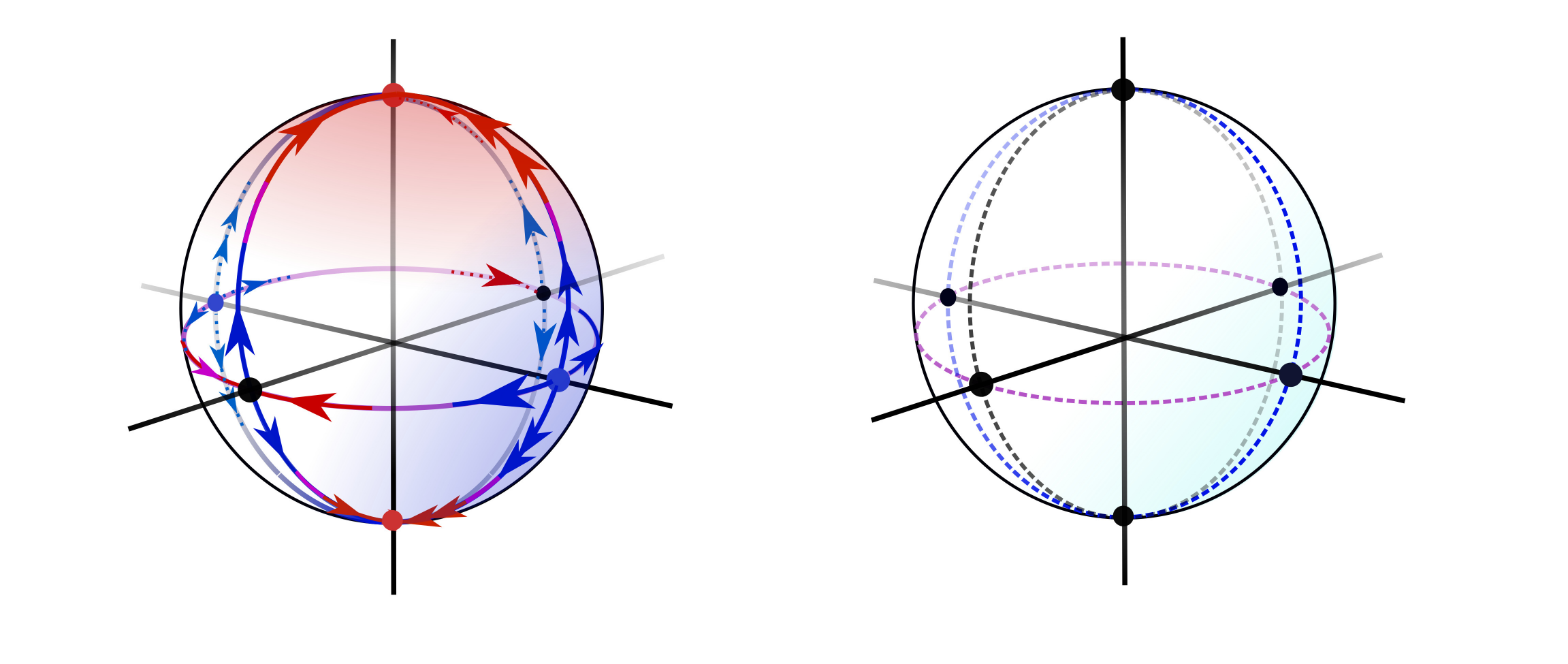}
           \put(27,39){$[\ve_3]$}
           \put(38,13){$[\ve_1]$}
           \put(9.5,11.5){$[\ve_2]$}
           \put(23.5,0){(a)}  
     \put(74,39){$[\ve_3]$}
      \put(84,13){$[\ve_1]=f_B([\ve_3])$}
      \put(57,11.5){$[\ve_2]$}   
      \put(70.2,0){(b)}       
\end{overpic}
\caption{(a) Dynamics of $f_A$. (b) Dynamics of $f_B$}
\label{fig:Cocliclos}
\end{figure}

Letting
 $f_0\eqdef f_A$ and $f_1\eqdef f_B$,  we consider the skew-product as in 
\eqref{e.toyzinho}
$$
F\colon \Sigma_2 \times 
 \mathbb{P}(\mathbb{R}^3) \to   \Sigma_2 \times 
 \mathbb{P}(\mathbb{R}^3), \quad (\underline{\ti}, x) = (\sigma(\,\underline{\ti}\,), f_{i_0} (x)), 
 \quad  \mbox{with}\quad \underline{\ti} =(\ldots  i_{-1}.i_0i_1\ldots). 
 $$
Comparing with Section~\ref{ss.toynonescaping},
 here $[\ve_3]$ and $[\ve_1]$ play the roles of $P$ and $Q$, 
respectively.
As in~\eqref{e.hoy},
we consider the ``saddles''   $\mathbf{E}_1=(0^\mathbb{Z}, [\ve_1])$
and $\mathbf{E}_3=(0^\mathbb{Z}, [\ve_3])$ with different "$\ut$-indices''. Note that $B (\ve_3)=\ve_1$ 
implies $f_1 ([\ve_3]) = [\ve_1]$ (as condition (d) in Section~\ref{ss.toynonescaping}) and  the global structure implies that $W^s ([\ve_3],f_0)$ intersects
$W^u([\ve_1], f_0)$, providing a cycle of 
co-index two between  $\mathbf{E}_1$ and  $\mathbf{E}_3$ for $F$ which is of non-escaping type.

The definition of $f_0$ implies the existence of the strong stable bundle of $[\ve_3]$ and 
the strong unstable bundle of $[\ve_1]$. We also have a dominated splittings and the corresponding strong manifolds,  we now provide  additional information about these objects.

Denote by $[\ve_i, \ve_j]$ the equivalence classes  in $\mathbb{P}(\mathbb{R}^3)$
of the  vectors in the plane  generated by $\ve_i$ and $\ve_j$, $i,j \in \{1,2,3\}$, $i<j$.
The $A$-invariance of these planes
and a direct calculation implies that both the strong stable bundle of $[\ve_3]$  and 
the strong unstable bundle
$[\ve_1]$ for $f_0$ are tangent to
   $[\ve_1,\ve_3]$. Thus we get the inclusion
 $$
\big( [\ve_1, \ve_3] \setminus \{ [\ve_1], [\ve_2]\} \big) \subset   W^\ss   ([\ve_3], f_0) \cap  W^\uu ([\ve_1], f_0).
$$ 
% 
% $W^\st ([\ve_3], f_0) \cap W^\ut ([\ve_1], f_0)$,  
% $W^\st ([\ve_3], f_0) \cap W^\ut ([\ve_2], f_0)$,
% and 
%  $W^\st ([\ve_2], f_0) \cap W^\ut ([\ve_1], f_0)$.
%  A direct calculation gives that the intersection 
%   $W^\st ([\ve_3], f_0) \cap W^\ut ([\ve_1], f_0)$ is simultaneously contained in 
%   $W^\ss   ([\ve_3], f_A) \cap  W^\uu ([\ve_1], f_A)$, which is a condition different from 
This  condition is different from (NE2) in the definition of non-escaping cycle and also 
from the one in the toy model in Section~\ref{ss.toynonescaping}, where the strong invariant manifolds are different.
This difficulty is bypassed using the map $B$.
From  \eqref{e.cyeigen}, we get that
   $$
   B([\ve_1, \ve_3]) \ne [\ve_1, \ve_3],
   $$
 which means that the $f_1$ maps $W^\uu ([\ve_1],f_0)$ outside of $W^\ss ([\ve_3],f_0)$.
The fact that $f^{-1}$ maps $W^\ss ([\ve_3], f_0)$ disjointly from $W^\uu ([\ve_1], f_0)$ also follows from
\eqref{e.cyeigen}. 

The fact that $Df_1$ maps the dominated splitting of $[\ve_3]$ into a splitting that
is in general position with respect to the one of $[\ve_1]$ also follows from \eqref{e.cyeigen}. 
Obtaining a similar condition to (e) in Section~\ref{ss.toynonescaping}. This completes our sketch.

\begin{remark}
Our construction provides a small skew-product perturbation with robust cycles of co-index two and one. 
However, this perturbation  it is not associated to a cocycle. A general question is if these perturbations can be obtained associated to a cocycle. Answering this question is beyond the scope of our paper.
\end{remark}

\bibliographystyle{siam}

\end{document}